\documentclass[11pt]{amsart}

\usepackage{amsmath, amssymb, amscd, mathrsfs, url, pinlabel,setspace,hyperref,verbatim}
\usepackage[margin=1.25in,marginparwidth=1in,centering,letterpaper,dvips]{geometry}
\usepackage{color,dcpic,latexsym,graphicx,epstopdf,comment}
\usepackage[all]{xy}
\usepackage{tikz,pgfplots}
\usepackage{tikz-cd,circuitikz}
\usepackage{enumitem,upquote,tabularx,textcomp,setspace,color}
\usepackage{makecell}
\usepackage{caption} 
\xyoption{rotate}

\newtheorem*{rep@theorem}{\rep@title}
\newcommand{\newreptheorem}[2]{%
\newenvironment{rep#1}[1]{%
 \def\rep@title{#2 \ref{##1}}%
 \begin{rep@theorem}}%
 {\end{rep@theorem}}}
\makeatother

\newtheorem {theorem}{Theorem}
\newreptheorem{theorem}{Theorem}
\newtheorem {lemma}[theorem]{Lemma}
\newtheorem {proposition}[theorem]{Proposition}
\newtheorem {corollary}[theorem]{Corollary}

\newtheorem {question}[theorem]{Question}

\numberwithin{equation}{section}
\numberwithin{theorem}{section}

\theoremstyle{definition}
\newtheorem{definition}[theorem]{Definition}
\newtheorem{construction}[theorem]{Construction}
\newtheorem{data}[theorem]{Data}
\newtheorem{notation}[theorem]{Notation}

\newtheorem{remark}[theorem]{Remark}

\newtheorem{example}[theorem]{Example}

\newlist{pcases}{enumerate}{1}
\setlist[pcases]{
  label=\bf{Case~\arabic*:}\protect\thiscase.~,
  ref=\arabic*,
  align=left,
  labelsep=0pt,
  leftmargin=0pt,
  labelwidth=0pt,
  parsep=0pt
}
\newcommand{\case}[1][]{%
  \if\relax\detokenize{#1}\relax
    \def\thiscase{}%
  \else
    \def\thiscase{~#1}%
  \fi
  \item
}

\DeclareFontFamily{U}{mathx}{\hyphenchar\font45}
\DeclareFontShape{U}{mathx}{m}{n}{
      <5> <6> <7> <8> <9> <10>
      <10.95> <12> <14.4> <17.28> <20.74> <24.88>
      mathx10
      }{}
\DeclareSymbolFont{mathx}{U}{mathx}{m}{n}
\DeclareFontSubstitution{U}{mathx}{m}{n}
\DeclareMathAccent{\widecheck}{0}{mathx}{"71}

\usetikzlibrary{calc,intersections}
\tikzset{every picture/.style=thick}
\tikzset{link/.style = { white, double = black, line width = 1.75pt, double distance = 1.25pt, looseness=1.75 }}
\tikzset{crossing/.style = {draw, circle, dotted, minimum size=0.5cm, inner sep=0, outer sep=0}}
\pgfplotsset{compat=1.12}

\usepackage{extarrows}
\usepackage{mathabx}
\usepackage{amsbsy}
\usepackage{appendix}
\usepackage{float}
\usepackage{subfigure}
\usepackage[numbers,sort&compress]{natbib}
\usepackage{amsthm}
\usepackage{geometry}
\usepackage{fancyhdr}%Heads
\usepackage{overpic}%pictures with text
\usepackage{mathrsfs}
\usepackage{colonequals}
\usepackage{microtype}
\usepackage{diagbox}
\usepackage{tabularx}
\usepackage{multirow}
\usepackage{tocvsec2}
\usepackage{mathtools}
\usepackage[all]{xy}
\usepackage{tikz-cd}
\usepackage{subcaption}

\newcommand{\bpf}{\begin{proof}}
\newcommand{\epf}{\end{proof}}
\newcommand{\bthm}{\begin{theorem}}
\newcommand{\ethm}{\end{theorem}}
\newcommand{\bprop}{\begin{proposition}}
\newcommand{\eprop}{\end{proposition}}
\newcommand{\bcor}{\begin{corollary}}
\newcommand{\ecor}{\end{corollary}}
\newcommand{\blem}{\begin{lemma}}
\newcommand{\elem}{\end{lemma}}
\newcommand{\bdefn}{\begin{definition}}
\newcommand{\edefn}{\end{definition}}
\newcommand{\bcons}{\begin{construction}}
\newcommand{\econs}{\end{construction}}
\newcommand{\bdata}{\begin{data}}
\newcommand{\edata}{\end{data}}
\newcommand{\bexmp}{\begin{example}}
\newcommand{\eexmp}{\end{example}}
\newcommand{\brem}{\begin{remark}}
\newcommand{\erem}{\end{remark}}
\newcommand{\bnot}{\begin{notation}}
\newcommand{\enot}{\end{notation}}
\newcommand{\benu}{\begin{enumerate}}
\newcommand{\benum}{\begin{enumerate}[leftmargin=*]}
\newcommand{\eenu}{\end{enumerate}}
\newcommand{\beq}{\begin{equation}}
\newcommand{\eeq}{\end{equation}}

\definecolor{lygreen}{HTML}{016646}
\usepackage[export]{adjustbox}
\title{Knot Floer complexes of generalized L-space knots}

\author{Shengyu Zou}
\address{Academy of Mathematics and Systems Science\\Chinese Academy of Sciences}
\email{syzou@amss.ac.cn}

\makeatletter

\begin{document}

\begin{abstract}
In this paper, we study height 4 knots, which generalize L-space knots and almost L-space knots studied by Ozsváth--Szabó, Baldwin--Sivek, and Binns.
A knot $K\subset S^3$ is called a height $m$ knot if $\operatorname{rank} \widehat{HF}(S^3_n(K))-n=m$ for all integers $n\geq 2g(K)-1$, where $\widehat{HF}$ is the hat version of Heegaard Floer homology introduced by Ozsváth--Szabó. In particular, height 0 and height 2 knots are exactly L-space
knots and almost L-space knots, respectively. We prove that the knot Floer complex $CFK^\infty(S^3,K)$ of a height 4 knot $K$ admits a basis that is horizontally and vertically simplified--an hv-basis for short. Moreover, we show that under such an hv-basis, the horizontal and vertical differentials fall into six cases, generalizing the results for L-space knots and almost L-space knots. The hv-basis of the knot Floer complex of $K$ can determine, via an explicit algorithm, the immersed curve of Hanselman-Rasmussen-Watson for the knot complement of $K$. Consequently, our results facilitate further exploration of computations of $\widehat{HF}$ for the gluing of two height 4 knot complements.
\end{abstract}

\maketitle
\section{Introduction}\label{section:introduction}

Knot Floer homology is a package of invariants for a knot $K$ in $S^3$, defined independently by Ozsv\'{a}th--Szab\'{o} \cite{OSknotinvar} and Rasmussen \cite{RasFloerhomoknotcomplement}. In this paper, we follow Ozsv\'{a}th--Szab\'{o}'s notation and consider $CFK^\infty(K) = CFK^\infty(S^3,K)$, which is a $\mathbb{Z}\oplus\mathbb{Z}$-filtered chain complex over $\mathbb{F}_2[U,U^{-1}]$. Equivalently, it is generated over $\mathbb{F}_2$ by elements of the form $[x,i,j]$ with $A(x)-j+i=0$, where $A(x)$ is the Alexander filtration level of $x$. The filtered chain homotopy type of $CFK^\infty(K)$ is a topological invariant of the oriented knot $K$, independent of the choice of an admissible, marked Heegaard diagram of $K$.

Heegaard Floer homology is a package of invariants for 3-manifolds due to Ozsv\'{a}th--Szab\'{o} \cite{OSHolodisk}. We will be interested in the hat version, denoted by $\widehat{HF}(Y)$ for a (smooth, connected, oriented) closed 3-manifold $Y$, which is a finitely generated graded abelian group. For a rational homology 3-sphere $Y$, the Euler characteristic $\chi(\widehat{HF}(Y))$ equals $|H_1(Y;\mathbb{Z})|$, where $|\cdot|$ denotes the order of the group. Thus, we have the following rank inequality: 

\vspace{-1.5em}
\begin{equation}
\label{rank HF no less than |H_1|}
\operatorname{rank}
\widehat{HF}(Y) \geq |H_1(Y;\mathbb{Z})|
\end{equation}
\vspace{-1.8em}

An L-space is a rational homology 3-sphere for which inequality (1.1) is an equality. This notion was introduced by Ozsv\'{a}th--Szab\'{o} \cite{OSLenspacesurgeries} to study knots admitting lens space surgeries, which generalizes lens spaces. L-spaces are also related to taut foliations and left-orderability, a connection formulated as the famous L-space conjecture \cite{L-spaceconj}.

An L-space knot is a knot $K$ for which $S_n^3(K)$ is an L-space for some $n \in \mathbb{N}$ (compare to \cite[Definition 1.10]{BaldwinSivekCharacslopesfor5_2}). Indeed, from \cite[Proposition 9.5]{OSrationalsurgeries}, such $n$ must satisfy $n\geq 2g(K)-1$, where $g(K)$ is the Seifert genus of $K$. By \cite{OSLenspacesurgeries}, the knot Floer complex of an L-space knot is determined by the Alexander polynomial of the knot. We will extend the definitions of L-spaces and L-space knots.  

\begin{definition}\label{dfn: height m space and height m knots} Let $Y$ be a rational homology 3-sphere, and let $m\in \mathbb{N}$ be an even integer. We call $Y$ a \textit{height $m$ (generalized L-)space} if $\operatorname{rank} \widehat{HF}(Y) = |H_1(Y;\mathbb{Z})|+m$.

A knot $K$ in $S^3$ is called a \textit{height $m$ knot} if there exists an integer $n \geq 2 g(K)-1$ such that $\operatorname{rank}\widehat{HF}(S_n^3(K)) = n+m$, i.e., $S^3_n(K)$ is a height $m$ space.
\end{definition}

\begin{remark}
We note that $m$ cannot be an odd integer, since $\chi(\widehat{HF}(Y))=|H_1(Y;\mathbb{Z})|$. It follows from \cite[Proposition 9.5]{OSrationalsurgeries} and \cite[Proposition 10.1]{BaldwinSivekFramedinstantonhomologyandconcordanceI} that if $K$ is a height $m$ knot, then for any integer $n \geq 2g(K)-1$, the surgery manifold $S^3_n(K)$ is a height $m$ space.

Height 0 spaces are exactly L-spaces. Height 2 spaces are equivalent to almost L-spaces studied by Baldwin--Sivek \cite{BaldwinSivekCharacslopesfor5_2} and Binns \cite{Binns}. In particular, Binns \cite{Binns} classified the filtered chain homotopy type of $CFK^\infty(K)$ for an almost L-space knot $K$.
\end{remark}

In this paper, we prove that for a height 4 knot $K$, the complex $CFK^\infty(K)$ always admits a certain simplified basis, called a \textit{horizontally and vertically simplified basis}, or an \textit{hv-basis} for short. A detailed description of hv-basis is given in Section \ref{subsection: Simplified basis}. Roughly speaking, an hv-basis is a filtered basis $\{\xi_i\}$ where the non-diagonal differentials---i.e. the horizontal and vertical differentials---are relatively simplified. More precisely, for each basis element $\xi_i$, there is at most one horizontal differential relating $\xi_i$, either starting or ending at $\xi_i$; and similarly for the vertical differentials.

Figure \ref{figure:trefoil and figure-eight CFK} illustrates $CFK^\infty(K)$ for the right-handed trefoil and the figure-eight knot, where the generators in the $i=0$ column form an hv-basis. The coordinates $i$ and $j$ correspond to the $U$-filtration level and the Alexander filtration level, respectively.

\begingroup
\setlength{\intextsep}{3pt}
\begin{figure}[H]
\centering
\setlength{\abovecaptionskip}{3pt}
\begin{minipage}{0.45\textwidth}
\centering
\includegraphics[scale=0.8]{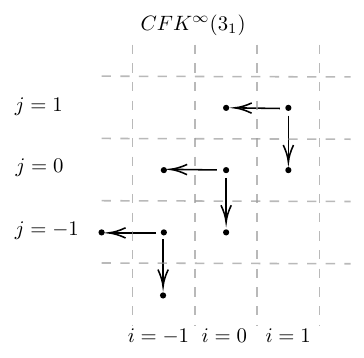}
\end{minipage}\hfill
\begin{minipage}{0.45\textwidth}
\centering
\includegraphics[scale=0.8]{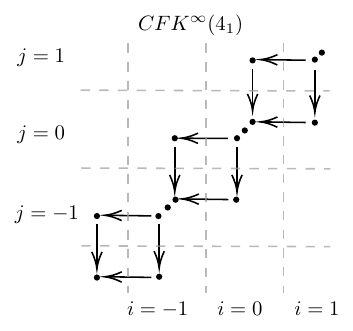}
\end{minipage}
\caption{Two examples of knot Floer complexes}
\label{figure:trefoil and figure-eight CFK}
\end{figure}
\endgroup

% \vspace{-0.8em}
% \begin{figure}[H]
%     \begin{minipage}{0.45\textwidth}
%         \centering
%         \includegraphics[scale=0.8]{fig/trefoil_CFK.pdf}
%     \end{minipage}
%     \begin{minipage}{0.45\textwidth}
%         \centering
%         \includegraphics[scale=0.8]{fig/figure_eight_CFK.pdf}
%     \end{minipage}
%     \caption{Two examples of knot Floer complexes}
%     \label{figure:trefoil and figure-eight CFK}
% \end{figure}
% \vspace{-1.8em}

\begin{remark}
These two examples are special because they have no diagonal differentials at all. For general examples, even under an hv-basis, there could exist some diagonal differentials, i.e., those that change two filtration levels simultaneously. The illustrations in the main theorem below only exhibit horizontal and vertical differentials; diagonal differentials are undetermined and omitted.
\end{remark}

We now state the main theorem of this paper. In Figure \ref{Figure: figure of main theorem}, the cases $(1), (2)$, and $(6)$ are classified into positive or negative variations depending on whether the differentials relating the generator at the corner of the twist end at it or begin at it.

\begin{theorem}\label{thm: classification of height 4}
If $K\subset S^3$ is a height 4 knot, then $CFK^{\infty}(K)$ admits an hv-basis (cf.~Section~\ref{subsection: Simplified basis}). Under this basis, the horizontal and vertical differentials take one of the following forms, as illustrated in Figure \ref{Figure: figure of main theorem}.

\begin{enumerate}[label=(\arabic*)]

\item A doubly-twisted staircase, either positive or negative.

\item A twisted staircase, either positive or negative, plus one box.

\item A staircase plus one twisted box.

\item A staircase plus two boxes, both located in the central position.

\item A staircase plus two separate boxes, located in symmetric positions.

\item A staircase with two separate twists, either positive or negative, located in symmetric positions. 

\end{enumerate}

\end{theorem}

\begingroup
\setlength{\intextsep}{1pt}
\begin{figure}[H]
	\centering
    \setlength{\abovecaptionskip}{3pt}
	% 第一行
	\begin{minipage}{0.24\linewidth}
		\centering
		\includegraphics[scale=0.7]{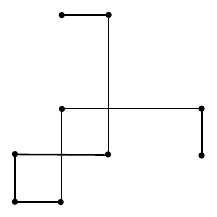}
		\caption*{$(1_{+})$}
	\end{minipage}
	\begin{minipage}{0.24\linewidth}
		\centering
		\includegraphics[scale=0.7]{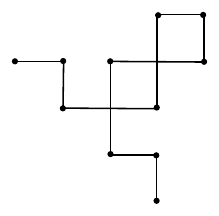}
		\caption*{$(1_{-})$}
	\end{minipage}
	\begin{minipage}{0.24\linewidth}
		\centering
		\includegraphics[scale=0.7]{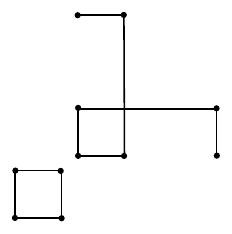}
		\caption*{$(2_{+})$}
	\end{minipage}
	\begin{minipage}{0.24\linewidth}
		\centering
		\includegraphics[scale=0.7]{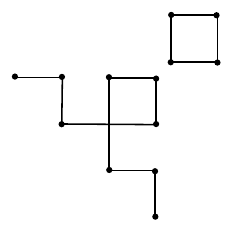}
		\caption*{$(2_{-})$}
	\end{minipage}
	
	% 第二行
	\begin{minipage}{0.48\linewidth}
		\centering
		\includegraphics[scale=0.7]{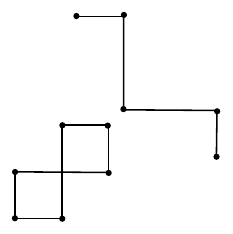}
		\caption*{(3)}
	\end{minipage}
	\begin{minipage}{0.48\linewidth}
		\centering
		\includegraphics[scale=0.7]{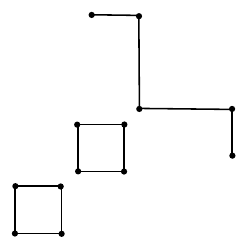}
		\caption*{(4)}
	\end{minipage}
	
	%第三行
	\begin{minipage}[b]{0.32\linewidth}
		\centering
		\includegraphics[scale=0.75]{fig/7-3}
		\caption*{(5)}
	\end{minipage}
	\begin{minipage}[b]{0.33\linewidth}
		\centering
		\includegraphics[scale=0.75]{fig/7-2}
		\caption*{$(6_{+})$}
	\end{minipage}
	\begin{minipage}[b]{0.33\linewidth}
		\centering
		\includegraphics[scale=0.75]{fig/7-8}
		\caption*{$(6_{-})$}
	\end{minipage}
    \caption{Six cases of the hv-basis for a height 4 knot}
    \label{Figure: figure of main theorem}
\end{figure}
\endgroup

The homological algebra lemmas and analytical framework employed in this paper build upon previous results in \cite{OSLenspacesurgeries} and \cite{Binns}; see the detailed discussion in Section \ref{subsection:proof methods}. I have generalized the key lemmas so that they apply to the height 4 case and, more generally, to arbitrary heights. Moreover, I have reorganized and reformulated the method of discussion used in \cite{Binns} into a systematic \textit{snake method}, establishing a foundational framework for treating knots of higher height.

The main innovations are twofold.  
First, I introduce the cancellation lemma to compute the homology of subquotient complexes, which significantly accelerates computations. This advancement makes it feasible to carry out the classification of differentials at height 4 and remains effective for even higher heights.  
Second, I revisit the origins of the homological algebra approach used in earlier work and directly impose constraints on possible differentials in the bent complexes via conditions on the dimensions of Heegaard Floer homology. This allows us to systematically rule out many pathological cases.

We compare our results with those for height 0 and 2 knots. Figure \ref{figure: hv-basis for knots of height 0 and 2} presents several examples.

\begingroup
\setlength{\intextsep}{2pt}
\begin{figure}[H]
    \centering
    \setlength{\abovecaptionskip}{3pt}
    \hspace*{0.02\textwidth}%
    \begin{minipage}[t]{0.28\textwidth}
        \centering
        \includegraphics[width=\textwidth, height=2.5cm, keepaspectratio]{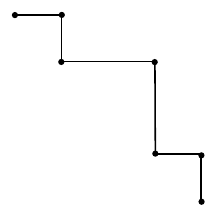}
        \par\vspace{0pt}
        \makebox[\linewidth][c]{\small a staircase}
    \end{minipage}%
    \hspace{0.04\textwidth}%
    \begin{minipage}[t]{0.28\textwidth}
        \centering
        \includegraphics[width=\textwidth, height=2.5cm, keepaspectratio]{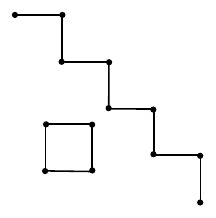}
        \par\vspace{0pt}
        \makebox[\linewidth][c]{\small a staircase plus one box}
    \end{minipage}%
    \hspace{0.04\textwidth}%
    \begin{minipage}[t]{0.28\textwidth}
        \centering
        \includegraphics[width=\textwidth, height=2.5cm, keepaspectratio]{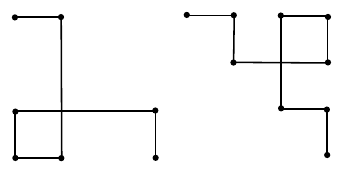}
       \par\vspace{0pt}
       \makebox[\linewidth][c]{\small twisted staircases}
    \end{minipage}%
    \hspace*{0.06\textwidth}
    \caption{The hv-basis for height 0 and height 2 knots}
    \label{figure: hv-basis for knots of height 0 and 2}
\end{figure}
\endgroup

For an L-space knot, there exists an hv-basis for its knot Floer complex, which must be a staircase \cite{OSLenspacesurgeries}. Given an almost L-space knot, one can always find an hv-basis, which is either a staircase plus one box, or a twisted staircase (which can be positive or negative) \cite{Binns}. 

% Need some applications plus remarks, 3~4 pages.

\begin{remark}
It is worth noting that the diagonal differentials can be completely determined for L-space knots and almost L-space knots; i.e., the hv-basis together with the horizontal and vertical differentials shown in Figure \ref{figure: hv-basis for knots of height 0 and 2} generates the whole $CFK^\infty(K)$ over $\mathbb{F}_2[U,U^{-1}]$ (two diagonal differentials shall be added in the twisted staircases cases, see \cite{Binns}). We expect that the diagonal differentials can also be determined for height 4 knots. The $1\times 1$ square component can be completely disconnected from the rest of the complex under specific conditions, see \cite[Lemma 4.12]{DavidPopKnotsoflowFloerwidth} for a discussion. A thorough solution may be somewhat more involved, and we leave this to future work. 

Neglecting the differences between positiveness and negativeness, the six cases differ according to the conditions on subquotient complexes $X_p$, except $(3)$ and $(4)$. See the table at the beginning of Section \ref{subsection: Dimension and Maslov grading}. 

In Sections \ref{subsection:further exploration} and \ref{subsection: Immersed curves}, we describe the connection between an hv-basis and the immersed curve of the knot complement introduced by Hanselman--Rasmussen--Watson \cite{HRWborderedHeegaardFloerimmersedcurves, HRWHeegaardFloerhomologyformanifoldswithtorusboundaryproperties}, which indicates that the hv-basis is sufficiently useful for many applications (e.g., computing $\widehat{HF}$ of closed 3-manifolds obtained by gluing two complement spaces of height 4 knots).
\end{remark}

\subsection{Proof methods}\label{subsection:proof methods}

The proof strategy of Theorem \ref{thm: classification of height 4} follows ideas similar to those for L-space knots and almost L-space knots in \cite{Binns,OSLenspacesurgeries}. First, we apply the large surgery formula and the height 4 condition to obtain constraints on the subquotient complexes $A_h$ of $CFK^{\infty}(K)$ for $h\in \mathbb{Z}$, where $A_h$ is usually called a \textit{bent complex} \cite[Definition 2.4]{RRL-spaceinterval}. More precisely, we obtain

\vspace{-1.5em}
\begin{equation}
    \sum_{h\in\mathbb{Z}}(\operatorname{rank}H_*(A_h)-1)=4.
\end{equation}
\vspace{-1.0em}

By Spin$^c$ conjugation, we have $H_*(A_h) \cong H_*(A_{-h})$. Hence $\operatorname{rank}H_*(A_h)$ must be of the form 

\vspace{-1.5em}
\begin{equation}
    1, \dots , 1, 5, 1, \dots , 1 \quad \text{ or } \quad 1 , \dots , 1 , 3 , 1 , \dots , 1 , 3 , 1 , \dots , 1. 
\end{equation}
\vspace{-1.2em}

We divide height 4 knots into the $(1,5,1)$ type and the $(1,3,1,3,1)$ type. For either type, we first establish some algebraic lemmas (cf.\ Lemmas \ref{lemma: 1 and 2k+1 then k or k+1} and \ref{lemma: 2k+1 and 1 then k or k+1}), which generalize those in \cite{Binns, OSLenspacesurgeries}. We then apply those lemmas to determine $H_*(X_p)$ and $H_*(Y_p)$, where $X_p$ is a certain subcomplex of $A_p$ and $Y_p$ is the associated quotient complex. Based on these preparations, we determine the dimension and Maslov grading at each Alexander filtration level, i.e., the entire $\widehat{HFK}(K)$.

Note that if $C=CFK^\infty(K)$ is \textit{reduced}---i.e., no differentials preserve the filtration level, then $C\{i=0\}$ is generated by $\widehat{HFK}(K)$. Precisely, the generators of $\widehat{HFK}(K)$ form a filtered basis, and we apply linear filtered basis changes to turn it into an hv-basis. We first determine the vertical differentials in $C\{i=0\}$ under the constraints of Maslov grading and the homologies of $Y_p$, proceeding from the lowest to highest Alexander filtration level. We call this method \textit{grid climbing} or the \textit{snake method}. As a result, we obtain a vertically simplified basis. Then we apply grid climbing again to determine the horizontal differentials and make the vertically simplified basis also horizontally simplified. 

Compared to \cite{OSLenspacesurgeries} and \cite{Binns}, there are more possible differentials connecting basis elements. We introduce new techniques for the $(1,5,1)$ type, which are also used extensively for the $(1,3,1,3,1)$ type. Precisely, we encounter cases where $\operatorname{rank} \widehat{HFK}(K,0)=5$ and $\operatorname{rank}$ $\widehat{HFK}(K,\pm 1)=3$, for which the homologies of the subquotient complexes are more difficult to compute. We apply the cancellation lemma (cf. Lemma \ref{lemma: cancellation lemma}) to make these computations more convenient. In the process of determining and simplifying horizontal differentials, there are cases where $H_*(X_p)$ and $H_*(Y_p)$ alone are insufficient to guarantee the existence of an hv-basis. We must refer again to the condition $H_*(A_0)=\mathbb{F}^5$ to constrain the horizontal differentials in $A_0$, which helps rule out many pathological cases.

\subsection{Further exploration}\label{subsection:further exploration}

The immersed curve package was developed by Hanselman--Rasmussen--Watson recently\cite{HRWborderedHeegaardFloerimmersedcurves, HRWHeegaardFloerhomologyformanifoldswithtorusboundaryproperties}. This package provides a geometric approach to encoding the structure
of the bordered Floer invariant $\widehat{CFD}(M)$ for a compact 3-manifold $M$ with torus boundary. Meanwhile, the information carried by the immersed curve of the knot complement is equivalent to that in the quotient complex $CFK_{\mathbb{F}[U,V]}(K)/(UV=0)$. We will describe the procedures for determining the immersed curve of the knot complement from the hv-basis of $CFK^\infty(K)$ in Section \ref{subsection: Immersed curves}. Potentially, the techniques in this paper could be translated into the language of immersed curves to obtain a more illustrative proof.

The pairing theorem of immersed curves \cite[Theorem 1.2]{HRWborderedHeegaardFloerimmersedcurves} provides a convenient way to compute $\operatorname{rank}\widehat{HF}(M_1\cup_{f} M_2)$, where $M_1\cup_{f}M_2$ is the 3-manifold obtained by gluing two 3-manifolds with torus boundary via an attaching map $f$. Based on Theorem \ref{thm: classification of height 4} and the above translation from an hv-basis to an immersed curve, we can compute the case where $M_1$ and $M_2$ are complements of height 4 knots.

With the help of immersed curves, it has been shown that if a closed 3-manifold $Y$ contains an essential torus, then $\operatorname{rank} \widehat{HF}(Y) \geq 5$, see \cite[Theorem 1.8, Theorem 7.20]{HRWborderedHeegaardFloerimmersedcurves}. The results of this paper may enable us to explore more examples of toroidal spaces of minimal or next-to-minimal rank (5 or 7). 

\begingroup
\setlength{\intextsep}{0pt}
\begin{figure}[H]
        \centering
        \setlength{\abovecaptionskip}{1pt}
        \includegraphics[scale=0.8]{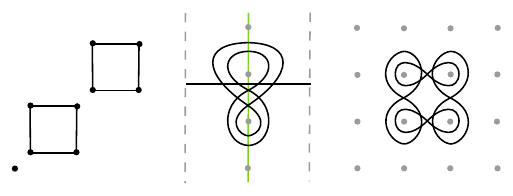}
    \caption{}
    \label{6_1 and 6_1}
\end{figure}
\endgroup

For example, the knot $6_1$ is a height 4 knot. Let $M_1$ and $M_2$ be the complements of $6_1$, with meridian-longitude framings $\mu_i,\lambda_i$. Let $h:\partial M_2\rightarrow \partial M_1$ be the attaching map such that $h(\mu_2)=\lambda_1$ and $h(\lambda_2)=\mu_1$. 
According to the pairing theorem of immersed curves, $\operatorname{rank}\widehat{HF}(M_1\cup_{h}M_2)=17$ (see Figure \ref{6_1 and 6_1} for an illustration). This is far larger than $7$; we may try a different gluing map to reduce it. 

The classification process can be further extended to higher heights. For instance, in the case of height 6, there are two candidate types, namely the $(1,7,1)$ type and the $(1,3,1,3,1,3,1)$ type. However, when seven generators are present, the number of possible differential configurations typically increases dramatically, significantly raising the computational complexity of a case-by-case classification. It remains unknown whether the current techniques and framework are sufficient to guarantee the existence of an $hv$-basis.

\subsection{Some questions}\label{subsection: some questions}

The classification of complexes in Theorem \ref{thm: classification of height 4} is purely algebraic. One may ask whether each complex among these types can be realized as $CFK^\infty(K)$ for some knot $K$ in $S^3$. As a partial result, the cases $(1_{+})$ and $(1_{-})$ can be realized as cables of L-space knots by computation techniques in \cite{Cablingintermsofimmersedcurves}; see also the discussion in \cite{Binns}. We therefore propose:

\begin{question}
    Can all cases of hv-basis in Theorem \ref{thm: classification of height 4} be realized by knots in $S^3$?
\end{question}

The $CFK^\infty(K)$ of $(3,2)$-cable of right-handed trefoil has an hv-basis of the case $(6_{+})$. The knot complexes of $6_1$ and $7_2$ fall into case $(4)$, while $6_2$ and $7_3$ fall into case $(5)$. To date, the author has not found knots in $S^3$ that realize cases $(1_{+}), (1_{-}), (2_{+}), (2_{-}), (3)$ and $(6_{-})$ up to now.

This question can be formulated more precisely if we take the length of the differentials into consideration. For example, consider a doubly-twisted staircase. Are there restrictions on the length of the steps? For L-space knots, the length of the bottommost step must be 1, and there exist a series of constraints among the lengths of other steps; see \cite{HeddenWatsonOnthegeographyandbotanyofknotFloerhomo, KrcatovichDavidArestrictiononAlexPolyofLspaceknots} for some concrete descriptions. 

% Main techniques

% Together with work before, these results reveal how the constraints from Heegaard Floer perspectives finally determine the possible knot Floer complex types, providing a framework for understanding the pattern of changes of knot Floer complexes as their related Floer-theoretic invariants become more complicated. Combined with results on topological properties of knots, it is promising to obtain further insights into the relationship between the topology of knots and their Floer-theoretic invariants.
 
Another direction worthy of attention and exploration is a topological and geometric characterization of height 4 spaces and height 4 knots. The L-space conjecture posits an equivalent characterization of L-spaces in terms of both an orderability condition on their fundamental groups and whether they admit taut foliations. Binns \cite{Binns} asked whether there is a version of the L-space conjecture for almost L-spaces. Similarly, we propose: 

\begin{question}
    Is there a version of the L-space conjecture for height 4 spaces?
\end{question}

To the best of the author's knowledge, there is currently no compelling guiding conjecture as to which knots are L-space knots. Yet efforts have been made to catalog which knots among $(1,1)$ knots are L-space knots. Greene--Lewallen--Vafaee \cite{11Lspaceknots} showed that a $(1,1)$ diagram represents an L-space knot if and only if it is coherent. Binns--Zhou \cite{11almostLspaceknots} showed that a standard $(1,1)$ diagram represents an almost L-space knot if and only if it is strongly almost coherent. Building on this work, it would be interesting to explore how to distinguish height 4 knots within $(1,1)$ knots.

\subsection{Organization}\label{subsection:organization}

This paper is organized as follows. 

% Section \ref{section:introduction} is an introduction to necessary concepts and our main theorem, including a brief discussion about the relationships between our results and previous work. After the statement of the main theorem, we explain the proof method, propose some questions and take a glance at potential future work. 

We provide a brief review of background knowledge of Heegaard Floer homology in Section
\ref{subsection: Heegaard Floer homology}. Then we give detailed definitions and illustrations of the primary algebraic concepts that will appear in later proofs, primarily concerning horizontally and vertically simplified bases(see Section \ref{subsection: Simplified basis}). The knot Floer complex over $\mathbb{F}[U,V]$ is discussed in Section \ref{subsection:Knot Floer complex over F[U,V]}, where we will describe the procedure for transforming $CFK^\infty(K)$ into $CFK_{\mathbb{F}[U,V]}(K)$. In Section \ref{subsection: Some algebraic lemmas}, we state two technical lemmas that will simplify subsequent computations. In Section \ref{subsection: Immersed curves} we describe the process of obtaining the immersed curve of the complement space of $K$ from an hv-basis of $CFK^\infty(K)$, through the complex $CFK_{\mathbb{F}[U,V]}(K)/(UV=0)$. For the hv-basis components appearing in our main theorem, we present their immersed curves as illustrations of this process.

Section \ref{section:Height 4 knots} serves as a common preparation for Sections \ref{section:(1,5,1) type} and \ref{section: (1,3,1,3,1) type}. We isolate the arguments shared by the following two sections, making them more self-contained and rendering the overall proof structure of the paper more streamlined.

Sections \ref{section:(1,5,1) type} and \ref{section: (1,3,1,3,1) type} contain detailed discussions about various cases. In these two sections, we explain in detail how to obtain an hv-basis in different cases. The results in different cases are collected at the beginning of each section. Readers interested in the details of the proof can look up the corresponding case. At the end of either section, we present results on those special cases that are not covered in our proof of the generic cases. We list and present the classification results in full detail to ensure completeness of our result and facilitate the reader's reference. The underlying ideas and methods for these special cases are essentially the same. To reduce the length of the paper, only the essential and different points are highlighted.

% In Appendix \ref{section:tables of classification results}, 

\subsection{Acknowledgment}

This paper is modified from the author's phd thesis. 

I am deeply grateful to Professor Fan Ye for introducing me to this project and for his continuous guidance throughout its development. At every stage, he patiently pointed out the direction forward while encouraging me to persist working. Through our numerous discussions, I have received invaluable instruction in fundamental research methods and conceptual thinking, which has profoundly shaped my approach to tackling academic challenges. This journey has been an exceptionally meaningful training experience, both intellectually and personally. Beyond mathematics, I have also benefited greatly from his thoughtful advice and unwavering support in life—guidance that has been of immense significance to me during this crucial stage of my academic career. His attitude toward research and life, as well as his approach to problem-solving, has inspired me to face the road ahead and its challenges with a more positive mind.

Thanks for Professor Jiajun Wang for guiding me into the field of Heegaard Floer homology. He has created many conveniences for my doctoral research career, making my journey in exploring Heegaard Floer homology all the more enjoyable. His support and encouragement are indispensable during my doctoral studies.

Thanks my classmates Cheng Chang and Xiliu Yang for their companion and generous help, it is my great fortune to have had the opportunity to study and conduct research alongside them during my doctoral studies.

Thanks for the love from my family members. Their diligence and positive attitude toward life have continually inspired me to keep moving forward.

I am grateful to my friends, including Honghuai Fang, Jia Li, Yongsheng Wu, Shengxuan Zhou, Linsheng Wang and Zekun Chen for their companionship. Fang is a wise and cheerful friend, the time with him has offered me profound insights in many aspects of life. Kun’s exceptional mathematical talent has left a deep impression on me, revealing the depth of mathematics and filling me with even greater awe for the subject. The journey with Zhou and Wang has been an inspiration to me. It is an invaluable opportunity to engage in discussions and exchanges of ideas with them, through which I have observed the working styles of outstanding researchers. My attitude towards research has been profoundly influenced by them. Much of my time was spent with Jia and Yongsheng, and many challenges were faced and handled together, their companion shall bring more energy to my futural life.

\section{Preliminaries}\label{section: preliminaries}

We present the background knowledge necessary for this paper in this section. We begin by introducing some concepts in Heegaard Floer homology and knot Floer homology, including $\widehat{HF}(Y)$, Spin$^c$ structures, and the complex $CFK^\infty(S^3,K)$. We then state large surgery formula and the Euler characteristic property. In Section \ref{subsection: Simplified basis}, we give detailed algebraic formulation of $CFK^{-}(K)$ and the definition of an hv-basis. In Section \ref{subsection:Knot Floer complex over F[U,V]}, we describe $CFK_{\mathbb{F}[U,V]}(K)$ and the transformation from $CFK^\infty(K)$ to $CFK_{\mathbb{F}[U,V]}(K)$. In Section \ref{subsection: Some algebraic lemmas}, we present some algebraic technical lemmas that will be applied in later proofs. Finally, in Section \ref{subsection: Immersed curves}, we illustrate how to draw the immersed curve of the complement of $K$ from an hv-basis for $CFK^\infty(K)$.

\subsection{Heegaard Floer homology}\label{subsection: Heegaard Floer homology}

Let $Y$ be a closed 3-manifold. The hat version of 
Heegaard Floer homology, denoted $\widehat{HF}(-)$, is an invariant of 3-manifolds due to Ozsv\'{a}th--Szab\'{o}. The homology group $\widehat{HF}(Y)$ is defined through Heegaard diagrams, involving symplectic topology and analysis; see \cite{OSpropertiesandexample, OSHolodisk}. This invariant can be defined over integer coefficients, but we will only consider the case over $\mathbb{F}=\mathbb{Z} / 2\mathbb{Z}$ throughout this paper. 

The homology group $H_1(Y;\mathbb{Z})$ encountered in this paper will always be taken with integer coefficients, which we abbreviate as $H_1(Y)$. $\widehat{HF}(Y)$ splits as a direct sum over equivalence classes of non-vanishing vector fields on $Y$, called Spin$^c$ structures. Let $Y$ be a rational homology 3-sphere. There is a non-canonical bijection between the set of Spin$^c$ structures on $Y$ and $H_1(Y)$. Suppose $Y$ is obtained by performing $n$-surgery on a knot $K$ in $S^3$. Then $K$ determines a canonical bijection between $\operatorname{Spin}^c(Y)$ and $\mathbb{Z}/n\mathbb{Z} \cong H_1(Y)$; see \cite{OSIntesurgery} for detailed descriptions and \cite{Binns} for a brief introduction. 

% Dehn-surgery is a well studied operation on 3-manifolds. Let $r \in \mathbb{Q}$ be a rational number and $K$ be a knot in $S^3$. The $r$-surgery on $K$, denoted $S_r^3(K)$, is the manifold obtained by removing a tubular neighborhood of $K$ from $S^3$, then regluing it with a framing determined by $r$. The framing consists of a meridian and preferred longitude, following the right-hand rule, see detailed description in \cite{rolfsen2003knots}.

There is a version of Heegaard Floer homology for knots, defined independently in \cite{OSknotinvar} and \cite{RasFloerhomoknotcomplement}. A convenient introduction to knot Floer homology can be found in \cite{HomLectureHeegaardFloer}. The invariant used in this paper is denoted by $CFK^{\infty}(S^3, K)$, following the notation of \cite{OSknotinvar}. We call it the knot Floer complex and write it as $CFK^\infty(K)$. The complex $CFK^\infty(K)$ is a finitely generated chain complex of $\mathbb{F}[U,U^{-1}]$-modules. The two filtrations are the Alexander filtration and the $U$-filtration. The $U$ action lowers the Maslov grading by two and the Alexander filtration level by one. 

The subquotient complex generated by those generators of $U$-filtration level 0 is denoted by $\widehat{CFK}(K)$. The homology of $\widehat{CFK}(K)$, as the associated graded object with respect to the Alexander filtration, is denoted by $\widehat{HFK}(K)$.

For a homogeneous element in $CFK^\infty(K)$, we set its $U$-filtration level and Alexander filtration level as its horizontal coordinate and vertical coordinate, respectively. This enables us to present $CFK^\infty(K)$ in the plane. Below is an illustration of $CFK^\infty(3_1)$. The Maslov grading for a generator is marked as the number near it. The arrows denote the differentials between generators.

\begingroup
\setlength{\intextsep}{1pt}
\begin{figure}[H]
    \setlength{\abovecaptionskip}{1pt}
    \begin{minipage}[c]{0.48\textwidth}
        \centering
        \includegraphics[scale=0.9]{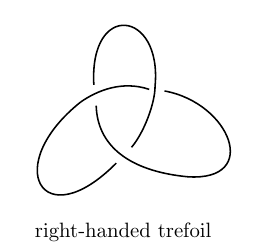}
    \end{minipage}
    \begin{minipage}[c]{0.48\textwidth}
        \centering
        \includegraphics[scale=0.75]{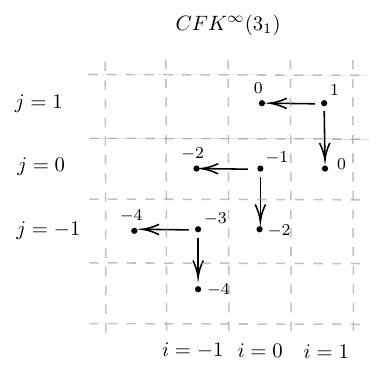}
    \end{minipage}	
    \caption{}
    \label{figure:CFK^infty of right-handed trefoil}
\end{figure}
\endgroup

We now state the large surgery formula, an isomorphism relating the Heegaard Floer homology of a surgery manifold to the bent complexes $A_m$.

% We now state an isomorphism, called large surgery formula, which connects Heegaard Floer homology of the surgery manifold to bent complexes $A_m$.

\begin{theorem}[{\cite[Theorem 4.4]{OSknotinvar}}] 
Given any knot $K \subset S^3$, there exists an integer $N$ such that for all $n \geqslant N$, there is an isomorphism of chain complexes

\vspace{-1.5em}
\begin{equation}
\label{theorem: large surgery formula}
\widehat{CF}(S_n^3(K),[m]) \cong CFK^\infty\{\max (i, j-m) = 0\} = A_m.
\end{equation}
\vspace{-1.8em}

\end{theorem}

The Euler characteristic of $\widehat{HF}(Y)$ satisfies the following equality.

\begin{theorem}[{\cite[Proposition 5.1]{OSpropertiesandexample}}] Suppose $\mathfrak{s}$ is an element of $\operatorname{Spin}^c(Y)$. The Euler characteristic of $\widehat{HF}(Y, \mathfrak{s})$ is given by:

\vspace{-1.5em}
\[
\chi(\widehat{HF}(Y,\mathfrak{s}))=
\begin{cases}
	1 & \text{if}\ b_1(Y)=0, \\
	0 & \text{if}\ b_1(Y)\neq 0\\
\end{cases}
\]
\vspace{-1.0em}

\label{theorem: Euler characteristic of Heegaard Floer homology}
\end{theorem}

\subsection{Simplified basis}\label{subsection: Simplified basis}

The main theorem of this paper states the existence of an hv-basis for $CFK^\infty(K)$ and classifies the possible cases. The term hv-basis is an abbreviation for horizontally and vertically simplified basis. We give a detailed algebraic formulation of it in this section. The material in this section is selected from \cite[Subsection 11.5]{LOTborderedHeeFloer}. 

\begin{definition}
    Let $\cdots\subset V_i\subset V_{i+1}\subset \cdots$ be an (ascending) filtration of a vector space $V$, so that $V={\bigcup}_{i=-\infty}^{\infty}V_i$. Assume that ${\bigcap}_{i=-\infty}^{\infty}V_i=0$. For $v\in V$, define $A(v)\coloneqq \operatorname{inf}\{i\in\mathbb{Z}\mid v\in V_i\}$; we call $A(v)$ the \textit{filtration level} of $v$. Let $\operatorname{gr}(V)$ denote the associated graded vector space, i.e.,
    \[
        \operatorname{gr}(V)\coloneqq \mathop{\oplus}\limits_{i=-\infty}^{\infty}\operatorname{gr}_i(V)
     \qquad   \operatorname{gr}_i(V)=V_i/V_{i-1}.
    \]

    Composing the projection $V_i\rightarrow V_i/V_{i-1}$ with the inclusion $V_i/V_{i-1}\rightarrow \operatorname{gr}(V)$ gives a map $\iota_i: V_i\rightarrow \operatorname{gr}(V)$. For $v\in V$ define $[v] \coloneqq \iota_{A(v)}(v)$.
\end{definition}

\begin{definition}
    A \textit{filtered basis} of a filtered vector space $V$ is a basis $\{v_1,\cdots, v_n\}$ for $V$ such that $\{[v_1],\cdots, [v_n]\}$ is a basis for $\operatorname{gr}(V)$. Similarly, we say that $\{v_1,\cdots,v_n\}$ is filtered linearly independent if $\{[v_1],\cdots,[v_n]\}$ is linearly independent in $\operatorname{gr}(V)$.
\end{definition}

\begin{remark}
    There is a filtered analogue of the basis extension theorem: if $\{v_1,\cdots, v_m\}$ is a filtered linearly independent set, then there exists a filtered basis for $V$ containing $\{v_1,\cdots,v_m\}$.
\end{remark}

Vertically simplified basis is originally defined on $CFK^{-}(K)$; see \cite{LOTborderedHeeFloer}. The definition can be given in a more general algebraic setting. Let $C$ be a finitely generated free $\mathbb{F}_2[U]$-module equipped with 

\begin{enumerate}[label=\textbullet]
    \item an integer-valued grading, which we refer to as the Maslov grading $M$

    \item an integer-valued ascending filtration, the Alexander filtration $A$, and

    \item a differential, denoted $\partial$
\end{enumerate}

Assume that these data satisfy the following compatibility conditions:

\begin{enumerate}[label=\textbullet]
    \item The differential $\partial$ drops the Maslov grading by one: for homogeneous $\xi\in C$, $M(\partial \xi) = M(\xi)-1$. 

    \item The differential $\partial$ respects the Alexander filtration: $A(\partial\xi) \leq A(\xi)$.

    \item Multiplication by $U$ is compatible with the Maslov grading in the sense that $M(U\cdot \xi) = M(\xi)-2$.

    \item Multiplication by $U$ drops the Alexander filtration by at least one in the sense that $A(U\cdot \xi) \leq A(\xi)-1$.
\end{enumerate}

Given such a $C$, we define $C^\infty \coloneqq C \otimes_{\mathbb{F}_2[U]} \mathbb{F}_2[U,U^{-1}]$. The complex $C^\infty$ has two filtrations: the Alexander filtration and the filtration by powers of $U^{-1}$. Precisely, the $s$-th filtered part of the second filtration consists of generators $U^{t}x$, where $t \ge -s$ and $x$ is a generator of $C$. 

We can illustrate $C^\infty$ graphically by picking an element $\xi_i$ in a filtered basis of $C$ over $\mathbb{F}_2[U]$ and plotting a generator of $C\otimes_{\mathbb{F}_2[U]} \mathbb{F}_2[U,U^{-1}]$ over $\mathbb{F}_2$ of the form $U^{-x}\cdot \xi_i$ with Alexander filtration level $y$ at the position $(x,y)$ in the plane. Then the differential of a generator at $(x,y)$ can be represented graphically by arrows connecting the point $(x,y)$ to the coordinates of other generators. These arrows necessarily point (non-strictly) to the left and downward; see Figure \ref{figure:CFK^infty of right-handed trefoil}.

\begin{definition}
    We call $C$ \textit{reduced} if $\partial$ strictly drops either the Alexander filtration or the $U$-filtration; i.e., for all generators $x$, we have generators $y_1,y_2$(may be zero) such that $\partial x = y_1 + U \cdot y_2$ where $A(y_1) < A(x)$.
\end{definition}

\begin{remark}
    Every complex $C$ described above is homotopy equivalent to a reduced one (cf.\ Lemma \ref{lemma: h or v basis exists}). Without loss of generality, we will assume throughout that the knot Floer complexes we deal with are reduced. The careful reader can further check that the basis change operations in the proofs in Sections \ref{section:(1,5,1) type} and \ref{section: (1,3,1,3,1) type} do not affect the property of being reduced.
\end{remark}

\begin{definition}
    We call a filtered basis $\{\xi_i\}$ over $\mathbb{F}_2[U]$ for $C$ \textit{vertically simplified} if for each basis vector $\xi_i$, either $\partial \xi_i \in U \cdot C$ or $\partial \xi_i \equiv \xi_{i+1} \pmod{U \cdot C}$. In the latter case, we say that there is a vertical arrow from $\xi_i$ to $\xi_{i+1}$. The length of the vertical arrow is the difference $A(\xi_i) - A(\xi_{i+1})$.
\end{definition}

\begin{remark}
    In a vertically simplified basis, there is at most one vertical arrow starting or ending at each basis element. The condition $\partial^2 = 0$ implies that no generator can serve as both the starting point and the endpoint of vertical differentials.
\end{remark}

\begin{definition}
    We call a filtered basis $\{\eta_i\}$ over $\mathbb{F}_2[U]$ for $C$ \textit{horizontally simplified} if for each basis vector $\eta_i$, either $A(\partial \eta_i) < A(\eta_i)$, or there exists an integer $m$ such that $\partial \eta_i = U^{m} \cdot \eta_{i+1} + \epsilon$, where $A(\eta_i) = A(U^{m} \cdot \eta_{i+1})$ and $A(\epsilon) < A(\eta_i)$. In the latter case, we say that there is a horizontal arrow from $\eta_i$ to $\eta_{i+1}$. The length of the horizontal arrow is the integer $m$.
\end{definition}

\begin{definition}
    A filtered basis $\{x_i\}$ for $C$ is called an \textit{hv-basis} if it is both horizontally and vertically simplified (up to reordering of the subscripts of the generators). 
\end{definition}

\begin{definition}
    Suppose $C^\infty = C \otimes_{\mathbb{F}_2[U]} \mathbb{F}_2[U,U^{-1}]$ as defined above. We say that \emph{$C^\infty$ has an hv-basis} if $C$ has an hv-basis. We adopt the convention that an hv-basis of $C^\infty$ is an hv-basis of the underlying $C$. 
\end{definition}

\begin{remark}
In this paper, $CFK^\infty(K) = CFK^{-}(K) \otimes_{\mathbb{F}_2[U]} \mathbb{F}_2[U, U^{-1}]$. An hv-basis of $CFK^\infty(K)$ is obtained from an hv-basis of $CFK^{-}(K)$.
\end{remark} 

%Examples_of_hv-basis, box, staircase, twisted staircase.

We now give definitions and illustrations of some examples of hv-bases.

\begin{definition}
    Suppose $C$ is a chain complex generated by $\{x_i \mid 1\leq i\leq 4\}$, with $A(x_1)-1 = A(x_2) = A(x_3) = A(x_4)+1$. Assume that there are horizontal differentials from $x_2$ to $x_1$ and from $x_4$ to $x_3$, and vertical differentials from $x_1$ to $x_3$ and from $x_2$ to $x_4$. Then we call $C^\infty = C \otimes \mathbb{F}[U, U^{-1}]$ a \textit{box}.
\end{definition}
    
\begin{definition}
    Suppose $C$ is a chain complex generated by $\{x_i \mid 1\leq i\leq 6\}$, with $A(x_1)-1 = A(x_2)-1 = A(x_3) = A(x_4) = A(x_5)+1 = A(x_6)+1$. Assume that there are 
    horizontal differentials from $x_3$ to $x_1$, from $x_5$ to $x_2$, and from $x_6$ to $x_4$, and vertical differentials from $x_3$ to $x_5$, from $x_1$ to $x_6$, and from $x_2$ to $x_4$. Then we call $C^\infty = C \otimes \mathbb{F}[U, U^{-1}]$ a \textit{twisted box}.
\end{definition}

\begin{definition}
    Suppose $C$ is a chain complex generated by $\{x_i \mid 0 \leq i \leq 2n\}$. Assume that for each $0\leq k\leq n-1$, there is a horizontal differential from $x_{2k+1}$ to $x_{2k}$ and a vertical differential from $x_{2k+1}$ to $x_{2k+2}$. Then we call $C^\infty = C \otimes \mathbb{F}[U, U^{-1}]$ a \textit{staircase}.
\end{definition}

\begin{definition}
    Suppose $C$ is a chain complex generated by $\{x_i \mid -2n\leq i\leq 2n\} \cup \{y_{-1},y_{1}\}$, where $n \geq 1$. Assume that for each $-n\leq k\leq n-1$ with $k\neq 0$, there is a horizontal differential from $x_{2k+1}$ to $x_{2k}$, together with horizontal differentials from $x_{1}$ to $y_{-1}$ and from $y_{1}$ to $x_{0}$ (the differential from $y_{1}$ to $x_{0}$ has length one). Also assume that, for each $-n\leq k\leq n-1$ with $k\neq -1$, there is a vertical differential from $x_{2k+1}$ to $x_{2k+2}$, together with vertical differentials from $x_{-1}$ to $y_{1}$ and from $y_{-1}$ to $x_{0}$ (the differential from $y_{-1}$ to $x_{0}$ has length one). Then we call $C^\infty = C \otimes \mathbb{F}[U, U^{-1}]$ a \textit{twisted staircase}.
\end{definition}

% \begingroup
% \setlength{\intextsep}{1pt}
% \begin{figure}[H]
%     \begin{minipage}{0.48\textwidth}
%         \centering
%         \includegraphics[scale=0.8]{fig/Examples_of_hv-basis-1.pdf}
%         \vspace{2.2em}
%         \caption*{A box and a twisted box}
%     \end{minipage}
%     \begin{minipage}{0.48\textwidth}
%         \centering
%         \setlength{\abovecaptionskip}{3pt}
%         \includegraphics[scale=0.8]{fig/Examples_of_hv-basis-2.pdf}
%         \caption*{A staircase}
%     \end{minipage}
% \end{figure}

% \begin{figure}[H]
%     \begin{minipage}{0.96\textwidth}
%         \centering
%         \setlength{\abovecaptionskip}{3pt}
%         \includegraphics[scale=0.7]{fig/Examples_of_hv-basis-3.pdf}
%         \caption*{A twisted staircase}
%     \end{minipage}
%     \vspace{-0.5em}
%     \caption{Examples of hv-bases}
%     \label{figure:examples_of_hv-bases}
% \end{figure}
% \endgroup

\begingroup
\setlength{\intextsep}{0pt}
\begin{figure}[H]
    \centering
    \begin{minipage}[t]{0.48\textwidth}
        \centering
        \parbox[c][3.2cm][c]{\linewidth}{\centering
            \includegraphics[width=\linewidth,height=3.2cm,keepaspectratio]{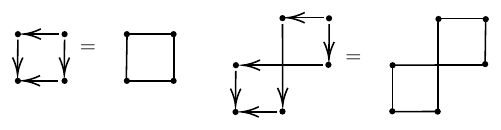}}
        \par\smallskip
        \makebox[\linewidth][c]{A box and a twisted box}
    \end{minipage}\hfill
    \begin{minipage}[t]{0.48\textwidth}
        \centering
        \parbox[c][3.2cm][c]{\linewidth}{\centering
            \includegraphics[width=\linewidth,height=3.2cm,keepaspectratio]{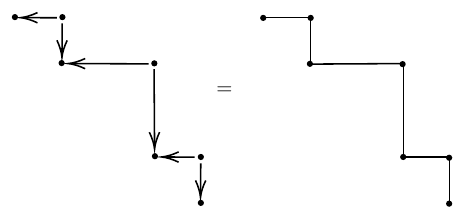}}
        \par\smallskip
        \makebox[\linewidth][c]{A staircase}
    \end{minipage}

    \medskip
    \includegraphics[scale=0.7]{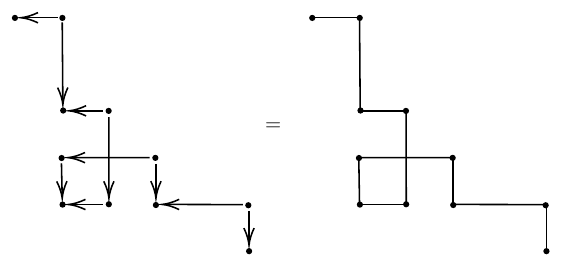}
    \par\smallskip
    A twisted staircase

    \setlength{\abovecaptionskip}{3pt}
    \caption{Examples of hv-bases}
    \label{figure:examples_of_hv-bases}
\end{figure}
\endgroup

\begin{remark}
Since differentials do not increase the $U$-filtration or Alexander filtration levels, the directions of arrows are clear, and we abbreviate them to lines. We neglect diagonal differentials throughout.
\end{remark}

The following lemma ensures the existence of a vertically simplified basis in all cases, which indicates that the essential difficulty in obtaining an hv-basis lies in making a vertically simplified basis also horizontally simplified.

\begin{lemma}[{\cite[Proposition 11.57]{LOTborderedHeeFloer}}]
\label{lemma: h or v basis exists}
Let $C$ be a $\mathbb{Z}$-filtered, $\mathbb{Z}$-graded, finitely generated chain complex over $\mathbb{F}[U]$ that is free as an $\mathbb{F}[U]$-module. Then $C$ is $\mathbb{Z}$-filtered, $\mathbb{Z}$-graded homotopy equivalent to a chain complex $C'$ that is reduced. Further, one can choose a basis for $C'$ over $\mathbb{F}[U]$ that is vertically simplified or, if one prefers, a basis that is horizontally simplified instead.
\end{lemma}

 % We then adjust this basis to be horizontally simplified simultaneously. Below we will show how to obtain a vertically simplified basis, starting with generators in different heights, with their Maslov grading being provided. The process is about determining and simplifying the differentials in this vertically simplified basis.

\subsection{Knot Floer complex over \texorpdfstring{$\mathbb{F}[U,V]$}{F[U,V]}} \label{subsection:Knot Floer complex over F[U,V]}

We briefly introduce the knot Floer complex over $\mathbb{F}[U, V]$, as a preparation for transforming an hv-basis to an immersed curve. This part is selected from \cite{HomLectureHeegaardFloer}.

Given a doubly-pointed Heegaard diagram $\mathcal{H}$ for a knot $K$ in $S^3$, denote the two base points by $w$ and $z$. We use variables $U$ and $V$ to record information about $w$ and $z$, and work with the polynomial ring $\mathbb{F}[U,V]$. We endow this ring with a bigrading $\operatorname{gr} = (\operatorname{gr}_U, \operatorname{gr}_V)$. We call $\operatorname{gr}_U$ the $U$-grading and $\operatorname{gr}_V$ the $V$-grading. The variables $U$ and $V$ have gradings $\operatorname{gr}(U) = (-2,0)$ and $\operatorname{gr}(V) = (0,-2)$, respectively.

The combination $A = \frac{1}{2}(\operatorname{gr}_U-\operatorname{gr}_V)$ is called the Alexander grading. Note that $A(U) = -1$ and $A(V) = 1$.

Let $CFK_{\mathbb{F}[U,V]}(\mathcal{H})$ be the free $\mathbb{F}[U,V]$-module generated by $\mathbb{T}_{\alpha} \cap \mathbb{T}_{\beta}$. This module is relatively bigraded as follows. For $\phi \in \pi_2(x,y)$, define
\begin{align}
  \operatorname{gr}_U(x) - \operatorname{gr}_U(y) & = \mu(\phi)-2n_w(\phi), \\
  \operatorname{gr}_V(x)-\operatorname{gr}_V(y) & = \mu(\phi)-2n_z(\phi). 
\end{align}

The relative gradings $\operatorname{gr}_U$ and $\operatorname{gr}_V$ are well defined and can be lifted to absolute $\mathbb{Z}$-gradings; see the detailed description in \cite{HomLectureHeegaardFloer}. The chain homotopy type of $CFK_{\mathbb{F}[U,V]}(\mathcal{H})$ is an invariant of $K$; we therefore denote it by $CFK_{\mathbb{F}[U,V]}(K)$.

The complex $CFK^{-}(K)$ is defined over the ring $\mathbb{F}[U]$, see \cite{LOTborderedHeeFloer}. The knot Floer complex $CFK^\infty(K)$ is related to $CFK^{-}(K)$ by $CFK^\infty(K) = CFK^{-}(K) \otimes_{\mathbb{F}[U]} \mathbb{F}[U, U^{-1}]$.

Suppose $\{\xi_i\}$ is a filtered basis for $CFK^\infty(K)$ such that each $\xi_i$ is homogeneous with $U$-grading $0$. The complex $CFK_{\mathbb{F}[U,V]}(K)$ is completely determined by $CFK^\infty(K)$ via the following algorithm.

The complex $CFK_{\mathbb{F}[U,V]}(K)$ is generated by $\{\xi_i\}$ over $\mathbb{F}[U,V]$. The differential $\partial \xi_i = \sum_{j} U^{\lambda_{ij}} \xi_j$ in $CFK^\infty(K)$ corresponds to the differential $\partial \xi_i = \sum_{j} U^{\lambda_{ij}} V^{\lambda_{ij} + A(\xi_i) - A(\xi_j)} \xi_j$ in $CFK_{\mathbb{F}[U,V]}(K)$. Let $\delta_{ij} = \lambda_{ij} + A(\xi_i) - A(\xi_j)$.

Zemke \cite[Section 1.5]{ZemkeLinkcobordismabsolutegrading} described the equivalence between $CFK^\infty(K)$ and the complex $CFL^\infty(S^3,K,\mathfrak{s}_0)_0$. Using Oszv\'{a}th--Szab\'{o}'s notation, $[x,i,j]$ in $CFK^\infty(K)$ corresponds to $U^{-i}V^{-j}x$ in $CFL^\infty(S^3,K,\mathfrak{s}_0)_0$. The $U$-action in $CFK^\infty(K)$ corresponds to the $UV$-action in $CFL^\infty(S^3,K,\mathfrak{s}_0)_0$.

If there is a diagonal differential from $\xi_i$ to $U^{\lambda_{ij}} \xi_j$, then $A(\xi_j) - A(\xi_i) < \lambda_{ij}$ and $\delta_{ij} > 0$. This diagonal differential is eliminated in the quotient $CFK_{\mathbb{F}[U,V]}(K)/(UV = 0)$. If $\partial\xi_i = U^{\lambda_{ij}} \xi_j$ is a horizontal differential, then $A(\xi_i) - A(\xi_j) + \lambda_{ij} = 0$. This horizontal differential corresponds to $\partial \xi_i = U^{\lambda_{ij}} \xi_j$ in $CFK_{\mathbb{F}[U, V]}(K)$. If $\partial \xi_i = \xi_j$ is a vertical differential, then $A(\xi_i)-A(\xi_j)$ is positive, and it corresponds to $\partial \xi_i = V^{A(\xi_i) - A(\xi_j)} \xi_j$ in $CFK_{\mathbb{F}[U, V]}(K)$. Thus, when passing from $CFK^\infty(K)$ to $CFK_{\mathbb{F}[U, V]}(K)$, the horizontal and vertical differentials are preserved while the diagonal differentials are eliminated.

% Work of Binns discusses the situation about a complex over $\mathbb{F}[U,U^{-1}]$, this is the true knot Floer complex we deal with. We state the difference between two kinds of formulations of knot Floer homology. Our target is expressing correctly what we have proved in our main theorem. The first part is the transformation between $\mathbb{F}[U,U^{-1}]$ and $\mathbb{F}[U,V]$. What we operate is the chain complex $CFK_{\mathbb{F}[U,V]}$, we talk on $\mathbb{F}[U,U^{-1}]$, $CFK_{\mathbb{F}[U,V]}$ is contained in $CFL^{\infty}()$, related by tensor with $\mathbb{F}[U, U^{-1}, V, V^{-1}]$. $CFK_{\mathbb{F}[U,V]}$ is generated by $[x,i,j]$ such that $A(x)+i-j=0$. Map $[x,i,j]$ to $U^{-i}V^{-j}x$ gives the isomorphism.

% 1. Which complex is hv-basis define on?

% 2. 

% Add the material on $CFK_{\mathbb{F}[U,V]}/(UV=0)$,  curve invariant and the transformation between $CFK^{-}(K)$ and $CFK_{\mathbb{F}[U,V]}(K)$ to this section.

\subsection{Some algebraic lemmas}\label{subsection: Some algebraic lemmas}

The cancellation lemma below is an efficient tool in our proof. An intuitive description is as follows: given a complex consisting of generators and differentials among these generators, we can eliminate one differential and the two generators connected by it and add some new differentials among the remaining generators according to certain rules. The resulting complex is homotopy equivalent to the previous one but has two fewer generators. 

\begin{lemma}[{\cite[Lemma 5.1]{RasFloerhomoknotcomplement}}]
Suppose $(C,d)$ is a chain complex freely generated by elements $y_i$, and write $d(y_i,y_j)$ for the $y_j$-coordinate of $d(y_i)$. If $d(y_k,y_l) = 1$, we can define a new complex $(\overline{C},\overline{d})$ with generators $\{y_i \mid i\neq k,l\}$ and differential

\vspace{-1.0em}
\[ \overline{d}(y_i) = d(y_i) + d(y_i,y_l) d(y_k) \]
\vspace{-1.2em}

which is chain homotopy equivalent to the first one.
\label{lemma: cancellation lemma}
\end{lemma}

\begin{remark}\label{rem: cancell lemma proof}
A proof of this lemma can be found in \cite{AnFloerWittencompandinfidimMorsetheory}. The chain homotopy equivalence is given by $\pi$ and $\iota$, where $\pi :C \to \overline{C}$ is the projection map and $\iota: \overline{C} \rightarrow C$ is given by $\iota(y_i) = y_i - d(y_i,y_l) y_k$. One may also find a proof of this lemma in \cite[Lemma 2.4]{BaldwinHeddenLipschitzCancellationlemma}.

We will sometimes call two generators connected by a differential a cancellation pair.
\end{remark}

\begin{definition}
    A \textit{knot-like complex} $C$ over $\mathbb{F}[U, V]$ is a free, finitely generated, bigraded chain complex over $\mathbb{F}[U, V]$ such that:
    \begin{enumerate}[label=(\arabic*)]
        \item $H_*(C/U)/V\text{-torsion}$\ $\cong$\ $\mathbb{F}[V]$, where $\operatorname{gr}_U(1) = 0$.

        \item $H_*(C/V)/U\text{-torsion}$\ $\cong$\ $\mathbb{F}[U]$, where $\operatorname{gr}_V(1) = 0$.
    \end{enumerate}
\end{definition}

\begin{definition}
    A basis $\{y_i\}_{i=0}^{2n}$ for a knot-like complex $C$ over $\mathbb{F}[U,V]$ is called \textit{vertically simplified} if $\partial y_i \equiv 0 \pmod{U}$ for $i$ even and $\partial y_i \equiv V^{m_i}y_{i+1} \pmod{U}$ for $i$ odd.
\end{definition}

The following lemma is proved in \cite[Proposition 2.8]{hom2022unknottingnumbercabling}. We call it the $1\times n$ lemma. It says that under a vertically simplified basis $\{y_i\}$, if there exists a vertical differential of length $n$ from $y_i$ to $y_j$, and $y_k$ is not the end of a vertical differential, then a horizontal differential of length one from $y_j$ to $y_k$ cannot exist.

\begin{lemma}\label{lemma: 1xn lemma}
    Let $\{y_i\}$ be a vertically simplified basis for a knot-like complex. Then for all even $i>0$ and odd $j$, we have $ \langle \partial y_i,y_j \rangle = 0  \pmod{(U^2,V)}$.     
\end{lemma}

\begin{remark}
    Once we transform $CFK^\infty(K)$ to a knot Floer complex over $\mathbb{F}[U, V]$ as described in Section \ref{subsection:Knot Floer complex over F[U,V]}, we always obtain a knot-like complex.
\end{remark}
    
\subsection{Immersed curves}\label{subsection: Immersed curves}

In this section, we describe the procedure for drawing the immersed curve of the complement space of $K$ from an hv-basis of $CFK^\infty(K)$. This part is selected from \cite{hom2022unknottingnumbercabling}.

Our immersed curve sits in $(\mathbb{R}/\mathbb{Z}) \times \mathbb{R}$ with punctures at each lattice point $(0, n + \frac{1}{2})$. The punctures may be thought of as pegs around which the immersed curves wind. For convenience, we identify $(\mathbb{R}/\mathbb{Z}) \times \mathbb{R}$ with $\big([-\frac{1}{2}, \frac{1}{2}]\big/ - \frac{1}{2} \sim \frac{1}{2}\big) \times \mathbb{R}$ and work with the latter.

Suppose $\{\xi_i \mid 0 \leq i \leq 2n\}$ is an hv-basis for $CFK^\infty(K)$. Each basis element $\xi_i$ corresponds to a point on the vertical line $\{0\}\times \mathbb{R}$. We abuse notation and denote this point also by $\xi_i$. The vertical coordinate of $\xi_i$ is given by the Alexander grading of $\xi_i$.

A horizontal differential from $\xi_i$ to $\xi_j$ of length $m_{ij}$ gives rise to an arc of length $m_{ij}$ from $\xi_i$ to $\xi_j$ to the right of the vertical line $\{0\}\times \mathbb{R}$, where $\xi_i$ is $m_{ij}$ units below $\xi_j$. We call such an arc a right arc of length $m_{ij}$. We also include an arc from the isolated $\xi_k$ (the homology generator under horizontal differentials) to $(\frac{1}{2}, 0)$, which we call the essential right arc.

Similarly, a vertical differential from $\xi_i$ to $\xi_j$ of length $n_{ij}$ gives rise to an arc of length $n_{ij}$ from $\xi_i$ to $\xi_j$ to the left of the vertical line $\{0\} \times \mathbb{R}$, where $\xi_i$ is now $n_{ij}$ units above $\xi_j$. We call such an arc a left arc of length $n_{ij}$. We also include an arc from the isolated $\xi_{l}$ (the homology generator under vertical differentials) to $(-\frac{1}{2},0)$, which we call the essential left arc.

In Figure \ref{Figure:Immersed_curve-1}, we present some examples of hv-bases and the corresponding immersed curves. The first line consists of all possible components of the hv-bases for height 4 knots. The second line illustrates the immersed curve for the first line. 

In the illustrations of hv-bases, the black dots at the center of each square represent generators, and the line segments connecting the black dots represent differentials. The length of each segment indicates the difference in filtration level. In the figures of immersed curves, the gray dots indicate the pegs $(0,n + \frac{1}{2})$ and the intersection points of the arcs with the green line $\{0\}\times \mathbb{R}$ are the generators of the hv-basis.

\begingroup
\setlength{\intextsep}{0pt}
\begin{figure}[H]
    \begin{minipage}[t]{0.18\textwidth}
        \centering
        \includegraphics[scale=0.85]{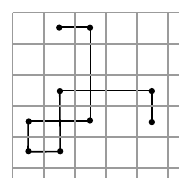}
    \end{minipage}
    \begin{minipage}[t]{0.17\textwidth}
        \centering
        \includegraphics[scale=0.95]{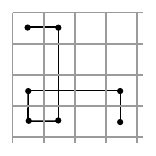}
    \end{minipage}
    \begin{minipage}[t]{0.17\textwidth}
        \centering
        \includegraphics[scale=0.95]{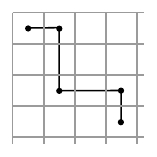}
    \end{minipage}
    \begin{minipage}[t]{0.18\textwidth}
        \centering
        \includegraphics[scale=0.85]{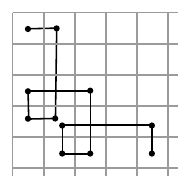}
    \end{minipage}
    \begin{minipage}[t]{0.14\textwidth}
        \centering
        \includegraphics[scale=1.0]{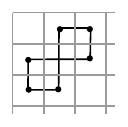}
    \end{minipage}
    \begin{minipage}[t]{0.12\textwidth}
        \centering
        \includegraphics[scale=0.9]{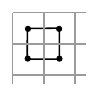}
    \end{minipage}
\end{figure}

\begin{figure}[H]
    \begin{minipage}[t]{0.18\textwidth}
        \centering
        \includegraphics[scale=0.9]{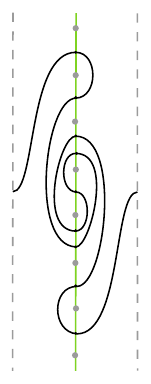}
    \end{minipage}
    \begin{minipage}[t]{0.17\textwidth}
        \centering
        \includegraphics[scale=1.0]{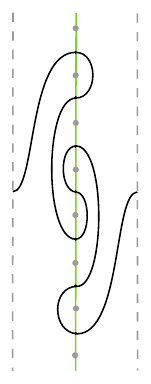}
    \end{minipage}
    \begin{minipage}[t]{0.17\textwidth}
        \centering
        \includegraphics[scale=1.0]{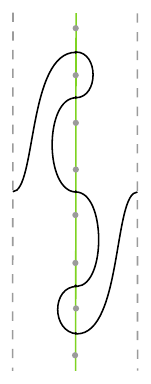}
    \end{minipage}
    \begin{minipage}[t]{0.18\textwidth}
        \centering
        \includegraphics[scale=0.9]{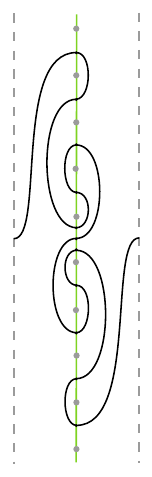}
    \end{minipage}
    \begin{minipage}[t]{0.14\textwidth}
        \centering
        \includegraphics[scale=0.85]{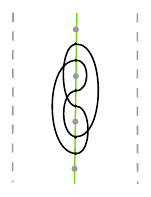}
    \end{minipage}
    \begin{minipage}[t]{0.12\textwidth}
        \centering
        \includegraphics[scale=0.70]{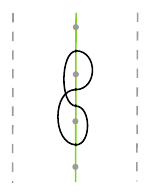}
    \end{minipage}
    \caption{Examples of hv-basis and their immersed curves}
    \label{Figure:Immersed_curve-1}
\end{figure}
\endgroup

\section{Height 4 knots}\label{section:Height 4 knots}

This Section serves as a common preparation for Sections \ref{section:(1,5,1) type} and \ref{section: (1,3,1,3,1) type}. We first divide height 4 knots into two types and describe some notations for later use. We then state our proof strategy, which is summarized into three steps. The information about $\widehat{HFK}(K)$ at general Alexander filtration levels is described in Corollaries \ref{cor: 151 outside 1 are 0 or 1-dim} and \ref{cor: 13131 outside h are 0 or 1-dim}.

Let $K\subseteq S^3$ be a knot. Recall that $K$ determines a canonical bijection between the set $\operatorname{Spin}^c(S^3_n(K))$ and the group $H_1(S^3_n(K))$, which is cyclic of order $n$. Thus, we identify $\mathbb{Z}/n\mathbb{Z}$ with the set of $\operatorname{Spin}^c$ structures on $S^3_n(K)$. For every $\mathfrak{s}\in \mathbb{Z}/n\mathbb{Z}$, we have $\operatorname{rank} \widehat{HF}(S^3_n(K),\mathfrak{s}) \geq \chi(\widehat{HF}(S^3_n(K),\mathfrak{s})) = 1$. The conjugation symmetry on $\widehat{HF}(S^3_n(K))$ imposes further restrictions. We describe this in Lemma \ref{lemma: 151 or 13131}.

\begin{lemma}\label{lemma: 151 or 13131}
Suppose $K$ is a height 4 knot and $n\geq 2g(K)-1$ is an integer. Then either
\begin{enumerate}[label=(\arabic*)]
\item $\operatorname{rank} \widehat{HF}(S^3_n(K),[i])=1$ for $[i]\neq 0$ and $\operatorname{rank}\widehat{HF}(S^3_n(K),[0])=5$, or 

\item there exists an $h\neq 0,$ such that $\operatorname{rank} \widehat{HF}(S^3_n(K),[i])=$
$\begin{cases}
\begin{array}{cl}
1, & [i]\neq [\pm h]\\
3, & [i]=[\pm h]
\end{array}
\end{cases}
$
\end{enumerate}
\end{lemma}

\begin{proof}
The Euler characteristic preserves the parity of the rank of the group. Thus, the condition $\chi(\widehat{HF}(S^3_n(K), [i])) = 1$ for each $[i] \in \mathbb{Z}/ n\mathbb{Z}$ implies that $\operatorname{rank} \widehat{HF}(S^3_n(K), [i])$ is odd for every $[i]$. Combined with the condition that  $\sum_{[i] \in \mathbb{Z}/ n\mathbb{Z}}\operatorname{rank} \widehat{HF}(S^3_n(K),[i]) = n+4,$ we conclude that there are two possibilities: either there is a unique $[i] \in \mathbb{Z}/n\mathbb{Z}$ such that $\operatorname{rank}\widehat{HF}(S^3_n(K), [i]) = 5$, or there are two distinct $[i] \in \mathbb{Z}/n\mathbb{Z}$ where $\operatorname{rank}\widehat{HF}(S^3_n(K),\  [i]) = 3$.

The conjugation symmetry gives that $\widehat{HF}(S^3_n(K),[i])$ and $\widehat{HF}(S^3_n(K),[-i])$ have equal rank. In the former case, the $\operatorname{Spin}^c$ structures $[i]$ in which $\operatorname{rank}\widehat{HF}(S^3_n(K),[i])=5$ must be self-conjugate, so $[i]=0$ or $[i]=[n/2]$, provided $n$ is even. The surgery formula in \cite{OSIntesurgery} states that $\widehat{HF}(S^3_n(K),[i])$ has rank one for $i\geq g(K)$, therefore $[i]$ must be 0 since $n/2\geq g(K)-1/2$. In the latter case, the two Spin$^c$ structures in which $\operatorname{rank}\widehat{HF}(S^3_n(K),[i])=3$ are conjugate, which corresponds to the second possibility. 
\end{proof}

\begin{definition}\label{dfn: 151 type and 13131 type}
For a height 4 knot $K$, if its Heegaard Floer homology satisfies case (1) in Lemma \ref{lemma: 151 or 13131}, we say that $K$ is of \emph{$(1,5,1)$ type}; if it satisfies case (2), we say $K$ is of \emph{$(1,3,1,3,1)$ type with parameter $h$}.
\end{definition}

We now set up some algebraic background necessary for later discussion.  

Assume $C$ is a finitely generated, $\mathbb{Z} \oplus \mathbb{Z}$-graded chain complex of $\mathbb{F}$-modules. We use $(i,j)$ to denote the two gradings. Let $\{\mathbb{A}\}$ denote the set of generators satisfying the conditions $\mathbb{A}$. For an appropriate set $S\subset \mathbb{Z} \oplus \mathbb{Z}$, we denote by $C(S)$ the (appropriate quotient of the) subcomplex consisting of generators with gradings in $S$. Let $H(S)$ denote the homology of $C(S)$. We assume that $C$ admits an invertible action, denoted by $U$, such that for $x\in C(\{(i,j)\})$, we have $Ux \in C(\{(i-1,j-1)\})$. For convenience, we will write $C\{(i,j)\}$ instead of $C(\{(i,j)\})$. We further assume that $U$ commutes with the differentials in $C$. This makes $C$ a complex of $\mathbb{F}[U, U^{-1}]$-module.  

Set $X_m = C\{i \leq 0,j = m\}, Y_m = C\{i = 0,j \leq m-1\}, UX_m = C\{i < 0,j = m-1\}$, and $A_m = X_m \cup Y_m = C\{max(i, j-m) = 0\}$. Here $A_m$ is called a bent complex. These are illustrated in Figure \ref{figure:X_m_and_Y_m}. 

%X_m,Y_m示意图
\begingroup
\setlength{\intextsep}{0pt}
\begin{figure}[H]
    \centering
    \setlength{\abovecaptionskip}{3pt}
	\includegraphics[scale=0.8]{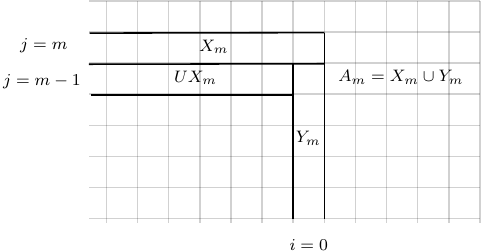}
    \caption{}
    \label{figure:X_m_and_Y_m}
\end{figure}
\endgroup

We define $Z_m$ as $C\{i \leq m-1,j = 0\}$.

Throughout, we conflate $C$ with the underlying knot Floer complex $CFK^\infty(K)$, where $i$ and $j$ correspond to the $U$-grading and the Alexander grading, respectively. 

Combined with Lemma \ref{theorem: Euler characteristic of Heegaard Floer homology}, we obtain:

\begin{enumerate}[label=(\arabic*)]
\item If $K$ is of $(1,5,1)$ type, then  $H_*(A_m) = $ 
$\begin{cases}
\begin{array}{cc}
\mathbb{F}^5, & m = 0, \\
\mathbb{F}, & m \neq 0. 
\end{array}
\end{cases}$

\item If $K$ is of $(1,3,1,3,1)$ type, then  $H_*(A_m) = $ 
$\begin{cases}
\begin{array}{cc}
	\mathbb{F}^3, & m= \pm h, \\
	\mathbb{F}, & m \neq \pm h. 
\end{array}
\end{cases}$
\end{enumerate}
\vspace{1em}

We now state our strategy for determining $CFK^\infty(K)$, which is divided into three steps.

\begin{enumerate}[label=(\arabic*)]

\item Determine $\widehat{HFK}(K)$, including the dimension and Maslov grading information at each Alexander filtration level.

\item Determine $\widehat{CFK}(K)$, i.e., the subquotient complex $C\{i = 0\}$. We begin with the filtered basis consisting of generators in $\widehat{HFK}(K)$ and turn it into a vertically simplified basis.

\item Determine $C\{j = 0\}$; we attempt to make the basis obtained in the previous step horizontally simplified while preserving its vertical simplicity.

\end{enumerate}

At the end of the third step, we obtain an hv-basis. This induces a disjoint decomposition of $CFK^\infty(K)$ as a vector space over $\mathbb{F}$, written as $CFK^\infty(K) = \bigcup_{i} U^i C_{hv}$. The subset $C_{hv}$ consists of this hv-basis together with all horizontal and vertical differentials, up to the action of powers of $U$.
 
%\item Try to eliminate differentials between different $U^i-$copies of $C_{hv},$ mainly through carrying filtered chain isotopy.

%\item Determine diagonal differentials in $C_{hv}.$ This will give us a tensor decomposition of chain complex: $CFK^\infty(S^3,K)=C_{hv}\otimes_{\mathbb{F}}\mathbb{F}[U,U^{-1}].$

%In this paper we will achieve step one to three. 

Following lemmas are preparations for step one. 

\begin{lemma}\label{lemma: 1 and 2k+1 then k or k+1}
Suppose $H_{*}(X_m\cup Y_m) = \mathbb{F}$ and $H_{*}(UX_{m}\cup Y_m) = \mathbb{F}^{2k+1}$. Then we have \(k \leq \operatorname{rank} H_{*}(X_{m}) \leq k+1\).
\end{lemma}\vspace{-1em}

\begin{proof} Natural inclusion and quotient relationships between subquotient complexes yield two short exact sequences, as shown on the left in Equation \eqref{eq:exact-sequencesof1,2k+1thenkork+1}.

\begin{equation}
\resizebox{0.90\linewidth}{!}{$
\xymatrix@C=10pt@R=11pt{
& & & 0 \ar[d] & \\
0 \ar[r] & C\{UX_m\} \ar[r] & C\{UX_m\cup Y_m\} \ar[r] & C\{Y_m\} \ar[r] \ar[d] & 0 \\
& & & C\{X_m\cup Y_m\} \ar[d] & \\
& & & C\{X_m\} \ar[d] & \\
& & & 0 &
}
\hspace{0.5em}
\raisebox{-10pt}{$
\xymatrix@C=8pt@R=20pt{
H_*(UX_m) \ar[dr]_{i_h} & & H_*(Y_m) \ar[ll]_{\delta_h} \ar[dr]_{i_v} & & H_*(X_m) \ar[ll]_{\delta_v} \\
& H_*(UX_m\cup Y_m) \ar[ur]_{p_h} & & H_*(X_m\cup Y_m) \ar[ur]_{p_v} &
}
$}
$}
\label{eq:exact-sequencesof1,2k+1thenkork+1}
\end{equation}

% \begin{figure}[H]
% \vspace{-0.5em}
% \centering
%     \begin{minipage}[c]{0.48\textwidth}
%     \centering
%         \includegraphics[scale=0.8]{fig/algebraic_lemma_figure-1}
%     \end{minipage}
%     \begin{minipage}[c]{0.48\textwidth}
%     \centering
%         \includegraphics[scale=0.8]{fig/algebraic_lemma_figure-2}
%     \end{minipage}	
%     \caption{}
%     \label{figure:1,2k+1thenkork+1}
% \end{figure}\vspace{-1.5em}

This induces short exact sequences on their homologies, as shown on the right in Equation \eqref{eq:exact-sequencesof1,2k+1thenkork+1}.

As noted in \cite[Section 3]{OSLenspacesurgeries}, the composition map $\delta_h\circ \delta_v=0.$ 

We assume $H_*(X_m) \cong H_*(UX_m) = \mathbb{F}^n$ and $H_*(Y_m) = \mathbb{F}^m$. The image of $i_v$ is either 0 or the whole $H_*(X_m\cup Y_m)$. If $i_v$ is the zero map, then the exactness on the right side tells that $n = m+1,m = n-1$. The map $\delta_v$ is surjective, which forces $\delta_h$ to be zero. The exactness on the left side then gives $2k+1 = n+n-1, n = k+1.$

If $i_v$ is surjective, then $p_v$ is the zero map and $m=n+1$. The map $\delta_v$ is injective. At this point, we resort to the relationship $\delta_v\circ U^{-1}\delta_h=0$, see the exact diagram in Equation \eqref{eq:exact-diagram}.

\begingroup
\setlength{\abovedisplayskip}{1pt}
\setlength{\abovedisplayshortskip}{1pt}
\setlength{\belowdisplayskip}{1pt}
\setlength{\belowdisplayshortskip}{1pt}
\begin{equation}
\label{eq:exact-diagram}
\scalebox{0.9}{
\xymatrix@C=12pt@R=22pt{
H_*(Y_m) \ar[dr] & & H_*(X_m) \ar[ll]_{\delta_v} \ar[dr] & & H_*(U^{-1}Y_m) \ar[ll]_{U^{-1}\delta_h} \\
& H_*(X_m\cup Y_m) \ar[ur] & & H_*(X_m\cup U^{-1}Y_m) \ar[ur] &
}
}
\end{equation}
\endgroup

% \begin{figure}[H]
% \vspace{-0.5em}
% \centering
% 	\includegraphics[scale=0.8]{fig/algebraic_lemma_figure-4}
%     \caption{}
%     \label{Figure: U translation homology}
% \end{figure}\vspace{-1.5em}

This yields $U^{-1}\delta_h=0$, then $\delta_h=0$ since $U^{-1}$ is an isomorphism. We get $2k+1=n+n+1$, $n=k$.

\begingroup
\setlength{\abovedisplayskip}{1pt}
\setlength{\abovedisplayshortskip}{1pt}
\setlength{\belowdisplayskip}{1pt}
\setlength{\belowdisplayshortskip}{1pt}
\begin{equation}
\label{eq:two-cases}
\resizebox{0.85\linewidth}{!}{$
\xymatrix@C=12pt@R=22pt{
\mathbb{F}^{k+1} \ar[dr] & & \mathbb{F}^{k} \ar[ll] \ar[dr]_{0} & & \mathbb{F}^{k+1} \ar[ll] \\
& \mathbb{F}^{2k+1} \ar[ur] & & \mathbb{F}^{1} \ar[ur] &
}
\hspace{1em}\&\hspace{1em}
\xymatrix@C=12pt@R=22pt{
\mathbb{F}^{k} \ar[dr] & & \mathbb{F}^{k+1} \ar[ll] \ar[dr] & & \mathbb{F}^{k} \ar[ll] \\
& \mathbb{F}^{2k+1} \ar[ur] & & \mathbb{F}^{1} \ar[ur]_{0} &
}
$}
\end{equation}
\endgroup

% \begin{figure}[H]
%     \vspace{-0.5em}
%     \centering
% 	\includegraphics[scale=1.0]{fig/algebraic_lemma_figure-3}
%     \caption{}
%     \label{fig/algebraic_lemma_figure-3}
% \end{figure}
% \vspace{-1.5em}

Equation \eqref{eq:two-cases} presents results in the two cases. 
\end{proof}

\begin{lemma}\label{lemma: 2k+1 and 1 then k or k+1}
Suppose $H_{*}(X_m \cup  Y_m) = \mathbb{F}^{2k+1} $ and $H_{*}(U X_{m} \cup Y_m) = \mathbb{F}.$ Then $k \leq \operatorname{rank} H_{*}(X_{m}) \leq k+1.$
\end{lemma}

\begin{proof} 
The proof of this lemma differs only slightly from that of Lemma \ref{lemma: 1 and 2k+1 then k or k+1}. We use the same homology diagram as in Equation \eqref{eq:exact-sequencesof1,2k+1thenkork+1} with the conditions on $H_*(UX_m\cup Y_m)$ and $H_*(X_m\cup Y_m)$ exchanged. We then discuss separately depending on whether $\operatorname{rank} \operatorname{im} i_h$ is $0$ or $1$.
\end{proof}

\begin{remark}\label{remark: (1,0)or(0,1) remark} 
The two lemmas above actually tell us that  $(\operatorname{rank} H_*(X_m)$, $\operatorname{rank}H_*(Y_m))$ = $(k, k+1)$ or $(k+1, k)$.
\end{remark}

\begin{lemma}\label{lemma: 1 and 1 then 1 or 0} Suppose $H_*(A_m) = \mathbb{F}, m=r-1, r, r+1$, then $H_*(\{(0,r)\}) = \mathbb{F}\ \text{or}\ 0$.
\end{lemma}

\begin{proof} 
Lemmas \ref{lemma: 1 and 2k+1 then k or k+1} tell us that $(H_*(X_{m}), H_*(Y_{m}))=(\mathbb{F},0)\ \text{or}\ (0,\mathbb{F}),m=r+1,r$. We then consider the following two short exact sequences: 

\vspace{-1.8em}
\begin{equation} 
0 \rightarrow UX_{r+1} \rightarrow X_r \rightarrow
C\{(0,r)\} \rightarrow 0,\quad 0 \rightarrow Y_{r} \rightarrow Y_{r+1} \rightarrow C\{(0,r)\} \rightarrow 0
\end{equation}
\vspace{-1.5em}

The induced short exact sequences of homologies give the result.
\end{proof}

\begin{corollary}\label{cor: 151 outside 1 are 0 or 1-dim} 
Assume $K$ is a height 4 knot of $(1,5,1)$ type then we have 
$\operatorname{rank} H_*(\{(0, A)\})$ $=\operatorname{rank} \widehat{HFK}(K, A) \leq 1$ whenever $|A| > 1$.
\end{corollary}

\begin{corollary}\label{cor: 13131 outside h are 0 or 1-dim} 
Assume $K$ is a height 4 knot of $(1,3,1,3,1)$ type with parameter $h$ then we have 
$\operatorname{rank} H_*(\{(0, A)\}) = \operatorname{rank} \widehat{HFK}(K, A) \leq 1$ whenever $|A| > h+1$ or $|A| < h-1$.
\end{corollary}

Recall that we assume $CFK^\infty(K)$ is reduced, i.e., there are no differentials from $x$ to $y$ if $x$ and $y$ have identical $(i, j)$ coordinates. We can then identify $C\{(0, j)\}$ with $\widehat{HFK}(K, j)$, as $\mathbb{F}$-vector spaces. The lemmas above allow us to determine the rank and Maslov grading of each $\widehat{HFK}(K,j)$. Together with homology information for each $H_*(Y_m)$, we can apply grid climbing to determine $C\{i=0\}$, as will be discussed in the following Sections.

\section{Height 4 knots of the (1,5,1) type}\label{section:(1,5,1) type}

In this section, we prove Theorem \ref{thm: classification of height 4} for height 4 knots of the $(1,5,1)$ type. We first state the main theorem of this section, then present the basic setup, which includes a detailed proof of Lemma \ref{lemma: 151 X_2 and Y_2} that presents the techniques used to handle all cases. In Section \ref{subsection: Dimension and Maslov grading}, we determine the dimensions and Maslov gradings in all cases. In Section \ref{subsection: Construction of hv-basis}, we first construct a vertically simplified basis and then modify it to obtain an hv-basis, thereby completing the proof of Theorem \ref{thm: 151 has hv-basis}. Note that in Sections \ref{subsection: Dimension and Maslov grading} and \ref{subsection: Construction of hv-basis} we work under the assumption that $g(K) > 1$. In Section \ref{subsection:3.4 Special case of genus one}, we present the classification results of the hv-basis for the special $g(K) = 1$ case.

\begin{theorem}\label{thm: 151 has hv-basis} 
Let $K$ be a height 4 knot of the $(1,5,1)$ type. Then $C=CFK^{\infty}(K)$ admits an hv-basis of one of the the following types:
\begin{enumerate}[label=(\arabic*)]
    \item a doubly-twisted staircase, either positive or negative;

    \item a twisted staircase (either positive or negative) plus one box;

    \item a staircase plus one twisted box;

    \item a staircase plus two boxes, both located in the central position. 
\end{enumerate}
These four possibilities are illustrated in Figure \ref{Figure: figure of main theorem for (1,5,1) type}.
\end{theorem}

\begingroup
\setlength{\intextsep}{1pt}
\begin{figure}[H]
	\centering
    \setlength{\abovecaptionskip}{1pt}
	% 第一行
	\begin{minipage}{0.24\linewidth}
		\centering
		\includegraphics[scale=0.7]{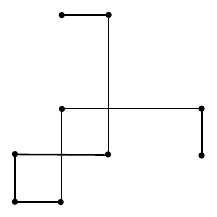}
		\caption*{$(1_{+})$}
	\end{minipage}
	\begin{minipage}{0.24\linewidth}
		\centering
		\includegraphics[scale=0.7]{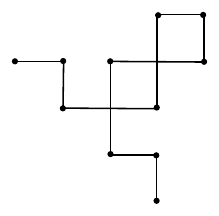}
		\caption*{$(1_{-})$}
	\end{minipage}
	\begin{minipage}{0.24\linewidth}
		\centering
		\includegraphics[scale=0.7]{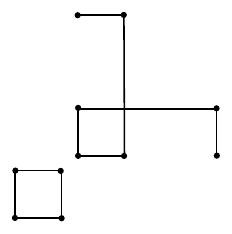}
		\caption*{$(2_{+})$}
	\end{minipage}
	\begin{minipage}{0.24\linewidth}
		\centering
		\includegraphics[scale=0.7]{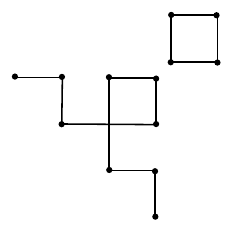}
		\caption*{$(2_{-})$}
	\end{minipage}
	
	% 第二行
	\begin{minipage}{0.48\linewidth}
		\centering
		\includegraphics[scale=0.7]{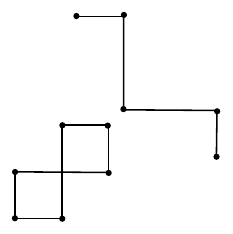}
		\caption*{(3)}
	\end{minipage}
	\begin{minipage}{0.48\linewidth}
		\centering
		\includegraphics[scale=0.7]{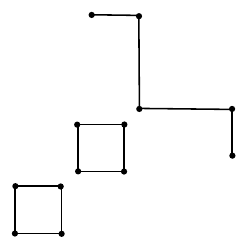}
		\caption*{(4)}
	\end{minipage}
    \caption{hv-basis for knots of the (1,5,1) type}
    \label{Figure: figure of main theorem for (1,5,1) type}
\end{figure}
\endgroup

\subsection{Basic setup}\label{subsection:3.1 Basic setup}

In this section, we first introduce the symmetry property \eqref{equation:(i,j) symmetry} of $CFK^\infty(K)$, and then present a detailed proof of Lemma \ref{lemma: 151 X_2 and Y_2} as an illustration to the proof methods of all cases.

The symmetry property on $CFK^\infty(K)$ with respect to commuting $(i,j)$ \cite[Section 3.5]{OSknotinvar} will be used repeatedly throughout this paper, which gives:

\begingroup
\setlength{\abovedisplayskip}{1pt}
\setlength{\abovedisplayshortskip}{1pt}
\setlength{\belowdisplayskip}{1pt}
\setlength{\belowdisplayshortskip}{1pt}
\begin{align*}
H_*(X_m) & = H_*(C\{i \leq 0, j = m\}) \\
         & \cong H_*(C\{i = m, j \leq 0\}) \\ 
         & \cong H_{*-2m}(C\{i = 0, j \leq -m\}) \\
         & = H_{*-2m}(Y_{-m+1}).
\end{align*}
\endgroup

The first and fourth equalities follow from the definitions, the second equality is a consequence of the symmetry property, and the third equality follows from the $U$-action. We summarize this as:
\begingroup
\setlength{\abovedisplayskip}{1pt}
\setlength{\abovedisplayshortskip}{1pt}
\setlength{\belowdisplayskip}{1pt}
\setlength{\belowdisplayshortskip}{1pt}
\begin{equation}\label{equation:(i,j) symmetry}
    H_*(X_m) \cong H_{*-2m}(Y_{-m+1})
\end{equation}
\endgroup

Applying \eqref{equation:(i,j) symmetry} to the cases $m = 2$ and $m = 1$, we obtain $H_*(X_2) \cong H_{*-4}(Y_{-1})$ and $H_*(X_1) \cong H_{*-2}(Y_0)$.

Now we are ready to determine the dimension information of 
$\widehat{HFK}(K, A)$ for $A = 1, 0, -1$. We will determine the Maslov grading simultaneously. As a starting point, we consider the generator of minimal Alexander grading at least $2$.

The following lemma has already been proved in \cite[Lemma 4.6]{Binns}; we prove it here as an illustration of the chasing process, which will be applied in every case we discuss later. At the same time, the proof gives us the Maslov grading information of $\widehat{HFK}(K, A)$ for $|A| \geq 2$.

\begin{lemma}   \label{lemma: 151 X_2 and Y_2}
Suppose $K$ is a height 4 knot of the $(1,5,1)$ type with genus greater than one. Then there exists a $d \in 2\mathbb{Z}$ such that either

\begin{enumerate}[label=(\arabic*),topsep=1pt,partopsep=0pt,itemsep=1pt,parsep=0pt]
\item $H_*(X_2) = 0$, $H_*(Y_2) = \mathbb{F}_{(d)}$, or

\item $H_*(X_2) = \mathbb{F}_{(d)}$, $H_*(Y_2) = 0$.
\end{enumerate}

Case (1) holds if the Maslov grading of the generator of minimal Alexander grading at least 2 is odd; case (2) holds if it is even.
\end{lemma}

\begin{proof}
% We will repeatedly apply above lemmas to chase the dimension and grading information, from top to below.

By \cite[Theorem 1.2]{OShologenusbound}, for $p > g(K)$, we have $H_*(X_p) = 0$ and $H_*(Y_p) = \widehat{HF}(S^3) = \mathbb{F}_{(0)}$. Thus, the short exact sequence $0 \rightarrow UX_{p+1} \rightarrow X_{p} \rightarrow C\{(0, p)\} \rightarrow 0$ yields $H_*(\{(0, p)\})$ $=$ $\widehat{HFK}(K, p) = 0$. We write $g=g(K)$ for brevity.

Consider the short exact sequence: $0 \rightarrow UX_{g+1} \rightarrow X_{g} \rightarrow C\{(0,g)\} \rightarrow 0$. The homology of $UX_{g+1}$ is 0 and $\operatorname{rank} H_*(X_{g})$ $=$ $\operatorname{rank} H_*(\{(0, g)\}) = 1$. By Lemma \ref{lemma: 1 and 2k+1 then k or k+1}, we obtain $H_*(Y_{g}) = 0$.

In the other short exact sequence, $0 \rightarrow Y_{g} \rightarrow Y_{g+1} \rightarrow C\{(0, g)\} \rightarrow 0$, the homology of $Y_{g}$ is 0. Therefore, $C\{(0,g)\}$ has the same homology as $Y_{g+1}$, which is one-dimensional with Maslov grading $0$. This is illustrated in Figure \ref{Figure:algebraic_lemma_figure-5}.

\begingroup
\setlength{\intextsep}{1pt}
\begin{figure}[H]
    \centering
    \setlength{\abovecaptionskip}{2pt}
	\includegraphics[scale=0.75]{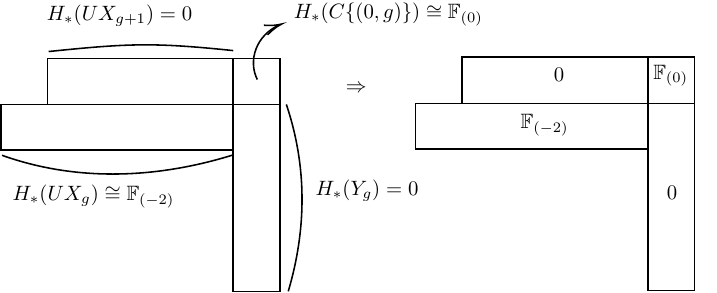}
    \caption{}
    \label{Figure:algebraic_lemma_figure-5}
\end{figure}
\endgroup

We continue downward to find the next non-zero homology term in $C\{i = 0\}$. The complex $X_p$ is transformed by $U$ repeatedly, with Maslov grading decreasing by 2 each time. Together with the homology $\widehat{HFK}(K, j)$ (zero- or one-dimensional) at the corresponding height level, we can determine $H_*(X_j)$. For example, if $H_*(\{(0,g-1)\}) = 0$, then $H_*(X_{g-1})$ is $\mathbb{F}_{(-2)}$ and $H_*(Y_{g-1})$ is $0$, respectively.

\begingroup
\setlength{\intextsep}{1pt}
\begin{figure}[H]
    \centering
    \setlength{\abovecaptionskip}{1pt}
	\includegraphics[scale=0.75]{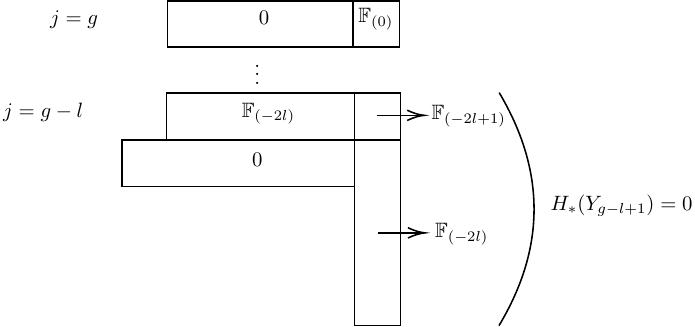}
    \caption{}
    \label{Figure:algebraic_lemma_figure-6}
\end{figure}
\endgroup
Denote by $g-l$ the height at which we encounter the next non-zero homology term (assume $g-l$ is far greater than $1$).
Then $H_*(Y_{g-l+1}) = 0$, $H_*(UX_{g-l+1}) = \mathbb{F}_{(-2l)}$, and $H_*(X_{g-l}) = 0$. The sequence $0 \rightarrow UX_{g-l+1} \rightarrow X_{g-l} \rightarrow C\{(0, g-l)\} \rightarrow 0$ gives $H_*(\{(0, g-l)\}) = \mathbb{F}_{(-2l+1)}$.
The sequence $0 \rightarrow Y_{g-l} \rightarrow Y_{g-l+1} \rightarrow C\{(0, g-l)\} \rightarrow 0$ then gives $H_*(Y_{g-l}) = \mathbb{F}_{(-2l)}$; see Figure \ref{Figure:algebraic_lemma_figure-6}.

Suppose we are still far above height 0, we continue to trace downward. Assume that we meet next nonzero homology term at height $g-l-t$ (assume $g-l-t$ is far greater than 1).

\begingroup
\setlength{\intextsep}{2pt}
\begin{figure}[H]
    \centering
    \setlength{\abovecaptionskip}{1pt}
	\includegraphics[scale=0.8]{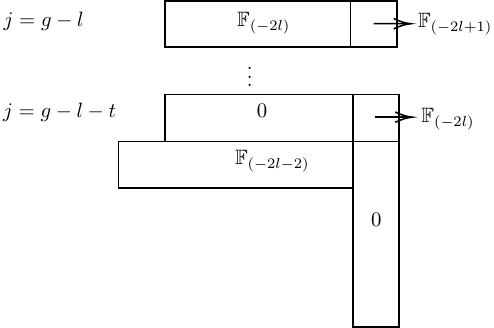}
    \caption{}
    \label{Figure:algebraic_lemma_figure-7}
\end{figure}
\endgroup
By analyzing the corresponding exact sequences we obtain $H_*(\{(0,g-l-t)\}) = \mathbb{F}_{(-2l)}$, $H_*(X_{g-l-t}) = \mathbb{F}_{(-2l)}$, and $H_*(Y_{g-l-t}) = 0$, see Figure \ref{Figure:algebraic_lemma_figure-7}.

\begingroup
\setlength{\intextsep}{2pt}
\begin{figure}[H]
    \centering
    \setlength{\abovecaptionskip}{1pt}
	\includegraphics[scale=0.8]{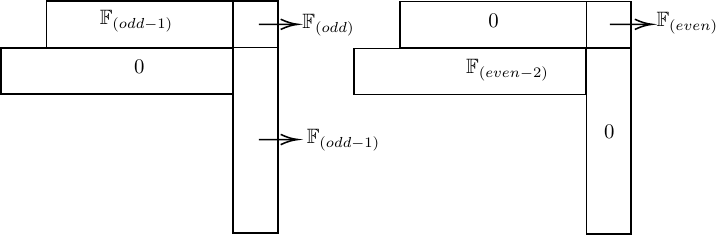}
    \caption{}
    \label{Figure:algebraic_lemma_figure-8}
\end{figure}
\endgroup

As long as we are far above height 0, the situation is identical each time. We summarize the results in Figure \ref{Figure:algebraic_lemma_figure-8}. A concrete description via formulas could be found in Ozsv\'{a}th--Szab\'{o} \cite[Theorem 1.2]{OSLenspacesurgeries}.
In short, the parity of the Maslov grading alternates as we trace downward, and the Maslov grading decreases by one when we go from an odd grading to an even grading.

The statement of the lemma then follows naturally.
\end{proof}

\begin{remark}
The above method for determining the vertical differentials is called \textit{grid climbing} or the \textit{snake method}.
\end{remark}

\begin{corollary}\label{cor: 151 genus>1 are fibered} Given a height 4 knot $K$ of the $(1,5,1)$ type, assume its genus is greater than 1, then $K$ must be fibered.
\end{corollary}

Let $x$ be the generator of lowest Alexander grading strictly greater than 1. Let $m$ denote the Maslov grading of $x$ and $A$ its Alexander grading (i.e., $\widehat{HFK}(K, A) = \mathbb{F}_{(m)}$). Recall the assumption that the genus of $K$ is greater than one.

Lemma \ref{lemma: 151 X_2 and Y_2} actually shows how to chase the dimension and Maslov grading. We now determine $H_*(\{(0,i)\}) = \widehat{HFK}(K, i)$ for $i = 0, \pm 1$. 

\subsection{Dimension and Maslov grading}\label{subsection: Dimension and Maslov grading}

In this section, we determine the dimension and Maslov grading of the generators at each height of $C\{i = 0\}$.

We discuss separately the different possibilities for $\operatorname{rank} H_*(X_m)$ with $m= 2, 1, 0, -1$. By Lemmas \ref{lemma: 1 and 2k+1 then k or k+1}, \ref{lemma: 2k+1 and 1 then k or k+1}, \ref{lemma: 151 X_2 and Y_2}, and Equation \eqref{equation:(i,j) symmetry}, there are four possible cases: 
\begin{equation}
(0,2,3,1),\  (0,3,2,1),\ (1,2,3,0),\ (1,3,2,0),
\end{equation}

where $(0,2,3,1)$ abbreviates $(\operatorname{rank} H_*(X_2), \dots, \operatorname{rank} H_*(X_{-1})) =(0,2,3,1)$.

The proofs of the cases $(0,2,3,1)$ and $(1,2,3,0)$ will be presented in detail. For the remaining two cases 
$(0,3,2,1)$ and $(1,3,2,0)$, we only sketch the arguments and highlight the main differences.

For clarity, we summarize the classification results for all cases in the following two tables before presenting the detailed proofs. The parameters $m$ and $A$ are defined after Corollary \ref{cor: 151 genus>1 are fibered}, $a$ and $b$ are independent parameters, $\alpha=m-2A+2$.

% To facilitate the reader's reference to the classification results for the corresponding cases, we list the computation results for all different cases in the tables on the following two pages before the detailed proofs begin.

\newpage

\[ (1,5,1), g(K)\geq 2 \]

\begin{table}[H]
\begin{center}
\begin{tabular}{|c|c|c|}
	\hline
	$rk H_*(X_p)$, $p = 2,1,0,-1$ & $\widehat{HFK}(K,i),i=1,0$ & hv-basis  \\
	\hline
	
	\multirow{4}{*}{(0,2,3,1)} & \makecell{ $\mathbb{F}_{(m+1)} \oplus \mathbb{F}_{(m+2)}$ \\ $\mathbb{F}_{(m+1)}$ } & 
	{\includegraphics[scale=0.58, valign=c]{fig/2-1}}
	\\
	
	\cline{2-3}
	
	& \makecell{ $\mathbb{F}_{(a)}\oplus \mathbb{F}_{(m+1)}$ \\ $\mathbb{F}_{(a-1)} \oplus \mathbb{F}_{(a-1)} \oplus \mathbb{F}_{(m)}$ }  &  {\includegraphics[scale=0.58, valign=c]{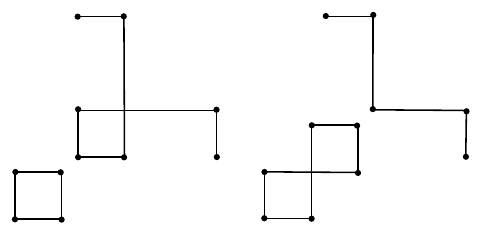}}
	\\
	
	\cline{2-3}
	
	& \makecell{ $\mathbb{F}_{(a)}\oplus \mathbb{F}_{(a+1)}$ \\ $\mathbb{F}_{(a-1)} \oplus \mathbb{F}_{(a)} \oplus \mathbb{F}_{(m-1)}$ }  & {\includegraphics[scale=0.58, valign=c]{fig/2-2}} \\
	
	\cline{2-3}
	
	& \makecell{ $\mathbb{F}_{(a)} \oplus \mathbb{F}_{(b)}$ \\ $\mathbb{F}^2_{(a-1)} \oplus \mathbb{F}^2_{(b-1)} \oplus \mathbb{F}_{(m-1)}$ } & {\includegraphics[scale=0.58, valign=c]{fig/2-4}} \\
	
	\hline
	
	\multirow{4}{*}{(0,3,2,1)} & \makecell{ $\mathbb{F}_{(m-3)} \oplus \mathbb{F}_{(m-2)}\oplus \mathbb{F}_{(m-1)}$ \\ $\mathbb{F}_{(m-4)}$ } & \includegraphics[scale=0.58, valign=c]{fig/2-5}
	\\
	
	\cline{2-3}
	
	& \makecell{ $\mathbb{F}_{(a)} \oplus \mathbb{F}_{(a+1)}\oplus \mathbb{F}_{(m-1)}$ \\ $\mathbb{F}_{(a-1)} \oplus \mathbb{F}_{(a)} \oplus \mathbb{F}_{(m-2)}$ }  & {\includegraphics[scale=0.58, valign=c]{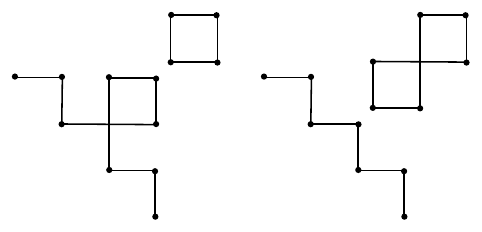}}
	\\
	
	\cline{2-3}
	
	& \makecell{$\mathbb{F}_{(a)}\oplus \mathbb{F}_{(m-2)}\oplus \mathbb{F}_{(m-1)}$ \\ $\mathbb{F}_{(a-1)} \oplus \mathbb{F}_{(a-1)} \oplus \mathbb{F}_{(m-3)}$ }  &{\includegraphics[scale=0.58, valign=c]{fig/2-6}} \\
	
	\cline{2-3}
	
	& \makecell{ $\mathbb{F}_{(a)} \oplus \mathbb{F}_{(b)}\oplus \mathbb{F}_{(m-1)}$ \\ $\mathbb{F}^2_{(a-1)} \oplus \mathbb{F}^2_{(b-1)} \oplus \mathbb{F}_{(m-2)}$ } & {\includegraphics[scale=0.58, valign=c]{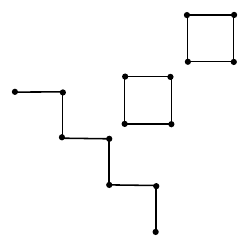}}\\
	\hline
\end{tabular}
\end{center}
\caption{}
\label{table:hv-basis of 151 type, g(K)>1-1}
\end{table}

\newpage

\begin{table}
\begin{center}
\begin{tabular}{|c|c|c|}
	\hline
	$rk H_*(X_p)$, $p = 2,1,0,-1$ & $\widehat{HFK}(K,i),i=1,0$ & hv-basis  \\
	\hline
	
	\multirow{4}{*}{(1,2,3,0)} & \makecell{ $\mathbb{F}_{(\alpha+1)} \oplus \mathbb{F}_{(\alpha+2)}\oplus \mathbb{F}_{(\alpha+3)}$ \\ $\mathbb{F}_{(\alpha+2)}$ } & {\includegraphics[scale=0.65, margin=0.05em, valign=c]{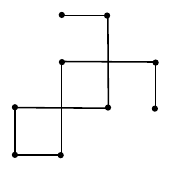}}
	\\
	
	\cline{2-3}
	
	& \makecell{ $\mathbb{F}_{(a-1)}\oplus \mathbb{F}_{(a)}\oplus \mathbb{F}_{(\alpha+1)}$ \\ $\mathbb{F}_{(a-2)} \oplus \mathbb{F}_{(a-1)} \oplus \mathbb{F}_{(\alpha)}$ }  &{\includegraphics[scale=0.65, margin=0.05em, valign=c]{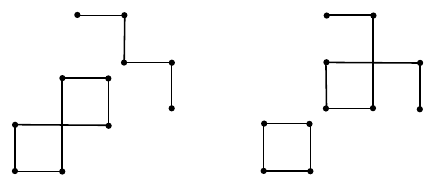}}
	\\
	
	\cline{2-3}
	
	& \makecell{ $\mathbb{F}_{(a)}\oplus \mathbb{F}_{(\alpha+1)}\oplus \mathbb{F}_{(\alpha+2)}$ \\ $\mathbb{F}_{(a-1)} \oplus \mathbb{F}_{(a-1)} \oplus \mathbb{F}_{(\alpha+1)}$ }  & {\includegraphics[scale=0.65, margin=0.05em, valign=c]{fig/3-2}} \\
	
	\cline{2-3}
	
	& \makecell{ $\mathbb{F}_{(a)} \oplus \mathbb{F}_{(b)} \oplus \mathbb{F}_{(\alpha+1)}$ \\ $\mathbb{F}^2_{(a-1)} \oplus \mathbb{F}^2_{(b-1)} \oplus \mathbb{F}_{(\alpha)}$ } &  {\includegraphics[scale=0.65, margin=0.05em, valign=c]{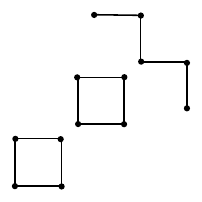}} \\
	
	\hline
	
	\multirow{4}{*}{(1,3,2,0)} & \makecell{ $\mathbb{F}_{(\alpha-2)} \oplus \mathbb{F}_{(\alpha-1)}$ \\ $\mathbb{F}_{(\alpha-3)}$ } &{\includegraphics[scale=0.65, margin=0.05em, valign=c]{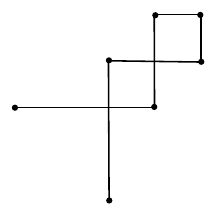}}
	\\
	
	\cline{2-3}
	
	& \makecell{ $\mathbb{F}_{(a)} \oplus \mathbb{F}_{(\alpha-1)}$ \\ $\mathbb{F}_{(a-1)} \oplus \mathbb{F}_{(a-1)} \oplus \mathbb{F}_{(\alpha-2)}$ }  & {\includegraphics[scale=0.65, margin=0.05em, valign=c]{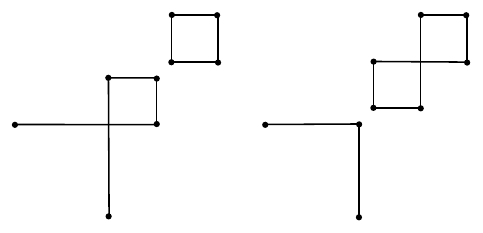}}
	\\
	
	\cline{2-3}
	
	& \makecell{ $\mathbb{F}_{(a)}\oplus \mathbb{F}_{(a+1)}$ \\ $\mathbb{F}_{(a-1)} \oplus \mathbb{F}_{(a)} \oplus \mathbb{F}_{(\alpha-1)}$ }  & {\includegraphics[scale=0.65, margin=0.05em, valign=c]{fig/3-6}} \\
	
	\cline{2-3}
	
	& \makecell{ $\mathbb{F}_{(a)} \oplus \mathbb{F}_{(b)}$ \\ $\mathbb{F}^2_{(a-1)} \oplus \mathbb{F}^2_{(b-1)} \oplus \mathbb{F}_{(\alpha-1)}$ } & {\includegraphics[scale=0.65, margin=0.05em, valign=c]{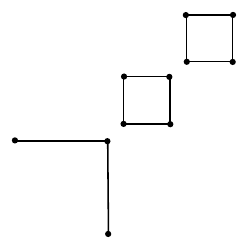}} \\
	
	\hline
\end{tabular}
\end{center}
\caption{}
\label{table:hv-basis of 151 type, g(K)>1-2}
\end{table}

\newpage

\settocdepth{section} % 临时只显示到 section
\subsubsection{(0,2,3,1) case}
\resettocdepth      % 恢复原 tocdepth

The following two short exact sequences \eqref{equation: exact sequence of X_p} and \eqref{equation: exact sequence of Y_p} will be used frequently later:
\begingroup
\setlength{\abovedisplayskip}{2pt}
\setlength{\abovedisplayshortskip}{2pt}
\setlength{\belowdisplayskip}{2pt}
\setlength{\belowdisplayshortskip}{2pt}
\begin{equation}\label{equation: exact sequence of X_p}
    0 \rightarrow UX_{p+1}\stackrel{i}{ \rightarrow} X_{p}\stackrel{j}{ \rightarrow} C\{(0,p)\} \rightarrow 0
\end{equation}
\begin{equation}\label{equation: exact sequence of Y_p}
    0 \rightarrow Y_p \stackrel{i}{\rightarrow} Y_{p+1}\stackrel{j}{\rightarrow} C\{(0, p)\} \rightarrow 0
\end{equation}
\endgroup

Taking $p = 1$ in Equation \eqref{equation: exact sequence of X_p}, we find that $H_*(\{(0, 1)\})$ has rank two. We assume it to be $\mathbb{F}_{(a)} \oplus \mathbb{F}_{(b)}$, which is also equal to $H_*(X_1)$.

Now consider Equation \eqref{equation: exact sequence of X_p} for $p = 0$. The complexes $UX_1$ and $X_0$ have homologies of rank two and three, respectively. The kernel of the map $i_*: H_*(UX_1) \rightarrow H_*(X_0)$ can have rank zero, one, or two, which forces $H_*(\{(0, 0)\})$ to have rank one, three, or five. We therefore divide our discussion into three subcases.

\textbf{Subcase (\romannumeral1)}: The group $H_*(\{(0, 0)\})$ has rank one. We assume it to be $\mathbb{F}_{(c)}$. 

In this subcase, the map $i_*$ is injective and $H_*(X_0) = \mathbb{F}_{(a-2)} \oplus \mathbb{F}_{(b-2)} \oplus \mathbb{F}_{(c)}$. By $\operatorname{Spin}^c$ conjugation \cite[Proposition 3.10]{OSknotinvar}, we obtain $H_*(\{(0, -1)\}) = \mathbb{F}_{(a-2)} \oplus \mathbb{F}_{(b-2)}$. Now consider Equation \eqref{equation: exact sequence of X_p} with $p = -1$, the boundary map $\delta: H_*(\{(0, -1)\}) \rightarrow H_*(UX_0)$ must be injective, and we have $\{a-3, b-3\} \subseteq \{a-4, b-4, c-2\}$. Note that $\{a-3, b-3\} = \{a-4, b-4\}$ is impossible. 

\begingroup
\setlength{\intextsep}{1.5pt}
\begin{figure}[H]
    \centering
    \setlength{\abovecaptionskip}{2pt}
	\includegraphics[scale=0.8]{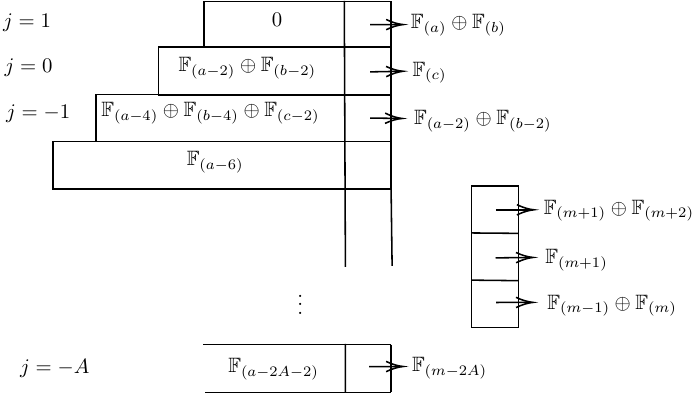}
    \caption{}
    \label{Figure:dim_gra_figure-1}
\end{figure}
\endgroup

The places of $a$ and $b$ are symmetric. Thus, we may assume $\{a-3, b-3\} = \{b-4, c-2\},$ which gives $b = a+1$ and $c = b-1 = a$. Then $H_*(X_{-1}) = \mathbb{F}_{(a-4)}.$ 

Observe that $H_*(X_{-A+1}) = \mathbb{F}_{(a-2A)}$. Consider Equation \eqref{equation: exact sequence of X_p} with $p = -A$, we obtain $H_*(X_{-A}) = 0$ and $H_*(\{(0, -A)\}) = \mathbb{F}_{(m-2A)}$. We then have $m-2A-1 = a-2A-2, a = m+1$; see Figure \ref{Figure:dim_gra_figure-1}.

\textbf{Subcase (\romannumeral2)}: Assume that $H_*(\{(0, 0)\}) = \mathbb{F}_{(c)} \oplus \mathbb{F}_{(d)} \oplus \mathbb{F}_{(e)}$.

Setting $p = 0$ in Equation \eqref{equation: exact sequence of X_p}, the image of $j_*$ must have rank two. We may assume $H_*(X_0) = \mathbb{F}_{(b-2)} \oplus \mathbb{F}_{(c)} \oplus \mathbb{F}_{(d)}$; then $e-1=a-2$ and $e=a-1$.

Setting $p = -1$ in Equation \eqref{equation: exact sequence of X_p}, we obtain that $H_*(X_{-1})$ has rank one. Since $c$ and $d$ are symmetric so far, $H_*(X_{-1})$ is either $\mathbb{F}_{(b-4)}$ or $\mathbb{F}_{(d-2)}$. We have either

\begingroup
\setlength{\abovedisplayskip}{2pt}
\setlength{\abovedisplayshortskip}{2pt}
\setlength{\belowdisplayskip}{2pt}
\setlength{\belowdisplayshortskip}{2pt}
\begin{equation}
\begin{cases}
\begin{array}{l}
	e=a-1  \\
	c-2=a-3\\
	d-2=b-3\\
	b-2A-2=m-2A-1
\end{array}
\end{cases} or\quad
\begin{cases}
\begin{array}{l}
	e=a-1  \\
	b-4=a-3\\
	c-2=b-3\\
	d-2A=m-2A-1
\end{array}
\end{cases}
\end{equation}
\endgroup
Results for the two possibilities are illustrated in Figures \ref{Figure:dim_gra_figure-2} and \ref{Figure:dim_gra_figure-3}.

\begingroup
\setlength{\intextsep}{2pt}
\begin{figure}[H]
    \centering
    \setlength{\abovecaptionskip}{2pt}
	\includegraphics[scale=0.8]{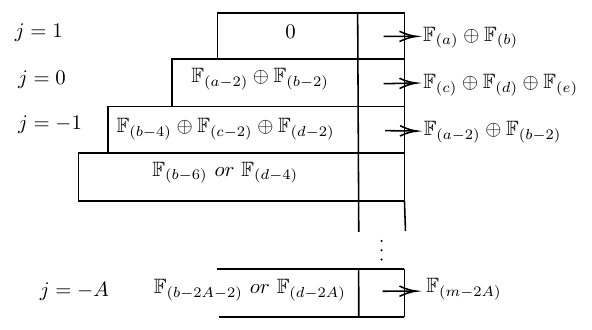}
    \caption{}
    \label{Figure:dim_gra_figure-2}
\end{figure}
\endgroup
\begingroup
\setlength{\intextsep}{2pt}
\begin{figure}[H]
    \centering
    \setlength{\abovecaptionskip}{2pt}
	\includegraphics[scale=0.8]{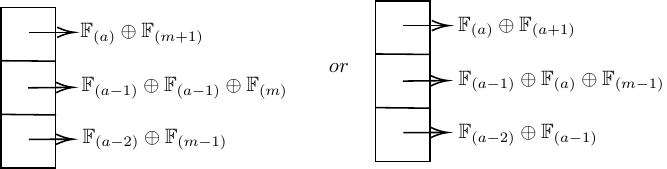}
    \caption{}
    \label{Figure:dim_gra_figure-3}
\end{figure}
\endgroup

\textbf{Subcase (\romannumeral3)}: Assume $H_*(\{(0, 0)\}) = \mathbb{F}_{(c)} \oplus \mathbb{F}_{(d)} \oplus \mathbb{F}_{(e)} \oplus \mathbb{F}_{(f)} \oplus \mathbb{F}_{(g)}$.

In this case $j_*$ is injective. We may assume $H_*(X_0) = \mathbb{F}_{(c)} \oplus \mathbb{F}_{(d)} \oplus \mathbb{F}_{(e)}$ and $f-1 = a-2, g-1 = b-2$.

The map $\delta: H_*(\{(0, -1)\}) \rightarrow H_*(UX_0)$ is injective. We assume that $a-3 = c-2$, $b-3 = d-2$, and $H_*(X_{-1}) = \mathbb{F}_{(e-2)}$. Tracing the $U$-action down to height $-A$, we obtain $e-2A=m-2A-1$, hence $e=m-1$.
\begingroup
\setlength{\intextsep}{2pt}
\begin{figure}[H]
    \centering
    \setlength{\abovecaptionskip}{2pt}
	\includegraphics[scale=0.75]{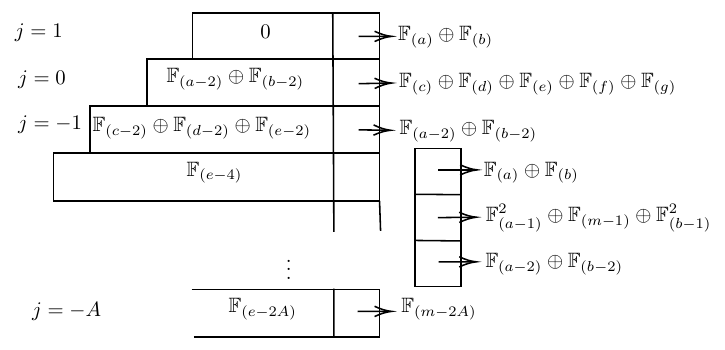}
    \caption{}
    \label{figure: case (iii) of (1,5,1), (0,2,3,1)}
\end{figure}
\endgroup
The results are illustrated in Figure \ref{figure: case (iii) of (1,5,1), (0,2,3,1)}.

\settocdepth{section} % 临时只显示到 section
\subsubsection{(0,3,2,1) case}
\resettocdepth      % 恢复原 tocdepth

Setting $p = 1$ in Equation \eqref{equation: exact sequence of X_p}, we obtain $\operatorname{rank} H_*(UX_2) = 0$ and $\operatorname{rank} H_*(X_1) = 3$. It follows that $\operatorname{rank} H_*(\{(0,1)\}) = 3$. We assume that $H_*(\{(0,1)\}) = \mathbb{F}_{(a)} \oplus \mathbb{F}_{(b)} \oplus \mathbb{F}_{(c)}$.

Setting $p = 0$ in Equation \eqref{equation: exact sequence of X_p}, we obtain $\operatorname{rank} H_*(UX_1) = 3$ and $\operatorname{rank} H_*(X_0) = 2$. The rank of $\operatorname{im}i_*$ determines $\operatorname{rank} H_*(\{(0, 0)\})$, which must be $1, 3, \text{or}\ 5$.  

\textbf{Subcase (i)}: Assume $H_*(\{(0, 0)\}) = \mathbb{F}_{(d)}$. 

In this case, the map $i_*$ is surjective, and the map $\delta:$ $H_*(\{(0, 0)\})$ $\rightarrow$ $H_*(UX_1)$ is injective. Without loss of generality, we may assume $H_*(X_0) = \mathbb{F}_{(a-2)} \oplus \mathbb{F}_{(b-2)}$ and $d-1 = c-2$.

\begingroup
\setlength{\intextsep}{2pt}
\begin{figure}[H]
    \centering
    \setlength{\abovecaptionskip}{2pt}
    \includegraphics[scale=0.8]{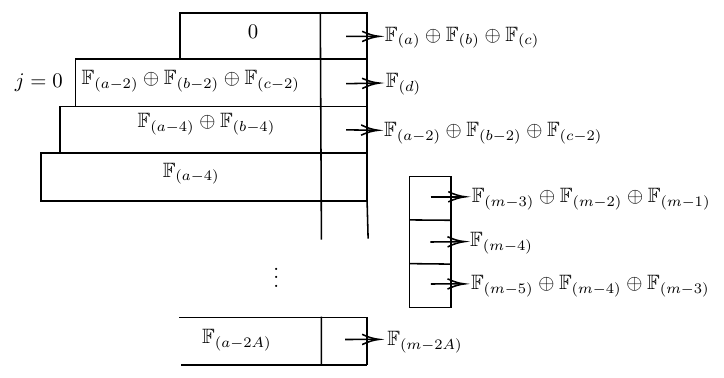}
    \caption{}
    \label{figure: case (i) of (1,5,1), (0,3,2,1)}
\end{figure}
\endgroup

We now focus on height $-1$ and take $p = -1$ in Equation \eqref{equation: exact sequence of X_p}. Since $a$ and $b$ are symmetric so far, we may assume either $\{b-3,c-3\} = \{a-4,b-4\}$ or $\{a-3,b-3\} = \{a-4,b-4\}$. The latter case is impossible by direct verification.

We obtain $b = a-1$, $c = b-1 = a-2$, and $H_*(X_{-1}) = \mathbb{F}_{(a-4)}$. At height $-A$, Equation \eqref{equation: exact sequence of X_p} tells us that $a-2A = m-2A-1$, so $a=m-1$. See Figure \ref{figure: case (i) of (1,5,1), (0,3,2,1)}.

\textbf{Subcase (ii)}: $H_*(\{(0,0)\}) = \mathbb{F}_{(d)} \oplus \mathbb{F}_{(e)} \oplus \mathbb{F}_{(f)}$.
\begingroup
\setlength{\intextsep}{2pt}
\begin{figure}[H]
    \centering
    \setlength{\abovecaptionskip}{2pt}
	\includegraphics[scale=0.8]{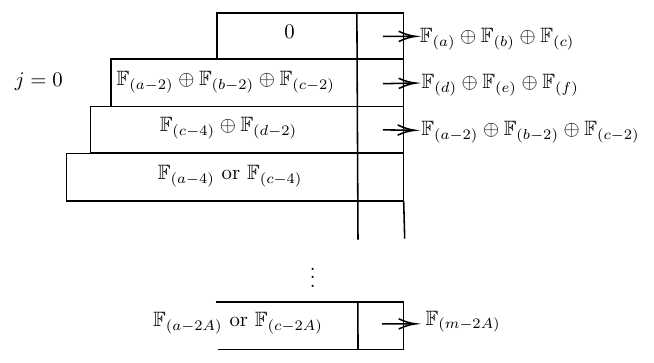}
    \caption{}
    \label{Figure:dim_gra_figure-11}
\end{figure}
\endgroup
Setting $p = 0$ in Equation \eqref{equation: exact sequence of X_p}, the image of $i_*$ has rank one, and the image of the boundary map $\delta$ has rank two. Without loss of generality, we assume $H_*(X_0) = \mathbb{F}_{(c-2)} \oplus \mathbb{F}_{(d)}$, $e-1 = a-2$ and $f-1 = b-2$. See Figure \ref{Figure:dim_gra_figure-11}.

At height -1, the symmetry between $a$ and $b$ gives two possibilities: either $\{b-3,c-3\} = \{c-4,d-2\}$ or $\{a-3,b-3\} = \{c-4,d-2\}$. In the first case, $c=d+1, b=c-1=d$ and $a=m-1$. In the latter case, $c=a+1$, $d=b-1$, and $c=m-1$. These are shown in Figure \ref{Figure:dim_gra_figure-13}.

\begingroup
\setlength{\intextsep}{2pt}
\begin{figure}[H]
    \centering
    \setlength{\abovecaptionskip}{2pt}
	\includegraphics[scale=0.8]{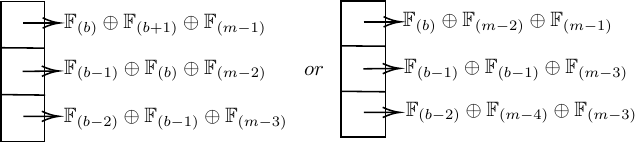}
    \caption{}
    \label{Figure:dim_gra_figure-13}
\end{figure}
\endgroup

\textbf{Subcase (iii)}: Assume $H_*(\{(0, 0)\}) = \mathbb{F}_{(d)} \oplus \mathbb{F}_{(e)} \oplus \mathbb{F}_{(f)} \oplus \mathbb{F}_{(g)} \oplus \mathbb{F}_{(h)}$.

In this case, taking $p = 0$ in Equation \eqref{equation: exact sequence of X_p}, the map $i_*$ is zero and the map $j_*$ is injective. The boundary map $\delta: H_*(\{(0, 0)\}) \rightarrow H_*(UX_1)$ is surjective. Without loss of generality, assume $H_*(UX_1) = \mathbb{F}_{(a-2)} \oplus \mathbb{F}_{(b-2)} \oplus \mathbb{F}_{(c-2)}$ and $H_*(X_0) = \mathbb{F}_{(d)} \oplus \mathbb{F}_{(e)}$. See Figure \ref{Figure:dim_gra_figure-14}.

\begingroup
\setlength{\intextsep}{2pt}
\begin{figure}[H]
    \centering
    \setlength{\abovecaptionskip}{2pt}
	\includegraphics[scale=0.8]{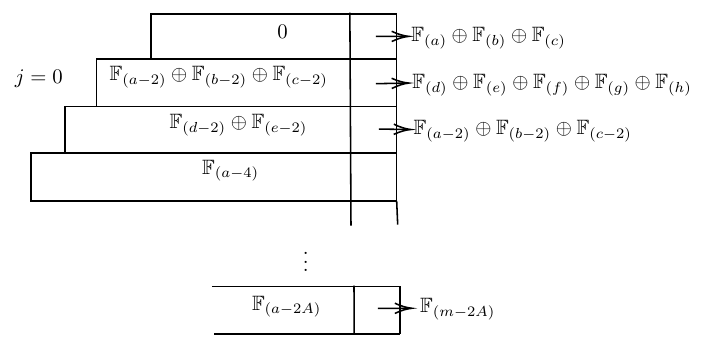}
    \caption{}
    \label{Figure:dim_gra_figure-14}
\end{figure}
\endgroup

Now set $p = -1$ in Equation \eqref{equation: exact sequence of X_p}. Since the places of $a, b, c$ are symmetric, we may assume $b-3 = d-2$ and $c-3 = e-2$. Then $H_*(X_{-1}) = \mathbb{F}_{(a-2)}$.

\begingroup
\setlength{\intextsep}{2pt}
\begin{figure}[H]
    \centering
    \setlength{\abovecaptionskip}{2pt}
	\includegraphics[scale=0.8]{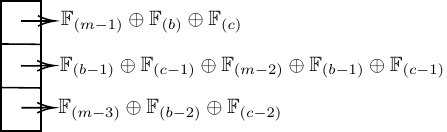}
    \caption{}
    \label{figure: (1,5,1) (0,3,2,1) case (iii)}
\end{figure}
\endgroup

At height $-A$, we have $a-2A = m-2A-1$. The results are illustrated in Figure \ref{figure: (1,5,1) (0,3,2,1) case (iii)}.

\settocdepth{section} % 临时只显示到 section
\subsubsection{(1,2,3,0) case}
\resettocdepth      % 恢复原 tocdepth

In this case, $m$ is even and $H_*(X_A) = \mathbb{F}_{(m)}$. Consider Equation \eqref{equation: exact sequence of X_p} with $p = 2$; we obtain $H_*(X_2) = \mathbb{F}_{(m-2A+4)}$. Let $\alpha = m-2A+2$; then $H_*(UX_2) = \mathbb{F}_{(\alpha)}$.

Taking $p = -1$ in Equation \eqref{equation: exact sequence of X_p}, we obtain $H_*(X_{-1}) = 0$. Moreover, the boundary map $\delta:$ $H_*(C\{(0, -1)\})$ $ \rightarrow$ $H_*(UX_0)$ is an isomorphism of degree -1, and both sides of this map have rank three. Similarly, taking $p = 0$ in Equation \eqref{equation: exact sequence of X_p}, we obtain $\operatorname{rank} H_*(\{(0, 0)\}) \in \{1, 3, 5\}$.

Assume $H_*(\{(0,1)\}) = \mathbb{F}_{(a)} \oplus \mathbb{F}_{(b)} \oplus \mathbb{F}_{(c)}$. Then by $\operatorname{Spin}^c$ conjugation, $H_*(\{(0,-1)\}) = \mathbb{F}_{(a-2)} \oplus \mathbb{F}_{(b-2)} \oplus \mathbb{F}_{(c-2)}$.

\textbf{Subcase (\romannumeral1)}: The group $\widehat{HFK}(K, 0)$ has rank one; assume it is $\mathbb{F}_{(d)}$.

Take $p = 1$ in Equation \eqref{equation: exact sequence of X_p}. The map $i_*$ is zero and $j_*$ is injective. We assume $H_*(X_1) = \mathbb{F}_{(a)} \oplus \mathbb{F}_{(b)}$ and $c-1 = \alpha$.

Now take $p = 0$ in Equation \eqref{equation: exact sequence of X_p}. The map $i_*$ is injective and $j_*$ is surjective. It follows that $H_*(X_0)$ $=$ $\mathbb{F}_{(a-2)} \oplus \mathbb{F}_{(b-2)} \oplus \mathbb{F}_{(d)}$.

Take $p = -1$ in Equation \eqref{equation: exact sequence of X_p}. Since $H_*(X_{-1})$ is zero, the map $\delta:$ $H_*(\{(0, -1)\})$ $\rightarrow$ $H_*(UX_0)$ is an isomorphism of degree -1. 

The relationships we obtain are
$\begin{cases}
\begin{array}{ll}
c=\alpha+1, & c-3=b-4, \\
a-3=d-2,    & b-3=a-4. \\
\end{array}
\end{cases}$

\begingroup
\setlength{\intextsep}{0pt}
\begin{figure}[H]
    \centering
    \setlength{\abovecaptionskip}{2pt}
	\includegraphics[scale=0.8]{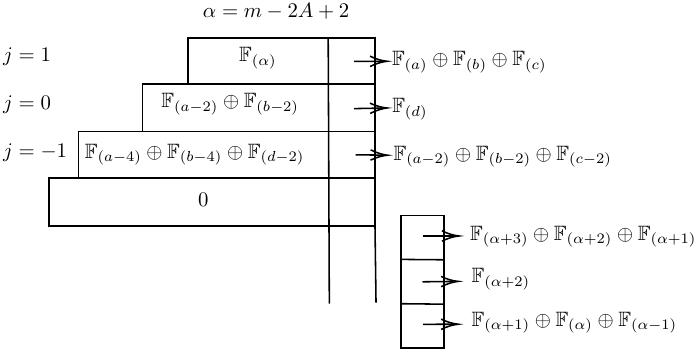}
    \caption{}
    \label{Figure:dim_gra_figure-6}
\end{figure}
\endgroup
From these relations, we can solve for all variables, as shown in Figure \ref{Figure:dim_gra_figure-6}.

\textbf{Subcase (\romannumeral2)}: Assume $\widehat{HFK}(K,0) = \mathbb{F}_{(d)} \oplus \mathbb{F}_{(e)} \oplus \mathbb{F}_{(f)}$.

Again, we assume $c-1 = \alpha$ and $H_*(X_1) = \mathbb{F}_{(a)} \oplus \mathbb{F}_{(b)}$.

Take $p = 0$ in Equation \eqref{equation: exact sequence of X_p}. The image of $i_*$ has rank one and the image of $j_*$ has rank two. We may assume $H_*(X_0) = \mathbb{F}_{(b-2)} \oplus \mathbb{F}_{(d)} \oplus \mathbb{F}_{(e)}$. The boundary map $\delta$ gives $f-1 = a-2$. See Figure \ref{Figure:dim_gra_figure-7}.

\begingroup
\setlength{\intextsep}{0pt}
\begin{figure}[H]
    \centering
    \setlength{\abovecaptionskip}{2pt}
	\includegraphics[scale=0.8]{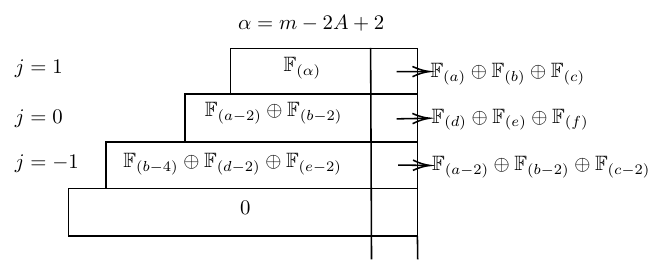}
    \caption{}
    \label{Figure:dim_gra_figure-7}
\end{figure}
\endgroup
Take $p = -1$ in Equation \eqref{equation: exact sequence of X_p}. The boundary map $\delta$ is an isomorphism of degree -1. So far, $d$ and $e$ are symmetric; we may assume $b-3 = d-2$. 
\begingroup
\setlength{\intextsep}{0pt}
\begin{figure}[H]
    \centering
    \setlength{\abovecaptionskip}{2pt}
	\includegraphics[scale=0.8]{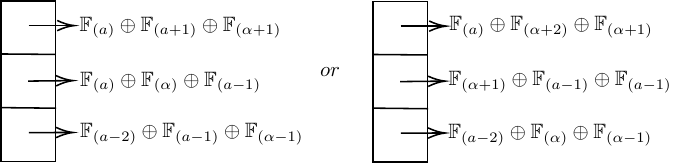}
    \caption{}
    \label{Figure:dim_gra_figure-8}
\end{figure}
\endgroup

For the undetermined variables, either
$a-3=b-4, c-3=e-2\ \text{or}\ a-3=e-2, c-3=b-4$. The results are summarized in Figure \ref{Figure:dim_gra_figure-8}.

\textbf{Subcase (\romannumeral3)}: Assume $\widehat{HFK}(K,0) = \mathbb{F}_{(d)} \oplus \mathbb{F}_{(e)} \oplus \mathbb{F}_{(f)} \oplus \mathbb{F}_{(g)} \oplus \mathbb{F}_{(h)}$.

Without loss of generality, we assume $c-1 = \alpha$ and $H_*(X_1) = \mathbb{F}_{(a)} \oplus \mathbb{F}_{(b)}$.

\begingroup
\setlength{\intextsep}{0pt}
\begin{figure}[H]
    \centering
    \setlength{\abovecaptionskip}{2pt}
	\includegraphics[scale=0.8]{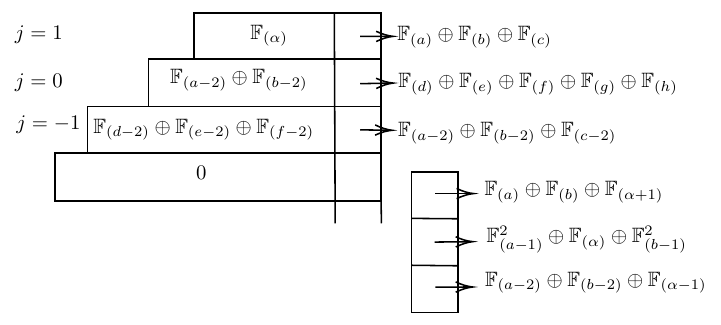}
    \caption{}
    \label{Figure:dim_gra_figure-9}
\end{figure}
\endgroup

In this case, the map $j_*:H_*(X_0) \rightarrow H_*(\{(0, 0)\})$ is injective and $\delta:$ $H_*(\{(0, 0)\})$ $\rightarrow$ $H_*(UX_1)$ is surjective. We may assume $g-1=a-2$ and $h-1=b-2$; then $H_*(X_0)$ $=$ $\mathbb{F}_{(d)} \oplus \mathbb{F}_{(e)} \oplus \mathbb{F}_{(f)}$. See Figure \ref{Figure:dim_gra_figure-9}.

The map $\delta:$ $H_*(\{(0, -1)\})$ $\rightarrow$ $H_*(UX_0)$ is an isomorphism of degree -1. Thus, we obtain $a-3=d-2$, $b-3=e-2$, and $c-3=f-2$. The computational result is shown in Figure \ref{Figure:dim_gra_figure-9}.
\settocdepth{section} % 临时只显示到 section
\subsubsection{(1,3,2,0) case}
\resettocdepth      % 恢复原 tocdepth

As in the $(1,2,3,0)$ case, $m$ is even and $H_*(X_A) = \mathbb{F}_{(m)}$. Let $\alpha = m-2A+2$; then $H_*(UX_2) = \mathbb{F}_{(\alpha)}$.

Take $p = -1$ in Equation \eqref{equation: exact sequence of X_p}. We obtain $H_*(X_{-1}) = 0$ and $H_*(UX_0)$ has rank two. The boundary map $\delta:$ $H_*(\{(0, -1)\})$ $\rightarrow$ $H_*(UX_0)$ is an isomorphism of degree -1.

Analysis of Equation \eqref{equation: exact sequence of X_p} with $p = 0$ shows that $\operatorname{rank} H_*(\{(0, 0)\}) \in \{1, 3, 5\}$.

Assume $H_*(\{(0,1)\}) = \mathbb{F}_{(a)} \oplus \mathbb{F}_{(b)}$. Then $H_*(\{(0,-1)\}) = \mathbb{F}_{(a-2)} \oplus \mathbb{F}_{(b-2)}$ by $\operatorname{Spin}^c$ conjugation.

\textbf{Subcase (i)}: Assume $H_*(\{(0,0)\}) = \mathbb{F}_{(c)}$. 

Note that $H_*(X_1) = \mathbb{F}_{(\alpha)} \oplus \mathbb{F}_{(a)} \oplus \mathbb{F}_{(b)}$. Consequently, $H_*(UX_1) = \mathbb{F}_{(\alpha-2)} \oplus \mathbb{F}_{(a-2)} \oplus \mathbb{F}_{(b-2)}$.

\begingroup
\setlength{\intextsep}{0pt}
\begin{figure}[H]
    \centering
    \setlength{\abovecaptionskip}{0pt}
	\includegraphics[scale=0.8]{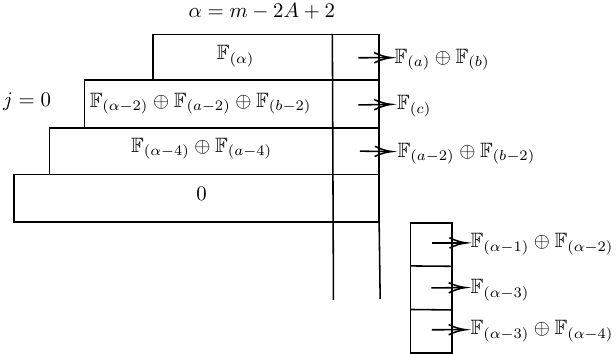}
    \caption{}
    \label{Figure:dim_gra_figure-16}
\end{figure}
\endgroup

Now take $p = 0$ in Equation \eqref{equation: exact sequence of X_p}. The places of $a$ and $b$ are symmetric. The boundary map $\delta$ sends $\mathbb{F}_{(c)}$ to either $\mathbb{F}_{(\alpha-2)}$ or $\mathbb{F}_{(b-2)}$.

If $\delta$ sends $\mathbb{F}_{(c)}$ to $\mathbb{F}_{(\alpha-2)}$, then $H_*(X_0) = \mathbb{F}_{(a-2)} \oplus \mathbb{F}_{(b-2)}$. Taking $p = -1$ in Equation \eqref{equation: exact sequence of X_p}, the map $\delta: H_*(\{(0,-1)\}) \rightarrow H_*(UX_0)$ is an isomorphism, and we obtain $\{a-3,b-3\}=\{a-4,b-4\}$, which is impossible. 

Thus, $\delta$ sends $\mathbb{F}_{(c)}$ to $\mathbb{F}_{(b-2)}$ and $c-1=b-2$.
Setting $p = -1$ in Equation \eqref{equation: exact sequence of X_p} yields $b-3=a-4$ and $a-3=\alpha-4$, as shown in Figure \ref{Figure:dim_gra_figure-16}.

\textbf{Subcase (ii)}: Assume $H_*(\{(0, 0)\}) = \mathbb{F}_{(c)} \oplus \mathbb{F}_{(d)} \oplus \mathbb{F}_{(e)}$.

\begingroup
\setlength{\intextsep}{1pt}
\begin{figure}[H]
    \centering
    \setlength{\abovecaptionskip}{2pt}
	\includegraphics[scale=0.8]{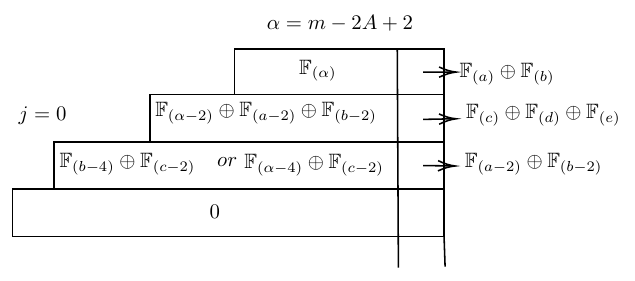}
    \caption{}
    \label{Figure:dim_gra_figure-17}
\end{figure}
\endgroup
\vspace{-4pt}

Again, $H_*(X_1) = \mathbb{F}_{(\alpha)} \oplus \mathbb{F}_{(a)} \oplus \mathbb{F}_{(b)}$. Take $p = 0$ in Equation \eqref{equation: exact sequence of X_p}. The images of $i_*$ and $j_*$ both have rank one. The places of $a$ and $b$ are symmetric, as are those of $c$, $d$, and $e$. There are two possibilities, either $H_*(X_0) = \mathbb{F}_{(\alpha-2)} \oplus \mathbb{F}_{(c)}$ or $H_*(X_0) = \mathbb{F}_{(b-2)} \oplus \mathbb{F}_{(c)}$; see Figure \ref{Figure:dim_gra_figure-17}. 

In the former case, $d-1=a-2$ and $e-1=b-2$. Consider Equation \eqref{equation: exact sequence of X_p} with $p = -1$; the boundary map gives $\{a-3,b-3\} = \{\alpha-4,c-2\}$. Hence $a-3 = \alpha-4$ and $b-3 = c-2$. 

In the second case, $d-1 = \alpha-2$ and $e-1=a-2$. Setting $p = -1$ in Equation \eqref{equation: exact sequence of X_p} gives $a-3=b-4$ and $b-3=c-2$.
\begingroup
\setlength{\intextsep}{1pt}
\begin{figure}[H]
    \centering
    \setlength{\abovecaptionskip}{2pt}
	\includegraphics[scale=0.8]{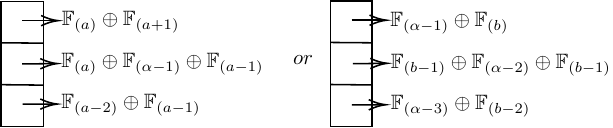}
    \caption{}
    \label{Figure:dim_gra_figure-18}
\end{figure}
\endgroup
\vspace{-4pt}

The computational results are illustrated in Figure
\ref{Figure:dim_gra_figure-18}.

\textbf{Subcase (iii)}: Assume $H_*(\{(0,0)\}) = \mathbb{F}_{(c)} \oplus \mathbb{F}_{(d)} \oplus \mathbb{F}_{(e)} \oplus \mathbb{F}_{(f)} \oplus \mathbb{F}_{(g)}$.

\begingroup
\setlength{\intextsep}{1pt}
\begin{figure}[H]
    \centering
    \setlength{\abovecaptionskip}{1pt}
	\includegraphics[scale=0.8]{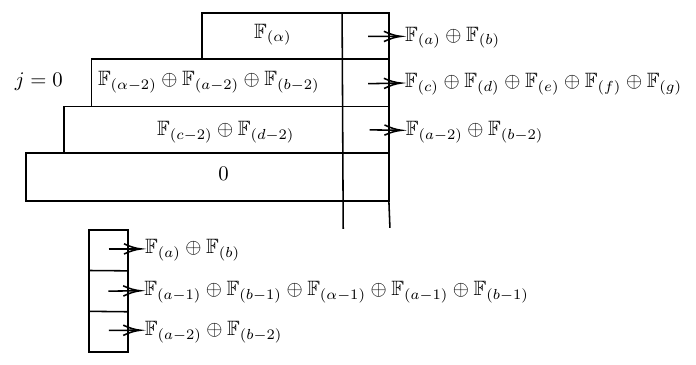}
    \caption{}
    \label{Figure:dim_gra_figure-19}
\end{figure}
\endgroup
In this case, take $p = 0$ in Equation \eqref{equation: exact sequence of X_p}. The map $i_*$ is zero, and $j_*$ is injective.

Chasing the Maslov grading in the exact sequences yields $e-1 = \alpha-2$, $f-1=a-2$, $g-1=b-2$, $a-3=c-2$ and $b-3=d-2$. The result is illustrated in Figure \ref{Figure:dim_gra_figure-19}.

\subsection{Construction of hv-basis}\label{subsection: Construction of hv-basis}

In this section, we prove the existence of an hv-basis for height 4 knots of the $(1,5,1)$ type. In Section \ref{subsection: Dimension and Maslov grading}, we obtained sixteen possibilities for the dimension and Maslov grading, which will be discussed separately.

The proof of Proposition \ref{prop: 151 table first row, positively two-twisted staircase} will present in detail the procedures for determining a vertically simplified basis.
\vspace{-0.5em}

\begin{proposition}\label{prop: 151 table first row, positively two-twisted staircase}
Let $K$ be a height 4 knot of the $(1,5,1)$ type. Suppose that for $p = 2, 1, 0, -1$, we have $\operatorname{rank}H_*({X_p}) = 0, 2, 3, 1$, and $\operatorname{rank}H_*({Y_p}) = 1, 3, 2, 0$, respectively. Assume that
$\widehat{HFK}(K, 1) = \mathbb{F}_{(m+1)} \oplus \mathbb{F}_{(m+2)}$ and $\widehat{HFK}(K, 0) = \mathbb{F}_{(m+1)}$. Then $CFK^\infty(K)$ admits an hv-basis, which is a positive doubly-twisted staircase. 
\end{proposition}

\vspace{-0.5em}

Before giving the proof, we state some notational conventions and lemmas. Recall that the generators of each $\widehat{HFK}(K,i)$ form a filtered basis for $C\{i = 0\}$. 

We denote by $x_{(A,m_A)}$ the generator with Alexander grading $A$ and Maslov grading $m_A$. If there is more than one generator with the same Maslov grading at height $A$, we label them $x_{(A,m_A,1)}$, $x_{(A,m_A,2)}$, $\cdots$, from left to right in the illustrations. Recall the assumption that $g(K) > 1$. We use $e$ and $o$ as abbreviations for "even" and "odd", respectively, indicating the parity of the Maslov grading of the generator. The complex $C\{i = 0\}$ is illustrated in Figure \ref{Figure: 151verbasis-1}. The dots denote the generators, and the number near each generator denotes its Maslov gradng. We often refer to the Alexander filtration level $j$ of a generator at its height.

\begingroup
\setlength{\intextsep}{2pt}
\begin{figure}[H]
    \centering
    \setlength{\abovecaptionskip}{2pt}
    \includegraphics[scale=0.7]{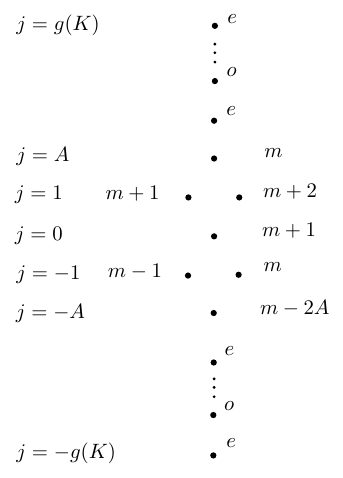}
    \caption{}
    \label{Figure: 151verbasis-1}
\end{figure}
\endgroup
\begin{lemma}\label{lem: differentials below height -A+1} The differentials among generators of height no greater than $-A$ can be described as follows: an adjacent pair of generators in the vertical direction is connected by a vertical differential, where the differential starts from the generators with odd Maslov grading.
\end{lemma}
\vspace{-1.5em}

\begin{proof}
We focus on how $H_*(X_p)$ and $H_*(Y_p)$ change as $p$ increases (while remaining no greater than $-A$). When $p < -g(K)$, we have $H_*(Y_p) = 0$ and  $H_*(X_{p}) = \mathbb{F}$ by Lemma \ref{lemma: 1 and 1 then 1 or 0}.

Suppose $p$ grows to $-g(K)+1$. Note that $H_*(\{(0,-g(K)\}) = \mathbb{F}$, and its Maslov grading must be even (see Remark \ref{rem: lowest generator in C{i=0} must have even Maslov grading} below). Equation \eqref{equation: exact sequence of Y_p} tells us that $H_*(Y_{-g(K)+1})$ becomes $\mathbb{F}$, while $H_*(X_{-g(K)+1})$ becomes 0. Thus, the complex $Y_{-g(K)+1}$ contains a single generator, and there are no differentials in $Y_{-g(K)+1}$.

As $p$ increases further, we eventually reach a height $q$ where $H_*(\{(0,q)\}) = \mathbb{F}_{(odd)}$. The pair $(H_*(X_p)$, $H_*(Y_p))$ then changes from $(0,\mathbb{F})$ to $(\mathbb{F},0)$. This forces $x_{(q, m_q)}$ to differentiate to $x_{(-g(K), -2g(K))}$. Thus, we obtain a pair of generators connected by a vertical differential. The Maslov grading difference and the condition $\partial^2=0$ prevent the generators above height $q$ from differentiating to this pair. 

As $p$ increases, the pair $(H_*(X_p), H_*(Y_p))$ alternates between $(0, \mathbb{F})$ and $(\mathbb{F}, 0)$. In a similar way, we conclude that these generators group into pairs and are connected by vertical differentials as claimed.
\end{proof}
\begin{remark}\label{rem: lowest generator in C{i=0} must have even Maslov grading}
The homology of $C\{i = 0\}$ is the $\widehat{HF}$ of $S^3$, supported in zero grading:
\begingroup
\setlength{\abovedisplayskip}{4pt}
\setlength{\abovedisplayshortskip}{2pt}
\setlength{\belowdisplayskip}{4pt}
\setlength{\belowdisplayshortskip}{2pt}
\[H_*(\{(0, g(K))\}) = \mathbb{F}_{(0)} = H_*(\{i = 0\}) \cong 
\widehat{HF}(S^3)\]
\endgroup
By $\operatorname{Spin}^c$ conjugation \cite[Lemma 3.10]{OSknotinvar}, $H_*(\{(0, -g(K))\}) = \mathbb{F}_{(-2g(K))}$.
\end{remark}
\noindent \textit{Proof of Proposition \ref{prop: 151 table first row, positively two-twisted staircase}}:

The differentials at height $-1, 0, 1$ are also determined under the constraints of $H_*(Y_p)$ and Maslov grading, by applying the grid climbing method and increasing the height by one each time. 

\begingroup
\setlength{\intextsep}{2pt}
\begin{figure}[H]
\centering
    \setlength{\abovecaptionskip}{2pt}
	\begin{minipage}[t]{0.3\textwidth}
       \includegraphics[scale=0.7]{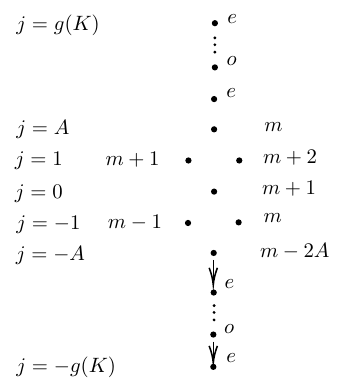} 
    \end{minipage}
	\begin{minipage}[t]{0.3\textwidth}
       \includegraphics[scale=0.7]{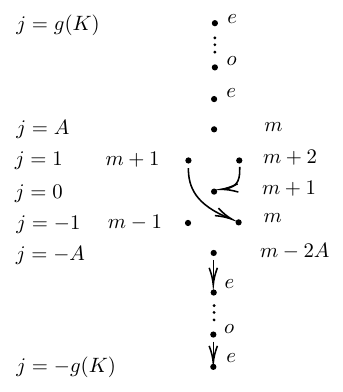} 
    \end{minipage}
    \begin{minipage}[t]{0.3\textwidth}
       \includegraphics[scale=0.7]{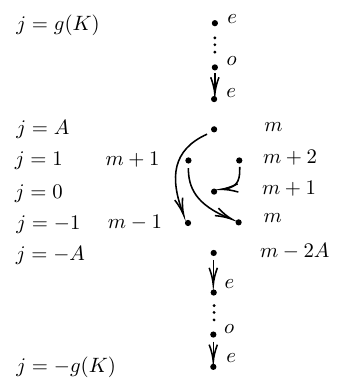} 
    \end{minipage}
    \caption{}
    \label{figure: (1,5,1) vertical differentials first case}
\end{figure} 
\endgroup

The result is illustrated in Figure \ref{figure: (1,5,1) vertical differentials first case}. Generators above height 1 are connected by differentials in a way similar to those in Lemma \ref{lem: differentials below height -A+1}.

Up to this point, we have obtained a vertically simplified basis for $CFK^{\infty}(K)$. We now proceed to make this vertically simplified basis also horizontally simplified. 

Recall that $C=CFK^{\infty}(K)$ is freely generated as an $\mathbb{F}[U,U^{-1}]$-module by the elements of the vertically simplified basis; i.e., the vertically simplified basis is a filtered basis for $CFK^\infty(K)$. We will show that the operations we perform to make this filtered basis horizontally simplified do not affect its vertically simplicity. Consequently, we obtain an hv-basis.  

An essential tool is Equation \eqref{equation:(i,j) symmetry}: $H_*(X_m)\cong H_{*-2m}(Y_{-m+1})$, see \cite[Proposition 3.10]{OSknotinvar} and also
\cite[Lemma 2.16]{HedLevsurformulaforknotFloer}. This gives identical dimensions and Maslov grading conditions (under symmetry) on $H_*(Z_m)$. The case we are dealing with is special in that we perform no basis changes at all to determine the differentials in $C\{i = 0\}$. Constraints from $H_*(Z_m)$ and Maslov grading then determine the horizontal differentials without ambiguity. 
\begingroup
\setlength{\intextsep}{2pt}
\begin{figure}[H]
    \centering
    \setlength{\abovecaptionskip}{2pt}
    \includegraphics[scale=0.8]{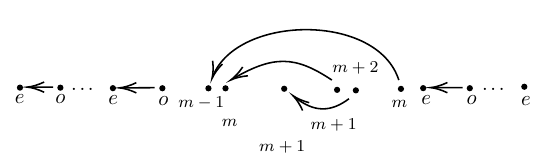}
    \caption{}
    \label{figure: (1,5,1) horizontal differentials first case}
\end{figure}
\endgroup

As a result, the differentials in $C\{j=0\}$ are shown in Figure \ref{figure: (1,5,1) horizontal differentials first case}.

The filtered basis we have obtained so far turns out to be an hv-basis, as illustrated in Figure \ref{figure: (1,5,1) hv-basis first case}. We place the generators and differentials via proper $U$-actions so that each basis element, vertical differential, and horizontal differential appears only once. We denote this subset by $C_{hv}$. 

\begingroup
\setlength{\intextsep}{2pt}
\begin{figure}[H]
    \centering
    \setlength{\abovecaptionskip}{2pt}
    \includegraphics[scale=0.8]{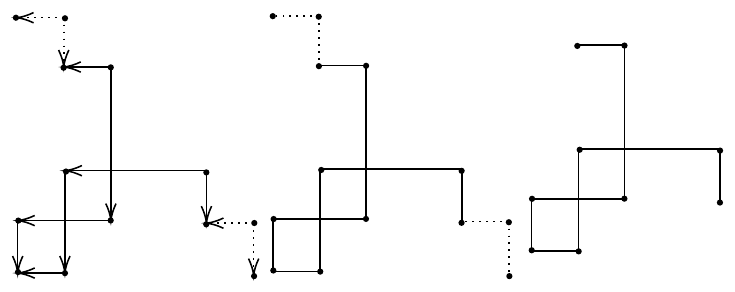}
    \caption{}
    \label{figure: (1,5,1) hv-basis first case}
\end{figure}
\endgroup

Since horizontal differentials must point leftward and vertical differentials must point downward, we may omit the arrows. The dashed lines indicate an indefinite number of generators and differentials. For convenience, we take the case with the minimal number of generators outside height $A$ as our representative. $\hfill \qed$

\vspace{1em}

From now on, we will take several pages to discuss the remaining cases. We introduce Lemma \ref{lemma: hori filtered basis change does not change vertical simplicity}, which describes the effect on vertical differentials when we perform certain filtered basis changes to simplify horizontal differentials.

We denote by $x_{(u,m_u)}$ the generator with coordinates $(0,u)$ and Maslov grading $m_u$, and by $y_{(v, m_v)}$ the generator with coordinates $(v,0)$ and Maslov grading $m_v$.

\vspace{-0.5em}
\begin{lemma}\label{lemma: hori filtered basis change does not change vertical simplicity}
Suppose $\{x_{(i, m_i)}\}$ is a vertically simplified basis for $CFK^\infty(K)$, then the translations $\{y_{(-i, m_{i}-2i)} = U^{i}x_{(i, m_i)}\}$ is a filtered basis for $C\{j = 0\}$; we denote it by $\{y_{(j, n_j)}\}$. We perform a filtered basis change: replace $y_{(j, n_j)}$ by $y'_{(j, n'_j)} = y_{(j, n_j)} + y_{(j-l, n_{j-l})}$, where $l > 0$ and $n_j = n_{j-l}$, while setting $y'_{(k, n'_k)} = y_{(k, n_k)}$ for $k \neq j$. If we translate $\{y'_{(j, n'_j)}\}$ back to a filtered basis $\{x'_{(-j, m'_{-j})} = U^{j} y'_{(j, n'_j)}\}$ for $C\{i = 0\}$, then $\{x'_{(i, m'_i)}\}$ is still a vertically simplified basis with the same vertical differentials as $\{x_{(i, m_i)}\}$, under the identification of $x_{(i,m_i)}$ with $x'_{(i,m'_i)}$. 
\end{lemma}
\vspace{-1.0em}

\begin{proof}

We examine the effects on vertical differentials when we carry the filtered basis change described in the lemma. First, consider differentials starting from $y'_{(j,n'_j)}$. Since $y'_{(j, n'_j)} = y_{(j, n_j)} + y_{(j-l, n_{j-l})}$, every differential from $y_{(j-l,n_{j-l})}$ to $z$ gives rise to a differential from $y'_{(j, n'_j)}$ to $z$.
Because $l > 0$, such a differential strictly decreases the $i$ coordinate (i.e., the $U$-filtration level). Hence, when translating $y'_{(j, n'_j)}$ back to a filtered basis for $C\{i = 0\}$, no vertical differentials are introduced. 

Second, consider differentials ending at $y'_{(j, n'_j)}$. A differential from $w$ to $y_{(j,n_j)}$ becomes two differentials: from $w$ to $y'_{(j, n'_j)}$ and from $w$ to $y_{(j-l, n_{j-l})}$. This also does not add any vertical differentials.
\vspace{-0.5em}
\end{proof}

\begin{remark}
Lemma \ref{lemma: hori filtered basis change does not change vertical simplicity} ensures that under the horizontal filtered basis changes described above, the filtered basis we obtain in the end is both vertically and horizontally simplified.
\end{remark}

% We call the height $1,0,-1$ levels the middle height part. 

\begin{proposition}\label{prop: 151 table second row, a positively twisted staircase plus one box or a staircase plus a negatively twisted box}
Let $K$ be a height 4 knot of the $(1,5,1)$ type. Suppose that for $p = 2, 1, 0, -1$, we have $\operatorname{rank}H_*({X_p}) = 0, 2, 3, 1$ and $\operatorname{rank}H_*({Y_p}) = 1, 3, 2, 0$, respectively. Assume that 
$\widehat{HFK}(K,1) = \mathbb{F}_{(a)} \oplus \mathbb{F}_{(m+1)}$ and $\widehat{HFK}(K,0) = \mathbb{F}_{(a-1)} \oplus \mathbb{F}_{(m)} \oplus \mathbb{F}_{(a-1)}$. Then $CFK^\infty(K)$ admits an hv-basis, which is either a positive twisted staircase plus one box or a staircase plus a twisted box.
\end{proposition}

\begin{proof}
We discuss separately according to the relationship between $a$ and $m$. 

\textit{If $a \notin \{m,m+1,m+2\}$}, under the constraints of Maslov grading and $H_*(Y_p)$, we apply the grid climbing method to determine an hv-basis in a similar way. See Figure \ref{figure: 151-2-1}.

\begingroup
\setlength{\intextsep}{2pt}
\begin{figure}[H]
    \setlength{\abovecaptionskip}{2pt}
    \begin{minipage}[t]{0.34\textwidth}
    \centering
       \includegraphics[scale=0.8]{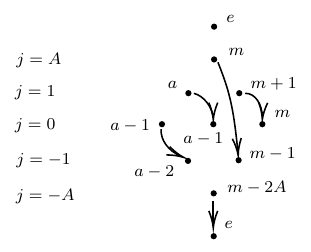} 
    \end{minipage}
	\begin{minipage}[t]{0.3\textwidth}
    \centering
       \includegraphics[scale=0.75]{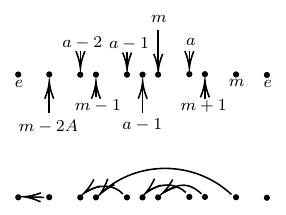} 
    \end{minipage}
    \begin{minipage}[t]{0.34\textwidth}
    \centering
       \includegraphics[scale=0.75]{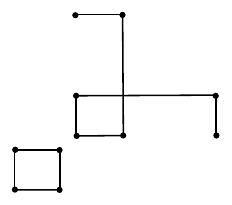} 
    \end{minipage}
    \caption{}
    \label{figure: 151-2-1}
\end{figure}
\endgroup

\textit{If $a = m$}, a vertical arrow starting from $x_{(1, m)}$ ends at either $x_{(0, m-1)}$ or $x_{(-1, m-1)}$; see Figure \ref{figure: 151-2-2}.

\begingroup
\setlength{\intextsep}{2pt}
\begin{figure}[H]
    \setlength{\abovecaptionskip}{2pt}
    \begin{minipage}[t]{0.48\textwidth}
    \centering
       \includegraphics[scale=0.8]{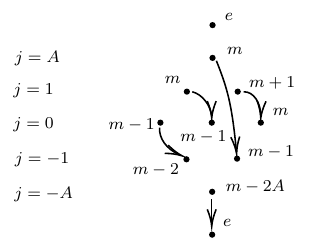} 
       \caption*{case (i)}
    \end{minipage}
	\begin{minipage}[t]{0.48\textwidth}
    \centering
       \includegraphics[scale=0.8]{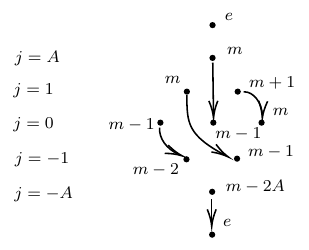} 
       \caption*{case (ii)}
    \end{minipage}
    \caption{}
    \label{figure: 151-2-2}
\end{figure}
\endgroup

\textbf{Case (\romannumeral 1)}: There is a cancellation pair between $C\{(0, 0)\}$ and $C\{(-1, 0)\}$, since $H_*(Z_{-1}) = 0, H_*(Z_{0}) = \mathbb{F}^2$, and $H_*(Z_1) = \mathbb{F}^3$. There are many possibilities for the differentials between $C\{(0, 0)\}$ and $C\{(-1, 0)\}$. We apply Lemma \ref{lemma: cancellation lemma} to $A_0$ (see Figure \ref{figure: H_*(A_0) and cancellation lemma}). 

\begingroup
\setlength{\intextsep}{2pt}
\begin{figure}[H]
    \centering
    \setlength{\abovecaptionskip}{2pt}
	\begin{minipage}[t]{0.4\textwidth}
       \includegraphics[scale=0.8]{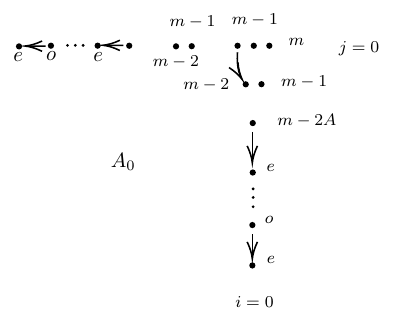} 
    \end{minipage}
	\begin{minipage}[t]{0.4\textwidth}
       \includegraphics[scale=0.8]{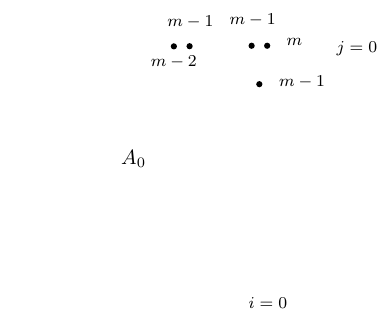} 
    \end{minipage}
    \caption{Cancellation lemma}
    \label{figure: H_*(A_0) and cancellation lemma}
\end{figure}
\endgroup

After cancellation, five generators remain in $A_0$, with possibly undetermined differentials between $C\{(0, 0)\}$ and $C\{(-1, 0)\}$. The condition $H_*(A_0) = \mathbb{F}^5$ prohibits differentials among the remaining five generators. Thus, the only differential from $C\{(0, 0)\}$ to $C\{(-1, 0)\}$ is the arrow from $y_{(0, m-1, 1)}$ to $y_{(-1, m-2)}$. 

Now consider the undetermined horizontal differentials. By Lemma \ref{lemma: 1xn lemma} and applying the $\partial^2=0$ condition to a $1\times 1$ square, there must be a differential from $y_{(1, m)}$ to $y_{(0, m-1, 2)}$. It is possible that $y_{(1, m)}$ also differentiates to $y_{(-1, m-1)}$. If so, we replace $y_{(0, m-1, 2)}$ by $y_{(0, m-1, 2)} + y_{(-1, m-1)}$. We then assume that $y_{(1, m)}$ differentiates only to $y_{(0, m-1, 2)}$. Figure \ref{figure: 151-2-3} shows the effect of this filtered basis change, which simplifies horizontal differentials but introduces additional diagonal differentials.

\begingroup
\setlength{\intextsep}{2pt}
\begin{figure}[H]
    \setlength{\abovecaptionskip}{2pt}
    \begin{minipage}[c]{0.48\textwidth}
    \centering
       \includegraphics[scale=0.8]{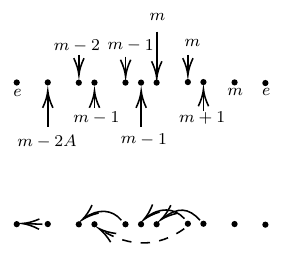} 
    \end{minipage}
	\begin{minipage}[c]{0.48\textwidth}
    \centering
       \includegraphics[scale=0.8]{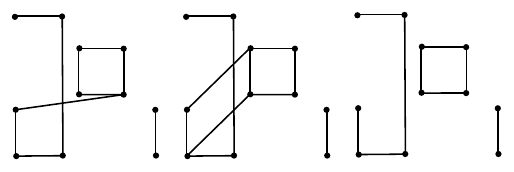} 
    \end{minipage}
    \caption{}
    \label{figure: 151-2-3}
\end{figure}
\endgroup

Next, if $y_{(A, m)}$ differentiates to $y_{(0, m-1, 2)} + y_{(-1, m-1)}$, we replace $y_{(A, m)}$ by $y_{(A, m)} + y_{(1, m)}$. We then obtain an hv-basis, which is a twisted staircase plus one box, as shown in Figure \ref{figure: 151-2-4}.

\begingroup
\setlength{\intextsep}{2pt}
\begin{figure}[H]
    \centering
    \setlength{\abovecaptionskip}{2pt}
       \includegraphics[scale=0.8]{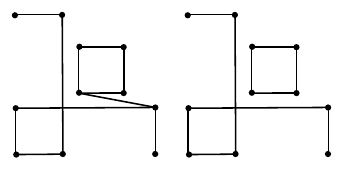} 
    \caption{}
    \label{figure: 151-2-4}
\end{figure}
\endgroup

\textbf{Case (\romannumeral 2)}: By the same argument as in Figure \ref{figure: H_*(A_0) and cancellation lemma}, the differentials between $C\{(0, 0)\}$ and $C\{(-1, 0)\}$ go from $y_{(0, m-1, 1)}$ to $y_{(-1, m-2)}$.

The generator $y_{(1, m)}$ cannot differentiate to $y_{(0, m-1, 2)}$ since this will result in a incomplete $1\times 1$ square, contradicting the $\partial^2=0$ condition, thus, $y_{(1, m)}$ must differentiate to $y_{(-1, m-1)}$. It follows that $y_{(1, m+1)}$ differentiates to $y_{(0, m)}$, because the homology conditions on $Z_m$ require two cancellation pairs starting from $C\{(1, 0)\}.$

It is possible that $y_{(A, m)}$ differentiates to $y_{(0, m-1,2)} + y_{(-1, m-1)}$, if so, we replace $y_{(A, m)}$ by $y_{(A, m)} + y_{(1, m)}$. We then obtain an hv-basis, which is a staircase plus a twisted box; see Figure \ref{figure: 151-2-5}.

\begingroup
\setlength{\intextsep}{2pt}
\begin{figure}[H]
    \centering
    \setlength{\abovecaptionskip}{2pt}
       \includegraphics[scale=0.8]{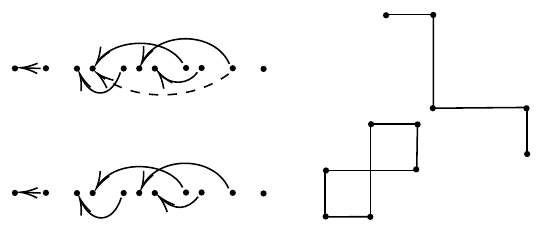} 
    \caption{}
    \label{figure: 151-2-5}
\end{figure}
\endgroup

\textit{If $a=m+1$}, we determine the vertically simplified basis for $C\{i = 0\}$ in a similar way, as illustrated in Figure \ref{figure: 151-2-6}.

\begingroup
\setlength{\intextsep}{2pt}
\begin{figure}[H]
    \centering
    \setlength{\abovecaptionskip}{2pt}
       \includegraphics[scale=0.8]{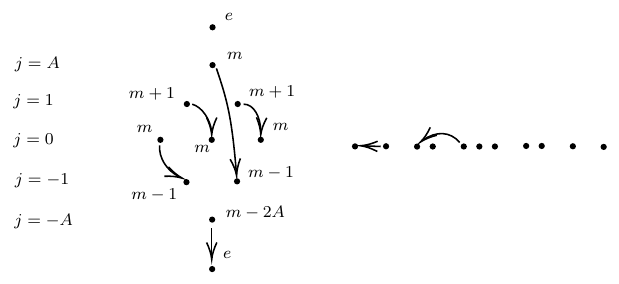} 
    \caption{}
    \label{figure: 151-2-6}
\end{figure}
\endgroup

In this case, by the same argument as in Figure \ref{figure: H_*(A_0) and cancellation lemma}, $y_{(0, m, 1)}$ differentiates to a linear combination of $y_{(-1, m-1, 1)}$ and $y_{(-1, m-1, 2)}$, while the generators $y_{(0, m, 2)}$ and $y_{(0, m, 3)}$ cannot differentiate to $C\{(-1, 0)\}$.

We perform a simultaneous basis change on $y_{(-1, m-1, 1)}$, $y_{(-1, m-1, 2)}$ 
and on $x_{(0,m,2)}$, $x_{(0,m,3)}$, using an appropriate matrix $\left(\begin{smallmatrix}
    p_{11} & p_{12}\\
    p_{21} & p_{22}
\end{smallmatrix}\right)$, as explained in Equation \ref{equation: simultaneous basis change}. This allows us to assume that the only differential from $C\{(0, 0)\}$ to $C\{(-1, 0)\}$ is the one from $y'_{(0, m, 1)}$ to $y'_{(-1, m-1, 1)}$, and that the differentials from $x'_{(1, m+1, 1)}$, $x'_{(1, m+1, 2)}$ to $x'_{(0, m, 2)}$, $x'_{(0, m, 3)}$ remain one to one, as in the vertically simplified basis $\{x_{(i, m_i)}\}$. 
\begingroup
\setlength{\abovedisplayskip}{3pt}
\setlength{\abovedisplayshortskip}{3pt}
\setlength{\belowdisplayskip}{3pt}
\setlength{\belowdisplayshortskip}{3pt}
\begin{equation}
\label{equation: simultaneous basis change}
\begin{cases}
\begin{array}{l}
(y'_{(-1,m-1,1)}, y'_{(-1,m-1,2)}) =(y_{(-1,m-1,1)}, y_{(-1,m-1,2)})
\begin{pmatrix}
    p_{11} & p_{12}\\
    p_{21} & p_{22}
\end{pmatrix} \\
(x'_{(0,m,2)}, x'_{(0,m,3)}) =(x_{(0,m,2)}, x_{(0,m,3)})\begin{pmatrix}
    p_{11} & p_{12}\\
    p_{21} & p_{22}
\end{pmatrix}
\end{array}
\end{cases}
\end{equation}
\endgroup

For convenience, we still denote $y_{(j, m_j)}$ instead of $y'_{(j, m_j)}$, and similarly for $x_{(i, m_i)}$. We now discuss the differentials from $y_{(1, m+1, 1)}, y_{(1, m+1, 2)}$ to $y_{(0, m, 2)}, y_{(0, m, 3)}$. By Lemma \ref{lemma: 1xn lemma}, there is an arrow from $y_{(1, m+1, 1)}$ to $y_{(0, m, 2)}$, and $y_{(1, m+1, 1)}$ does not differentiate to $y_{(0, m, 3)}$. This forces $y_{(1, m+1, 2)}$ to differentiate to $y_{(0, m, 3)}$. If $y_{(1, m+1, 2)}$ also differentiates to $y_{(0, m, 2)}$, we replace $y_{(1, m+1, 2)}$ by $y_{(1, m+1, 2)} + y_{(1, m+1, 1)}$. This adds a vertical differential from $x_{(A, m)}$ to $x_{(-1, m-1, 1)}$. We can replace $x_{(A, m)}$ by $x_{(A, m)} + x_{(0, m, 1)}$ to eliminate this additional vertical differential.

\begingroup
\setlength{\intextsep}{2pt}
\begin{figure}[H]
    \centering
    \setlength{\abovecaptionskip}{2pt}
       \includegraphics[scale=0.75]{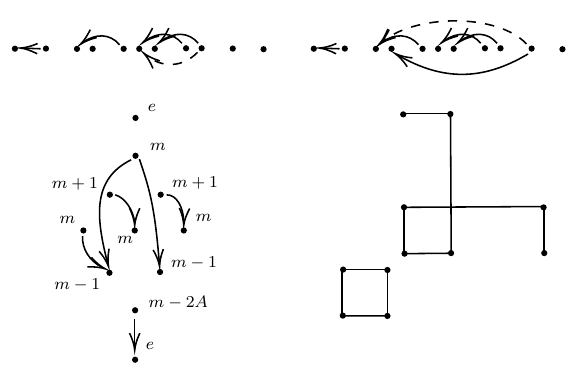} 
    \caption{}
    \label{figure: 151-2-7}
\end{figure}
\endgroup

The generator $y_{(A, m)}$ differentiates to a linear combination of $y_{(-1, m-1, 1)}$ and $y_{(-1, m-1, 2)}$. If necessary, we replace $y_{(A, m)}$ by $y_{(A, m)} + y_{(0, m, 1)}$ and assume that $y_{(A, m)}$ differentiates to $y_{(-1, m-1, 2)}$. The hv-basis we obtain is a positively twisted staircase plus one box, see Figure \ref{figure: 151-2-7}.

\textit{If $a=m+2$}, the vertically simplified basis can be determined in a similar way, under the constraints of Maslov grading and $H_*(Y_p)$, as shown on the left in Figure \ref{figure: 151-2-8}.

\begingroup
\setlength{\intextsep}{2pt}
\begin{figure}[H]
    \setlength{\abovecaptionskip}{2pt}
    \begin{minipage}[c]{0.66\textwidth}
    \centering
       \includegraphics[scale=0.8]{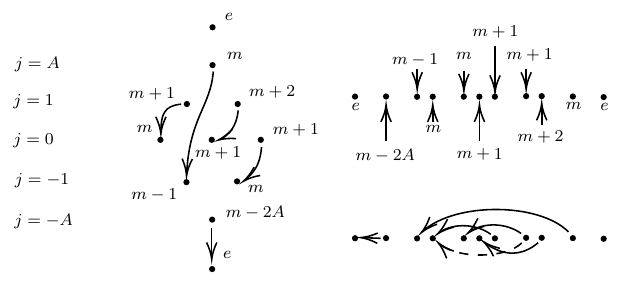} 
    \end{minipage}
	\begin{minipage}[c]{0.3\textwidth}
    \centering
       \includegraphics[scale=0.8]{fig/151-2-5.pdf} 
    \end{minipage}
    \caption{}
    \label{figure: 151-2-8}
\end{figure}
\endgroup

In this case, applying Lemma \ref{lemma: cancellation lemma} to $A_0$ shows that the arrow from $y_{(0, m+1, 2)}$ to $y_{(-1, m)}$ is the only differential between $C\{(0, 0)\}$ and $C\{(-1, 0)\}$. By Lemma \ref{lemma: 1xn lemma}, the generator $y_{(1, m+2)}$ differentiates to $y_{(0, m+1, 1)}$, and $y_{(1, m+1)}$ differentiates to $y_{(0, m)}$. Note that $y_{(1, m+1)}$ may also differentiate to $y_{(-1,m)}$, in that case we replace $y_{(1, m+1)}$ by $y_{(1, m+1)} + y_{(0, m+1, 2)}$. The hv-basis we obtain is a positively twisted staircase plus one box.
\end{proof}

% \begin{lemma}\label{lemma: 1x2 rectangle lemma}
% Assume three generators $x,y,z$ in $CFK^\infty(K)$ have coordinates (1,2),(1,0),(0,0) respectively. $x$ differentiates to $y$, $y$ differentiates to $z$. Then either
% \begin{enumerate}[label=(\arabic*)]
%     \item a horizontal or a vertical differential starts from $x$,

%     \item or a horizontal or a vertical differential ends at $z$.
    
% \end{enumerate}
% \end{lemma}

% \begin{proof}

% The relative positions of these three generators are illustrated below. Similar to the argument in the proof of Lemma \ref{lemma: 1xn lemma}, if $x$ shall differentiate to $z$ through two steps, $x$ can only first differentiate to $(0,2),(0,1),(1,1),(1,0)$ positions. These four possibilities correspond to the four possible results in the statement.

% \begin{figure}[H]
%     \centering
%        \includegraphics[scale=0.8]{fig/1times2_rectangle_lemma.pdf} 
%     \caption{}
%     \label{figure: 1x2 rectangle lemma}
% \end{figure}

% \end{proof}
\vspace{-0.5em}

\begin{proposition}\label{proposition: 151 table third row, either a positive twisted staircase plus a box or a staircase plus a twisted box}
Let $K$ be a height 4 knot of the $(1,5,1)$ type. Suppose that for $p = 2,1,0,-1$, we have $\operatorname{rank}H_*({X_p}) = 0,2,3,1$ and $\operatorname{rank}H_*({Y_p}) = 1,3,2,0$, respectively. Assume that 
$\widehat{HFK}(K,1) = \mathbb{F}_{(a)} \oplus \mathbb{F}_{(a+1)}$ and $\widehat{HFK}(K,0) = \mathbb{F}_{(a)} \oplus \mathbb{F}_{(m-1)} \oplus \mathbb{F}_{(a-1)}$. Then $CFK^\infty(K)$ admits an hv-basis, which is either a positive twisted staircase plus one box or a staircase plus a twisted box.
\end{proposition}
\vspace{-0.5em}

\begin{proof}
\textit{If $a\notin\{m-1,m\}$}, the vertically simplified basis and the horizontal differentials can be determined under the constraints of Maslov grading, $H_*(Y_p)$ and $H_*(Z_p)$. The hv-basis we obtain is a staircase plus a twisted box, as shown in Figure \ref{figure: 151-2-9}.

\begingroup
\setlength{\intextsep}{2pt}
\begin{figure}[H]
    \setlength{\abovecaptionskip}{2pt}
    \begin{minipage}[c]{0.66\textwidth}
    \centering
       \includegraphics[scale=0.8]{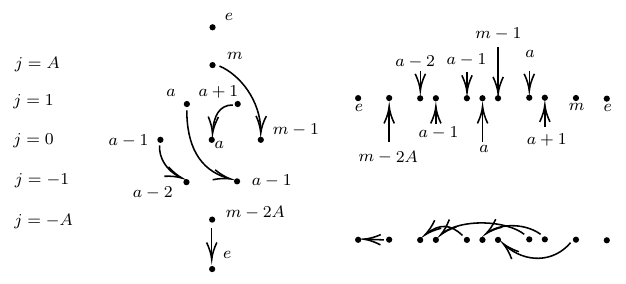} 
    \end{minipage}
	\begin{minipage}[c]{0.3\textwidth}
    \centering
       \includegraphics[scale=0.8]{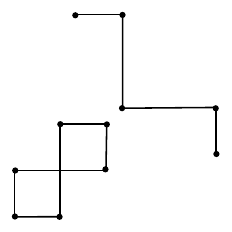} 
    \end{minipage}
    \caption{}
    \label{figure: 151-2-9}
\end{figure}
\endgroup

\textit{If $a=m$}, there are two possible vertically simplified bases.

\textbf{Case (\romannumeral 1)}: The vertically simplified basis is the same as the above one. By Lemma \ref{lemma: cancellation lemma} and the condition $H_*(A_0) = \mathbb{F}^5$, the generator $y_{(0, m-1, 1)}$ differentiates to $y_{(-1, m-2)}$. The generator $y_{(0, m-1, 2)}$ does not differentiate to $y_{(-1, m-2)}$ due to the $\partial^2=0$ constraint. The generator $y_{(1, m)}$ cannot differentiate to $y_{(0, m-1, 2)}$ by Lemma \ref{lemma: 1xn lemma}, since there is a vertical differential from $x_{(0,m-1,1)}$ to $x_{(-1,m-2)}$. By Maslov grading, $y_{(1,m)}$ must differentiate to $y_{(-1, m-1)}$. The hv-basis we obtain is again a staircase plus a twisted box; see Figure \ref{figure: 151-2-10}.

\begingroup
\setlength{\intextsep}{2pt}
\begin{figure}[H]
    \setlength{\abovecaptionskip}{2pt}
    \begin{minipage}[c]{0.66\textwidth}
    \centering
       \includegraphics[scale=0.8]{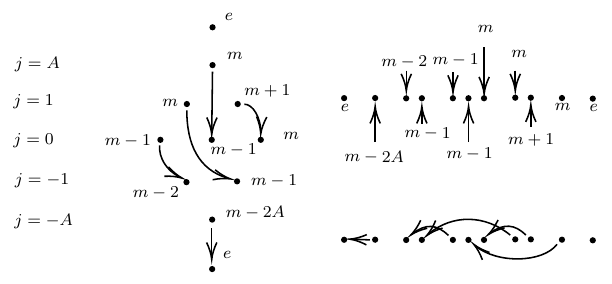} 
    \end{minipage}
	\begin{minipage}[c]{0.3\textwidth}
    \centering
       \includegraphics[scale=0.8]{fig/151-2-16.pdf} 
    \end{minipage}
    \caption{}
    \label{figure: 151-2-10}
\end{figure}
\endgroup

\textbf{Case (\romannumeral 2)}: The differential from $C\{(0, 0)\}$ to $C\{(-1, 0)\}$ is the same as in Case (i), and $y_{(1, m)}$ differentiates to a linear combination of $y_{(0, m-1, 2)}$ and $y_{(-1, m-1)}$. We replace $y_{(0, m-1, 2)}$ by $y_{(0, m-1, 2)} + y_{(-1, m-1)}$ if necessary and assume that $y_{(1, m)}$ differentiates to $y_{(0, m-1, 2)}$. 

If $y_{(A, m)}$ differentiates to $y_{(0, m-1, 2)} + y_{(-1, m-1)}$, we replace $y_{(A, m)}$ by $y_{(A, m)} + y_{(1, m)}$. We then obtain an hv-basis, which is a positively twisted staircase plus one box; see Figure \ref{figure: 151-2-11}.

\begingroup
\setlength{\intextsep}{2pt}
\begin{figure}[H]
    \setlength{\abovecaptionskip}{2pt}
    \begin{minipage}[c]{0.66\textwidth}
    \centering
       \includegraphics[scale=0.8]{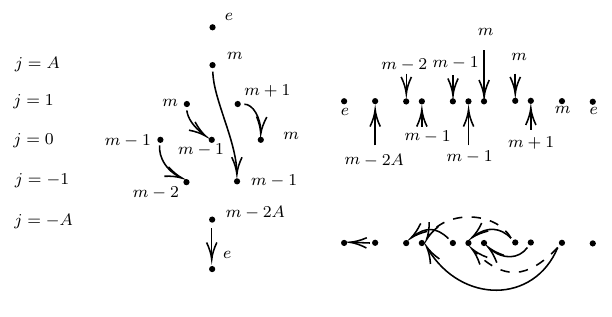} 
    \end{minipage}
	\begin{minipage}[c]{0.3\textwidth}
    \centering
       \includegraphics[scale=0.8]{fig/151-2-5.pdf} 
    \end{minipage}
    \caption{}
    \label{figure: 151-2-11}
\end{figure}
\endgroup

\textit{If $a = m-1$}, there is a differential from $y_{(0, m-2)}$ to $y_{(-1, m-3)}$, and for similar reasons $y_{(1, m-1)}$ must differentiate to $y_{(-1, m-2)}$. We discuss the differentials from the two generators with Maslov grading $m$, $y_{(1, m)}$ and $y_{(A, m)}$, to $y_{(0, m-1, 1)}$ and $y_{(0, m-1, 2)}$, as follows. 

Suppose $y_{(1, m)} = U^{-1} x_{(-1, m-2)}$ differentiates to $y_{(0, m-1, 1)}$. By Lemma \ref{lemma: 1xn lemma}, $U^{-1}x_{(1, m-1)}$ must differentiate horizontally to $x_{(A, m)}$, but we know that $x_{(-1,m-3)}$ is not the starting point of a horizontal differential, a contradiction (see the right side of Figure \ref{figure: 151-2-12}).

\begingroup
\setlength{\intextsep}{2pt}
\begin{figure}[H]
    \setlength{\abovecaptionskip}{2pt}
    \begin{minipage}[c]{0.66\textwidth}
    \centering
       \includegraphics[scale=0.8]{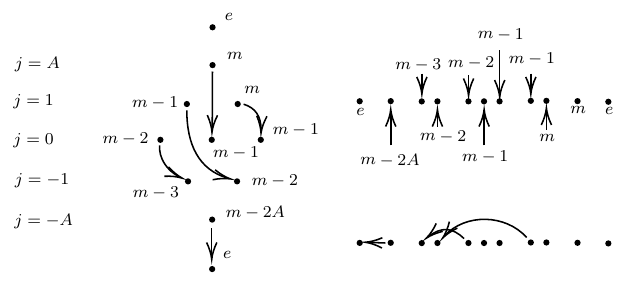} 
    \end{minipage}
	\begin{minipage}[c]{0.3\textwidth}
    \centering
       \includegraphics[scale=0.8]{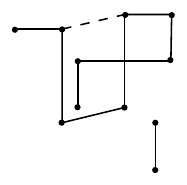} 
    \end{minipage}
    \caption{}
    \label{figure: 151-2-12}
\end{figure}
\endgroup

This forces $y_{(1, m)}$ to differentiate to $y_{(0, m-1, 2)}$. Then $y_{(A, m)}$ differentiates to either $y_{(0, m-1, 1)}$ or $y_{(0, m-1, 1)} + y_{(0, m-1, 2)}$. In the latter case, we replace $y_{(A, m)}$ by $y_{(A, m)} + y_{(1, m)}$. The hv-basis we obtain is a staircase plus a twisted box, as shown in Figure \ref{figure: 151-2-13}.

\begingroup
\setlength{\intextsep}{2pt}
\begin{figure}[H]
    \setlength{\abovecaptionskip}{2pt}
    \begin{minipage}[c]{0.55\textwidth}
    \centering
       \includegraphics[scale=0.8]{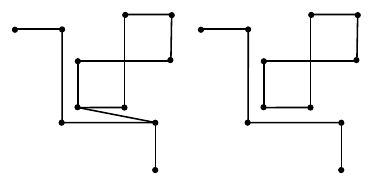} 
    \end{minipage}
	\begin{minipage}[c]{0.4\textwidth}
    \centering
       \includegraphics[scale=0.8]{fig/151-2-16.pdf} 
    \end{minipage}
    \caption{}
    \label{figure: 151-2-13}
\end{figure}
\endgroup

\end{proof}

\begin{proposition}\label{prop: 151 table 4th row, a staircase plus two positive boxes}
Let $K$ be a height 4 knot of the $(1,5,1)$ type. Suppose that for $p = 2, 1, 0, -1$ we have $\operatorname{rank}H_*({X_p}) = 0, 2, 3, 1$ and $\operatorname{rank}H_*({Y_p}) = 1, 3, 2, 0$, respectively. Assume that
$\widehat{HFK}(K,1) = \mathbb{F}_{(a)} \oplus \mathbb{F}_{(b)}$ and $\widehat{HFK}(K,0) = \mathbb{F}^2_{(a-1)} \oplus \mathbb{F}_{(m-1)} \oplus \mathbb{F}^2_{(b-1)}$. Then $CFK^\infty(K)$ admits an hv-basis, which is a staircase plus two boxes, both located in the central position.
\end{proposition}

\begin{proof}
\textit{If $a \neq m, b \neq m, a \neq b$}, through necessary basis changes, we can assume the generators at height 0 are placed in a suitable order and obtain a vertically simplified basis as shown on the left of Figure \ref{figure: 151-2-14}. 

\begingroup
\setlength{\intextsep}{2pt}
\begin{figure}[H]
    \setlength{\abovecaptionskip}{2pt}
    \begin{minipage}[t]{0.66\textwidth}
    \centering
       \includegraphics[scale=0.8]{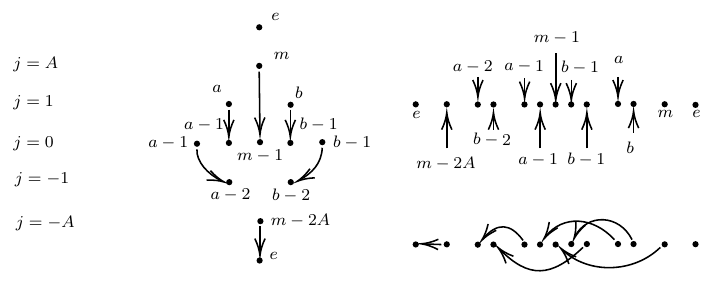} 
    \end{minipage}
	\begin{minipage}[t]{0.3\textwidth}
    \centering
       \includegraphics[scale=0.8]{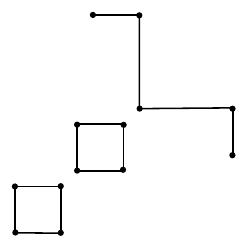} 
    \end{minipage}
    \caption{}
    \label{figure: 151-2-14}
\end{figure}
\endgroup

By the cancellation lemma \ref{lemma: cancellation lemma}, we carry the same argument as in figure \ref{figure: H_*(A_0) and cancellation lemma}. We can determine the differentials between $C\{(0, 0)\}$ and $C\{(-1, 0)\}$. The differentials from $C\{(1, 0)\}$ to $C\{(0, 0)\}$ are determined by applying Lemma \ref{lemma: 1xn lemma} and the contraints of $H_*(Z_m)$ and Maslov grading. The hv-basis we obtain is a staircase plus two boxes, both located in the central position.

\textit{If $a=m, b\neq m$}, we can place the three generators with Maslov grading $m-1$ at height $0$ in a suitable order to obtain a vertically simplified basis identical to the case just discussed. The remaining differentials are determined by constraints of $H_*(Z_m)$ and Maslov grading, and applying Lemma \ref{lemma: 1xn lemma} or the $\partial^2=0$ constraint. We again obtain an hv-basis as shown in Figure \ref{figure: 151-2-14}.

\textit{If $a=b=m$}, we perform the simultaneous basis change described in Equation \eqref{equation: simultaneous basis change} on the differentials from $y_{(0,m-1,1)}$ and $y_{(0,m-1,5)}$ to $y_{(-1,m-2,1)}$ and $y_{(-1,m-2,2)}$, preserving the one-to-one correspondence from $y_{(-1,m-2,1)}$ and $y_{(-1,m-2,2)}$ to $Uy_{(0,m-1,2)}$ and $Uy_{(0,m-1,4)}$; see Figure \ref{figure: 151-2-15}. This yields an identical vertically simplified basis. The remaining discussion is similar, and we obtain an hv-basis as shown in Figure \ref{figure: 151-2-14}.

\begingroup
\setlength{\intextsep}{2pt}
\begin{figure}[H]
    \setlength{\abovecaptionskip}{2pt}
    \begin{minipage}[c]{0.66\textwidth}
    \centering
       \includegraphics[scale=0.8]{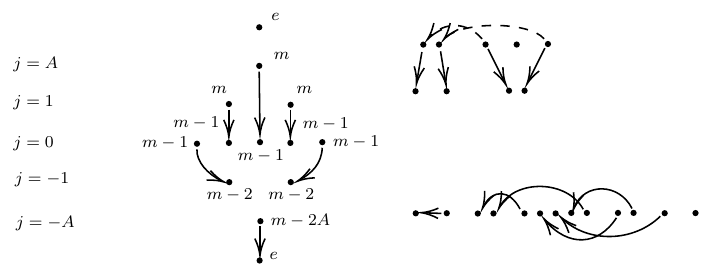} 
    \end{minipage}
	\begin{minipage}[c]{0.3\textwidth}
    \centering
       \includegraphics[scale=0.8]{fig/151-4-2.pdf} 
    \end{minipage}
    \caption{}
    \label{figure: 151-2-15}
\end{figure}
\endgroup

\textit{If $a=b\neq m$}, we perform a simultaneous basis change on the differentials from $y_{(0,a-1,1)}$ and $y_{(0,a-1,4)}$ to $y_{(-1,a-2,1)}$ and $y_{(-1,a-2,2)}$, preserving the one-to-one correspondence from $y_{(-1,a-2,1)}$ and $y_{(-1,a-2,2)}$ to $Uy_{(0,a-1,2)}$ and $Uy_{(0,a-1,3)}$. Through a similar discussion, we obtain an hv-basis, which is again a staircase plus two boxes.
\end{proof}

\begin{proposition}\label{prop: 151 table 5th row, a negative twisted staircase}
Let $K$ be a height 4 knot of the $(1,5,1)$ type. Suppose that for $p = 2,1,0,-1$, we have $\operatorname{rank}H_*({X_p}) = 0,3,2,1$ and $\operatorname{rank}H_*({Y_p}) = 1,2,3,0$, respectively. Assume that
$\widehat{HFK}(K,1) = \mathbb{F}_{(m-1)} \oplus \mathbb{F}_{(m-2)} \oplus \mathbb{F}_{(m-3)}$ and $\widehat{HFK}(K,0) = \mathbb{F}_{(m-4)}$, then $CFK^\infty(K)$ admits an hv-basis, which is a negative doubly-twisted staircase.
\end{proposition}

\begingroup
\setlength{\intextsep}{2pt}
\begin{figure}[H]
    \setlength{\abovecaptionskip}{2pt}
    \begin{minipage}[c]{0.66\textwidth}
    \centering
       \includegraphics[scale=0.8]{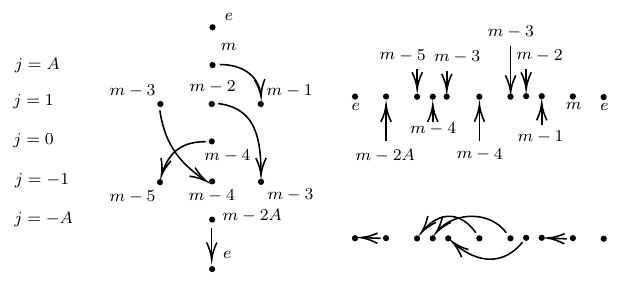} 
    \end{minipage}
	\begin{minipage}[c]{0.3\textwidth}
    \centering
       \includegraphics[scale=0.8]{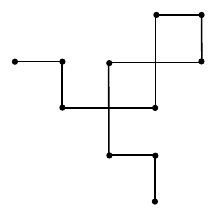} 
    \end{minipage}
    \caption{}
    \label{figure: 151-5-1}
\end{figure}
\endgroup
\begin{proof}
The given Maslov grading conditions and homologies of $Y_m$ can determine vertical and horizontal differentials without refering to basis changes, similar to the situation in Proposition \ref{prop: 151 table first row, positively two-twisted staircase}. As shown in Figure \ref{figure: 151-5-1}, the hv-basis we obtain is a negatively doubly-twisted staircase.
\end{proof}

\begin{proposition}\label{prop: 151 table 6th row}
Let $K$ be a height 4 knot of the $(1,5,1)$ type. Suppose that for $p = 2,1,0,-1$, we have $\operatorname{rank}H_*({X_p}) = 0,3,2,1$ and $\operatorname{rank}H_*({Y_p}) = 1,2,3,0$, respectively. Assume that
$\widehat{HFK}(K,1) = \mathbb{F}_{(a)} \oplus \mathbb{F}_{(a+1)} \oplus \mathbb{F}_{(m-1)}$ and $\widehat{HFK}(K,0) = \mathbb{F}_{(a-1)} \oplus \mathbb{F}_{(a)} \oplus \mathbb{F}_{(m-2)}$. Then $CFK^\infty(K)$ admits an hv-basis, which is either a negative twisted staircase plus one box or a staircase plus a twisted box.
\end{proposition}
\vspace{-1.5em}

\begin{proof}

We discuss separately according to the relationship between $a$ and $m$.

\textit{If $a\notin\{m-2,m-1\}$}, the vertical and horizontal differentials are determined by constraints of Maslov grading and homologies of $Y_m$ and $Z_m$. The hv-basis we obtain is a staircase plus a twisted box; see Figure \ref{figure: 151-6-1}.

\begingroup
\setlength{\intextsep}{2pt}
\begin{figure}[H]
    \setlength{\abovecaptionskip}{2pt}
    \begin{minipage}[t]{0.66\textwidth}
    \centering
       \includegraphics[scale=0.8]{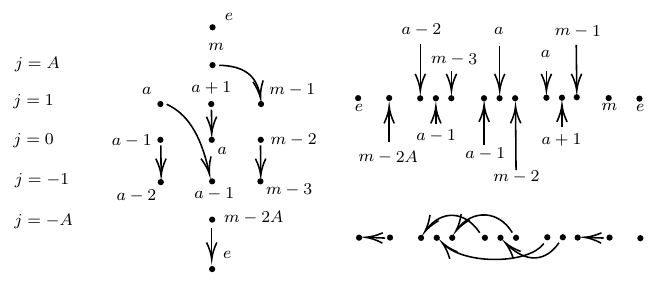} 
    \end{minipage}
	\begin{minipage}[t]{0.3\textwidth}
    \centering
       \includegraphics[scale=0.8]{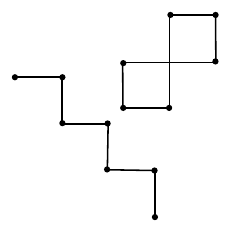} 
    \end{minipage}
    \caption{}
    \label{figure: 151-6-1}
\end{figure}
\endgroup

\textit{If $a = m-2$}, up to reordering and basis changes on $x_{(-1,m-3,1)}$ and $x_{(-1,m-3,2)}$, and on $x_{(0,m-2,1)}$ and $x_{(0,m-2,2)}$, we assume the obtained vertically simplified basis is identical to the one above.

In $C\{j=0\}$, by Lemma \ref{lemma: cancellation lemma} the constaint $H_*(A_0)\cong \mathbb{F}^5$, the generator $y_{(0,m-3)}$ must differentiate to $y_{(-1,m-4)}$, while $y_{(0,m-2,1)}$ does not differentiate leftward. The generator $y_{(0,m-2,2)}$ differentiates to a linear combination of $y_{(-1,m-3,1)}$ and $y_{(-1,m-3,2)}$; see Figure \ref{figure: 151-6-2}. We discuss separately according to $\partial y_{(0,m-2,2)}$.

\begingroup
\setlength{\intextsep}{2pt}
\begin{figure}[H]
    \centering
    \setlength{\abovecaptionskip}{2pt}
       \includegraphics[scale=0.8]{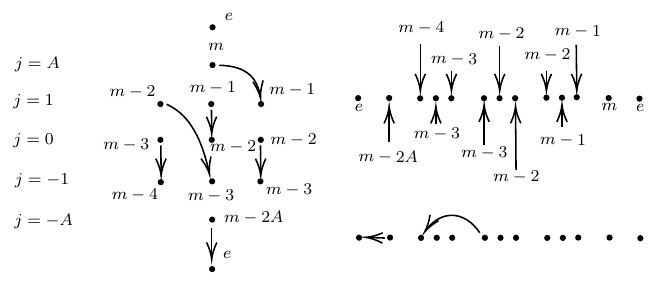} 
    \caption{}
    \label{figure: 151-6-2}
\end{figure}
\endgroup

\textbf{Case (i)}: If $y_{(0,m-2,2)}$ differentiates to $y_{(-1,m-3,1)}$, then by Lemma \ref{lemma: 1xn lemma}, $y_{(1,m-1,2)}$ differentiates to $y_{(0,m-2,1)}$. It follows that $y_{(1,m-2)}$ differentiates to either $y_{(-1,m-3,2)}$ or $y_{(-1,m-3,1)}+y_{(-1,m-3,2)}$. In the latter case, we replace $y_{(1,m-2)}$ by $y_{(1,m-2)}+y_{(0,m-2,2)}$. Thus, we assume that $y_{(1,m-2)}$ differentiates to $y_{(-1,m-3,2)}$.

\begingroup
\setlength{\intextsep}{0pt}
\begin{figure}[H]
    \centering
    \setlength{\abovecaptionskip}{1pt}
       \includegraphics[scale=0.8]{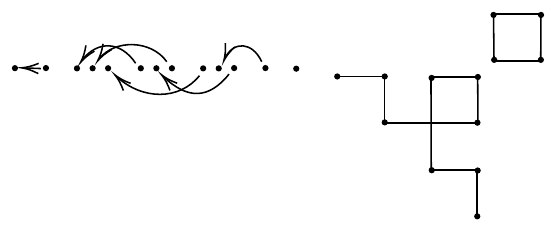} 
    \caption{}
    \label{figure: 151-6-3}
\end{figure}
\endgroup

By Maslov grading constraints and the condition $\partial^2 = 0$, the generator $y_{(A, m)}$ must differentiate to $y_{(1, m-1, 1)}$. The hv-basis we obtain is shown in Figure \ref{figure: 151-6-3}.

\textbf{Case(ii)}: If $y_{(0,m-2,2)}$ differentiates to $y_{(-1,m-3,1)}+y_{(-1,m-3,2)}$, we replace $y_{(-1,m-3,1)}$ by $y_{(-1,m-3,1)}+y_{(-1,m-3,2)}$. This operation does not change vertical differentials and we return to case (i).

\textbf{Case (iii)}: If $y_{(0, m-2, 2)}$ differentiates to $y_{(-1,m-3,2)}$, then $y_{(1,m-2)}$ differentiates to $y_{(-1,m-3,1)}$ or $y_{(-1,m-3,1)}+y_{(-1,m-3,2)}$. In the latter case, we replace $y_{(-1,m-3,1)}$ by $y_{(-1,m-3,1)}$ $+y_{(-1,m-3,2)}$. Hence we may assume that $y_{(1,m-2)}$ differentiates to $y_{(-1,m-3,1)}$.

\begingroup
\setlength{\intextsep}{0pt}
\begin{figure}[H]
    \centering
    \setlength{\abovecaptionskip}{1pt}
       \includegraphics[scale=0.8]{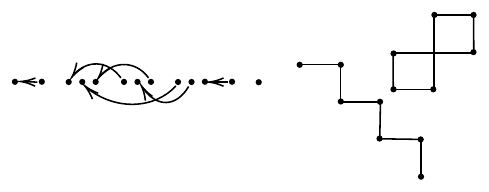} 
    \caption{}
    \label{figure: 151-6-4}
\end{figure}
\endgroup

By Lemma \ref{lemma: 1xn lemma}, the generator $y_{(1,m-1,1)}$ differentiates to $y_{(0,m-2,1)}$, while $y_{(1,m-1,2)}$ does not differentiate leftward. Finally, the generator $y_{(A,m)}$ must differentiate to $y_{(1,m-1,2)}$. We then obtain an hv-basis, which is a staircase plus a twisted box.

\textit{If $a = m-1$}, up to a reordering of $x_{(-1,m-3,1)}$ and $x_{(-1,m-3,2)}$, and of $x_{(0,m-1,1)}$ and $x_{(0,m-1,2)}$, we can obtain a vertically simplified basis as shown in Figure \ref{figure: 151-6-5}.

\begingroup
\setlength{\intextsep}{0pt}
\begin{figure}[H]
    \centering
    \setlength{\abovecaptionskip}{1pt}
       \includegraphics[scale=0.8]{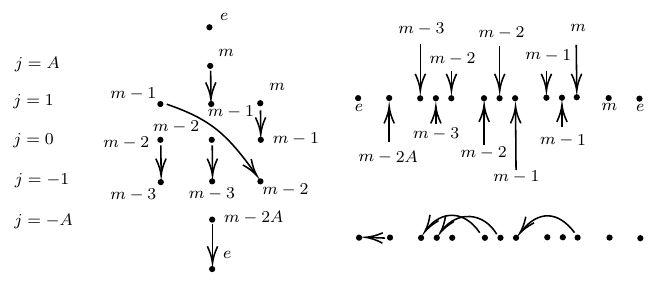} 
    \caption{}
    \label{figure: 151-6-5}
\end{figure}
\endgroup

When determining the horizontal differentials, we perform a simultaneous basis change as described in Equation \eqref{equation: simultaneous basis change} on $x_{(0,m-2,1)}, x_{(0,m-2,2)}$ and $x_{(-1,m-1,1)}, x_{(-1,m-1,2)}$, preserving the differentials from $x_{(0,m-2,1)}$, $x_{(0,m-2,2)}$ to $x_{(-1,m-3,1)}$, $x_{(-1,m-3,2)}$, while making $y_{(0,m-2,1)}(x_{(0,m-2,1)})$ and $y_{(0,m-2,2)}(x_{(0,m-2,2)})$ differentiate to $y_{(-1,m-3,1)}$ and $y_{(-1,m-3,2)}$, respectively.

We now determine the differentials from $y_{(1,m-1,1)}$ and $y_{(1,m-1,2)}$ to $y_{(-1,m-2)}$, and from $y_{(A,m)}$ to $y_{(1,m-1,1)}$ and $y_{(1,m-1,2)}$. If there were a differential from $y_{(1,m-1,2)}$ to $y_{(-1,m-2)}$, then by the $\partial^2x_{(0,m-2,2)}=0$ constraint there would have to be a diagonal differential from $y_{(-1,m-3,2)}$ to $Uy_{(-1,m-2)}$, which leads to a contradiction when considering $\partial^2x_{(-A,m)}=0$, since $x_{(A,m)}$ is not the starting point of another vertical differential and $x_{(1,m)}$ is not the ending point of a vertical differential.

\begingroup
\setlength{\intextsep}{0pt}
\begin{figure}[H]
    \centering
        \setlength{\abovecaptionskip}{1pt}
       \includegraphics[scale=0.8]{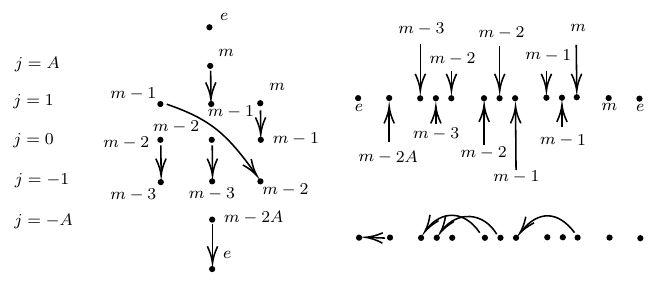} 
    \caption{}
    \label{figure: 151-6-6}
\end{figure}
\endgroup

Consequently, the preimage of $y_{(-1,m-2)}$ is $y_{(1,m-1,1)}$, and $y_{(A,m)}$ differentiates to $y_{(1,m-1,2)}$. We then obtain an hv-basis, which is a staircase plus a twisted box; see Figure \ref{figure: 151-6-6}.
\end{proof}
\vspace{-0.5em}

\begin{proposition}\label{prop: 151 table 7th row, a negative twisted staircase plus a box}
Let $K$ be a height 4 knot of the $(1,5,1)$ type. Suppose that for $p = 2, 1, 0, -1$, we have $\operatorname{rank}H_*({X_p}) = 0, 3, 2, 1$ and $\operatorname{rank}H_*({Y_p}) = 1, 2, 3, 0$, respectively. Assume that 
$\widehat{HFK}(K,1) = \mathbb{F}_{(a)} \oplus \mathbb{F}_{(m-2)} \oplus \mathbb{F}_{(m-1)}$ and $\widehat{HFK}(K,0) = \mathbb{F}_{(a-1)} \oplus \mathbb{F}_{(a-1)} \oplus \mathbb{F}_{(m-3)}$, then $CFK^\infty(K)$ admits an hv-basis, which is a negative twisted staircase plus one box.
\end{proposition}

\begin{proof}
We separate our discussion based on the relationship between $a$ and $m$.

\textit{If $a\notin\{m-2, m-1\}$}, up to a reordering of $x_{(0,a-1,1)}$ and $x_{(0,a-1,2)}$, the vertical and horizontal differentials can be determined by $H_*(Y_p)$, $H_*(Z_p)$ and Maslov grading. We obtain an hv-basis whose configuration is a negatively twisted staircase plus one box; see Figure \ref{figure: 151-6-7}.

\begingroup
\setlength{\intextsep}{0pt}
\begin{figure}[H]
    \setlength{\abovecaptionskip}{1pt}
    \begin{minipage}[t]{0.66\textwidth}
    \centering
       \includegraphics[scale=0.8]{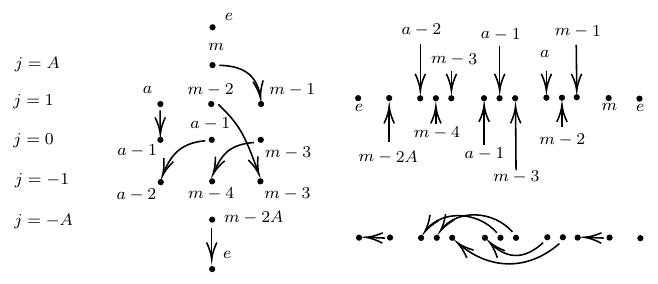} 
    \end{minipage}
	\begin{minipage}[t]{0.3\textwidth}
    \centering
       \includegraphics[scale=0.8]{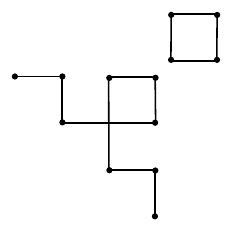} 
    \end{minipage}
    \caption{}
    \label{figure: 151-6-7}
\end{figure}
\endgroup

\textit{If $a=m-2$}, we can obtain an identical vertically simplified basis. By performing a simultaneous basis change, we assume that $y_{(0,m-3,2)}$ and $y_{(0,m-3,3)}$ differentiate to $y_{(-1,m-4,1)}$ and $y_{(-1,m-4,2)}$, respectively. By Lemma \ref{lemma: 1xn lemma}, the generator $y_{(1,m-2,1)}$ must differentiate to either $y_{(0,m-3,1)}$ or $y_{(0,m-3,1)}+y_{(-1,m-3)}$, and $y_{(1,m-2,2)}$ must differentiate to $y_{(-1,m-3)}$. If $y_{(1,m-2,1)}$ differentiates to $y_{(0,m-3,1)}+y_{(-1,m-3)}$, we replace $y_{(0,m-3,1)}$ by $y_{(0,m-3,1)}+y_{(-1,m-3)}$. The hv-basis we obtain is still a negatively twisted staircase plus a box; see Figure \ref{figure: 151-6-8}.

\begingroup
\setlength{\intextsep}{1pt}
\begin{figure}[H]
    \setlength{\abovecaptionskip}{1pt}
    \centering
       \includegraphics[scale=0.8]{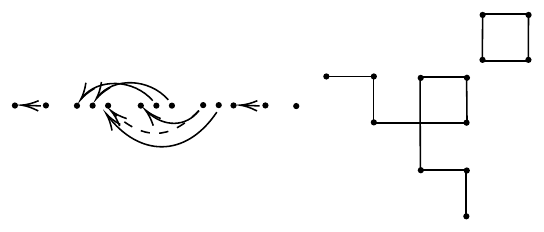} 
    \caption{}
    \label{figure: 151-6-8}
\end{figure}
\endgroup

The case $a=m-1$ is the same as the case $a=m-2$ in Proposition \ref{prop: 151 table 6th row}.
\end{proof}
\vspace{-0.5em}

\begin{proposition}\label{prop: 151 table 8th row, a staircase plus two boxes}
Let $K$ be a height 4 knot of the $(1,5,1)$ type. Suppose that for $p = 2,1,0,-1$, we have $\operatorname{rank}H_*({X_p}) = 0,3,2,1$ and $\operatorname{rank}H_*({Y_p}) = 1,2,3,0$, respectively. Assume that
$\widehat{HFK}(K,1) = \mathbb{F}_{(a)} \oplus \mathbb{F}_{(b)} \oplus \mathbb{F}_{(m-1)}$ and $\widehat{HFK}(K,0) = \mathbb{F}^2_{(a-1)} \oplus \mathbb{F}^2_{(b-1)} \oplus \mathbb{F}_{(m-2)}$. Then $CFK^\infty(K)$ admits an hv-basis, which is a staircase plus two boxes, both located in the central position.
\end{proposition}

\begin{proof}

We discuss separately according to the relationships among $a,b$ and $m$. The proof is similar to that of Proposition \ref{prop: 151 table 4th row, a staircase plus two positive boxes}.

\textit{If $a\neq m-1, b\neq m-1, a\neq b$}, we can determine a vertically simplified basis under the constraints of $H_*(Y_p)$ and Maslov grading. The horizontal differentials can be determined by constraints of $H_*(Z_m)$ and $H_*(Y_m)$ as well. The hv-basis we obtain is a staircase plus two boxes, both located in the central position.

\begingroup
\setlength{\intextsep}{0pt}
\begin{figure}[H]
    \setlength{\abovecaptionskip}{1pt}
    \begin{minipage}[t]{0.68\textwidth}
    \centering
       \includegraphics[scale=0.8]{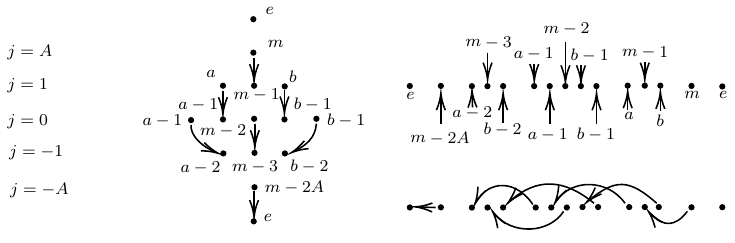} 
    \end{minipage}
	\begin{minipage}[t]{0.28\textwidth}
    \centering
       \includegraphics[scale=0.8]{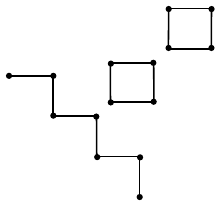} 
    \end{minipage}
    \caption{}
    \label{figure: 151-8-1}
\end{figure}
\endgroup

\textit{If $a=m-1, a\neq b$}, up to a reordering of $x_{(0,m-2,1)},x_{(0,m-2,2)}$ and $x_{(0,m-2,3)}$, we can assume that the vertically simplified basis we obtain is identical to the previous case. When determining the horizontal differentials, we apply a simultaneous basis change and Lemma \ref{lemma: 1xn lemma}. The hv-basis we obtain is still a staircase plus two boxes.

\textit{If $a=b=m-1$}, we fix a proper order on the five generators at height $0$ such that the resulting vertically simplified basis is the same as before. The horizontal differentials can then be determined by applying simultaneous basis change and Lemma \ref{lemma: 1xn lemma}. We obtain an identical hv-basis.

\textit{If $a=b \neq m-1$}, in this case the generators with Maslov gradings $a,a-1,a-2$ and those with gradings $m-1,m-2,m-3$ can be discuss separately in the sense that no vertical differentials overlap them. We obtain an identical hv-basis in the end. 
\end{proof}
\vspace{-0.5em}

Recall that the parameter $\alpha$ equals $m-2A+2$.

\vspace{-0.5em}
\begin{proposition}\label{prop: 151 table 9th row, a positive doubly-twisted staircase}
Let $K$ be a height 4 knot of the $(1,5,1)$ type. Suppose that for $p = 2,1,0,-1$, we have $\operatorname{rank}H_*({X_p}) = 1,2,3,0$ and $\operatorname{rank}H_*({Y_p}) = 0,3,2,1$, respectively. Assume that
$\widehat{HFK}(K,1) = \mathbb{F}_{(\alpha+1)} \oplus \mathbb{F}_{(\alpha+2)} \oplus \mathbb{F}_{(\alpha+3)}$ and $\widehat{HFK}(K,0) = \mathbb{F}_{(\alpha+2)}$, then $CFK^\infty(K)$ admits an hv-basis, which is a positive doubly-twisted staircase.
\end{proposition}
\vspace{-0.5em}

\begin{proof}
This case is similar to Proposition \ref{prop: 151 table 5th row, a negative twisted staircase}. The generators at each Alexander filtration level have distinct Maslov grading. Therefore, the constraints of Maslov grading and the homology of $Y_m$ determine the vertical differentials and yield a vertically simplified basis without refering to a discussion of possible differentials, as illustrated on the left in Figure \ref{figure: 151-9-1}. The horizontal differentials are determined symmetrically. The hv-basis we obtain is a positive doubly-twisted staircase; see the right side of Figure \ref{figure: 151-9-1}.

\begingroup
\setlength{\intextsep}{0pt}
\begin{figure}[H]
    \setlength{\abovecaptionskip}{1pt}
    \begin{minipage}[c]{0.66\textwidth}
    \centering
       \includegraphics[scale=0.8]{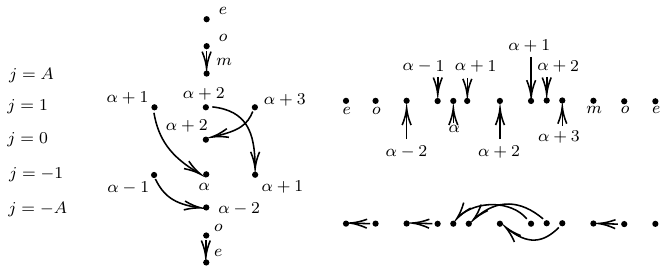} 
    \end{minipage}
	\begin{minipage}[c]{0.3\textwidth}
    \centering
       \includegraphics[scale=0.8]{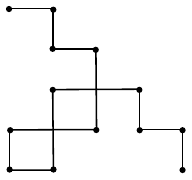} 
    \end{minipage}
    \caption{}
    \label{figure: 151-9-1}
\end{figure}
\endgroup

\end{proof}

\begin{proposition}\label{prop: 151 table 10th row}
Let $K$ be a height 4 knot of the $(1,5,1)$ type. Suppose that for $p = 2, 1, 0, -1$, we have $\operatorname{rank}H_*({X_p}) = 1,2,3,0$ and $\operatorname{rank}H_*({Y_p}) = 0,3,2,1$, respectively. Assume that
$\widehat{HFK}(K,1) = \mathbb{F}_{(a-1)} \oplus \mathbb{F}_{(a)} \oplus \mathbb{F}_{(\alpha+1)}$ and $\widehat{HFK}(K,0) = \mathbb{F}_{(a-2)} \oplus \mathbb{F}_{(a-1)} \oplus \mathbb{F}_{(\alpha)}$. Then $CFK^\infty(K)$ admits an hv-basis, which is either a staircase plus a twisted box or a positive twisted staircase plus one box.
\end{proposition}
\begin{proof}
Similarly, we discuss separately according to the relationship between $a$ and $\alpha$.

\begingroup
\setlength{\intextsep}{0pt}
\begin{figure}[H]
    \setlength{\abovecaptionskip}{1pt}
    \begin{minipage}[c]{0.66\textwidth}
    \centering
       \includegraphics[scale=0.8]{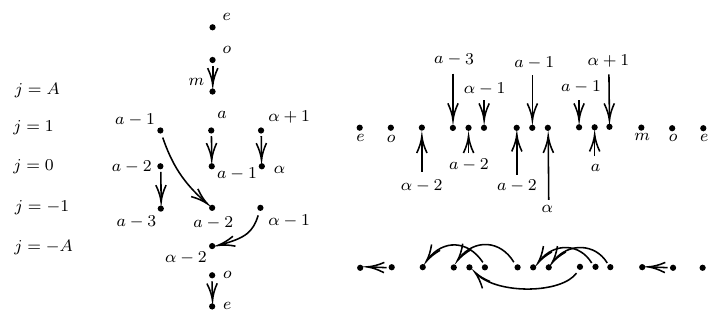} 
    \end{minipage}
	\begin{minipage}[c]{0.3\textwidth}
    \centering
       \includegraphics[scale=0.8]{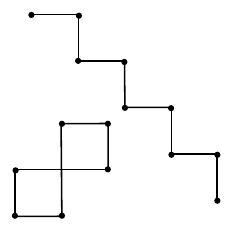} 
    \end{minipage}
    \caption{}
    \label{figure: 151-10-1}
\end{figure}
\endgroup

\textit{If $a\notin\{\alpha+1,\alpha+2\}$}, the homology of $Y_m$ and the Maslov grading constraints determine the differentials (see Figure \ref{figure: 151-10-1}). We obtain a staircase plus a twisted box.

% As shown in Figure \ref{figure: 151-10-1}, the hv-basis we obtain is a staircase plus a twisted box.

\textit{If $a=\alpha+1$}. Maslov grading and $H_*(Y_m)$ determine a vertically simplified basis identical to the case above. By the $H_*(A_0)\cong \mathbb{F}^5$ condition and Lemma \ref{lemma: cancellation lemma}, similar to the analysis in Figure \ref{figure: H_*(A_0) and cancellation lemma}, we know that $y_{(0,\alpha-1)}$ differentiates to $y_{(-1,\alpha-2)}$. The generator $y_{(1,\alpha)}$ differentiates to a linear combination of $y_{(-1,\alpha-1,1)}$ and $y_{(-1,\alpha-1,2)}$ by Maslov grading. 

\begingroup
\setlength{\intextsep}{0pt}
\begin{figure}[H]
    \setlength{\abovecaptionskip}{1pt}
    \centering
       \includegraphics[scale=0.8]{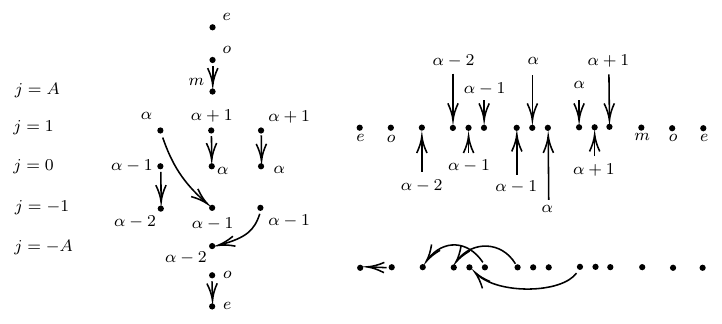} 
       \caption{}
       \label{figure: 151-10-2}
\end{figure}
\endgroup

 The preimage of $y_{(-A,\alpha-2)}$ is a linear combination of $y_{(-1,\alpha-1,1)}$ and $y_{(-1,\alpha-1,2)}$. We perform a simultaneous basis change, after which $y_{(-1,\alpha-1,2)}$ differentiates to $y_{(-A,\alpha-2)}$ and $y_{(1,\alpha)}$ differentiates to $y_{(-1,\alpha-1,1)}$, and the vertical one-to-one correspondence between $x_{(1,\alpha+1,1)}, x_{(1,\alpha+1,2)}$ and $x_{(0,\alpha,1)},x_{(0,\alpha,2)}$ is preserved, as shown in Figure \ref{figure: 151-10-2}.

We claim that $y_{(1,\alpha+1,1)}$ cannot differentiate to $y_{(0,\alpha,2)}$. If it did, then $x_{(1,\alpha)}$ would differentiate twice to $Uy_{(0,\alpha,2)}$. By Lemma \ref{lemma: 1xn lemma}, there would have to be a diagonal differential from $x_{(1,\alpha)}$ to $Ux_{(1,\alpha+1,2)}$. Then $x_{(1,\alpha)}$ would differentiate twice to $y_{(-A,\alpha-2)}$. There does not exist second path from $x_{(1,\alpha)}$ to $y_{(-A,\alpha-2)}$, contradicting $\partial^2 x_{(1,\alpha)}=0$.

It follows that $y_{(1,\alpha+1,1)}$ differentiates to $y_{(0,\alpha,1)}$, and $y_{(1,\alpha+1,2)}$ differentiates to either $y_{(0,\alpha,2)}$ or $y_{(0,\alpha,2)}+y_{(0,\alpha,1)}$. In the latter case, we replace $y_{(1,\alpha+1,2)}$ by $y_{(1,\alpha+1,2)}+y_{(1,\alpha+1,1)}$, which does not affect the vertical differentials. We then obtain an hv-basis, which is a staircase plus a twisted box; see Figure \ref{figure: 151-10-3}.

\begingroup
\setlength{\intextsep}{0pt}
\begin{figure}[H]
    \setlength{\abovecaptionskip}{1pt}
    \centering
       \includegraphics[scale=0.8]{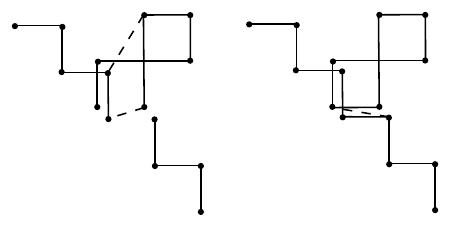} 
       \caption{}
       \label{figure: 151-10-3}
\end{figure}
\endgroup

\textit{If $a=\alpha+2$}, the hv-basis we obtain is shown in Figure \ref{figure: 151-10-4}. We separate our discussion according to the preimage of $y_{(-A,\alpha-2)}$. 
The differentials from $y_{(0,\alpha,1)}$ to $\mathbb{F}<y_{(-1,\alpha-1,1)}>\oplus\ \mathbb{F}<y_{(-1,\alpha-1,2)}>$, together with those from $\mathbb{F}<y_{(-1,\alpha-1,1)}> \oplus \mathbb{F}\ <y_{(-1,\alpha-1,2)}>$ to $y_{(-A,\alpha-2)}$, form a short exact sequence by constraints of $H_*(Z_m)$. Maslov grading forces $y_{(1,\alpha+2)}$ to differentiate to $y_{(0,\alpha+1)}$; see Figure \ref{figure: 151-10-4}

\begingroup
\setlength{\intextsep}{0pt}
\begin{figure}[H]
    \setlength{\abovecaptionskip}{1pt}
    \centering
       \includegraphics[scale=0.8]{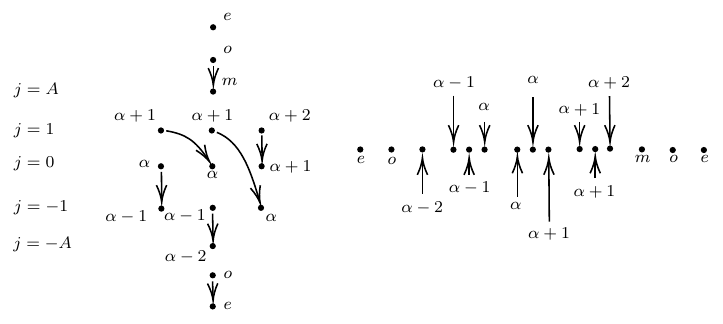} 
       \caption{}
       \label{figure: 151-10-4}
\end{figure}
\endgroup

\textbf{Case (i)}: The preimage of $y_{(-A,\alpha-2)}$ is $y_{(-1,\alpha-1,2)}$. Then $y_{(0,\alpha,1)}$ differentiates to $y_{(-1,\alpha-1,1)}$. By Lemma \ref{lemma: 1xn lemma}, the generator $y_{(1,\alpha+1,1)}$ differentiates to $y_{(0,\alpha,2)}$ to form a $1\times 1$ square.

It is possible that $y_{(1, \alpha+1, 1)}$ differentiates to $y_{(0, \alpha, 2)} + y_{(-1, \alpha)}$. If so, we replace $y_{(0, \alpha, 2)}$ by $y_{(0, \alpha, 2)} + y_{(-1, \alpha)}$. Thus, we may assume that $y_{(1, \alpha+1, 1)}$ differentiates to $y_{(0, \alpha, 2)}$.

The differentials starting from $y_{(1, \alpha+1, 2)}$ are undetermined. If $y_{(1,\alpha+1,2)}$ differentiates to $y_{(-1,\alpha)}+y_{(0,\alpha,2)}$, we replace $y_{(1,\alpha+1,2)}$ by $y_{(1,\alpha+1,2)}+y_{(1,\alpha+1,1)}$. This does not change the vertical differentials and we thus assume $y_{(1, \alpha+1, 2)}$ differentiates to $y_{(-1,\alpha)}$. The hv-basis we obtain is a positive twisted staircase plus one box; see Figure \ref{figure: 151-10-5}.

\begingroup
\setlength{\intextsep}{0pt}
\begin{figure}[H]
    \setlength{\abovecaptionskip}{1pt}
    \centering
       \includegraphics[scale=0.8]{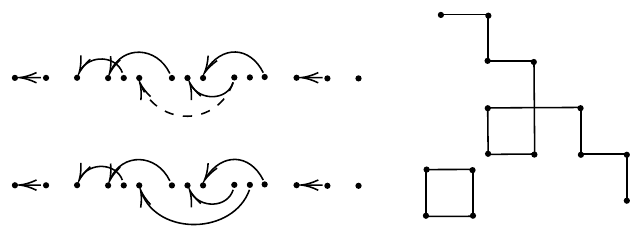} 
       \caption{}
       \label{figure: 151-10-5}
\end{figure}
\endgroup

\textbf{Case (ii)}: The preimage of $y_{(-A,\alpha-2)}$ is $y_{(-1,\alpha-1,1)} + y_{(-1,\alpha-1,2)}$. Then $y_{(0,\alpha,1)}$ differentiates to $y_{(-1,\alpha-1,1)}+y_{(-1,\alpha-1,2)}$. We replace $y_{(-1,\alpha-1,1)}$ by $y_{(-1,\alpha-1,1)}+y_{(-1,\alpha-1,2)}$ and $x_{(0,\alpha,2)}$ by $x_{(0,\alpha,2)}+x_{(-1,\alpha)}$. This case is then reduced to Case (i); see Figure \ref{figure: 151-10-6}.

\begingroup
\setlength{\intextsep}{0pt}
\begin{figure}[H]
    \centering
    \setlength{\abovecaptionskip}{1pt}
       \includegraphics[scale=0.8]{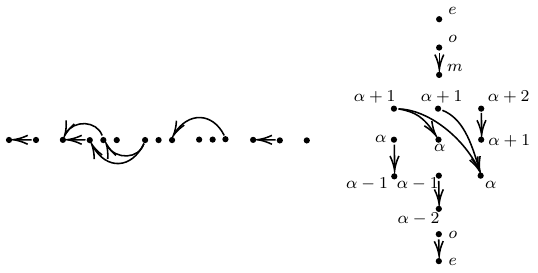} 
       \caption{}
       \label{figure: 151-10-6}
\end{figure}
\endgroup

\textbf{Case (iii)}: The preimage of $y_{(-A,\alpha-2)}$ is $y_{(-1,\alpha-1,1)}$ and $y_{(0,\alpha,1)}$ differentiates to $y_{(-1,\alpha-1,2)}$.

\begingroup
\setlength{\intextsep}{0pt}
\begin{figure}[H]
    \centering
    \setlength{\abovecaptionskip}{1pt}
       \includegraphics[scale=0.8]{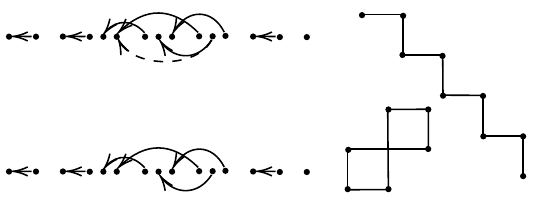} 
       \caption{}
       \label{figure: 151-10-7}
\end{figure}
\endgroup

 By Lemma \ref{lemma: 1xn lemma} and the $\partial^2=0$ constraint, $y_{(1,\alpha+1,1)}$ cannot differentiate to $y_{(0,\alpha,2)}$. Therefore, the generator $y_{(1,\alpha+1,1)}$ differentiates to $y_{(-1,\alpha)}$.
The generator $y_{(1,\alpha+1,2)}$ must differentiate to either $y_{(0,\alpha,2)}$ or $y_{(-1,\alpha)}+y_{(0,\alpha,2)}$. In the latter case, we replace $y_{(0,\alpha,2)}$ by $y_{(0,\alpha,2)}+y_{(-1,\alpha)}$. We obtain a staircase plus a twisted box, as shown in Figure \ref{figure: 151-10-7}.
\end{proof}

\begin{proposition}\label{prop: 151 table 11th row, a negative twisted staircase plus a box or a }
Let $K$ be a height 4 knot of the $(1,5,1)$ type. Suppose that for $p = 2,1,0,-1$, we have $\operatorname{rank}H_*({X_p}) = 1,2,3,0$ and $\operatorname{rank}H_*({Y_p}) = 0,3,2,1$, respectively. Assume that
$\widehat{HFK}(K,1) = \mathbb{F}_{(a)} \oplus \mathbb{F}_{(\alpha+1)} \oplus \mathbb{F}_{(\alpha+2)}$ and $\widehat{HFK}(K,0) = \mathbb{F}_{(a-1)} \oplus \mathbb{F}_{(a-1)} \oplus \mathbb{F}_{(\alpha+1)}$. Then $CFK^\infty(K)$ admits an hv-basis, which is either a positive twisted staircase plus one box or a staircase plus a twisted box.
\end{proposition}

\begin{proof}
Again, we discuss separately according to the relationship between $a$ and $\alpha$.

\textit{If $a\notin\{\alpha+1,\alpha+2\}$}, both the vertical and horizontal differentials are completely determined by $H_*(Y_p), H_*(Z_p)$ and Maslov grading. The hv-basis we obtain is shown in Figure \ref{figure: 151-11-1}, which is a positive twisted staircase plus one box.

\begingroup
\setlength{\intextsep}{0pt}
\begin{figure}[H]
    \setlength{\abovecaptionskip}{1pt}
    \begin{minipage}[c]{0.66\textwidth}
    \centering
       \includegraphics[scale=0.75]{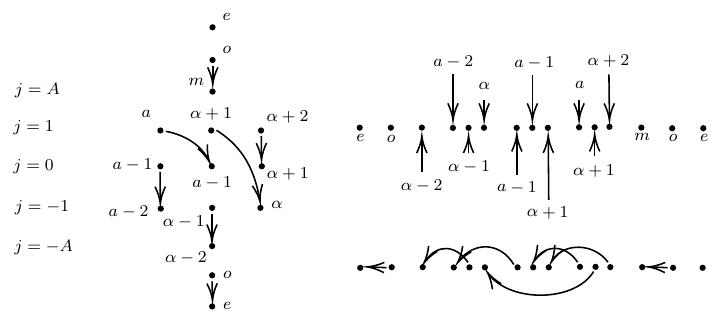} 
    \end{minipage}
	\begin{minipage}[c]{0.3\textwidth}
    \centering
       \includegraphics[scale=0.75]{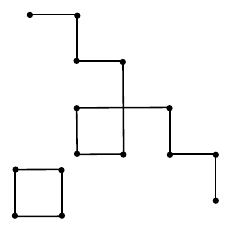} 
    \end{minipage}
    \caption{}
    \label{figure: 151-11-1}
\end{figure}
\endgroup

\textit{If $a=\alpha+1$}, this is identical to the case $a=\alpha+2$ in Proposition \ref{prop: 151 table 10th row}. The hv-basis we obtain is either a positive twisted staircase plus one box or a staircase plus a twisted box.

\textit{If $a=\alpha+2$}, by Maslov grading and $H_*(Y_m)$ constraints, we obtain a vertically simplified basis as shown in the left of Figure \ref{figure: 151-11-2}. The generator $y_{(-1,\alpha-1)}$ differentiates to $y_{(-A,\alpha-2)}$. By performing a simultaneous basis change, we may assume that $y_{(0,\alpha+1,3)}$ differentiates to $y_{(-1,\alpha,2)}$ while keeping the vertical differentials from $x_{(1,\alpha+2,1)}$ and $x_{(1,\alpha+2,2)}$ to $x_{(0,\alpha+1,1)}$ and $x_{(0,\alpha+1,2)}$ unchanged. Lemma \ref{lemma: 1xn lemma} tells us that $y_{(1,\alpha+2,2)}$ must differentiate to $y_{(0,\alpha+1,2)}$ to form a $1\times 1$ square box; see Figure \ref{figure: 151-11-2}.

\begingroup
\setlength{\intextsep}{0pt}
\begin{figure}[H]
    \centering
    \setlength{\abovecaptionskip}{1pt}
       \includegraphics[scale=0.8]{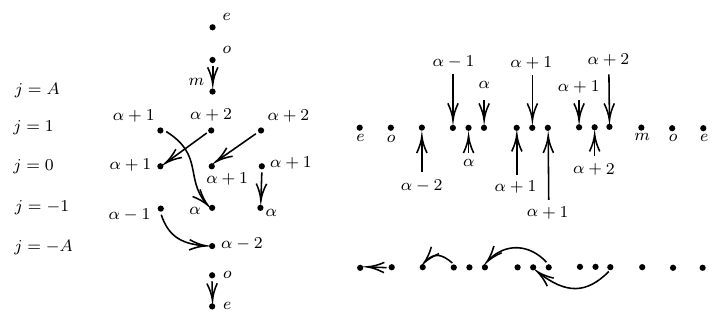} 
       \caption{}
       \label{figure: 151-11-2}
\end{figure}
\endgroup

The generator $y_{(1,\alpha+1)}$ may differentiate to $y_{(-1,\alpha,1)}+y_{(-1,\alpha,2)}$. If so, we replace $y_{(1,\alpha+1)}$ by $y_{(1,\alpha+1)}+y_{(0,\alpha+1,3)}$. We then assume that $y_{(1,\alpha+1)}$ differentiates to $y_{(-1,\alpha,1)}$.

The generator $y_{(1,\alpha+2,1)}$ may differentiate to $y_{(0,\alpha+1,1)}+y_{(0,\alpha+1,2)}$. If so, we replace $y_{(1,\alpha+2,1)}$ by $y_{(1,\alpha+2,1)}+y_{(1,\alpha+2,2)}$, and $x_{(1,\alpha+1)}$ by $x_{(1,\alpha+1)}+x_{(0,\alpha+1,3)}$. We obtain an hv-basis, which is a positive twisted staircase plus one box; see Figure \ref{figure: 151-11-3}.

\begingroup
\setlength{\intextsep}{0pt}
\begin{figure}[H]
    \setlength{\abovecaptionskip}{1pt}
    \begin{minipage}[c]{0.66\textwidth}
    \centering
       \includegraphics[scale=0.8]{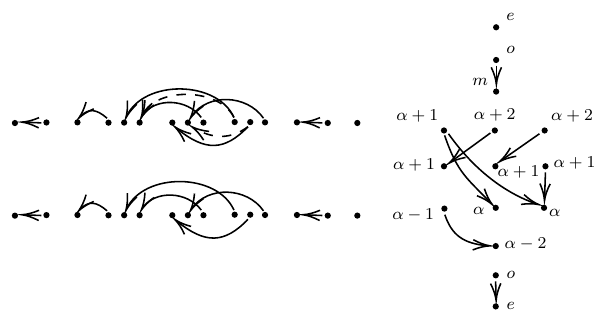} 
    \end{minipage}
	\begin{minipage}[c]{0.3\textwidth}
    \centering
       \includegraphics[scale=0.8]{fig/151-11-2.pdf} 
    \end{minipage}
    \caption{}
    \label{figure: 151-11-3}
\end{figure}
\endgroup
\end{proof}

\begin{proposition}\label{prop: 151 table 12th row, a staircase plus two boxes}
Let $K$ be a height 4 knot of the $(1,5,1)$ type. Suppose that for $p = 2,1,0,-1$, we have $\operatorname{rank}H_*({X_p}) = 1,2,3,0$ and $\operatorname{rank}H_*({Y_p}) = 0,3,2,1$, respectively. Assume that
$\widehat{HFK}(K,1) = \mathbb{F}_{(a)} \oplus \mathbb{F}_{(b)} \oplus \mathbb{F}_{(\alpha+1)}$ and $\widehat{HFK}(K,0) = \mathbb{F}^2_{(a-1)} \oplus \mathbb{F}^2_{(b-1)} \oplus \mathbb{F}_{(\alpha)}$. Then $CFK^\infty(K)$ admits an hv-basis, which is a staircase plus two boxes, both located in the central position.
\end{proposition}

\begin{proof}
The proof is similar to that of Propositions \ref{prop: 151 table 4th row, a staircase plus two positive boxes} and \ref{prop: 151 table 8th row, a staircase plus two boxes}. 

\textit{If $a\neq b, a\neq \alpha+1, b\neq \alpha+1$}, the vertical differentials are determined by constraints from $H_*(Y_p)$ and Maslov grading. Up to a reordering of generators with identical Maslov grading, we obtain a vertically simplified basis as shown on the left in Figure \ref{figure: 151-12-1}. By constraints of $H_*(Z_m)$ and Maslov grading, and applying Lemma \ref{lemma: 1xn lemma}, the horizontal differentials can be determined. The hv-basis we obtain is a staircase plus two boxes; see Figure \ref{figure: 151-12-1}.

\begingroup
\setlength{\intextsep}{0pt}
\begin{figure}[H]
    \setlength{\abovecaptionskip}{1pt}
    \begin{minipage}[c]{0.66\textwidth}
    \centering
       \includegraphics[scale=0.8]{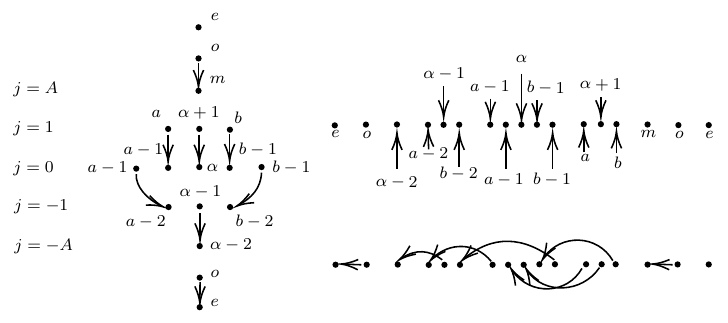} 
    \end{minipage}
	\begin{minipage}[c]{0.3\textwidth}
    \centering
       \includegraphics[scale=0.8]{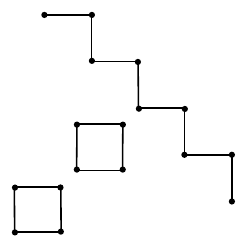} 
    \end{minipage}
    \caption{}
    \label{figure: 151-12-1}
\end{figure}
\endgroup

\textit{If $a=b\neq \alpha+1$}, the vertical differentials is determined under constraints of Maslov grading and $H_*(Y_m)$. After reordering generators with identical Maslov grading, we obtain a vertically simplified basis identical to the previous case. 

\begingroup
\setlength{\intextsep}{0pt}
\begin{figure}[H]
    \setlength{\abovecaptionskip}{1pt}
    \begin{minipage}[c]{0.66\textwidth}
    \centering
       \includegraphics[scale=0.8]{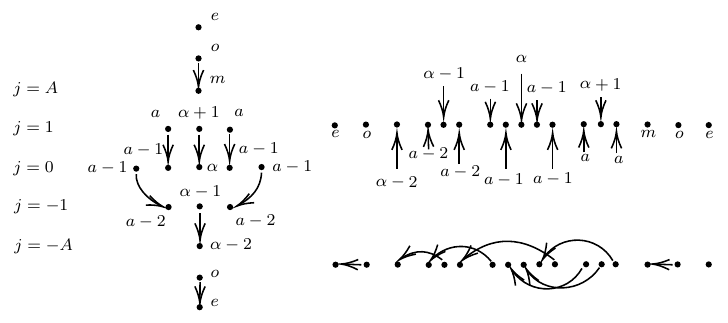} 
    \end{minipage}
	\begin{minipage}[c]{0.3\textwidth}
    \centering
       \includegraphics[scale=0.8]{fig/151-12-2.pdf} 
    \end{minipage}
    \caption{}
    \label{figure: 151-12-2}
\end{figure}
\endgroup

We perform a simultaneous basis change on $y_{(-1,a-2,1)},y_{(-1,a-2,2)}$ and $y_{(0,a-1,2)},y_{(0,a-1,3)}$, after which $y_{(0,a-1,1)}$ and $y_{(0,a-1,4)}$ differentiate to $y_{(-1,a-2,1)}$ and $y_{(-1,a-2,2)}$, respectively. The horizontal differentials are determined symmetrically; see Figure \ref{figure: 151-12-2}.

\textit{If $a=b=\alpha+1$}, again, by a suitable reordering of generators and applying simultaneous basis change together with Lemma \ref{lemma: 1xn lemma}, we can arrange the vertical and horizontal differentials to be identical to the previous case. The hv-basis we obtain is a staircase plus two boxes. 

\textit{If $a=\alpha+1, b\neq \alpha+1$}, in this case, the generators with Maslov gradings $\alpha+1,\alpha,\alpha-1$ and those with gradings $b,b-1,b-2$ are separated into two parts. We use Maslov grading and $H_*(Y_m), H_*(Z_m)$ constraints to determine the differentials. The hv-basis we obtain is identical.
\end{proof}

\vspace{-0.5em}
\begin{proposition}\label{prop: 151 table 13th row, a negative twisted staircase}
Let $K$ be a height 4 knot of the $(1,5,1)$ type. Suppose that for $p = 2,1,0,-1$, we have $\operatorname{rank}H_*({X_p}) = 1,3,2,0$ and $\operatorname{rank}H_*({Y_p}) = 0,2,3,1$, respectively. Assume that
$\widehat{HFK}(K,1) = \mathbb{F}_{(\alpha-2)} \oplus \mathbb{F}_{(\alpha-1)}$ and $\widehat{HFK}(K,0) = \mathbb{F}_{(\alpha-3)}$. Then $CFK^\infty(K)$ admits an hv-basis, which is a negative doubly-twisted staircase.
\end{proposition}

\vspace{-0.5em}
\begin{proof}
In this case, the constraints of Maslov grading and $H_*(Y_p), H_*(Z_m)$ determine the vertical and horizontal differentials through grid climbing method; see Figure \ref{figure: 151-13-1}.

\begingroup
\setlength{\intextsep}{0pt}
\begin{figure}[H]
    \setlength{\abovecaptionskip}{0pt}
    \begin{minipage}[c]{0.66\textwidth}
    \centering
       \includegraphics[scale=0.7]{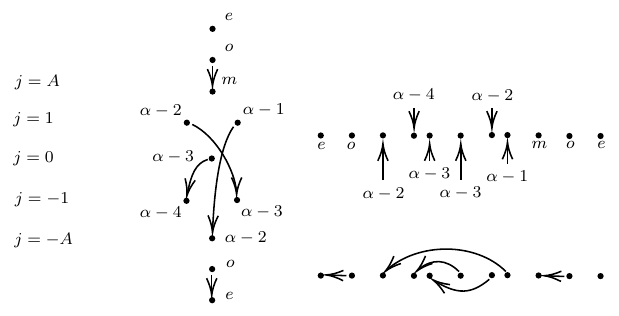} 
    \end{minipage}
	\begin{minipage}[c]{0.3\textwidth}
    \centering
       \includegraphics[scale=0.7]{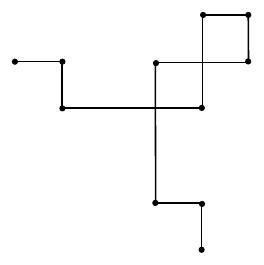} 
    \end{minipage}
    \caption{}
    \label{figure: 151-13-1}
\end{figure}
\endgroup

The hv-basis we obtain is a negative doubly-twisted staircase.
\end{proof}
\vspace{-0.5em}

\begin{proposition}\label{prop: 151 table 14th row}
Let $K$ be a height 4 knot of the $(1,5,1)$ type. Suppose that for $p = 2,1,0,-1$, we have $\operatorname{rank}H_*({X_p}) = 1,3,2,0$ and $\operatorname{rank}H_*({Y_p}) = 0,2,3,1$, respectively. Assume that
$\widehat{HFK}(K,1) = \mathbb{F}_{(a)} \oplus \mathbb{F}_{(\alpha-1)}$ and $\widehat{HFK}(K,0) = \mathbb{F}_{(a-1)} \oplus \mathbb{F}_{(a-1)} \oplus \mathbb{F}_{(\alpha-2)}$. Then $CFK^\infty(K)$ admits an hv-basis, which is either a staircase plus a twisted box or a negative twisted staircase plus one box.
\end{proposition}
\vspace{-0.5em}

\begin{proof}

\textit{If $a\notin\{\alpha-1,\alpha\}$}, up to a reordering of $x_{(0,a-1,1)}$ and $x_{(0,a-1,2)}$, we can assume the obtained vertical differentials are as shown on the left in Figure \ref{figure: 151-14-1}. In this case, the horizontal differentials are determined by Maslov grading and the cancellation lemma (Lemma \ref{lemma: cancellation lemma}) without refering to basis change. We then obtain an hv-basis, which is a negative twisted staircase plus one box, as shown in Figure \ref{figure: 151-14-1}. 

\begingroup
\setlength{\intextsep}{0pt}
\begin{figure}[H]
    \setlength{\abovecaptionskip}{0pt}
    \begin{minipage}[c]{0.66\textwidth}
    \centering
       \includegraphics[scale=0.7]{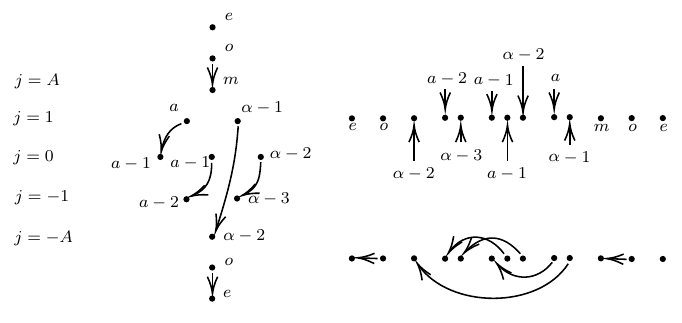} 
    \end{minipage}
	\begin{minipage}[c]{0.3\textwidth}
    \centering
       \includegraphics[scale=0.7]{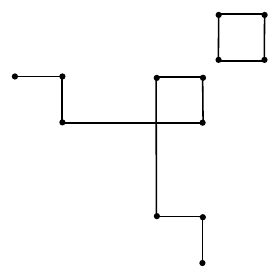} 
    \end{minipage}
    \caption{}
    \label{figure: 151-14-1}
\end{figure}
\endgroup

\textit{If $a=\alpha-1$}, the vertically simplified basis we obtain is identical. The condition $H_*(A_0)\cong \mathbb{F}^5$ requires that differentials from $y_{(0,\alpha-2,2)}, y_{(0,\alpha-2,3)}$ to $y_{(-1,\alpha-3,1)},y_{(-1,\alpha-3,2)}$ build an isomorphism between these two-dimensional spaces. We perform a simultaneous basis change on $y_{(0,\alpha-2,2)}$, $y_{(0,\alpha-2,3)}$ and $y_{(1,\alpha-1,1)}$, $y_{(1,\alpha-1,2)}$. We then assume that $y_{(0,\alpha-2,2)}$ differentiates to $y_{(-1,\alpha-3,1)}$, and $y_{(0,\alpha-2,3)}$ differentiates to $y_{(-1,\alpha-3,2)}$, preserving the vertical differentials between $x_{(0,\alpha-2,2)}$, $x_{(0,\alpha-2,3)}$ and $x_{(-1,\alpha-3,1)}$, $x_{(-1,\alpha-3,2)}$. See Figure \ref{figure: 151-14-2}.

\begingroup
\setlength{\intextsep}{0pt}
\begin{figure}[H]
    \setlength{\abovecaptionskip}{0pt}
    \begin{minipage}[c]{0.66\textwidth}
    \centering
       \includegraphics[scale=0.70]{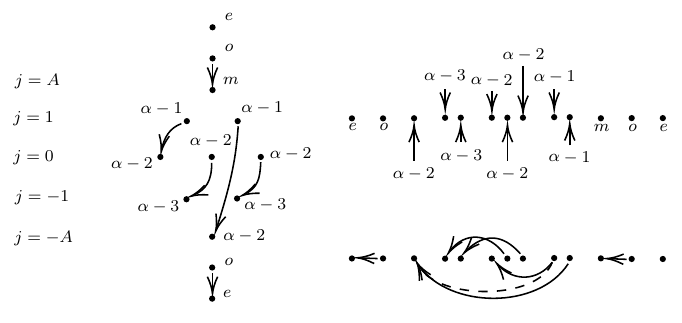} 
    \end{minipage}
	\begin{minipage}[c]{0.3\textwidth}
    \centering
       \includegraphics[scale=0.70]{fig/151-14-2.pdf} 
    \end{minipage}
    \caption{}
    \label{figure: 151-14-2}
\end{figure}
\endgroup

By Lemma \ref{lemma: 1xn lemma} and the condition $\partial^2 x_{(0,\alpha-2,2)}=0$, the preimage of $y_{(0,\alpha-2,1)}$ is $y_{(1,\alpha-1,1)}$. If the generator $y_{(1,\alpha-1,1)}$ differentiates to $y_{(0,\alpha-2,1)}$ $+y_{(-A,\alpha-2)}$, we replace $y_{(0,\alpha-2,1)}$ by $y_{(0,\alpha-2,1)}+y_{(-A,\alpha-2)}$. Thus, we may assume that $y_{(1,\alpha-1,1)}$ differentiates to $y_{(0,\alpha-2,1)}$. It follows that $y_{(1,\alpha-1,2)}$ differentiates to $y_{(-A,\alpha-2)}$. We then obtain an hv-basis, which is a negative twisted staircase plus one box, as shown in Figure \ref{figure: 151-14-2}.

\textit{If $a=\alpha$}, there are two possible vertically simplified basis.

\textbf{Case (i)}: The generator $x_{(0,\alpha-1,2)}$ differentiates to $x_{(-1,\alpha-2)}$. The vertically simplified basis we obtain is shown on the left in Figure \ref{figure: 151-14-3}.

Consider the constraint $H_*(A_0)\cong\mathbb{F}^5$ and apply Lemma \ref{lemma: cancellation lemma}, the generator $y_{(0,\alpha-2)}$ differentiates to $y_{(-1,\alpha-3)}$. Note that $y_{(0,\alpha-1,2)}$ differentiates to $y_{(-1,\alpha-2)}$ or $y_{(-1,\alpha-2)}+y_{(-A,\alpha-2)}$, and $y_{(1,\alpha)}$ must differentiate to $y_{(0,\alpha-1,1)}$ by constraint $\partial^2 y_{(0,\alpha-1,2)}=0$. After replacing $y_{(-1,\alpha-2)}$ by $y_{(-1,\alpha-2)}+y_{(-A,\alpha-2)}$, we may assume that $y_{(0,\alpha-1,2)}$ differentiates to $y_{(-1,\alpha-2)}$.

The generator $y_{(1,\alpha-1)}$ differentiates to either $y_{(-A,\alpha-2)}$ or $y_{(-A,\alpha-2)}+y_{(-1,\alpha-2)}$. In the latter case, we replace $y_{(1,\alpha-1)}$ by $y_{(1,\alpha-1)}+y_{(0,\alpha-1,2)}$ and assume that $y_{(1,\alpha-1)}$ differentiates to $y_{(-A,\alpha-2)}$. We obtain an hv-basis, which is a negative twisted staircase plus one box; see Figure \ref{figure: 151-14-3}.

\begingroup
\setlength{\intextsep}{0pt}
\begin{figure}[H]
    \setlength{\abovecaptionskip}{1pt}
    \begin{minipage}[c]{0.66\textwidth}
    \centering
       \includegraphics[scale=0.7]{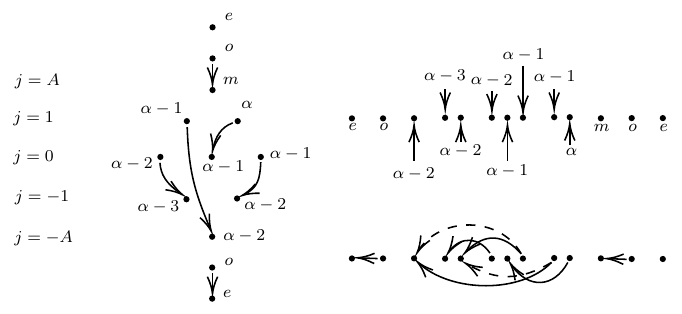} 
    \end{minipage}
	\begin{minipage}[c]{0.3\textwidth}
    \centering
       \includegraphics[scale=0.7]{fig/151-14-2.pdf} 
    \end{minipage}
    \caption{}
    \label{figure: 151-14-3}
\end{figure}
\endgroup

\textbf{Case (ii)}: In our obtained hv-basis, the generator $x_{(0,\alpha-1,2)}$ differentiates to $x_{(-A,\alpha-2)}$ and $x_{(1,\alpha-1)}$ differentiates to $x_{(-1,\alpha-2)}$.

By Lemma \ref{lemma: 1xn lemma} and the $\partial^2y_{(0,\alpha-1,2)}=0$ condition, the generator $y_{(0,\alpha-1,2)}$ cannot differentiate to $y_{(-1,\alpha-2)}$; therefore it must differentiate to $y_{(-A,\alpha-2)}$. By Maslov grading constraints, $y_{(0,\alpha-2)}$ differentiates to $y_{(-1,\alpha-3)}$ and $y_{(1,\alpha)}$ differentiates to $y_{(0,\alpha-1,1)}$.

\begingroup
\setlength{\intextsep}{2pt}
\begin{figure}[H]
    \setlength{\abovecaptionskip}{1pt}
    \begin{minipage}[t]{0.66\textwidth}
    \centering
       \includegraphics[scale=0.7]{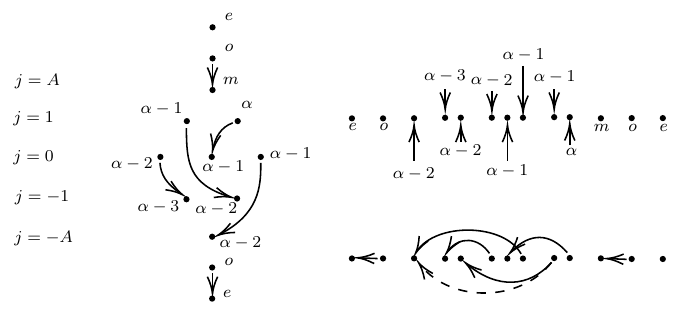} 
    \end{minipage}
	\begin{minipage}[t]{0.3\textwidth}
    \centering
       \includegraphics[scale=0.7]{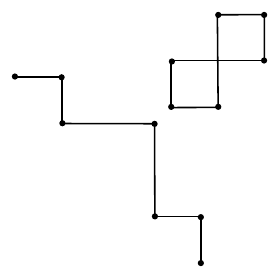} 
    \end{minipage}
    \caption{}
    \label{figure: 151-14-4}
\end{figure}
\endgroup

If $y_{(1,\alpha-1)}$ differentiates to $y_{(-1,\alpha-2)}+y_{(-A,\alpha-2)}$, we replace $y_{(-1,\alpha-2)}$ by $y_{(-1,\alpha-2)}+y_{(-A,\alpha-2)}$ and then assume that $y_{(1,\alpha-1)}$ differentiates to $y_{(-1,\alpha-2)}$. The resulting hv-basis is a staircase plus a twisted box, as shown in Figure \ref{figure: 151-14-4}.
\end{proof}

\begin{proposition}\label{prop: 151 table 15th row, a negative twisted staircase plus a box or a }
Let $K$ be a height 4 knot of the $(1,5,1)$ type. Suppose that for $p = 2,1,0,-1$, we have $\operatorname{rank}H_*({X_p}) = 1,3,2,0$ and $\operatorname{rank}H_*({Y_p}) = 0,2,3,1$, respectively. Assume that
$\widehat{HFK}(K,1) = \mathbb{F}_{(a)} \oplus \mathbb{F}_{(a+1)}$ and $\widehat{HFK}(K,0) = \mathbb{F}_{(a-1)} \oplus \mathbb{F}_{(a)} \oplus \mathbb{F}_{(\alpha-1)}$. Then $CFK^\infty(K)$ admits an hv-basis, which is a negative twisted staircase plus one box or a staircase plus a twisted box.
\end{proposition}
\vspace{-10pt}
\begin{proof}

\textit{If $a\notin\{\alpha,\alpha-1\}$}, the vertical differentials can be determined by Maslov grading and constraints of $H_*(Y_p)$; see Figure \ref{figure: 151-15-1}. The horizontal differentials are determined symmetrically. We obtain an hv-basis, a staircase plus a twisted box.

\begingroup
\setlength{\intextsep}{0pt}
\begin{figure}[H]
    \setlength{\abovecaptionskip}{1pt}
    \begin{minipage}[t]{0.66\textwidth}
    \centering
       \includegraphics[scale=0.7]{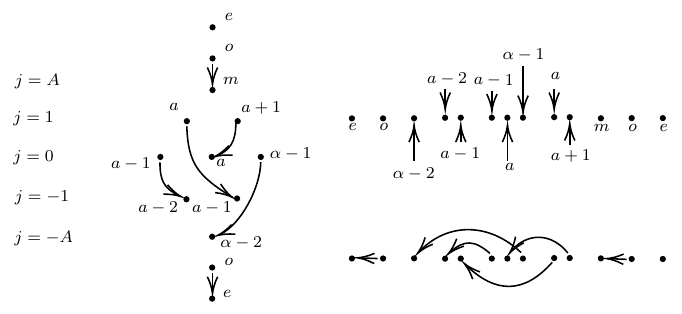} 
    \end{minipage}
	\begin{minipage}[t]{0.3\textwidth}
    \centering
       \includegraphics[scale=0.7]{fig/151-14-6.pdf} 
    \end{minipage}
    \caption{}
    \label{figure: 151-15-1}
\end{figure}
\endgroup

\textit{If $a=\alpha-1$}, this is the same as the case $a=\alpha$ in Proposition \ref{prop: 151 table 14th row}.

\textit{If $a=\alpha$}, there is only one possible configuration for the vertical differentials; see the left side of Figure \ref{figure: 151-15-2}. By the constraints of Maslov grading and the homologies of $Z_m$, the generators $y_{(1,\alpha)}$ and $y_{(1,\alpha+1)}$ differentiate to $y_{(-1,\alpha-1)}$ and $y_{(0,\alpha)}$, respectively (see Figure \ref{figure: 151-15-2}). We separate the discussion into two cases according to $\partial y_{(0,\alpha-1,1)}$.

\begingroup
\setlength{\intextsep}{0pt}
\begin{figure}[H]
    \centering
    \setlength{\abovecaptionskip}{1pt}
       \includegraphics[scale=0.7]{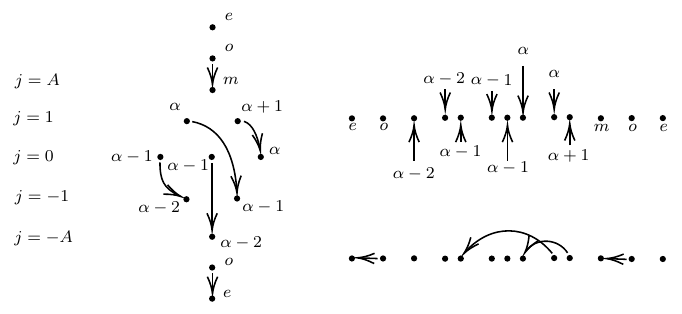} 
    \caption{}
    \label{figure: 151-15-2}
\end{figure}
\endgroup

\textbf{Case (i)}: The generator $y_{(0,\alpha-1,1)}$ differentiates to $y_{(-1,\alpha-2)}$. Then $y_{(0,\alpha-1,2)}$ differentiates to either $y_{(-A,\alpha-2)}$ or $y_{(-A,\alpha-2)}+y_{(-1,\alpha-2)}$. The latter possibility can be excluded by Lemma \ref{lemma: 1xn lemma} and the constraint $\partial^2 y_{(0,\alpha-1,2)}=0$, since there are no horizontal differential from $y_{(A,m)}$ to $y_{(1,\alpha+1)}$. We obtain an hv-basis whose configuration is a staircase plus a twisted box, as shown in Figure \ref{figure: 151-15-3}.

\begingroup
\setlength{\intextsep}{0pt}
\begin{figure}[H]
    \centering
    \setlength{\abovecaptionskip}{1pt}
       \includegraphics[scale=0.7]{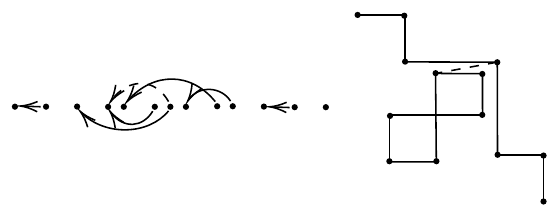} 
    \caption{}
    \label{figure: 151-15-3}
\end{figure}
\endgroup

\textbf{Case (ii)}: The generator $y_{(0,\alpha-1,1)}$ differentiates to $y_{(-1,\alpha-2)}+y_{(-A,\alpha-2)}$. From the argument above, we know that $y_{(0,\alpha-1,2)}$ must differentiate to $y_{(-A,\alpha-2)}$. Replacing $y_{(-1,\alpha-2)}$ by $y_{(-1,\alpha-2)}+y_{(-A,\alpha-2)}$ yields an hv-basis that is a staircase plus a twisted box; see Figure \ref{figure: 151-15-4}.

\begingroup
\setlength{\intextsep}{0pt}
\begin{figure}[H]
    \setlength{\abovecaptionskip}{0pt}
    \begin{minipage}[c]{0.48\textwidth}
    \centering
       \includegraphics[scale=0.7]{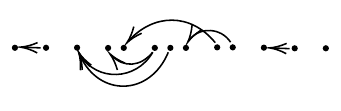} 
    \end{minipage}
	\begin{minipage}[c]{0.48\textwidth}
    \centering
       \includegraphics[scale=0.65]{fig/151-14-6.pdf} 
    \end{minipage}
    \caption{}
    \label{figure: 151-15-4}
\end{figure}
\endgroup
\end{proof}

\begin{proposition}\label{prop: 151 table 16th row, a staircase plus two boxes}
Let $K$ be a height 4 knot of the $(1,5,1)$ type. Suppose that for $p = 2,1,0,-1$ we have $\operatorname{rank}H_*({X_p}) = 1,3,2,0$ and $\operatorname{rank}H_*({Y_p}) = 0,2,3,1$, respectively. Assume that
$\widehat{HFK}(K,1) = \mathbb{F}_{(a)} \oplus \mathbb{F}_{(b)}$ and $\widehat{HFK}(K,0) = \mathbb{F}^2_{(a-1)} \oplus \mathbb{F}^2_{(b-1)} \oplus \mathbb{F}_{(\alpha-1)}$. Then $CFK^\infty(K)$ admits an hv-basis, which is a staircase plus two boxes, both located in the central position.
\end{proposition}
\vspace{-12pt}

\begin{proof}
The proof is similar to those in Propositions \ref{prop: 151 table 4th row, a staircase plus two positive boxes}, \ref{prop: 151 table 8th row, a staircase plus two boxes}, and \ref{prop: 151 table 12th row, a staircase plus two boxes}.

\textit{If} $a\neq \alpha, b\neq \alpha, a\neq b$, up to a reordering of generators with identical Maslov grading, we obtain a vertically simplified basis as illustrated on the left in Figure \ref{figure: 151-16-1}. When determining the horizontal differentials, we need to apply Lemma \ref{lemma: 1xn lemma} and trace the $1\times 1$ squares. The hv-basis we obtain is a staircase plus two boxes; see Figure \ref{figure: 151-16-1}.

\begingroup
\setlength{\intextsep}{0pt}
\begin{figure}[H]
    \setlength{\abovecaptionskip}{0pt}
    \begin{minipage}[t]{0.66\textwidth}
    \centering
       \includegraphics[scale=0.7]{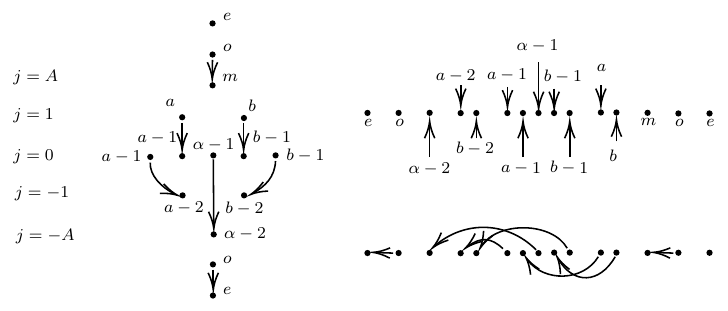} 
    \end{minipage}
	\begin{minipage}[t]{0.3\textwidth}
    \centering
       \includegraphics[scale=0.7]{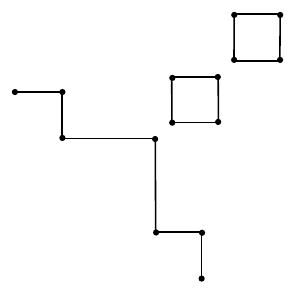} 
    \end{minipage}
    \caption{}
    \label{figure: 151-16-1}
\end{figure}
\endgroup

\textit{If} $a=b\neq \alpha$, Maslov grading tells us that $y_{(0,\alpha-1)}$ must differentiate to $y_{(-A,\alpha-2)}$. By applying a simultaneous basis change to $y_{(-1,a-2,1)},y_{(-1,a-2,2)}$ and $y_{(0,a-1,2)},y_{(0,a-1,3)}$, we may assume that $y_{(0,a-1,1)}$ and $y_{(0,a-1,4)}$ differentiate to $y_{(-1,a-2,1)}$ and $y_{(-1,a-2,2)}$, respectively, preserving the vertical one-to-one correspondence between $x_{(1,a,1)},x_{(1,a,2)}$ and $x_{(0,a-1,2)}, x_{(0,a-1,3)}$. 
\begingroup
\setlength{\intextsep}{0pt}
\begin{figure}[H]
    \centering
    \setlength{\abovecaptionskip}{0pt}
       \includegraphics[scale=0.7]{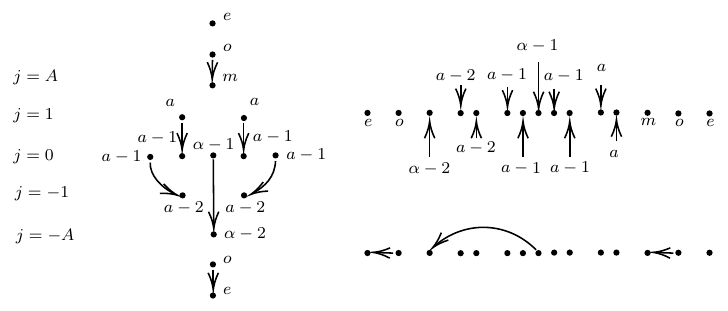} 
    \caption{}
    \label{figure: 151-16-2}
\end{figure}
\endgroup

The remaining differentials are determined similarly by the constraints of Maslov grading and $H_*(Y_p)$. Again, we obtain a staircase plus two boxes.

\textit{If} $a=b=\alpha$, we carry basis changes and reorder the generators at the same Alexander filtration level, then assume that $x_{(0,\alpha-1,1)}, x_{(0,\alpha-1,3)}$, and $x_{(0,\alpha-1,5)}$ differentiate to $x_{(-1,\alpha-2,1)}$, $x_{(-A,\alpha-2)}$, and $x_{(-1,\alpha-2,2)}$, respectively. This yields an identical vertically simplified basis. We use the Maslov grading, $H_*(Z_m)$ constraints and Lemma \ref{lemma: 1xn lemma} to determine the remaining horizontal differentials. The hv-basis obtained in this case is again a staircase plus two boxes.

\textit{If} $a=\alpha, b\neq \alpha$, in this case there are no differentials between the generators with Maslov gradings $\alpha,\alpha-1,\alpha-2$ and those with gradings $a,a-1,a-2$. The constraints of Maslov grading and $H_*(Y_p), H_*(Z_m)$ yield an identical hv-basis, a staircase plus two boxes.
\end{proof}

\subsection{Special case of genus one}\label{subsection:3.4 Special case of genus one}

For a height 4 knot $K$ of the $(1,5,1)$ type, if $g(K) = 1$, the classification results of its hv-basis are exhibited in Table \ref{table: 151 g(K) = 1}.

The hv-basis in the case $g(K) = 1$ can be viewed as a truncation of the general case. The discussion is identical to that of the middle height level part in the general case. We also apply the cancellation lemma (Lemma \ref{lemma: cancellation lemma}), the constraint from $H_*(A_0)$, and the $1\times n$ Lemma \ref{lemma: 1xn lemma}. The relative simplicity lies in that we start directly from height $1$ and trace downward. There are only two cases for $H_*(X_p)$ with $p = 1,0,-1$, namely $(2,3,1)$ and $(3,2,1)$, due to the condition $H_*(X_2) = 0$ and Equation \eqref{equation:(i,j) symmetry}.

\begin{table}[H]
\centering
\begin{tabular}{|c|c|c|}
\hline

\makecell{$rk H_*(X_p)$ \\ $p = 1,0,-1$} & $\widehat{HFK}(K,i),i=1,0$ & hv-basis  \\

\hline

\multirow{4}{*}{(2,3,1)} & \makecell{$\mathbb{F}_{(a)}\oplus \mathbb{F}_{(a-1)}$ \\ $\mathbb{F}_{(a-1)}$ } &  {\includegraphics[scale=0.8, valign=c]{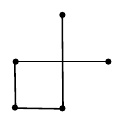}}

\\

\cline{2-3}

& \makecell{$\mathbb{F}_{(a)}\oplus \mathbb{F}_{(a-1)}$ \\ $\mathbb{F}_{(a-1)}\oplus \mathbb{F}_{(d)}\oplus \mathbb{F}_{(a-2)}$ } &{\includegraphics[scale=0.8, valign=c]{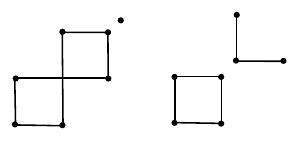}} \\

\cline{2-3}

& \makecell{ $\mathbb{F}_{(a)}\oplus \mathbb{F}_{(b)}$ \\ $\mathbb{F}_{(a-1)} \oplus \mathbb{F}_{(b-1)} \oplus \mathbb{F}_{(b-1)}$ } &{\includegraphics[scale=0.8, valign=c]{fig/1-2}} \\

\cline{2-3}

& \makecell{ $\mathbb{F}_{(a)}\oplus \mathbb{F}_{(b)}$ \\ $\mathbb{F}_{(a-1)}\oplus \mathbb{F}_{(b-1)} \oplus \mathbb{F}_{(e)} \oplus \mathbb{F}_{(a-1)} \oplus \mathbb{F}_{(b-1)}$ } & {\includegraphics[scale=0.8, valign=c]{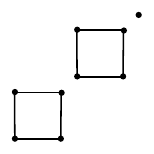}} \\

\hline

\multirow{4}{*}{(3,2,1)} & \makecell{ $\mathbb{F}_{(a)}\oplus \mathbb{F}_{(a+1)}\oplus \mathbb{F}_{(a-1)}$ \\ $\mathbb{F}_{(a-2)}$ } & {\includegraphics[scale=0.8, valign=c]{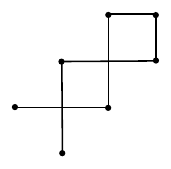}}\\

\cline{2-3}

& \makecell{ $\mathbb{F}_{(a)}\oplus \mathbb{F}_{(a+1)}\oplus \mathbb{F}_{(b)}$ \\ $\mathbb{F}_{(a-1)}\oplus \mathbb{F}_{(b-1)}\oplus \mathbb{F}_{(b-1)}$ } & {\includegraphics[scale=0.8,  valign=c]{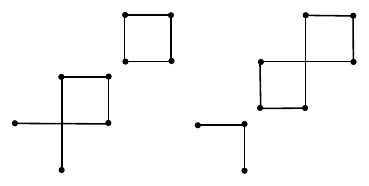}}\\

\cline{2-3}

& \makecell{ $\mathbb{F}_{(a)}\oplus \mathbb{F}_{(a+1)}\oplus \mathbb{F}_{(b)}$ \\ $\mathbb{F}_{(a)}\oplus \mathbb{F}_{(a-1)}\oplus \mathbb{F}_{(b-1)}$ } & {\includegraphics[scale=0.8, valign=c]{fig/1-6}}\\

\cline{2-3}

& \makecell{ $\mathbb{F}_{(a)}\oplus \mathbb{F}_{(b)}\oplus \mathbb{F}_{(c)}$ \\ $\mathbb{F}_{(b-1)}\oplus \mathbb{F}_{(c-1)}\oplus \mathbb{F}_{(a-1)}\oplus \mathbb{F}_{(b-1)}\oplus \mathbb{F}_{(c-1)}$ } & {\includegraphics[scale=0.8, valign=c]{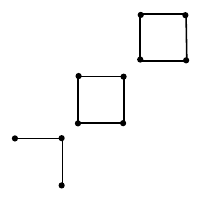}}\\

\hline
\end{tabular}
\caption{}
\label{table: 151 g(K) = 1}
\end{table}

\newpage

\section{Height 4 knots of the (1,3,1,3,1) type}\label{section: (1,3,1,3,1) type}

In this section, we classify the hv-basis for height 4 knots of the $(1,3,1,3,1)$ type. Similar to the $(1,5,1)$ type, we discuss separately according to the different cases of $\operatorname{rank} H_*(X_m)$. We state our main theorem at the beginning and then present the discussion case by case. Given a height 4 knot $K$ of the $(1,3,1,3,1)$ type with parameter $h$, in Section \ref{subsection:5.1 Cases of generic genera} we classify the hv-basis of $K$ under the assumption that the genus of $K$ is greater than 2 but not equal to $h+1$. In Section \ref{subsection:5.2 Cases of special genera}, we list the classification results for the cases $g(K) = 2$ and $g(K) = h+1$. 

The following is the main theorem of this section.

\begin{theorem}\label{thm: 131 has hv-basis} 
Suppose $K$ is a height 4 knot of the $(1,3,1,3,1)$ type with parameter $h$. Then $C  =  CFK^{\infty}(K)$ admits an hv-basis. Under this hv-basis, the horizontal and vertical differentials take the form of either a staircase with two separate twists, where the two twists are located at heights $\pm h$, or a staircase plus two separate boxes, where the two boxes are also located at height $\pm h$. The results are illustrated in Figure \ref{Figure: figure of main theorem for (1,3,1,3,1) type}.
\end{theorem}

\begingroup
\setlength{\intextsep}{0pt}
\begin{figure}[H]
	\centering
	\setlength{\abovecaptionskip}{1pt}
	\begin{minipage}[b]{0.32\linewidth}
		\centering
		\includegraphics[scale = 0.75]{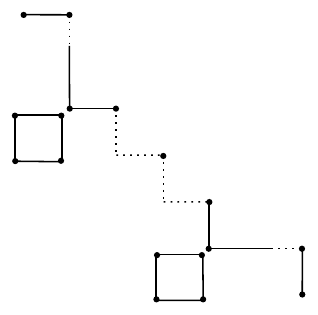}
		\caption*{(5)}
	\end{minipage}
	\begin{minipage}[b]{0.33\linewidth}
		\centering
		\includegraphics[scale = 0.75]{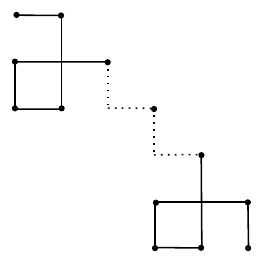}
		\caption*{$(6_{+})$}
	\end{minipage}
	\begin{minipage}[b]{0.33\linewidth}
		\centering
		\includegraphics[scale = 0.75]{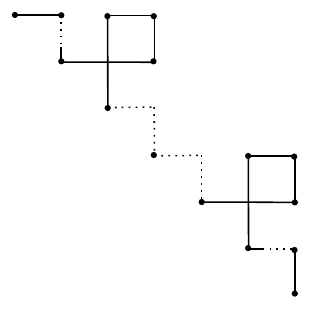}
		\caption*{$(6_{-})$}
	\end{minipage}
    \caption{hv-basis for height 4 knots of the $(1,3,1,3,1)$ type}
    \label{Figure: figure of main theorem for (1,3,1,3,1) type}
\end{figure}
\endgroup

\subsection{Cases of generic genera}\label{subsection:5.1 Cases of generic genera}

In this section, we classify the hv-basis for a height 4 knot $K$ of the $(1,3,1,3,1)$ type with parameter $h$, where $g(K) > 2$ and $g(K) \neq h+1$. Consider $CFK^\infty(K)$; let $x$ denote the generator of lowest Alexander grading strictly greater than $h+1$. We use $m$ to denote its Maslov grading and $A$ to denote its Alexander grading (i.e., $\widehat{HFK}(K, A)  =  \mathbb{F}_{(m)}$).

Applying Lemmas \ref{lemma: 1 and 2k+1 then k or k+1} and \ref{lemma: 2k+1 and 1 then k or k+1} and analyze the short exact sequence consisting of $UX_{p+1},X_{p}$ and $C\{(0,p)\}$, there are 16 possibilities for $H_*(X_m)$ with $m = h+2, h+1, h, h-1$, as shown in the tables below. Before giving the proof, we list the classification results for all cases in Tables \ref{table:(1,3,...,3,1)-1}, \ref{table:(1,3,...,3,1)-2}, \ref{table:(1,3,...,3,1)-3}, and \ref{table:(1,3,...,3,1)-4}.

\newpage

\[ (1,3,1,3,1),g(K) > h+1 \]

\begin{table}[H]
\centering
\begin{tabular}{|c|c|c|}
	\hline
	
	\makecell{$rk H_*(X_p)$ \\ $p = h+2,h+1,h,h-1$ } & \makecell{$\widehat{HFK}(K,i)$\\ $i = h+1,h,h-1$} & hv-basis \\
	
	\hline
	
	(0,1,1,0)& \makecell{$\mathbb{F}_{(a)}$\\ $\mathbb{F}_{(a-1)}\oplus\mathbb{F}_{(a-1)}$ \\ $\mathbb{F}_{(a-2)}$ } & {\includegraphics[scale = 0.58, margin =1pt, valign = c]{fig/7-0}} \\
	
	\hline
	
	(0,1,1,1)& \makecell{$\mathbb{F}_{(a)}$\\ $\mathbb{F}_{(a-1)}\oplus\mathbb{F}_{(a-1)}$ \\ $\mathbb{F}_{(m-1)}\oplus\mathbb{F}_{(a-2)}$ } & {\includegraphics[scale = 0.58, margin = 1pt, valign = c]{fig/7-1}} \\
	
	\hline
	
	\multirow{2}{*}{(0,1,2,0)}& \makecell{$\mathbb{F}_{(m+1)}$ \\ $\mathbb{F}_{(m)}$ \\ $\mathbb{F}_{(m-2)}\oplus\mathbb{F}_{(m-1)}$ }& {\includegraphics[scale = 0.58, margin = 1pt, valign = c]{fig/7-2}}  \\
	
	\cline{2-3}
	
	& \makecell{$\mathbb{F}_{(a)}$ \\ $\mathbb{F}_{(a-1)}\oplus\mathbb{F}_{(a-1)}\oplus\mathbb{F}_{(m-1)}$ \\ $\mathbb{F}_{(a-2)}\oplus\mathbb{F}_{(m-2)}$ } & {\includegraphics[scale = 0.58, margin = 1pt, valign = c]{fig/7-3}}  \\
	
	\hline
	
	\multirow{2}{*}{(0,1,2,1)}& \makecell{$\mathbb{F}_{(m+1)}$ \\ $\mathbb{F}_{(m)}$ \\ $\mathbb{F}_{(m-1)}$} & {\includegraphics[scale = 0.58, margin = 1pt, valign = c]{fig/7-4}}  \\
	
	\cline{2-3}
	
	& \makecell{$\mathbb{F}_{(a)}$ \\ $\mathbb{F}_{(a-1)}\oplus\mathbb{F}_{(a-1)}\oplus\mathbb{F}_{(m-1)}$ \\ $\mathbb{F}_{(a-2)}$ } & {\includegraphics[scale = 0.58, margin = 1pt, valign = c]{fig/7-5}}  \\
	
	\hline
\end{tabular}
\caption{}
\label{table:(1,3,...,3,1)-1}
\end{table}
\newpage

\begin{table}[H]
\centering
\begin{tabular}{|c|c|c|}
	\hline
	
	\multirow{2}{*}{(0,2,1,0)}& \makecell{$\mathbb{F}_{(m-2)}\oplus\mathbb{F}_{(m-1)}$ \\ $\mathbb{F}_{(m-3)}$ \\ $\mathbb{F}_{(m-4)}$} & {\includegraphics[scale = 0.65, margin = 1pt, valign = c]{fig/7-6}} \\
	
	\cline{2-3}
	
	& \makecell{ $\mathbb{F}_{(a)}\oplus\mathbb{F}_{(m-1)}$ \\ $\mathbb{F}_{(a-1)}\oplus\mathbb{F}_{(a-1)}\oplus\mathbb{F}_{(m-2)}$  \\ $\mathbb{F}_{(a-2)}$ } & {\includegraphics[scale = 0.65, margin = 1pt, valign = c]{fig/7-7}} \\
	
	\hline
	
	\multirow{2}{*}{(0,2,1,1)}& \makecell{$\mathbb{F}_{(m-2)}\oplus\mathbb{F}_{(m-1)}$ \\ $\mathbb{F}_{(m-3)}$ \\ $\mathbb{F}_{(m-4)}\oplus\mathbb{F}_{(m-3)}$} &{\includegraphics[scale = 0.65, margin = 1pt, valign = c]{fig/7-8}} 
	\\
	
	\cline{2-3}
	
	& \makecell{ $\mathbb{F}_{(m-1)}\oplus\mathbb{F}_{(b)}$ \\ $\mathbb{F}_{(m-2)}\oplus\mathbb{F}_{(b-1)}\oplus\mathbb{F}_{(b-1)}$  \\ $\mathbb{F}_{(m-3)}\oplus\mathbb{F}_{(b-2)}$ } & {\includegraphics[scale = 0.65, margin = 1pt, valign = c]{fig/7-9}} 
	\\
	
	\hline
	
	(0,2,2,0)& \makecell{ $\mathbb{F}_{(m-1)}\oplus\mathbb{F}_{(b)}$ \\ $\mathbb{F}_{(b-1)}\oplus\mathbb{F}_{(b-1)}$  \\ $\mathbb{F}_{(m-4)}\oplus\mathbb{F}_{(b-2)}$ } & {\includegraphics[scale = 0.65, margin = 1pt, valign = c]{fig/7-10}}  \\
	
	\hline
	
	(0,2,2,1)& \makecell{ $\mathbb{F}_{(m-1)}\oplus\mathbb{F}_{(a)}$ \\ $\mathbb{F}_{(a-1)}\oplus\mathbb{F}_{(a-1)}$  \\ $\mathbb{F}_{(a-2)}$ } & {\includegraphics[scale = 0.65, margin = 1pt, valign = c]{fig/7-11}}  \\
	
	\hline
\end{tabular}
\caption{}
\label{table:(1,3,...,3,1)-2}
\end{table}
\newpage

\begin{table}[H]
\centering
\begin{tabular}{|c|c|c|}
	
	\hline
	
	(1,1,1,0)& \makecell{$\mathbb{F}_{(a)}\oplus\mathbb{F}_{(\alpha+1)}$\\ $\mathbb{F}_{(a-1)}\oplus\mathbb{F}_{(a-1)}$ \\ $\mathbb{F}_{(a-2)}$ } & {\includegraphics[scale = 0.7, margin = 1pt, valign = c]{fig/7-0-1}} \\
	
	\hline
	(1,1,1,1)& \makecell{$\mathbb{F}_{(a)}\oplus\mathbb{F}_{(\alpha+1)}$\\ $\mathbb{F}_{(a-1)}\oplus\mathbb{F}_{(a-1)}$ \\ $\mathbb{F}_{(a-2)}\oplus\mathbb{F}_{(\alpha)}$ } & {\includegraphics[scale = 0.7, margin = 1pt, valign = c]{fig/7-1-1}} 
	\\
	
	\hline
	
	\multirow{2}{*}{(1,1,2,0)}& \makecell{$\mathbb{F}_{(\alpha+1)}\oplus\mathbb{F}_{(\alpha+2)}$ \\ $\mathbb{F}_{(\alpha+1)}$ \\ $\mathbb{F}_{(\alpha-1)}\oplus\mathbb{F}_{(\alpha)}$ }&{\includegraphics[scale = 0.7, margin = 1pt, valign = c]{fig/7-2-1}}  
	\\
	
	\cline{2-3}
	
	& \makecell{$\mathbb{F}_{(a)}\oplus\mathbb{F}_{(\alpha+1)}$ \\ $\mathbb{F}_{(a-1)}\oplus\mathbb{F}_{(a-1)}\oplus\mathbb{F}_{(\alpha)}$ \\ $\mathbb{F}_{(a-2)}\oplus\mathbb{F}_{(\alpha-1)}$ }  & {\includegraphics[scale = 0.7, margin = 1pt, valign = c]{fig/7-3-1}}  
	\\
	
	\hline
	
	\multirow{2}{*}{(1,1,2,1)}& \makecell{$\mathbb{F}_{(\alpha+1)}\oplus\mathbb{F}_{(\alpha+2)}$ \\ $\mathbb{F}_{(\alpha+1)}$ \\ $\mathbb{F}_{(\alpha)}$} &{\includegraphics[scale = 0.7, margin = 1pt, valign = c]{fig/7-4-1}} 
	\\
	
	\cline{2-3}
	
	& \makecell{$\mathbb{F}_{(a)}\oplus\mathbb{F}_{(\alpha+1)}$ \\ $\mathbb{F}_{(a-1)}\oplus\mathbb{F}_{(a-1)}\oplus\mathbb{F}_{(\alpha)}$ \\ $\mathbb{F}_{(a-2)}$ } & {\includegraphics[scale = 0.7, margin = 1pt, valign = c]{fig/7-5-1}} 
	\\
	
	\hline
\end{tabular}
\caption{}
\label{table:(1,3,...,3,1)-3}
\end{table}
\newpage

\begin{table}[H]
\centering
\begin{tabular}{|c|c|c|}
	\hline	
	
	\multirow{2}{*}{(1,2,1,0)}& \makecell{$\mathbb{F}_{(\alpha-1)}$ \\ $\mathbb{F}_{(\alpha-2)}$ \\ $\mathbb{F}_{(\alpha-3)}$} & {\includegraphics[scale = 0.65, margin = 1pt, valign = c]{fig/7-6-1}} 
	\\
	
	\cline{2-3}
	
	& \makecell{ $\mathbb{F}_{(a)}$ \\ $\mathbb{F}_{(a-1)}\oplus\mathbb{F}_{(a-1)}\oplus\mathbb{F}_{(\alpha-1)}$  \\ $\mathbb{F}_{(a-2)}$ } & {\includegraphics[scale = 0.65, margin = 1pt, valign = c]{fig/7-7-1}} 
	\\
	
	\hline
	
	\multirow{2}{*}{(1,2,1,1)}& \makecell{$\mathbb{F}_{(\alpha-1)}$ \\ $\mathbb{F}_{(\alpha-2)}$ \\ $\mathbb{F}_{(\alpha-3)}\oplus\mathbb{F}_{(\alpha-2)}$} & {\includegraphics[scale = 0.65, margin = 1pt, valign = c]{fig/7-8-1}} 
	\\
	
	\cline{2-3}
	
	& \makecell{ $\mathbb{F}_{(a)}$ \\ $\mathbb{F}_{(a-1)}\oplus\mathbb{F}_{(a-1)}\oplus\mathbb{F}_{(\alpha-1)}$  \\ $\mathbb{F}_{(a-2)}\oplus\mathbb{F}_{(\alpha-2)}$ } & {\includegraphics[scale = 0.65, margin = 1pt, valign = c]{fig/7-9-1}} 
	\\
	
	\hline
	
	(1,2,2,0)& \makecell{ $\mathbb{F}_{(a)}$ \\ $\mathbb{F}_{(a-1)}\oplus\mathbb{F}_{(a-1)}$  \\ $\mathbb{F}_{(a-2)}\oplus\mathbb{F}_{(\alpha-3)}$ } & {\includegraphics[scale = 0.65, margin = 1pt, valign = c]{fig/7-10-1}}  
	\\
	
	\hline
	
	(1,2,2,1)& \makecell{ $\mathbb{F}_{(a)}$ \\ $\mathbb{F}_{(a-1)}\oplus\mathbb{F}_{(a-1)}$  \\ $\mathbb{F}_{(a-2)}$ } & {\includegraphics[scale = 0.65, margin = 1pt, valign = c]{fig/7-11-1}} 
	\\
	
	\hline
\end{tabular}
\caption{}
\label{table:(1,3,...,3,1)-4}
\end{table}

\newpage
\settocdepth{section} % 临时只显示到 section
\subsubsection{(0,1,1,0) case}
\resettocdepth      % 恢复原 tocdepth

We start with the $(0,1,1,0)$ case. SImilar to the argument in Lemma \ref{lemma: 151 X_2 and Y_2}, we know that the condition $H_*(X_{h+2}) = 0$ is equivalent to $m$ being odd. 

Recall the short exact sequence \eqref{equation: exact sequence of X_p}:
\begingroup
\setlength{\abovedisplayskip}{3pt}
\setlength{\abovedisplayshortskip}{3pt}
\setlength{\belowdisplayskip}{3pt}
\setlength{\belowdisplayshortskip}{3pt}
\[ 0\rightarrow  UX_{p+1}\rightarrow X_{p}\rightarrow C\{(0,p)\}\rightarrow 0 \]
\endgroup
Taking $p = h+1,h-1$, we obtain that $\widehat{HFK}(K,h\pm 1)$ are one-dimensional. For $p = h$, we obtain that $\widehat{HFK}(K,h)$ is either zero or two-dimensional. We will frequently use the symmetric property of homology, $H_*(X_{-p}) \cong H_{*+2p}(Y_{p+1})$; see Equation \eqref{equation:(i,j) symmetry}.

\textbf{Subcase (\romannumeral1)}: Assume $\widehat{HFK}(K,h) = 0$, $\widehat{HFK}(K,h+1) = \mathbb{F}_{(a)}$, and $\widehat{HFK}(K,h-1) = \mathbb{F}_{(b)}$; see the left of Figure \ref{figure:(13131) (0,1,1,0) case(i)}(later computation gives $b=a-3$).

\begingroup
\setlength{\intextsep}{1pt}
\begin{figure}[H]
\centering
    \setlength{\abovecaptionskip}{0pt}
    \begin{minipage}[c]{0.48\textwidth}
	   \includegraphics[scale = 0.75]{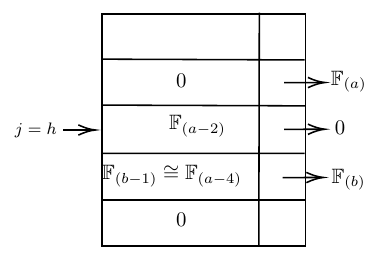}
    \end{minipage}
    \begin{minipage}[c]{0.48\textwidth}
	   \includegraphics[scale = 0.75]{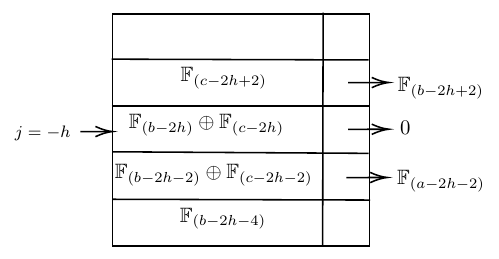}
    \end{minipage}
    \caption{}
    \label{figure:(13131) (0,1,1,0) case(i)}
\end{figure}
\endgroup
\vspace{-8pt}
Taking $p = h+1$ in Equation \eqref{equation: exact sequence of X_p}, we obtain $H_*(X_{h+1})\cong H_*(\{(0,h+1)\}) = \mathbb{F}_{(a)}$. Then $H_*({UX_{h+1}})\cong H_*({X_h}) \cong = \mathbb{F}_{(a-2)}$, hence $H_*({UX_h}) = \mathbb{F}_{(a-4)}$. The boundary map $\delta: H_*({\{(0,h-1)\}})\rightarrow H_*({UX_h})$ is an isomorphism of degree -1, then $b-1=a-4, H_*({\{(0,h-1)\}}) = \mathbb{F}_{(a-3)}$; see Figure \ref{figure:(13131) (0,1,1,0) case(i)}. 

By Spin$^c$ conjugation \cite[Proposition 3.10]{OSknotinvar}, we have $H_*({\{0,-h+1\}}) = \mathbb{F}_{(b-2h+2)}$ and
$H_*({\{0,-h-1\}}) = \mathbb{F}_{(a-2h-2)}$.

Recall $\operatorname{rank} H_*(X_p)\cong H_{*-2p}(Y_{-p+1})$, so $\operatorname{rank} H_*(Y_{-h+2})=0=\operatorname{rank} H_*(X_{h-1})$. Lemma \ref{lemma: 1 and 1 then 1 or 0} tells us that $\operatorname{rank}\ H_*(X_{-h+2})=1$. We assume $H_*({X_{-h+2}}) = \mathbb{F}_{(c-2h+4)}$, as illustrated in Figure \ref{figure:(13131) (0,1,1,0) case(i)}. We thus obtain $H_*({UX_{-h}}) = \mathbb{F}_{(b-2h-2)}\oplus\mathbb{F}_{(c-2h-2)}$. The boundary map $\delta$ must send $\mathbb{F}_{(a-2h-2)}$ to $\mathbb{F}_{(c-2h-2)}$, since $b-2h-2=a-2h-5$. We then get $c-2h-2 = a-2h-3, c=a-1$ and $H_*({X_{-h-1}}) = \mathbb{F}_{(b-2h-2)}=\mathbb{F}_{(a-2h-5)}$. Applying the $U$-action to send $X_{-h-1}$ downward to $UX_{-A}$, we obtain $m-2A-1 = b-2A = a-2A-3$, hence $a = m+2$, then $b=m-1, c=m+1$.

We thus obtain $\widehat{HFK}(K,h+1) = \mathbb{F}_{(m+2)}$, $\widehat{HFK}(K,h) = 0, \widehat{HFK}(K,h-1) = \mathbb{F}_{(m-1)}$.

The homology group $H_*(X_{-h+2})$ $ = $ $\mathbb{F}_{(c-2h+4)}$ $ = $ $\mathbb{F}_{(m-2h+5)}$. Assume the first nonzero generator above height $-h+1$ is located at Alexander grading $-k$ level, then $\widehat{HFK}(K,-k) = \mathbb{F}_{(m+1-2k)}$ and $\widehat{HFK}(K,k) = \mathbb{F}_{(m+1)}$. Equivalently, the first nonzero homology generator below height $h-1$ is $\mathbb{F}_{(m+1)}$.

\begingroup
\setlength{\intextsep}{1pt}
\begin{figure}[H]
    \setlength{\abovecaptionskip}{0pt}
    \begin{minipage}[t]{0.48\textwidth}
    \centering
       \includegraphics[scale = 0.75]{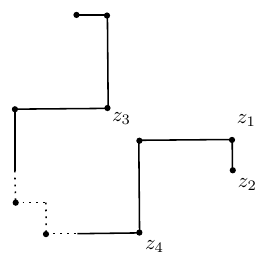} 
    \end{minipage}
	\begin{minipage}[t]{0.48\textwidth}
    \centering
       \includegraphics[scale = 0.75]{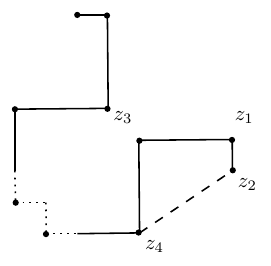} 
    \end{minipage}
    \caption{}
    \label{figure:(1,3,1,3,1) first case}
\end{figure}
\endgroup
\vspace{-8pt}

Under the constraints of Maslov grading and $H_*(Y_p)$, we obtain an hv-basis as depicted in Figure \ref{figure:(1,3,1,3,1) first case}.

We denote part of the generators by $z_i$ (see the right of Figure \ref{figure:(1,3,1,3,1) first case}). The length of the vertical differential between $z_3$ and $z_4$ is 2 and the length of the horizontal differential from $z_1$ to $z_3$ is no less than 3. To satisfy the condition $\partial^2 z_1 = 0$, there must be a diagonal differential from $z_2$ to $z_4$, since the only generators appearing in the rectangle region spanned by $z_1,z_3,z_4$ are $U-$ translations of $z_1,z_2,z_4$. There does not exist other paths from $z_1$ to $z_4$ by Maslov grading constraints. We now check the condition $\partial^2 z_2=0$. Since all horizontal and vertical differentials are already determined, there can not exist a second path from $z_2$ to $\partial z_4$. This gives a contradiction.

Thus, we exclude this subcase.

\textbf{Subcase (\romannumeral 2)}: Assume that $\widehat{HFK}(K,h+1) = \mathbb{F}_{(a)},$ $\widehat{HFK}(K,h) = \mathbb{F}_{(a-1)}$ $\oplus$ $\mathbb{F}_{(b)}$, and $\widehat{HFK}(K,h-1) = \mathbb{F}_{(b-1)}$. See Figure \ref{figure:(13131) (0,1,1,0) case(ii)}.

\begingroup
\setlength{\intextsep}{2pt}
\begin{figure}[H]
	\centering
    \setlength{\abovecaptionskip}{1pt}
    \begin{minipage}[c]{0.40\textwidth}
	   \includegraphics[scale = 0.8]{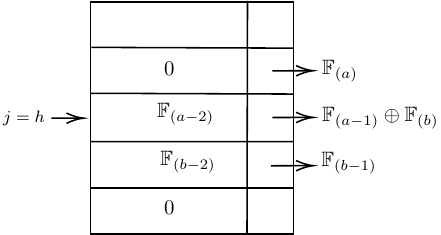}
    \end{minipage}
    \begin{minipage}[c]{0.52\textwidth}
	   \includegraphics[scale = 0.8]{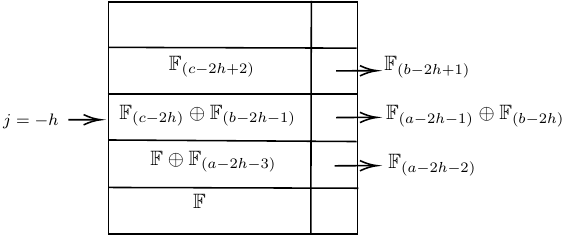}
    \end{minipage}
    \caption{}
    \label{figure:(13131) (0,1,1,0) case(ii)}
\end{figure}
\endgroup
\vspace{-8pt}

In the negative height part, assume $H_*({X_{-h+2}}) = \mathbb{F}_{(c-2h+4)}$, then $H_*(UX_{-h+2})=\mathbb{F}_{(c-2h+2)}$. We obtain
$H_*({X_{-h+1}}) = \mathbb{F}_{(c-2h+2)}$ $\oplus$ $\mathbb{F}_{(b-2h+1)}$. The boundary map $\delta: H_*({\{(0,-h-1)\}})\rightarrow H_*({UX_{-h}})$ sends $\mathbb{F}_{(a-2h-2)}$ to an $\mathbb{F}_{(a-2h-3)}$ component (in the direct sum) in $H_*({UX_{-h}})$, hence there is an $\mathbb{F}_{(a-2h-1)}$ component in $H_*(X_{-h})$.

We discuss separately according to the map $\delta: H_*({\{(0,-h)\}})\rightarrow H_*({UX_{-h+1}})$.

\begin{itemize}

    \item $\delta$ sends $\mathbb{F}_{(b-2h)}$ to $\mathbb{F}_{(b-2h-1)}$.
    
We obtain $H_*({X_{-h}})$ $ = $ $\mathbb{F}_{(a-2h-1)}$ $\oplus$ $\mathbb{F}_{(c-2h)}$, so $H_*({UX_{-h}})$ $ = $ $\mathbb{F}_{(a-2h-3)}$ $\oplus$ $\mathbb{F}_{(c-2h-2)}$ and $H_*({X_{-h-1}})$ $ = $ $\mathbb{F}_{(c-2h-2)}$. At height $-A$, we have $c-2A = m-2A-1$, hence $c = m-1$.

    \item $\delta$ sends $\mathbb{F}_{(a-2h-1)}$ to $\mathbb{F}_{(c-2h)}$.
    
Then $H_*({X_{-h}})$ $ = $ $\mathbb{F}_{(b-2h-1)}$ $\oplus$ $\mathbb{F}_{(b-2h)}$, so $H_*({UX_{-h}})$ $ = $ $\mathbb{F}_{(b-2h-3)}$ $\oplus$ $\mathbb{F}_{(b-2h-2)}$. Either $b-2h-3 = a-2h-3$ and $b-2A = m-2A-1$, or $b-2h-2 = a-2h-3$ and $b-2A-1 = m-2A-1$.

    \item $\delta$ sends $\mathbb{F}_{(a-2h-1)}$ to $\mathbb{F}_{(b-2h-1)}$.
    
Then $H_*({X_{-h}})$ $ = $ $\mathbb{F}_{(b-2h)}$ $\oplus$ $\mathbb{F}_{(c-2h)}$, so $H_*({UX_{-h}})$ $ = $ $\mathbb{F}_{(b-2h-2)}$ $\oplus$ $\mathbb{F}_{(c-2h-2)}$. Here $b-2h-2 = a-2h-3$, so $H_*({X_{-h-1}})$ $ = $ $\mathbb{F}_{(c-2h-2)}$. At height $-A$, we have $c-2A = m-2A-1$.
    
\end{itemize}

We list the computational results in Table \ref{table:131-3}.

\begingroup
\setlength{\intextsep}{2pt}
\begin{table}[H]
\centering
\setlength{\abovecaptionskip}{1pt}
\begin{tabular}{|c|c|c|c|c|}
\hline
$a$ & $a$ & $a$ & $m+1$ & $m-1$ \\
\hline
$b$ & $b$ & $a-1$ & $m$ & $m-1$ \\
\hline
$c$ & $m-1$ & $m-1$ & $m-1$ & $m-3$ \\
\hline
\end{tabular}
\caption{}
\label{table:131-3}
\end{table}
\endgroup

The situations in second and third columns are included in the first column. We first discuss the fourth column. 

\textit{If} $a = m-1,b = m-1,c = m-3$. the Maslov grading and homologies of $Y_p$ determine the vertical and horizontal differentials by grid climbing method, as shown in Figure \ref{figure:(13131) (0,1,1,0)-1}.

\begingroup
\setlength{\intextsep}{0pt}
\begin{figure}[H]
	\centering
    \setlength{\abovecaptionskip}{0pt}
    \begin{minipage}[c]{0.48\textwidth}
	   \includegraphics[scale = 0.7]{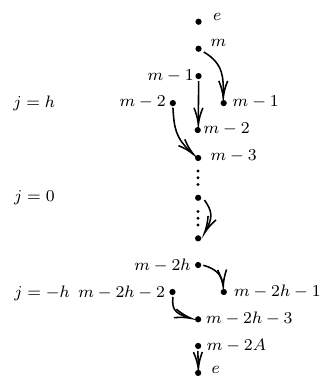}
    \end{minipage}
    \begin{minipage}[c]{0.48\textwidth}
	   \includegraphics[scale = 0.7]{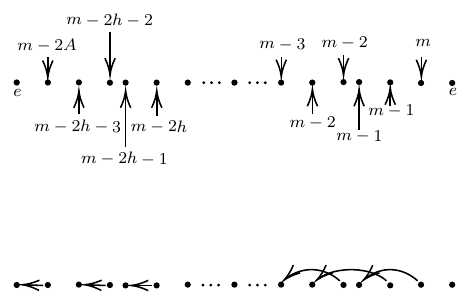}
    \end{minipage}
    \caption{}
    \label{figure:(13131) (0,1,1,0)-1}
\end{figure}
\endgroup
\vspace{-8pt}

Consider the $1\times 2$ rectangle with vertices $x_{(h+1,m-1)}$, $x_{(h-1,m-2)}$ and $U^{-h+1}y_{(-h,m-2h-1)}$. A vertical differential from the generator $Ux_{(A,m)}$ goes downward to $U^{-h+1}y_{(-h,m-2h-1)}$ (see Figure \ref{figure:(13131) (0,1,1,0)-2}). Apply Lemma \ref{lemma: 1xn lemma} and the constraint $\partial^2 x_{(h+1,m-1)}=0$, since the horizontal and vertical differentials are already exhausted, there does not exist another path from $x_{(h+1,m-1)}$ to $U^{-h+1}y_{(-h,m-2h-1)}$. We thus rule out this case.
\begingroup
\setlength{\intextsep}{0pt}
\begin{figure}[H]
    \centering
    \setlength{\abovecaptionskip}{0pt}
	\includegraphics[scale = 0.7]{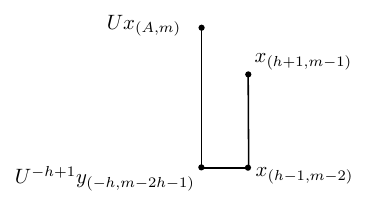}
    \caption{}
    \label{figure:(13131) (0,1,1,0)-2}
\end{figure}
\endgroup
\vspace{-8pt}

\textit{If $c = m-1$}, we discuss separately according to the relationships among $a, b, \text{and}\ m$. 

% In summarize, we either exclude the possibility or obtain an hv-basis whose configuration is a staircase plus two separate boxes.

\textit{If $b = a-1, a\neq m, a\neq m+1$}, we can obtain a vertically simplified basis by Maslov grading and $H_*(Y_p)$ constraints, as shown in Figure \ref{figure:(13131) (0,1,1,0)-3}. The horizontal differentials can also be determined by $H_*(Z_p)$ and Maslov grading. The hv-basis we obtain is a staircase plus two boxes located at heights $\pm h$. The number of generators between height $h-1$ and $-h+1$ is indeterminate. We draw three generators between height $h-1$ and $-h+1$ in the illustration, using dashed lines to represent differentials of indeterminate length (see Figure \ref{figure:(13131) (0,1,1,0)-3}).

\begingroup
\setlength{\intextsep}{0pt}
\begin{figure}[H]
	\centering
    \setlength{\abovecaptionskip}{0pt}
    \begin{minipage}[c]{0.4\textwidth}
	   \includegraphics[scale = 0.7]{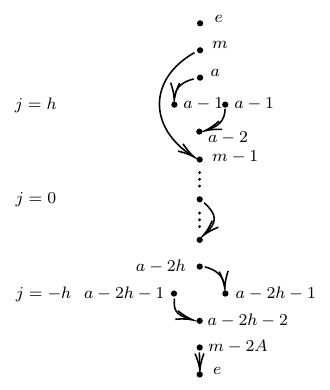}
    \end{minipage}
    \begin{minipage}[c]{0.3\textwidth}
	   \includegraphics[scale = 0.7]{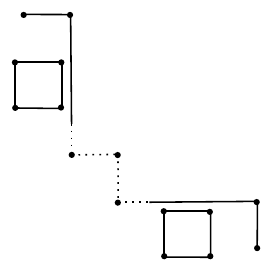}
    \end{minipage}
    \caption{}
    \label{figure:(13131) (0,1,1,0)-3}
\end{figure}
\endgroup
\vspace{-8pt}

\textit{If $b = a-1 = m$}, there are two possible vertically simplified bases; see Figure \ref{figure:(13131) (0,1,1,0)-4}.

\begingroup
\setlength{\intextsep}{1pt}
\begin{figure}[H]
    \centering
    \setlength{\abovecaptionskip}{0pt}
	\includegraphics[scale = 0.7]{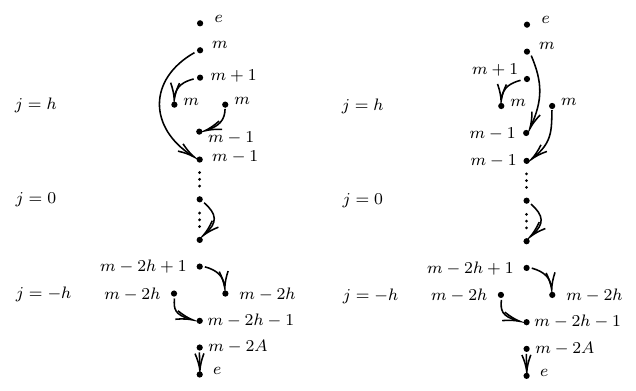}
    \caption{}
    \label{figure:(13131) (0,1,1,0)-4}
\end{figure}
\endgroup

The left one is similar to the previous case. We can determine the horizontal differentials and obtain an hv-basis, which is a staircase plus two boxes located at symmetric heights. For the right one, we determine the horizontal differentials by $H_*(Z_m)$ and Maslov grading constraints. Consider the generator $x_{(h,m,2)}$ in the obtained hv-basis, we can check that $\partial^2 x_{(h,m,2)}\neq 0$. Thus, we rule out this case. 

\begingroup
\setlength{\intextsep}{1pt}
\begin{figure}[H]
    \centering
    \setlength{\abovecaptionskip}{0pt}
	\includegraphics[scale = 0.7]{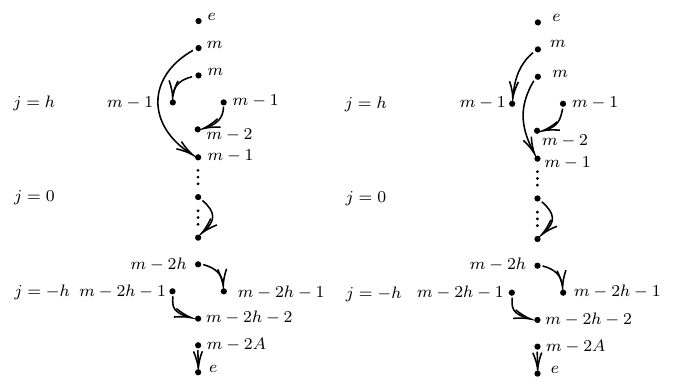}
    \caption{}
    \label{figure:(13131) (0,1,1,0)-5}
\end{figure}
\endgroup

\textit{If $b = a-1 = m-1$}, there are also two possible vertically simplified bases (see Figure \ref{figure:(13131) (0,1,1,0)-5}). We can turn the left one into an hv-basis by Maslov grading and $H_*(Z_m)$ constraints, which is a staircase plus two boxes at symmetric heights. For the right one, we determine the horizontal differentials and find that $\partial^2 x_{(h,m-1,2)}\neq 0$, thus we rule out this case.

\textit{If $b\neq a-1, b\neq m, a\neq m, a\neq m+1$}, we can obtain a vertically simplified basis as shown in Figure \ref{figure:(13131) (0,1,1,0)-6}. We then determine the horizontal differentials and find that $\partial^2 x_{(h+1,a)}\neq 0$, which rules out this case. 

\begingroup
\setlength{\intextsep}{1pt}
\begin{figure}[H]
    \centering
    \setlength{\abovecaptionskip}{0pt}
	\includegraphics[scale = 0.7]{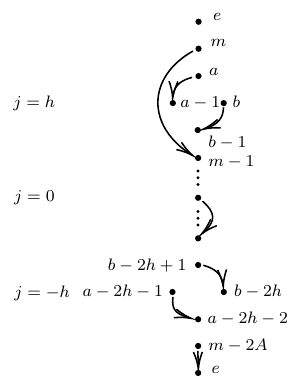}
    \caption{}
    \label{figure:(13131) (0,1,1,0)-6}
\end{figure}
\endgroup

\textit{If $b\neq a-1, b = m$}. Similarly, we first determine a vertically simplified basis, then determine the horizontal differentials and find $\partial^2 x_{(h+1,a)}\neq 0$. Thus we exclude this case.

\textit{If $b\neq a-1, b = m, a = m$}, we determine a vertically simplified basis first, during determining the horizontal differentials we find that $\partial^2 x_{(h+1,m)}\neq 0$. This case is excluded.

\textit{If $b\neq a-1, b\neq m, a = m$}, in the vertically simplified basis we obtain, we go to check that $\partial^2 x_{(h,b)}\neq 0$, which excludes this case.

\textit{If $b\neq a-1, b\neq m, a = m+1$}, again we check that $\partial^2 x_{(h,b)}\neq 0$, which rules out this case.

\settocdepth{section} % 临时只显示到 section
\subsubsection{(0,1,1,1) case}
\resettocdepth      % 恢复原 tocdepth

In the short exact sequence \eqref{equation: exact sequence of X_p}, taking $p = h+1,h,h-1$ shows that $\widehat{HFK}(K,h+1)$ is one-dimensional, $\widehat{HFK}(K,h)$ is either zero- or two-dimensional, and $\widehat{HFK}(K,h-1)$ is two-dimensional. 

\textbf{Subcase (i)}: Assume $\widehat{HFK}(K,h) = 0$.

We use the conjugation symmetry and analyze the short exact sequences \eqref{equation: exact sequence of X_p}, which gives $\widehat{HFK}(K,h+1) = \mathbb{F}_{(m+2)}$, $\widehat{HFK}(K,h-1)$ $ = $ $\mathbb{F}_{(m-1)}$ $\oplus$ $\mathbb{F}_{(m+1)}$. The first nonzero generator above height $-h+1$ has Maslov grading $m-2h+4$. In this case, Maslov grading and $H_*(Y_p), H_*(Z_m)$ determine the vertical and horizontal differentials. We obtain an hv-basis as illustrated in Figure \ref{figure:(13131) (0,1,1,1)-1}.

\begingroup
\setlength{\intextsep}{1pt}
\begin{figure}[H]
    \centering
    \setlength{\abovecaptionskip}{0pt}
	\includegraphics[scale = 0.8]{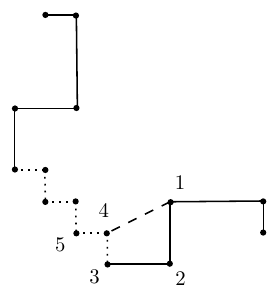}
    \caption{}
    \label{figure:(13131) (0,1,1,1)-1}
\end{figure}
\endgroup

We use numbers to indicate generators, as shown in the figure above. To satisfy $\partial^2 1=0$, there must be a diagonal differential from $1$ to $4$, creating a path from $1$ to $5$ under two successive differentiations. However, no second path from $1$ to $5$ could exist because all horizontal and vertical differentials are already determined. This is a contradiction, and we rule out this case.

\textbf{Subcase (ii)}: Assume $\widehat{HFK}(K,h+1) = \mathbb{F}_{(a)}$, $\widehat{HFK}(K,h) = $ $\mathbb{F}_{(a-1)}$ $\oplus$ $\mathbb{F}_{(b)}$, and $\widehat{HFK}(K,h-1) = \mathbb{F}_{(b-1)}$ $\oplus$ $\mathbb{F}_{(c)}$; see Figure \ref{figure:(13131) (0,1,1,1)-2}.

\begingroup
\setlength{\intextsep}{1pt}
\begin{figure}[H]
	\centering
    \setlength{\abovecaptionskip}{1pt}
    \begin{minipage}[c]{0.45\textwidth}
	   \includegraphics[scale = 0.8]{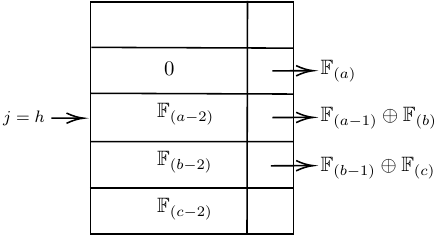}
    \end{minipage}
    \begin{minipage}[c]{0.52\textwidth}
	   \includegraphics[scale = 0.8]{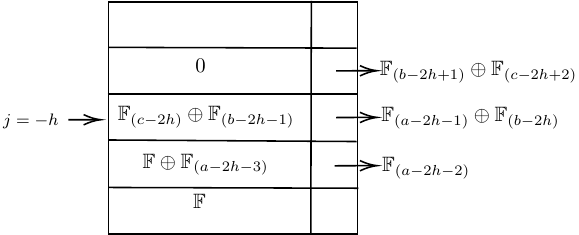}
    \end{minipage}
    \caption{}
    \label{figure:(13131) (0,1,1,1)-2}
\end{figure}
\endgroup

We use Maslov grading constraint and analyze the exact sequences \eqref{equation: exact sequence of X_p}. We list the five possible results in Table \ref{table:131-4}. The situations in the first four columns are included into the first column. We deal with the fifth column first.

\begingroup
\setlength{\intextsep}{1pt}
\begin{table}[H]
\centering
\setlength{\abovecaptionskip}{1pt}
\begin{tabular}{|c|c|c|c|c|c|}
\hline
$a$  &  $a$    &  $a$   &  $a$    &  $m+1$ &  $m-1$  \\
\hline
$b$  &  $b$    &  $m$   &  $a-1$  &  $m$   &  $m-1$  \\
\hline
$c$  &  $m-1$  &  $m-1$ &  $m-1$  &  $m-1$ &  $m-3$  \\
\hline
\end{tabular}
\caption{}
\label{table:131-4}
\end{table}
\endgroup

\textit{If $a = m-1, b = m-1, c = m-3$}, we use Maslov grading and $H_*(Y_p), H_*(Z_m)$ to constrain the horizontal and vertical differentials. We obtain an hv-basis, part of which is illustrated in Figure \ref{figure:(13131) (0,1,1,1)-3}.

\begingroup
\setlength{\intextsep}{0pt}
\begin{figure}[H]
    \centering
    \setlength{\abovecaptionskip}{0pt}
	\includegraphics[scale = 0.8]{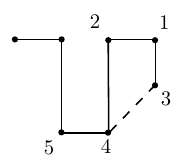}
    \caption{}
    \label{figure:(13131) (0,1,1,1)-3}
\end{figure}
\endgroup

As shown in the figure above, consider the generators labeled 2, 4, and 5. There could not exist a second path from 2 to 5, which gives a contradiction since $\partial^2 2=0$, and this rules out this case.

For $c = m-1$, we discuss separately according to the relationship between $a$ and $b$.

\textit{If $b = a-1, a\neq m, a\neq m+1$}, under the constraints of Maslov grading and $H_*(Y_p), H_*(Z_m)$ we obtain an hv-basis as shown in Figure \ref{figure:(13131) (0,1,1,1)-4}. The resulting hv-basis is a staircase plus two boxes, located at symmetric heights.

\textit{If $b = a-1 = m-1$ or $b = a-1 = m$}, we shall discuss different possibilities of vertical simplified basis, then determine the corresponding horizontal differentials. Either the corresponding possibility is excluded by the $\partial^2=0 $ constraint or we obtain an hv-basis identical to the one shown in Figure \ref{figure:(13131) (0,1,1,1)-4}.

\begingroup
\setlength{\intextsep}{0pt}
\begin{figure}[H]
	\centering
    \setlength{\abovecaptionskip}{0pt}
    \begin{minipage}[c]{0.5\textwidth}
        \centering
	   \includegraphics[scale = 0.7]{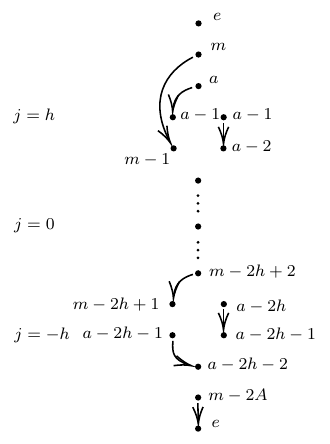}
    \end{minipage}
    \begin{minipage}[c]{0.48\textwidth}
        \centering
	   \includegraphics[scale = 0.8]{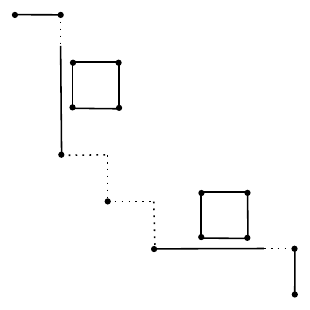}
    \end{minipage}
    \caption{}
    \label{figure:(13131) (0,1,1,1)-4}
\end{figure}
\endgroup

\textit{If $b \neq a-1, a \neq m, a \neq m+1$}, we can rule out this case by checking that $\partial^2 x_{(h+1,a)} \neq 0$.

\textit{If $b \neq a-1, a = m, b = m$}, we can rule out this case by checking that $\partial^2 x_{(h,m)} \neq 0$.

\textit{If $b \neq a-1, a = m+1, b \neq m$}, we can rule out this case since $\partial^2 x_{(h,b)} \neq 0$.

\textit{If $b \neq a-1, b \neq m, a = m$}, we can rule out this case since $\partial^2 x_{(h,b)} \neq 0$.

\settocdepth{section} % 临时只显示到 section
\subsubsection{(0,1,2,0) case}
\resettocdepth      % 恢复原 tocdepth

From the short exact sequence \eqref{equation: exact sequence of X_p}, we obtain that $\widehat{HFK}(K,h+1)$ is one-dimensional, $\widehat{HFK}(K,h)$ is one- or three- dimensional, and $\widehat{HFK}(K,h-1)$ is two-dimensional.

\textbf{Subcase (i)}: Assume $\widehat{HFK}(K,h+1) = \mathbb{F}_{(a)}$, $\widehat{HFK}(K,h) = \mathbb{F}_{(b)}$, and $\widehat{HFK}(K,h-1)$ $ = $ $\mathbb{F}_{(a-3)}$ $\oplus$ $\mathbb{F}_{(b-1)}$.

Denote the first non-zero generator below height $h-1$ by $\mathbb{F}_{(d)}$. Assume $H_*(UX_{-h+2})$ $\cong$ $\mathbb{F}_{(c)}$. We use Maslov grading and exact sequences \eqref{equation: exact sequence of X_p} to determine $a,b,c,d$. There are three possibilities, as shown in Table \ref{table:131-5}.

\begingroup
\setlength{\intextsep}{2pt}
\begin{table}[H]
\centering
\setlength{\abovecaptionskip}{1pt}
\begin{tabular}{|c|c|c|c|}
\hline
$a$  &  $m-1$    &  $m+1$    &  $m+2$       \\
\hline
$b$  &  $m-1$    &  $m$      &  $m+1$       \\
\hline
$c$  &  $m-2h-3$ &  $m-2h-1$ &  $m-2h+1$    \\
\hline
$d$  &  $m-5$    &  $m-3$    &  $m-1$       \\
\hline
\end{tabular}
\caption{}
\label{table:131-5}
\end{table}
\endgroup
\vspace{-8pt}

For the first column, we first determine a vertical simplified basis under constriant of Maslov grading and $H_*(Y_p)$. Then we determine the horizontal differentials under constraints of $H_*(Z_m)$ and Maslov grading. We then check that $\partial^2 x_{(h+1,m-1)} \neq 0$, which helps to rule out this possibility.

For the second column, we determine a vertically simplified basis under constraints of Maslov grading and $H_*(Y_p)$, as shown on the left in Figure \ref{figure:(13131) (0,1,2,0)-1}. The horizontal differentials are determined by $H_*(Z_m)$ and Maslov grading constraint. The hv-basis we obtain is a staircase with two twists, located at symmetric heights; see the right side of Figure \ref{figure:(13131) (0,1,2,0)-1}. 

\begingroup
\setlength{\intextsep}{1pt}
\begin{figure}[H]
	\centering
    \setlength{\abovecaptionskip}{1pt}
    \begin{minipage}[c]{0.4\textwidth}
        \centering
	   \includegraphics[scale = 0.7]{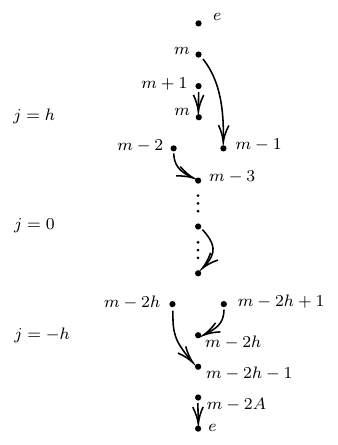}
    \end{minipage}
    \begin{minipage}[c]{0.5\textwidth}
        \centering
	   \includegraphics[scale = 0.8]{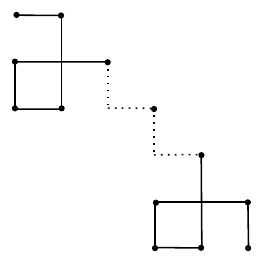}
    \end{minipage}
    \caption{}
    \label{figure:(13131) (0,1,2,0)-1}
\end{figure}
\endgroup
\vspace{-8pt}

For the third column, we determine an hv-basis under constraints of Maslov grading and $H_*(Y_p), H_*(Z_m)$. We then check that $\partial^2 x_{(h-1,m-1)}\neq 0$, which rules out this possibility. The left and right sides of Figure \ref{figure:(13131) (0,1,2,0)-2} illustrate the contradictions $\partial^2 1\neq 0$ for the first and third columns, respectively.

\begingroup
\setlength{\intextsep}{0pt}
\begin{figure}[H]
    \centering
    \setlength{\abovecaptionskip}{0pt}
	\includegraphics[scale = 0.8]{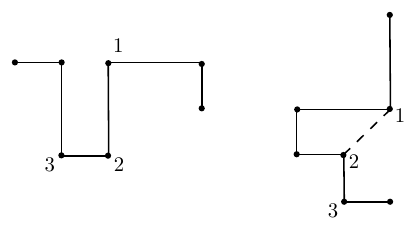}
    \caption{}
    \label{figure:(13131) (0,1,2,0)-2}
\end{figure}
\endgroup

\textbf{Subcase (ii)}: Assume $\widehat{HFK}(K,h+1) = \mathbb{F}_{(a)}$, $\widehat{HFK}(K,h) = \mathbb{F}_{(a-1)}\oplus\mathbb{F}_{(b)}\oplus\mathbb{F}_{(c)}$ and $\widehat{HFK}(K,h-1) = \mathbb{F}_{(b-1)}\oplus\mathbb{F}_{(c-1)}$.

\begingroup
\setlength{\intextsep}{1pt}
\begin{figure}[H]
	\centering
    \setlength{\abovecaptionskip}{0pt}
    \begin{minipage}[c]{0.4\textwidth}
        \centering
	   \includegraphics[scale = 0.7]{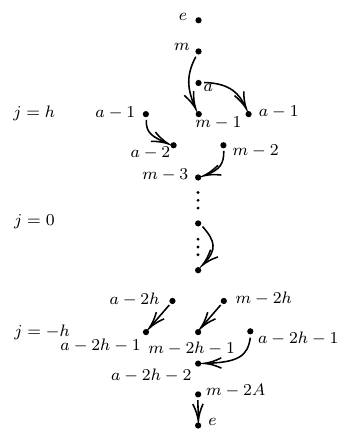}
    \end{minipage}
    \begin{minipage}[c]{0.5\textwidth}
        \centering
	   \includegraphics[scale = 0.8]{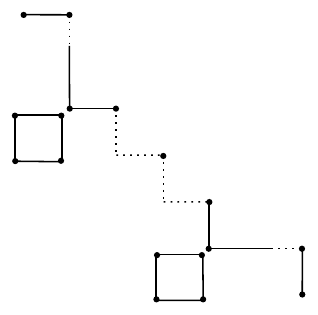}
    \end{minipage}
    \caption{}
    \label{figure:(13131) (0,1,2,0)-3}
\end{figure}
\endgroup

Maslov grading and $H_*(Y_p)$ constraints gives $c = m-1$, and the first non-zero generator below height $h-1$ has Maslov grading $m-3$. We discuss separately according to the relationships among $a, b, \text{and}\ m$. 

\textit{If $b = a-1,a\neq m$}, we determine an hv-basis under the constraints of Maslov grading and $H_*(Y_p), H_*(Z_m)$. The hv-basis we obtain is a staircase plus two boxes, located at symmetric heights, as shown in Figure \ref{figure:(13131) (0,1,2,0)-3}.

\textit{If $b = a-1 = m-1$}, up to a reordering of generators with identical Maslov gradings and Alexander filtration levels, we obtain an identical vertically simplified basis. The horizontal differentials are determined by the constraints of Maslov grading and $H_*(Z_m)$. The hv-basis we obtain is again a staircase plus two boxes, as shown in Figure \ref{figure:(13131) (0,1,2,0)-3}.

\textit{If $b\neq a-1, a\neq m, b\neq m-1$}, we can verify $\partial^2 x_{(h+1,a)}\neq 0$ in the obtained hv-basis, which rules out this case.

\textit{If $b\neq a-1, b = m-1$}, in the obtained hv-basis we find $\partial^2 x_{(h+1,a)}\neq 0$, which rules out this case.

\textit{If $b\neq a-1, a = m$}, in the obtained hv-basis we can check that $\partial^2 x_{(h,b)}\neq 0$, which rules out this case.

\settocdepth{section} % 临时只显示到 section
\subsubsection{(0,1,2,1) case}
\resettocdepth      % 恢复原 tocdepth

Taking $p = h+1,h,h-1$ in Equation \eqref{equation: exact sequence of X_p}, we obtain that $\widehat{HFK}(K,h+1)$ and $\widehat{HFK}(K,h-1)$ are one-dimensional, while $\widehat{HFK}(K,h)$ is either one- or three- dimensional.

\textbf{Subcase (i)}: Assume $\widehat{HFK}(K,h+1) = \mathbb{F}_{(a)}$, $\widehat{HFK}(K,h) = \mathbb{F}_{(b)}$, and $\widehat{HFK}(K,h-1) = \mathbb{F}_{(c)}$.

We list the computational results in Table \ref{table:131-6}. There are three possibilities.

\begingroup
\setlength{\intextsep}{2pt}
\begin{table}[H]
\centering
\setlength{\abovecaptionskip}{1pt}
\begin{tabular}{|c|c|c|c|}
\hline
$a$  &  $m+2$ &  $m-1$ &  $m+1$    \\
\hline
$b$  &  $m+1$ &  $m-1$ &  $m$      \\
\hline
$c$  &  $m-1$ &  $m-2$ &  $m-1$    \\
\hline
\end{tabular}
\caption{}
\label{table:131-6}
\end{table}
\endgroup

For the first column, we determine the vertical and horizontal differentials under constraints of Maslov grading and $H_*(Y_p), H_*(Z_m)$. In the obtained hv-basis, we can verify $\partial^2 x_{(h-1,m-1)} \neq 0$, which rules out this possibility.

For the second column, in the obtained hv-basis we check that $\partial^2 x_{(h+1,m-1)}\neq 0$, which rules out this possibility.

For the third column, the hv-basis we obtain is a staircase with two positive twists, as shown in Figure \ref{figure:(13131) (0,1,2,1)-1}.

\begingroup
\setlength{\intextsep}{1pt}
\begin{figure}[H]
	\centering
    \setlength{\abovecaptionskip}{0pt}
    \begin{minipage}[c]{0.4\textwidth}
        \centering
	   \includegraphics[scale = 0.7]{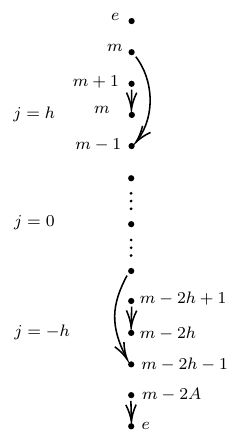}
    \end{minipage}
    \begin{minipage}[c]{0.5\textwidth}
        \centering
	   \includegraphics[scale = 0.8]{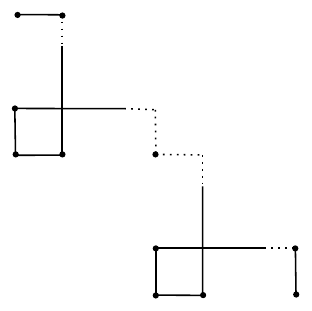}
    \end{minipage}
    \caption{}
    \label{figure:(13131) (0,1,2,1)-1}
\end{figure}
\endgroup

\textbf{Subcase (ii)}: Assume $\widehat{HFK}(K,h+1) = \mathbb{F}_{(a)}$, $\widehat{HFK}(K,h) = \mathbb{F}_{(a-1)}\oplus\mathbb{F}_{(b)}\oplus\mathbb{F}_{(c)}$, and $\widehat{HFK}(K,h-1) = \mathbb{F}_{(c-1)}$.

Maslov grading and $H_*(Y_p)$ constraints yield $b = m-1$. We discuss separately according to the relationships among $a, c, \text{and}\ m$.

\textit{If $c\neq a-1$}, there could be different possibilities of vertically simplified basis. In the obtained hv-basis we can verify either $\partial^2$ $x_{(h+1,a)}$ $\neq 0$ or $\partial^2$ $x_{(h,c)}$ $\neq 0$. We then rule out this possibility.

\textit{If $c = a-1$ and $a\neq m$}, we determine the vertical and horizontal differentials under the constraints of Maslov gradings and $H_*(Y_p), H_*(Z_m)$. The hv-basis we obtain is a staircase plus two boxes, as shown in Figure \ref{figure:(13131) (0,1,2,1)-2}.

\begingroup
\setlength{\intextsep}{1pt}
\begin{figure}[H]
	\centering
    \setlength{\abovecaptionskip}{0pt}
    \begin{minipage}[c]{0.4\textwidth}
        \centering
	   \includegraphics[scale = 0.7]{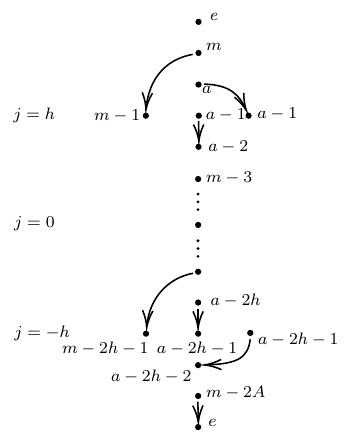}
    \end{minipage}
    \begin{minipage}[c]{0.5\textwidth}
        \centering
	   \includegraphics[scale = 0.8]{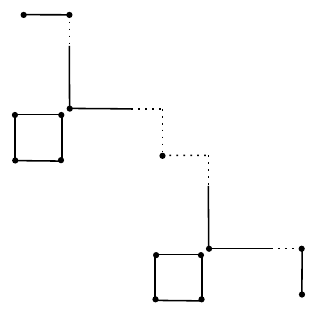}
    \end{minipage}
    \caption{}
    \label{figure:(13131) (0,1,2,1)-2}
\end{figure}
\endgroup

\textit{If $c = a-1 = m-1$}, We determine the vertical and horizontal differentials under the constraints of Maslov gradings and $H_*(Y_p), H_*(Z_m)$, and reorder the generators with identical Maslov gradings and Alexander filtration levels properly. The hv-basis we obtain is again a staircase plus two boxes, located at symmetric heights; see Figure \ref{figure:(13131) (0,1,2,1)-2}.

\settocdepth{section} % 临时只显示到 section
\subsubsection{(0,2,1,0) case}
\resettocdepth      % 恢复原 tocdepth

In this case $\widehat{HFK}(K,h+1)$ is two-dimensional, $\widehat{HFK}(K,h-1)$ is one-dimensional, and $\widehat{HFK}(K,h)$ is one- or three-dimensional.

\textbf{Subcase (i)}: Assume that $\widehat{HFK}(K,h+1) = \mathbb{F}_{(a)}\oplus\mathbb{F}_{(b)}$, $\widehat{HFK}(K,h) = \mathbb{F}_{(b-1)}$, and $\widehat{HFK}(K,h-1) = \mathbb{F}_{(a-3)}$.

Denote the first nonzero generator below height $h-1$ by $\mathbb{F}_{(c)}$. There are three possibilities on $a,b,c$, as shown in Table \ref{table:131-7}:

\begingroup
\setlength{\intextsep}{1pt}
\begin{table}[H]
\centering
\setlength{\abovecaptionskip}{0pt}
\begin{tabular}{|c|c|c|c|}
\hline
$a$  &  $m-1$ &  $m$   &  $m-1$    \\
\hline
$b$  &  $m-3$ &  $m-1$ &  $m-2$  \\
\hline
$c$  &  $m-5$ &  $m-1$ &  $m-3$  \\
\hline
\end{tabular}
\caption{}
\label{table:131-7}
\end{table}
\endgroup

For the first column, we determine an hv-basis under constraints of Maslov grading and $H_*(Y_p), H_*(Z_m)$. We then check that $\partial^2 x_{(h,m-4)}\neq 0$, which rules out this possibility.

For the second column, in the obtained hv-basis we find that $\partial^2 x_{(h,m-2)} \neq 0$, which rules out this possibility.

For the third column, we determine the vertical and horizontal differentials under constraints of Maslov grading and $H_*(Y_p). H_*(Z_m)$. We obtain an hv-basis whose configuration is a staircase with two negative twists, as shown in Figure \ref{figure:(13131) (0,2,1,0)-1}.

\begingroup
\setlength{\intextsep}{1pt}
\begin{figure}[H]
	\centering
    \setlength{\abovecaptionskip}{0pt}
    \begin{minipage}[c]{0.4\textwidth}
        \centering
	   \includegraphics[scale = 0.7]{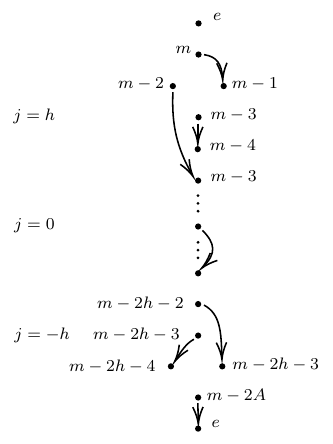}
    \end{minipage}
    \begin{minipage}[c]{0.5\textwidth}
        \centering
	   \includegraphics[scale = 0.8]{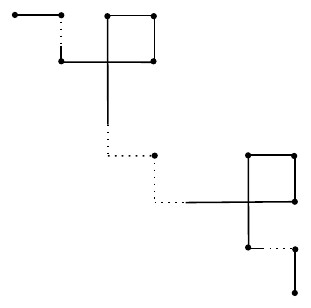}
    \end{minipage}
    \caption{}
    \label{figure:(13131) (0,2,1,0)-1}
\end{figure}
\endgroup

\textbf{Subcase (ii)}: Assume $\widehat{HFK}(K,h+1) = \mathbb{F}_{(a)}\oplus\mathbb{F}_{(b)}$, $\widehat{HFK}(K,h) = \mathbb{F}_{(a-1)}\oplus\mathbb{F}_{(b-1)}\oplus\mathbb{F}_{(c)}$, and $\widehat{HFK}(K,h-1) = \mathbb{F}_{(c-1)}$.

Denote the first nonzero generator below height $h-1$ by $\mathbb{F}_{(d)}$. We use conjugation symmetry and exact sequences \eqref{equation: exact sequence of X_p} to determine $a,b,c,d$. We could obtain $d = m-3$ and one of $a$ and $b$ must be $m-1$. We assume $a = m-1$ without loss of generality. We discuss separately according to the relationships among $b, c, \text{and}\ m$.

\textit{If $c = b-1, b\neq m-1$}, we determine a vertically simplified basis under constraints of Maslov grading and $H_*(Y_p)$. The horizontal differentials are determined under constraints of Maslov grading and $H_*(Z_m)$. The obtained hv-basis is a staircase plus two boxes, located at symmetric heights; see Figure \ref{figure:(13131) (0,2,1,0)-2}.

\begingroup
\setlength{\intextsep}{1pt}
\begin{figure}[H]
	\centering
    \setlength{\abovecaptionskip}{0pt}
    \begin{minipage}[c]{0.4\textwidth}
        \centering
	   \includegraphics[scale = 0.7]{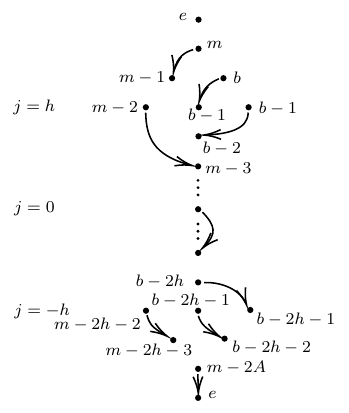}
    \end{minipage}
    \begin{minipage}[c]{0.5\textwidth}
        \centering
	   \includegraphics[scale = 0.8]{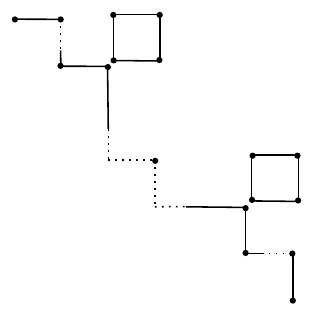}
    \end{minipage}
    \caption{}
    \label{figure:(13131) (0,2,1,0)-2}
\end{figure}
\endgroup
\vspace{-8pt}

\textit{If $c = b-1, b = m-1$}, under the constraints of Maslov grading and $H_*(Y_p), H_*(Z_m)$, we obtain an hv-basis identical to the previous case after reordering the generators with identical Maslov grading and Alexander filtration levels properly, which is again a staircase plus two boxes located at symmetric heights.

The case $c\neq b-1$ can be ruled out, as explained below:

\textit{If $c \neq b-1, b \neq m-1, c \neq m-2$}, we determine a vertically simplified basis first and then the horizontal differentials. In the obtained hv-basis we can verify $H_*(A_h) \neq \mathbb{F}^3$.

\textit{If $c\neq b-1, b = m-1, c\neq m-2$}, in the obtained hv-basis we can verify $\partial^2 x_{(h,c)}\neq 0$.

\textit{If $c\neq b-1, b\neq m-1, c = m-2$}, in the obtained hv-basis we can verify $\partial^2 x_{(h+1,b)}\neq 0$.

\settocdepth{section} % 临时只显示到 section
\subsubsection{(0,2,1,1) case}
\resettocdepth      % 恢复原 tocdepth

In this case $\widehat{HFK}(K,h+1)$ and $\widehat{HFK}(K,h-1)$ are two-dimensional, while $\widehat{HFK}(K,h)$ is either one- or three-dimensional.

\textbf{Subcase (i)}: $\widehat{HFK}(K,h+1) = \mathbb{F}_{(a)}\oplus\mathbb{F}_{(b)}$, $\widehat{HFK}(K,h) = \mathbb{F}_{(b-1)}$, and $\widehat{HFK}(K,h-1) = \mathbb{F}_{(a-3)}\oplus\mathbb{F}_{(c)}$.

Computation yields three possibilities, as shown in Table \ref{table:131-8}. 

\begingroup
\setlength{\intextsep}{2pt}
\begin{table}[H]
\centering
\setlength{\abovecaptionskip}{1pt}
\begin{tabular}{|c|c|c|c|}
\hline
$a$  &  $m$ &  $m-1$   &  $m-1$    \\
\hline
$b$  &  $m-1$ &  $m-2$ &  $m-3$  \\
\hline
$c$  &  $m-1$ &  $m-3$ &  $m-5$  \\
\hline
\end{tabular}
\caption{}
\label{table:131-8}
\end{table}
\endgroup
\vspace{-8pt}
For the first column, we construct an hv-basis under the constraints of Maslov grading and $H_*(Y_p), H_*(Z_m)$. We then verify $\partial^2 x_{(A,m)}\neq 0$, which rules out this case.

\begingroup
\setlength{\intextsep}{1pt}
\begin{figure}[H]
	\centering
    \setlength{\abovecaptionskip}{1pt}
    \begin{minipage}[c]{0.4\textwidth}
        \centering
	   \includegraphics[scale = 0.7]{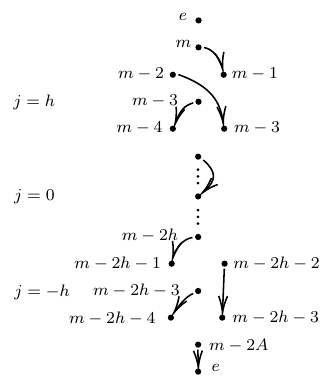}
    \end{minipage}
    \begin{minipage}[c]{0.5\textwidth}
        \centering
	   \includegraphics[scale = 0.8]{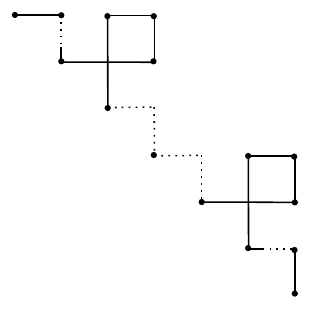}
    \end{minipage}
    \caption{}
    \label{figure:(13131) (0,2,1,1)-1}
\end{figure}
\endgroup

For the second column, Maslov grading and $H_*(Y_p)$ can determine a vertically simplified basis; see the left of Figure \ref{figure:(13131) (0,2,1,1)-1}. The horizontal differentials are determined by constraints of Maslov grading and $H_*(Z_m)$. The hv-basis we obtain is a staircase with two negative twists, located at symmetric heights; see the right of Figure \ref{figure:(13131) (0,2,1,1)-1}.

For the third column, in the obtained hv-basis, we check that $\partial^2 y_{(h+1,m-3)}\neq 0$, which rules out this possibility.

\textbf{Subcase (ii)}: $\widehat{HFK}(K,h+1) = \mathbb{F}_{(a)}\oplus\mathbb{F}_{(b)}$, $\widehat{HFK}(K,h) = \mathbb{F}_{(a-1)}\oplus\mathbb{F}_{(b-1)}\oplus\mathbb{F}_{(c)}$, and $\widehat{HFK}(K,h-1) = \mathbb{F}_{(c-1)}\oplus\mathbb{F}_{(d)}$.

Analyze the exact sequences \eqref{equation: exact sequence of X_p}, we obtain $d = m-3$ and $a=m-1$. We discuss separately according to the relationships between $b,c,m$.

\textit{If $c = b-1, b\neq m-1$}, we first determine a vertically simplified basis under constraints of Maslov grading and $H_*(Y_p)$, as shown in the left side of Figure \ref{figure:(13131) (0,2,1,1)-2}. The horizontal differentials are determined under constraints of Maslov grading and $H_*(Z_m)$. The obtained hv-basis is a staircase plus two boxes, located in symmetric heights; see Figure \ref{figure:(13131) (0,2,1,1)-2}.

\begingroup
\setlength{\intextsep}{1pt}
\begin{figure}[H]
	\centering
    \setlength{\abovecaptionskip}{1pt}
    \begin{minipage}[c]{0.4\textwidth}
        \centering
	   \includegraphics[scale = 0.7]{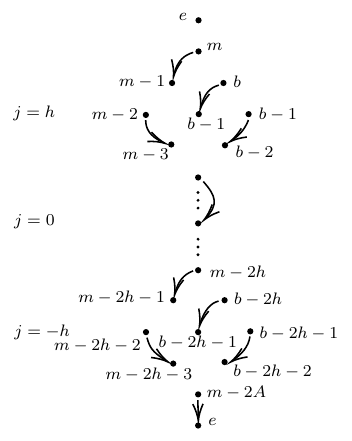}
    \end{minipage}
    \begin{minipage}[c]{0.5\textwidth}
        \centering
	   \includegraphics[scale = 0.8]{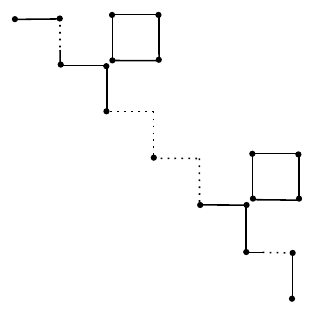}
    \end{minipage}
    \caption{}
    \label{figure:(13131) (0,2,1,1)-2}
\end{figure}
\endgroup

\textit{If $c = b-1, b = m-1$}, under the constraints of Maslov grading and $H_*(Y_p), H_*(Z_m)$, we obtain an identical hv-basis, as shown in the right of Figure \ref{figure:(13131) (0,2,1,1)-2}.

\textit{If $c\neq b-1, b\neq m-1, c\neq m-2$}, we determine an hv-basis and then verify $H_*(A_h)\neq \mathbb{F}^3$. This contradiction rules out this possibility.

\textit{If $c\neq b-1, b = m-1$}, we first determine an hv-basis and then verify $\partial^2 x_{(h,c)}\neq 0$, which rules out this case.

\textit{If $c\neq b-1, c = m-2$}, we first determine an hv-basis and then verify $\partial^2 x_{(h+1,b)}\neq 0$, which rules out this case.

\settocdepth{section} % 临时只显示到 section
\subsubsection{(0,2,2,0) case}
\resettocdepth      % 恢复原 tocdepth

In this case, $\operatorname{rank} \widehat{HFK}(K,j) = 
\begin{cases}
\begin{array}{cl}
  2    & j = h+1     \\  
  0\ \text{or}\ 2 & j = h       \\
  2    & j = h-1  \\
\end{array}
\end{cases}
$

\textbf{Subcase (i)}: Assume $\widehat{HFK}(K,h+1) = \mathbb{F}_{(a)}\oplus\mathbb{F}_{(b)}$, $\widehat{HFK}(K,h) = 0$, and $\widehat{HFK}(K,h-1) = \mathbb{F}_{(a-3)}\oplus\mathbb{F}_{(b-3)}$. 

Denote the first non-zero generator below height $h-1$ by $\mathbb{F}_{(c)}$. Analyzing exact sequences \eqref{equation: exact sequence of X_p} yields $a = m-1$, $b = m-3$, and $c = b-4 = m-7$. The hv-basis is determined by Maslov gradings and $H_*(Y_p), H_*(Z_m)$, as illustrated in Figure \ref{figure:(13131) (0,2,2,0)-1}. We can verify $\partial^2 x_{(-h-1,m-2h-5)}\neq 0$, which rules out this case.

\begingroup
\setlength{\intextsep}{1pt}
\begin{figure}[H]
	\centering
    \setlength{\abovecaptionskip}{1pt}
    \begin{minipage}[c]{0.4\textwidth}
        \centering
	   \includegraphics[scale = 0.7]{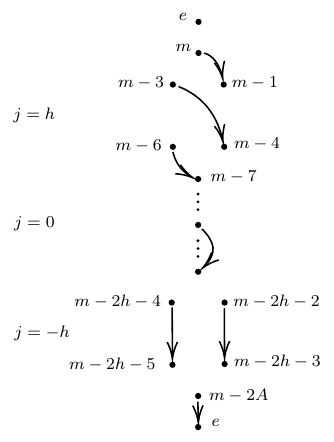}
    \end{minipage}
    \begin{minipage}[c]{0.5\textwidth}
        \centering
	   \includegraphics[scale = 0.8]{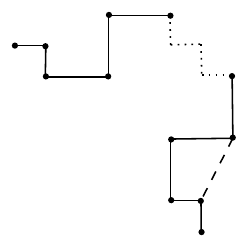}
    \end{minipage}
    \caption{}
    \label{figure:(13131) (0,2,2,0)-1}
\end{figure}
\endgroup

\textbf{Subcase (ii)}: Assume $\widehat{HFK}(K,h+1) = \mathbb{F}_{(a)}\oplus\mathbb{F}_{(b)}$, $\widehat{HFK}(K,h) = \mathbb{F}_{(b-1)}\oplus\mathbb{F}_{(c)}$, and $\widehat{HFK}(K,h-1) = \mathbb{F}_{(a-3)}\oplus\mathbb{F}_{(c-1)}$. 

Denote the first nonzero generator below height $h-1$ by $\mathbb{F}_{(d)}$. Different possibilities of $a,b,c,d$ are listed in Table \ref{table:131-9}. The rightmost three columns are included in the second column.

\begingroup
\setlength{\intextsep}{1pt}
\begin{table}[H]
\centering
\setlength{\abovecaptionskip}{1pt}
\begin{tabular}{|c|c|c|c|c|}
\hline
$a$  &  $m$ &  $m-1$ &  $m-1$  &  $m-1$  \\
\hline
$b$  &  $m-1$ &  $b$  &  $b$  &  $m-2$  \\
\hline
$c$  &  $m-1$ &  $c$  &  $m-3$  &  $m-3$  \\
\hline
$d$  &  $m-3$ &  $m-5$ & $m-5$  & $m-5$  \\
\hline
\end{tabular}
\caption{}
\label{table:131-9}
\end{table}
\endgroup

For the first column, we determine an hv-basis under the constraints of Maslov gradings and $H_*(Y_p), H_*(Z_m)$. We then verify $\partial^2 x_{(h+1,m-1)} \neq 0$, which rules out this case. 

\begingroup
\setlength{\intextsep}{1pt}
\begin{figure}[H]
	\centering
    \setlength{\abovecaptionskip}{1pt}
    \begin{minipage}[c]{0.4\textwidth}
        \centering
	   \includegraphics[scale = 0.7]{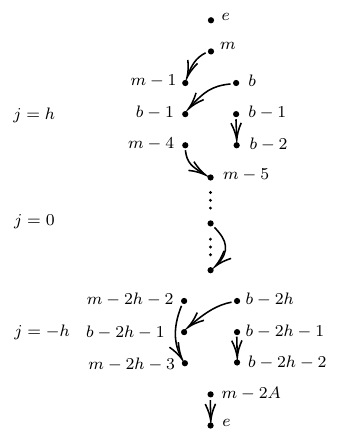}
    \end{minipage}
    \begin{minipage}[c]{0.5\textwidth}
        \centering
	   \includegraphics[scale = 0.8]{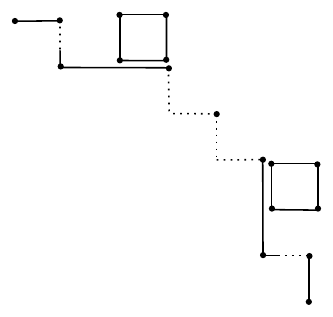}
    \end{minipage}
    \caption{}
    \label{figure:(13131) (0,2,2,0)-2}
\end{figure}
\endgroup

For the second column, \textit{if $c = b-1, b\neq m-1, b\neq m-2$}, we determine an hv-basis under the constraints of Maslov gradings and $H_*(Y_p), H_*(Z_m)$. The obtained vertically simplified basis and hv-basis are shown in Figure \ref{figure:(13131) (0,2,2,0)-2}, which is a staircase plus two boxes, located at symmetric heights.

\textit{If $c = b-1, b\in \{m-1, m-2\}$}, under the constraints of Maslov grading and $H_*(Y_p), H_*(Z_m)$, we can obtain an identical hv-basis; see Figure \ref{figure:(13131) (0,2,2,0)-2}.

\textit{If $c\neq b-1, b\neq m-1, c\neq m-2, m-3$}, in the obtained hv-basis we can verify $\partial^2 x_{(h,c)}\neq 0$.

\textit{If $c\neq b-1$ and $c\in \{m-2, m-3\}$}, in the obtained hv-basis we verify $\partial^2 x_{(h+1, b)}\neq 0$.

\textit{If $c\neq b-1, b = m-1$, and $c\neq m-3$}, in the obtained hv-basis we can check that $\partial^2 x_{(h,c)}\neq 0$.

\textit{If $c = m-3$ and $b = m-1$}, in the obtained hv-basis we find that $H_*(A_h)\neq \mathbb{F}^3$.

\settocdepth{section} % 临时只显示到 section
\subsubsection{(0,2,2,1) case}
\resettocdepth      % 恢复原 tocdepth

In this case, $\operatorname{rank} \widehat{HFK}(K, j) = 
\begin{cases}
\begin{array}{cl}
  2    & j = h+1     \\  
  0\  \text{or}\  2 & j = h       \\
  1    & j = h-1  \\
\end{array}
\end{cases}
$

\textbf{Subcase (i)}: Assume $\widehat{HFK}(K,h+1) = \mathbb{F}_{(a)}\oplus\mathbb{F}_{(b)}$, $\widehat{HFK}(K,h) = 0$, and $\widehat{HFK}(K,h-1) = \mathbb{F}_{(a-3)}$. 

Denote the first nonzero generator above height $-h+1$ by $\mathbb{F}_{(c)}$.
Computation yields $a = m-1, b = m-3$, and $c = m-2h-4$. Maslov grading and homologies of $Y_p,Z_m$ determine all the vertical and horizontal differentials. In the obtained hv-basis, we can verify $\partial^2 x_{(-h-1,m-2h-5)} \neq 0$, which rules out this case.

\textbf{Subcase (ii)}: Assume $\widehat{HFK}(K,h+1) = \mathbb{F}_{(a)}\oplus\mathbb{F}_{(b)}$, $\widehat{HFK}(K,h) = \mathbb{F}_{(a-1)}\oplus\mathbb{F}_{(c)}$, and $\widehat{HFK}(K,h-1) = \mathbb{F}_{(d)}$. Denote the first non-zero generator above height $-h+1$ by $\mathbb{F}_{(e)}$.

The four possibilities of $a,b,c,d,e$ are listed in Table \ref{table:131-10}.

\begingroup
\setlength{\intextsep}{1pt}
\begin{table}[H]
\centering
\setlength{\abovecaptionskip}{1pt}
\begin{tabular}{|c|c|c|c|c|}
\hline
$a$  &  $m-1$ &  $a$ &  $a$  &  $m-2$  \\
\hline
$b$  &  $m$ &  $m-1$  &  $m-1$  &  $m-1$  \\
\hline
$c$  &  $m-1$ &  $c$  &  $m-3$  &  $m-3$  \\
\hline
$d$  &  $m-3$ & $c-1$ & $m-4$  & $m-4$  \\
\hline
$e$  &  $m-2h$ & $m-2h-2$  &  $m-2h-2$  &  $m-2h-2$ \\  
\hline
\end{tabular}
\caption{}
\label{table:131-10}
\end{table}
\endgroup
\vspace{-8pt}

The third and fourth columns are included in the second column.

For the first column, we determine an hv-basis under the constraints of Maslov gradings and $H_*(Y_p), H_*(Z_m)$. We can verify $\partial^2x_{(-h-1,m-2h-3)}\neq 0$ in the obtained hv-basis, which rules out this case.

\begingroup
\setlength{\intextsep}{1pt}
\begin{figure}[H]
	\centering
    \setlength{\abovecaptionskip}{1pt}
    \begin{minipage}[c]{0.4\textwidth}
        \centering
	   \includegraphics[scale = 0.7]{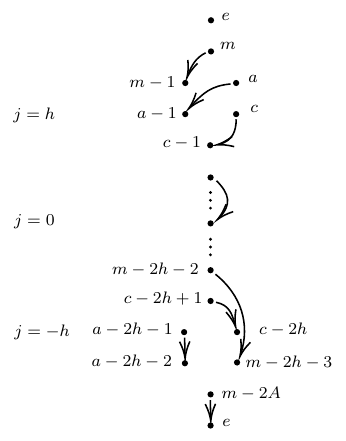}
    \end{minipage}
    \begin{minipage}[c]{0.5\textwidth}
        \centering
	   \includegraphics[scale = 0.8]{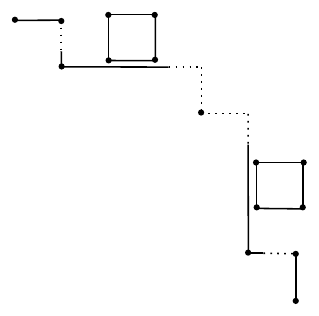}
    \end{minipage}
    \caption{}
    \label{figure:(13131) (0,2,2,1)-1}
\end{figure}
\endgroup
\vspace{-8pt}

For the second column, \textit{if $c = a-1, a \neq m-2$}, we can obtain an hv-basis whose configuration is a staircase plus two boxes, as shown in Figure 
\ref{figure:(13131) (0,2,2,1)-1}. If $c=a-1=m-3$, the different possibilities of vertically simplified basis can be either ruled out by the condition $\partial^2=0$ or maded into an identical hv-basis.  

The case where $c \neq a-1$ can be ruled out, as explained below:

\textit{If $c\neq a-1, a\neq m-1, c\notin\{m-2, m-3\}$}, in the obtained hv-basis we can verify $\partial^2 x_{(h,c)}\neq 0$.

\textit{If $c\neq a-1, a = m-1, c\neq m-3$}, in the obtained hv-basis we again find that $\partial^2 x_{(h,c)}\neq 0$.

\textit{If $c\neq a-1, a = m-1, c = m-3$}, in the obtained  hv-basis we can verify $\partial^2 x_{(h+1,m-1)}\neq 0$.

\textit{If $c\neq a-1, a\neq m-1, c\in\{m-2,m-3\}$}, in the obtained hv-basis we can check that $\partial^2 x_{(h+1, a)}\neq 0$. 

\vspace{-8pt}
\settocdepth{section} % 临时只显示到 section
\subsubsection{(1,1,1,0) case}
\resettocdepth      % 恢复原 tocdepth

In this case, we have $\operatorname{rank} \widehat{HFK}(K,j) = 
\begin{cases}
\begin{array}{cl}
  2    & j = h+1     \\  
  0\ \text{or}\ 2 & j = h       \\
  1    & j = h-1  \\
\end{array}
\end{cases}
$

For the remaining eight cases where $H_*(X_{h+2}) = \mathbb{F}$, set $\alpha = m-2A+2h+2$.

\textbf{Subcase (i)}: Assume $\widehat{HFK}(K,h+1) = \mathbb{F}_{(a)}\oplus\mathbb{F}_{(\alpha+1)}$, $\widehat{HFK}(K,h)=0$ and $\widehat{HFK}(K,h-1) = \mathbb{F}_{(a-3)}$. 

Denote the first nonzero generator below height $h-1$ by $\mathbb{F}_{(c)}$.
Analyzing the exact sequences \eqref{equation: exact sequence of X_p} gives $a = \alpha+3$ and $c = \alpha+2$. We determine an hv-basis under the constraints of Maslov grading and $H_*(Y_p), H_*(Z_m)$. In the obtained hv-basis, we can verify $\partial^2 x_{(-h+1, \alpha-2h+2)}\neq 0$. Thus, we rule out this case.

\textbf{Subcase (ii)}: Assume $\widehat{HFK}(K,h+1) = \mathbb{F}_{(a)}\oplus\mathbb{F}_{(\alpha+1)}$, $\widehat{HFK}(K,h) = \mathbb{F}_{(a-1)}\oplus\mathbb{F}_{(b)}$, and $\widehat{HFK}(K,h-1) = \mathbb{F}_{(b-1)}$.

Denote the first nonzero generator below height $h-1$ by $\mathbb{F}_{(c)}$. Computation yields five possibilities, as shown in Table \ref{table: 131-11}. The rightmost four columns can be included into the second column.

\begingroup
\setlength{\intextsep}{1pt}
\begin{table}[H]
\centering
\setlength{\abovecaptionskip}{1pt}
\begin{tabular}{|c|c|c|c|c|c|}
\hline
$a$  &  $\alpha$ &  $a$ &  $a$  &  $\alpha+2$ & $a$ \\
\hline
$b$  &  $\alpha$ &  $b$  &  $\alpha+1$  &  $\alpha+1$ & $a-1$ \\
\hline
$c$  &  $\alpha-2$ &  $\alpha$  &  $\alpha$  &  $\alpha$ & $\alpha$ \\
\hline
\end{tabular}
\caption{}
\label{table: 131-11}
\end{table}
\endgroup
\vspace{-8pt}

For the first column, we determine an hv-basis under constraints of Maslov grading and $H_*(Y_p), H_*(Z_m)$. We then verify $\partial^2 x_{(-h, \alpha-5)}\neq 0$, which rules out this case.

\begingroup
\setlength{\intextsep}{1pt}
\begin{figure}[H]
	\centering
    \setlength{\abovecaptionskip}{1pt}
    \begin{minipage}[c]{0.4\textwidth}
        \centering
	   \includegraphics[scale = 0.7]{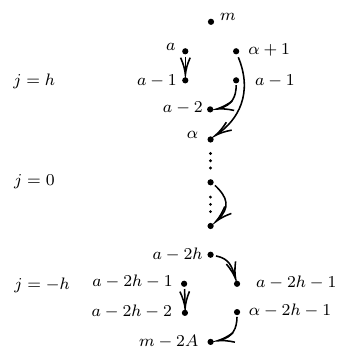}
    \end{minipage}
    \begin{minipage}[c]{0.5\textwidth}
        \centering
	   \includegraphics[scale = 0.8]{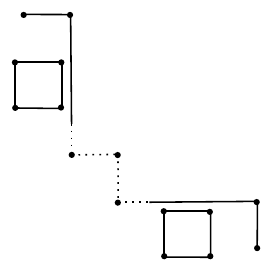}
    \end{minipage}
    \caption{}
    \label{figure:(13131) (1,1,1,0)-1}
\end{figure}
\endgroup

For the second column, we discuss separately according to the relationships among $a, b$, and $\alpha$.

\textit{If $b = a-1$}, except for the pathological case where $b = a-1 = \alpha+1$(ruled out by $\partial^2=0$ condition), we can obtain an hv-basis under the constraints of Maslov grading and homologies of $Y_p, Z_m$. The hv-basis we obtain is a staircase plus two boxes, as shown in Figure \ref{figure:(13131) (1,1,1,0)-1}.  

\textit{If} $b \neq a-1$, we first determine an hv-basis under constraints of Maslov grading and $H_*(Y_p), H_*(Z_m)$. We then verify at least one of $\partial^2 x_{(h+1,a)}$ and $\partial^2x_{(h,b)}$ is non-zero, which rules out this case.

\vspace{-8pt}
\settocdepth{section} % 临时只显示到 section
\subsubsection{(1,1,1,1) case}
\resettocdepth      % 恢复原 tocdepth

In this case, we have $\operatorname{rank} \widehat{HFK}(K,j) = 
\begin{cases}
\begin{array}{cl}
  2    & j = h+1     \\  
  0\ or\ 2 & j = h       \\
  2    & j = h-1  \\
\end{array}
\end{cases}
$

Recall that $\alpha = m-2A+2h+2$.

\textbf{Subcase (i)}: Assume $\widehat{HFK}(K,h+1) = \mathbb{F}_{(a)}\oplus\mathbb{F}_{(\alpha+1)}$, $\widehat{HFK}(K,h) = 0$, and $\widehat{HFK}(K,h-1) = \mathbb{F}_{(a-3)}\oplus\mathbb{F}_{(b)}$. 

Denote the first nonzero generator above height $-h+1$ by $\mathbb{F}_{(c)}$. Computation yields $a = \alpha+3, b = \alpha+2$ and $c = b-2h+3$. We determine an hv-basis under the constraints of Maslov grading and $H_*(Y_p), H_*(Z_m)$. Then we verify $\partial^2 x_{(-h+1, a-2h-1)}\neq 0$, which rules out this case.

\textbf{Subcase (ii)}: Assume $\widehat{HFK}(K,h+1) = \mathbb{F}_{(a)}\oplus\mathbb{F}_{(\alpha+1)}$, $\widehat{HFK}(K,h) = \mathbb{F}_{(a-1)}\oplus\mathbb{F}_{(b)}$, and $\widehat{HFK}(K,h-1) = \mathbb{F}_{(b-1)}\oplus\mathbb{F}_{(c)}$.

Denote the first nonzero generator above height $-h+1$ by $\mathbb{F}_{(d)}$. There are five possibilities in $a,b,c,d$, as shown in Table \ref{table:131-12}. The first column represents an individual case, while the remaining four columns can be included in the second column.

\begingroup
\setlength{\intextsep}{1pt}
\begin{table}[H]
\centering
\setlength{\abovecaptionskip}{1pt}
\begin{tabular}{|c|c|c|c|c|c|}
\hline
$a$  &  $\alpha$ &  $a$ &  $a$  &  $\alpha+2$ & $a$ \\
\hline
$b$  &  $\alpha$ &  $b$  &  $\alpha+1$  &  $\alpha+1$ & $a-1$ \\
\hline
$c$  &  $\alpha-2$ &  $\alpha$  &  $\alpha$  &  $\alpha$ & $\alpha$ \\
\hline
$d$  &  $c-2h+3$ &  $c-2h+3$  &  $c-2h+3$  &  $c-2h+3$ & $c-2h+3$  \\
\hline
\end{tabular}
\caption{}
\label{table:131-12}
\end{table}
\endgroup
\vspace{-8pt}

For the first column, we construct an hv-basis under the constraints of Maslov grading and $H_*(Y_p), H_*(Z_m)$. We then verify $\partial^2 x_{(-h-1, \alpha-2h-2)}\neq 0$, which rules out this case.

\begingroup
\setlength{\intextsep}{1pt}
\begin{figure}[H]
	\centering
    \setlength{\abovecaptionskip}{1pt}
    \begin{minipage}[c]{0.4\textwidth}
        \centering
	   \includegraphics[scale = 0.7]{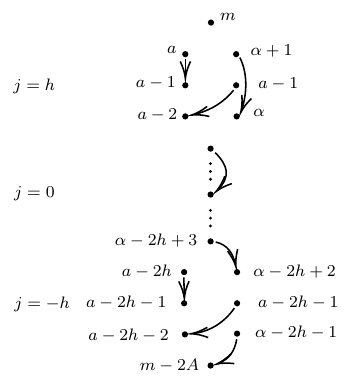}
    \end{minipage}
    \begin{minipage}[c]{0.5\textwidth}
        \centering
	   \includegraphics[scale = 0.8]{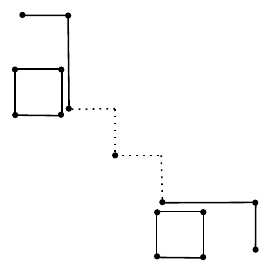}
    \end{minipage}
    \caption{}
    \label{figure:(13131) (1,1,1,1)-1}
\end{figure}
\endgroup

For the second column, \textit{if $b = a-1, a\in\{\alpha+1,\alpha+2\}$}, we construct an hv-basis under constraints of Maslov grading and $H_*(Y_p), H_*(Z_m)$. The hv-basis is a staircase plus two boxes located at symmetric heights, as shown in Figure \ref{figure:(13131) (1,1,1,1)-1}.

\textit{If $b\neq a-1$}, we discuss separately according to the relationships between $a$ and $b$. In each case, we determine an hv-basis and check that at least one of $\partial^2 x_{(h,b)}$ and $\partial^2 x_{(h+1,a)}$ is non-zero, this contradiction rules out this case.

 \settocdepth{section} % 临时只显示到 section
\subsubsection{(1,1,2,0) case}
\resettocdepth      % 恢复原 tocdepth

In this case, we have $\operatorname{rank} \widehat{HFK}(K,j) = 
\begin{cases}
\begin{array}{cl}
  2    & j = h+1     \\  
  1\ \text{or}\ 3 & j = h       \\
  2    & j = h-1  \\
\end{array}
\end{cases}
$

\textbf{Subcase (i)}: Assume $\widehat{HFK}(K,h+1)$ $ = $ $\mathbb{F}_{(a)}$ $\oplus$ $\mathbb{F}_{(\alpha+1)}$, $\widehat{HFK}(K,h-1)$ $ = $ $\mathbb{F}_{(a-3)}$ $\oplus$ $\mathbb{F}_{(b-1)}$, and $\widehat{HFK}(K,h)$ $ = $ $\mathbb{F}_{(b)}$, 

Denote the first nonzero generator below height $h-1$ by $\mathbb{F}_{(c)}$. Analyzing the exact sequences \eqref{equation: exact sequence of X_p} yields three possibilities, as shown in Table \ref{table:131-13}.

\begingroup
\setlength{\intextsep}{2pt}
\begin{table}[H]
\centering
\setlength{\abovecaptionskip}{1pt}
\begin{tabular}{|c|c|c|c|}
\hline
$a$  &  $\alpha$ &  $\alpha+2$ &  $\alpha+3$   \\
\hline
$b$  &  $\alpha$ &  $\alpha+1$  &  $\alpha+2$   \\
\hline
$c$  &  $\alpha-4$ &  $\alpha-2$  &  $\alpha$   \\
\hline
\end{tabular}
\caption{}
\label{table:131-13}
\end{table}
\endgroup

For the first column, we determine an hv-basis under the constraints of Maslov grading and homologies of $Y_p, Z_m$. In the obtained hv-basis, we check that $\partial^2 y_{(h+1,\alpha)}\neq 0$, which rules out this case.

For the second column, we construct a vertically simplified basis as shown in the left of Figure \ref{figure:(13131) (1,1,2,0)-1}. We then determine the horizontal differentials. The obtained hv-basis is a staircase with two positive twists; see the right side of Figure \ref{figure:(13131) (1,1,2,0)-1}.

\begingroup
\setlength{\intextsep}{2pt}
\begin{figure}[H]
	\centering
    \setlength{\abovecaptionskip}{1pt}
    \begin{minipage}[c]{0.4\textwidth}
        \centering
	   \includegraphics[scale = 0.7]{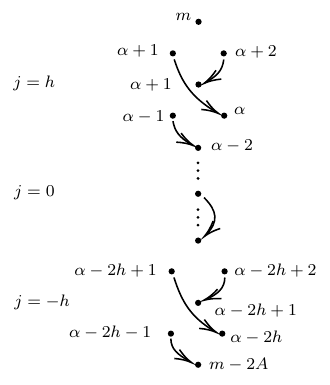}
    \end{minipage}
    \begin{minipage}[c]{0.5\textwidth}
        \centering
	   \includegraphics[scale = 0.8]{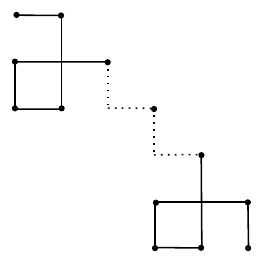}
    \end{minipage}
    \caption{}
    \label{figure:(13131) (1,1,2,0)-1}
\end{figure}
\endgroup

For the third column, in the obtained hv-basis we verify $\partial^2 x_{(-h+1, \alpha-2h+2)}\neq 0$, which rules out this case.

\textbf{Subcase (ii)}: Assume $\widehat{HFK}(K,h+1) = \mathbb{F}_{(a)}\oplus\mathbb{F}_{(\alpha+1)}$, $\widehat{HFK}(K,h) = \mathbb{F}_{(a-1)}\oplus\mathbb{F}_{(b)}\oplus\mathbb{F}_{(c)}$, and $\widehat{HFK}(K,h-1) = \mathbb{F}_{(b-1)}\oplus\mathbb{F}_{(c-1)}$. 

Denote the first non-zero generator below height $h-1$ by $\mathbb{F}_{(d)}$. Computation gives $b = \alpha, d = \alpha-2$. We discuss separately according to the relationships among $a,c$ and $\alpha$.

\begingroup
\setlength{\intextsep}{1pt}
\begin{figure}[H]
	\centering
    \setlength{\abovecaptionskip}{0pt}
    \begin{minipage}[c]{0.4\textwidth}
        \centering
	   \includegraphics[scale = 0.7]{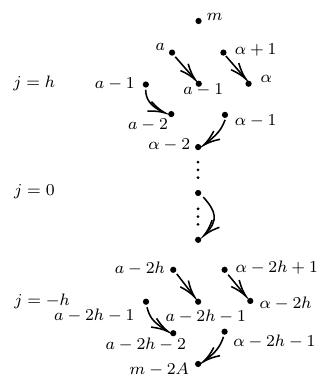}
    \end{minipage}
    \begin{minipage}[c]{0.5\textwidth}
        \centering
	   \includegraphics[scale = 0.8]{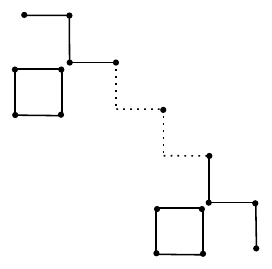}
    \end{minipage}
    \caption{}
    \label{figure:(13131) (1,1,2,0)-2}
\end{figure}
\endgroup
\vspace{-8pt}

\textit{If $c = a-1$}, we determine an hv-basis under the constraints of Maslov grading and homologies of $Y_p, Z_m$. The obtained hv-basis is a staircase plus two boxes, as shown in Figure \ref{figure:(13131) (1,1,2,0)-2}.

\textit{If $c\neq a-1$}, we determine an hv-basis depending on the relationships between undetermined variables. We check that either $\partial^2 x_{(h+1, a)} \neq 0$ or $\partial^2 x_{(h,c)}\neq 0$, which rule out this case.

\settocdepth{section} % 临时只显示到 section
\subsubsection{(1,1,2,1) case}
\resettocdepth       % 恢复原 tocdepth

Computation gives $\operatorname{rank} \widehat{HFK}(K,j) = 
\begin{cases}
\begin{array}{cl}
  2    & j = h+1     \\  
  1\ \text{or}\ 3 & j = h       \\
  1    & j = h-1  \\
\end{array}
\end{cases}
$

We divide our discussion into two subcases, depending on $\operatorname{rank} \widehat{HFK}(K,h)$.

\textbf{Subcase (i)}: Assume that $\widehat{HFK}(K,h+1) = \mathbb{F}_{(a)}\oplus\mathbb{F}_{(\alpha+1)}$, $\widehat{HFK}(K,h) = \mathbb{F}_{(b)}$, and $\widehat{HFK}(K,h-1) = \mathbb{F}_{(c)}$. 

Denote the first nonzero generator above height $-h+1$ by $\mathbb{F}_{(d)}$. Computation yields three possibilities, as shown in Table \ref{table:131-14}.

\begingroup
\setlength{\intextsep}{2pt}
\begin{table}[H]
\centering
\setlength{\abovecaptionskip}{1pt}
\begin{tabular}{|c|c|c|c|}
\hline
$a$  &  $\alpha+3$ &  $\alpha$ &  $\alpha+2$   \\
\hline
$b$  &  $\alpha+2$ &  $\alpha$  &  $\alpha+1$   \\
\hline
$c$  &  $\alpha$ &  $\alpha-1$  &  $\alpha$   \\
\hline
$d$  &  $\alpha-2h+3$ &  $\alpha-2h-1$  &  $\alpha-2h+1$   \\
\hline
\end{tabular}
\caption{}
\label{table:131-14}
\end{table}
\endgroup
\vspace{-8pt}

For the first column, we determine an hv-basis and check that $\partial^2 x_{(-h+1, \alpha-2h+2)} \neq 0$, which rules out this possibility.

For the second column, we determine an hv-basis and verify $\partial^2$ $x_{(-h-1, \alpha-2h-2)}$ $\neq 0$.

\begingroup
\setlength{\intextsep}{1pt}
\begin{figure}[H]
	\centering
    \setlength{\abovecaptionskip}{0pt}
    \begin{minipage}[c]{0.4\textwidth}
        \centering
	   \includegraphics[scale = 0.7]{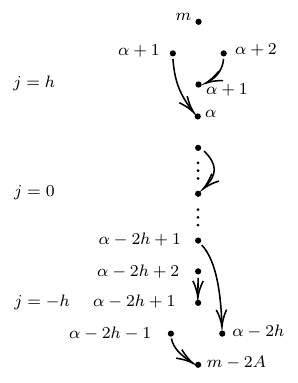}
    \end{minipage}
    \begin{minipage}[c]{0.5\textwidth}
        \centering
	   \includegraphics[scale = 0.8]{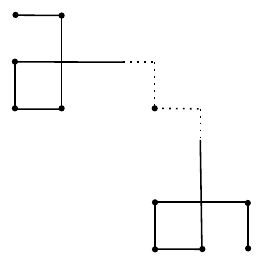}
    \end{minipage}
    \caption{}
    \label{figure:(13131) (1,1,2,1)-1}
\end{figure}
\endgroup

For the third column, we determine an hv-basis under the constraints of Maslov grading and homologies of $Y_p, Z_m$. The obtained hv-basis is a staircase with two positive twists, as shown in Figure \ref{figure:(13131) (1,1,2,1)-1}.

\textbf{Subcase (ii)}: Assume $\widehat{HFK}(K,h+1) = \mathbb{F}_{(a)}\oplus\mathbb{F}_{(\alpha+1)}$, $\widehat{HFK}(K,h) = \mathbb{F}_{(a-1)}\oplus\mathbb{F}_{(b)}\oplus\mathbb{F}_{(c)}$, and $\widehat{HFK}(K,h-1) = \mathbb{F}_{(c-1)}$. 

Denote the first nonzero generator above height $-h+1$ by $\mathbb{F}_{(d)}$. Computation gives $b = \alpha, d = \alpha-2h+1$. We discuss separately according to the relationships among $a,c$ and $\alpha$.

\textit{If $c = a-1$}, we determine an hv-basis under the constraints of Maslov grading and homologies of $Y_p, Z_m$. The obtained hv-basis is a staircase plus two boxes, as shown in Figure \ref{figure:(13131) (1,1,2,1)-2}.

\begingroup
\setlength{\intextsep}{1pt}
\begin{figure}[H]
	\centering
    \setlength{\abovecaptionskip}{1pt}
    \begin{minipage}[c]{0.4\textwidth}
        \centering
	   \includegraphics[scale = 0.7]{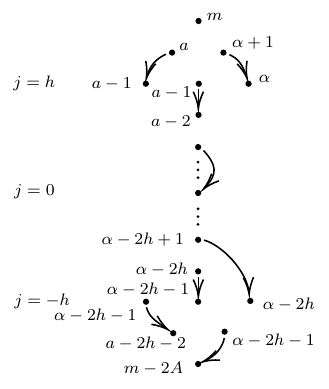}
    \end{minipage}
    \begin{minipage}[c]{0.5\textwidth}
        \centering
	   \includegraphics[scale = 0.8]{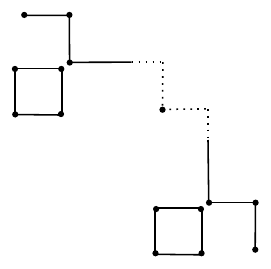}
    \end{minipage}
    \caption{}
    \label{figure:(13131) (1,1,2,1)-2}
\end{figure}
\endgroup

\textit{If $c\neq a-1$}, we discuss separately according to the  values of $a,c$. We then verify $\partial^2 x_{(h+1,a)}$ and $\partial^2 x_{(h,c)}$ do not vanish simultaneously, which rules out this possibility.

\settocdepth{section} % 临时只显示到 section
\subsubsection{(1,2,1,0) case}
\resettocdepth      % 恢复原 tocdepth

Computation gives $\operatorname{rank} \widehat{HFK}(K,j) = 
\begin{cases}
\begin{array}{cl}
  1    & j = h+1     \\  
  1\ \text{or} \ 3 & j = h       \\
  1    & j = h-1  \\
\end{array}
\end{cases}
$

We divide our discussion into two subcases, depending on $\operatorname{rank} \widehat{HFK}(K,h)$.

\textbf{Subcase (i)}: Assume $\widehat{HFK}(K,h+1) = \mathbb{F}_{(a)}$, $\widehat{HFK}(K,h) = \mathbb{F}_{(b)}$, and $\widehat{HFK}(K,h-1) = \mathbb{F}_{(c)}$. 

Let the first non-zero generator below height $h-1$ be denoted by $\mathbb{F}_{(d)}$. The 
computation yields three possibilities, as shown in Table \ref{table:131-15}.

\begingroup
\setlength{\intextsep}{1pt}
\begin{table}[H]
\centering
\setlength{\abovecaptionskip}{1pt}
\begin{tabular}{|c|c|c|c|}
\hline
$a$  &  $\alpha+1$ &  $\alpha-2$ &  $\alpha-1$   \\
\hline
$b$  &  $\alpha-1$ &  $\alpha-3$  &  $\alpha-2$   \\
\hline
$c$  &  $\alpha-2$ &  $\alpha-3$  &  $\alpha-3$   \\
\hline
$d$  &  $\alpha$ &  $\alpha-4$  &  $\alpha-2$   \\
\hline
\end{tabular}
\caption{}
\label{table:131-15}
\end{table}
\endgroup

For the first column, we determine an hv-basis under the constraints of Maslov grading and homologies of $Y_p, Z_m$. We then check that $\partial^2 x_{(-h, \alpha-2h-1)}\neq 0$, which rules out this possibility.

For the second column, we determine an hv-basis and verify $\partial^2 x_{(-h,\alpha-2h-3)}\neq 0$.

For the third column, the obtained hv-basis is a staircase with two negative twists, located at symmetric heights, as shown in Figure \ref{figure:(13131) (1,2,1,0)-1}.

\begingroup
\setlength{\intextsep}{1pt}
\begin{figure}[H]
	\centering
    \setlength{\abovecaptionskip}{0pt}
    \begin{minipage}[c]{0.4\textwidth}
        \centering
	   \includegraphics[scale = 0.7]{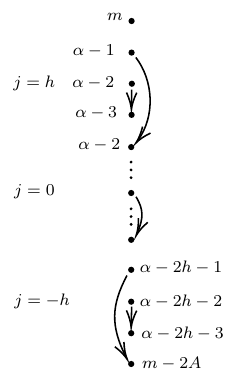}
    \end{minipage}
    \begin{minipage}[c]{0.5\textwidth}
        \centering
	   \includegraphics[scale = 0.8]{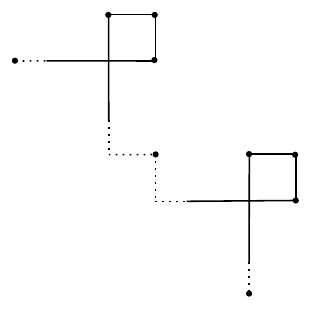}
    \end{minipage}
    \caption{}
    \label{figure:(13131) (1,2,1,0)-1}
\end{figure}
\endgroup
\vspace{-6pt}

\textbf{Subcase (ii)}: Assume $\widehat{HFK}(K,h+1) = \mathbb{F}_{(a)}$, $\widehat{HFK}(K,h) = \mathbb{F}_{(a-1)}\oplus\mathbb{F}_{(b)}\oplus\mathbb{F}_{(\alpha-1)}$, and $\widehat{HFK}(K,h-1) = \mathbb{F}_{(b-1)}$.

Denote the first non-zero generator below height $h-1$ by $\mathbb{F}_{(c)}$. Computation gives $c = \alpha-2$. We discuss separately according to the relationships among $a, b$, and $\alpha$.

\textit{If $b = a-1$}, we construct an hv-basis under the constraints of Maslov grading and $H_*(Y_p), H_*(Z_m)$, which is a staircase plus two boxes, as shown in Figure \ref{figure:(13131) (1,2,1,0)-2}.

\begingroup
\setlength{\intextsep}{1pt}
\begin{figure}[H]
	\centering
    \setlength{\abovecaptionskip}{0pt}
    \begin{minipage}[c]{0.4\textwidth}
        \centering
	   \includegraphics[scale = 0.7]{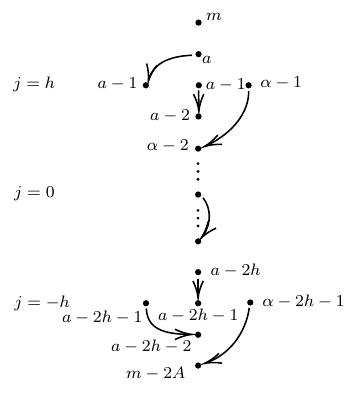}
    \end{minipage}
    \begin{minipage}[c]{0.5\textwidth}
        \centering
	   \includegraphics[scale = 0.8]{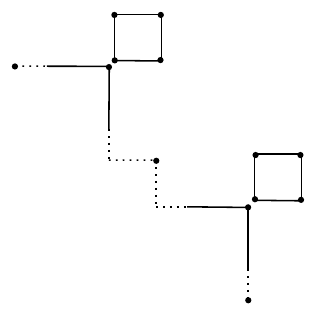}
    \end{minipage}
    \caption{}
    \label{figure:(13131) (1,2,1,0)-2}
\end{figure}
\endgroup
\vspace{-8pt}

\textit{If $b \neq a-1$}, there are different cases depending on the relationships among $a, b$ and $\alpha$. In each case we find either $\partial^2 x_{(h+1,a)}\neq 0$ or $\partial^2 x_{(h,b)}\neq 0$, which rules out this possibility.

\vspace{-10pt}
\settocdepth{section} % 临时只显示到 section
\subsubsection{(1,2,1,1) case}
\resettocdepth      % 恢复原 tocdepth

In this case, we have $\operatorname{rank} \widehat{HFK}(K,j) = 
\begin{cases}
\begin{array}{cl}
  1    & j = h+1     \\  
  1\ \text{or} \ 3 & j = h       \\
  2    & j = h-1  \\
\end{array}
\end{cases}
$

We divide our discussion into two subcases, depending on $\operatorname{rank} \widehat{HFK}(K,h)$.

\textbf{Subcase (i)}: Assume $\widehat{HFK}(K,h+1) = \mathbb{F}_{(a)}$, $\widehat{HFK}(K,h) = \mathbb{F}_{(b)}$, and $\widehat{HFK}(K,h-1) = \mathbb{F}_{(a-1)}\oplus\mathbb{F}_{(b-1)}$. 

Denote the first non-zero generator above height $-h+1$ by $\mathbb{F}_{(c)}$. Computation yields three possibilities, as shown in Table \ref{table:131-16}.

\begingroup
\setlength{\intextsep}{1pt}
\begin{table}[H]
\centering
\setlength{\abovecaptionskip}{1pt}
\begin{tabular}{|c|c|c|c|}
\hline
$a$  &  $\alpha+1$ &  $\alpha-1$  &  $\alpha-2$   \\
\hline
$b$  &  $\alpha-1$ &  $\alpha-2$  &  $\alpha-3$   \\
\hline
$c$  &  $\alpha-1$ &  $\alpha-3$  &  $\alpha-5$   \\
\hline
\end{tabular}
\caption{}
\label{table:131-16}
\end{table}
\endgroup
\vspace{-8pt}

For the first column, we construct an hv-basis under the constraints of Maslov grading and the homologies of $Y_p, Z_m$, we then find $\partial^2 x_{(-h, \alpha-2h-1)}\neq 0$.

For the second column, the hv-basis we obtain is a staircase with two negative twists located at symmetric heights, as shown in Figure \ref{figure:(13131) (1,2,1,1)-1}.

\begingroup
\setlength{\intextsep}{1pt}
\begin{figure}[H]
	\centering
    \setlength{\abovecaptionskip}{1pt}
    \begin{minipage}[c]{0.4\textwidth}
        \centering
	   \includegraphics[scale = 0.7]{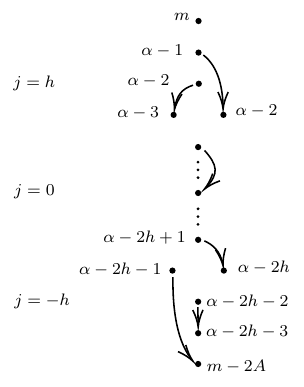}
    \end{minipage}
    \begin{minipage}[c]{0.5\textwidth}
        \centering
	   \includegraphics[scale = 0.8]{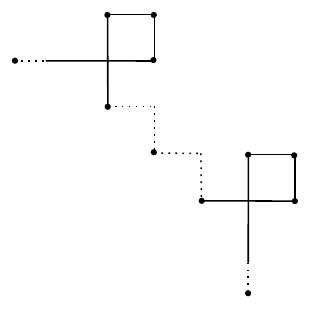}
    \end{minipage}
    \caption{}
    \label{figure:(13131) (1,2,1,1)-1}
\end{figure}
\endgroup

For the third column, in the obtained hv-basis we verify $\partial^2 x_{(-h+1,\alpha-2h-2)} \neq 0$, which rules out this possibility.

\textbf{Subcase (ii)}: Assume $\widehat{HFK}(K,h+1) = \mathbb{F}_{(a)}$, $\widehat{HFK}(K,h) = \mathbb{F}_{(a-1)}\oplus\mathbb{F}_{(b)}\oplus\mathbb{F}_{(\alpha-1)}$, and $\widehat{HFK}(K,h-1) = \mathbb{F}_{(b-1)}\oplus\mathbb{F}_{(c)}$. 

Denote the first nonzero generator above height $-h+1$ by $\mathbb{F}_{(d)}$. Computation gives $c = \alpha-2$ and $d = \alpha-2h+1$. We discuss separately according to the relationships among $a, b$, and $\alpha$.

\textit{If $b = a-1$}, we determine an hv-basis under the constraints of Maslov grading and homologies of $Y_p, Z_m$, which is a staircase plus two boxes, as shown in Figure \ref{figure:(13131) (1,2,1,1)-2}.

\begingroup
\setlength{\intextsep}{1pt}
\begin{figure}[H]
	\centering
    \setlength{\abovecaptionskip}{1pt}
    \begin{minipage}[c]{0.4\textwidth}
        \centering
	   \includegraphics[scale = 0.7]{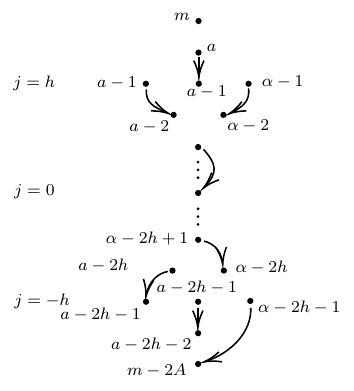}
    \end{minipage}
    \begin{minipage}[c]{0.5\textwidth}
        \centering
	   \includegraphics[scale = 0.8]{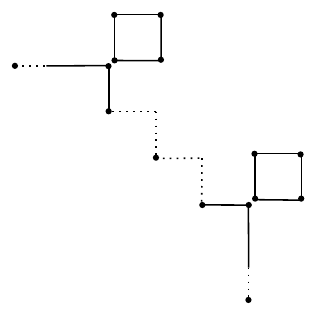}
    \end{minipage}
    \caption{}
    \label{figure:(13131) (1,2,1,1)-2}
\end{figure}
\endgroup

\textit{If $b \neq a-1$}, we divide our discussion into cases where $a$ and $b$ take different values. We find either $\partial^2 x_{(h+1,a)}\neq 0$ or $\partial^2 x_{(h,b)}\neq 0$, which rules out this possibility.

\settocdepth{section} % 临时只显示到 section
\subsubsection{(1,2,2,0) case}
\resettocdepth      % 恢复原 tocdepth

Computation gives $\operatorname{rank} \widehat{HFK}(K,j) = 
\begin{cases}
\begin{array}{cl}
  1    & j = h+1     \\  
  0\ \text{or}\ 2 & j = h       \\
  2    & j = h-1  \\
\end{array}
\end{cases}
$

We divide our discussion into two subcases, depending on $\operatorname{rank} \widehat{HFK}(K,h)$.

\textbf{Subcase (i)}: Assume $\widehat{HFK}(K,h+1) = \mathbb{F}_{(a)}$, $\widehat{HFK}(K,h) = 0$, and $\widehat{HFK}(K,h-1) = \mathbb{F}_{(a-3)}\oplus\mathbb{F}_{(\alpha-3)}$. 

Denote the first nonzero generator below height $h-1$ by $\mathbb{F}_{(b)}$. Computation gives $a = \alpha-2$ and $b = a-4 = \alpha-6$.

We use Maslov grading and $H_*(Y_p), H_*(Z_m)$ to constrain the differentials and construct an hv-basis. We then verify $\partial^2 x_{(-h-1,\alpha-2h-4)}\neq 0$, which rules out this possibility.

\textbf{Subcase (ii)}: Assume that $\widehat{HFK}(K,h+1) = \mathbb{F}_{(a)}$, $\widehat{HFK}(K,h) = \mathbb{F}_{(a-1)}\oplus\mathbb{F}_{(b)}$, and $\widehat{HFK}(K,h-1) = \mathbb{F}_{(b-1)}\oplus\mathbb{F}_{(c)}$. 

Denote the first nonzero generator below height $h-1$ by $\mathbb{F}_{(d)}$. Computation gives $d = c-1$. There are four possibilities, as shown in Table \ref{table:131-17}.

\begingroup
\setlength{\intextsep}{1pt}
\begin{table}[H]
\centering
\setlength{\abovecaptionskip}{1pt}
\begin{tabular}{|c|c|c|c|c|}
\hline
$a$   &  $\alpha+1$ &  $a$ &  $\alpha$  &  $a$  \\
\hline
$b$  &  $\alpha-1$ & $b$ &  $\alpha-1$  &  $a-1$  \\
\hline
$c$  &  $\alpha-1$ & $\alpha-3$ &  $\alpha-3$  &  $\alpha-3$  \\
\hline
\end{tabular}
\caption{}
\label{table:131-17}
\end{table}
\endgroup

The latter three columns can be absorbed into the second column.

For the first column, we construct an hv-basis and check that $\partial^2 x_{(-h,\alpha-2h-1)} \neq 0$, which rules out this possibility.

For the second column, \textit{if $b = a-1$}, the obtained hv-basis is a staircase plus two boxes, as shown in Figure \ref{figure:(13131) (1,2,2,0)-1}.

\begingroup
\setlength{\intextsep}{1pt}
\begin{figure}[H]
	\centering
    \setlength{\abovecaptionskip}{1pt}
    \begin{minipage}[c]{0.4\textwidth}
        \centering
	   \includegraphics[scale = 0.7]{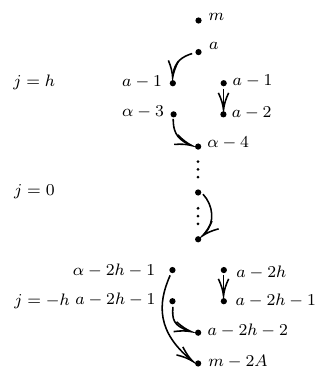}
    \end{minipage}
    \begin{minipage}[c]{0.5\textwidth}
        \centering
	   \includegraphics[scale = 0.8]{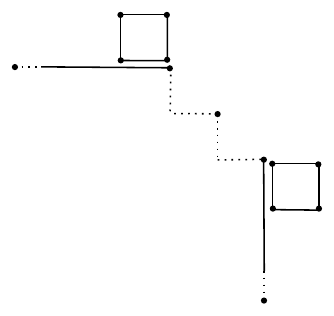}
    \end{minipage}
    \caption{}
    \label{figure:(13131) (1,2,2,0)-1}
\end{figure}
\endgroup

\textit{If $b\neq a-1$}, we construct an hv-basis and find that either $\partial^2 x_{(h+1,a)}\neq 0$ or $\partial^2 x_{(h,b)}\neq 0$ in the obtained hv-basis, which rules out this possibility.

\settocdepth{section} % 临时只显示到 section
\subsubsection{(1,2,2,1) case}
\resettocdepth      % 恢复原 tocdepth

We have $\operatorname{rank} \widehat{HFK}(K,j) = 
\begin{cases}
\begin{array}{cl}
  1    & j = h+1     \\  
  0\ \text{or}\ 2 & j = h       \\
  1    & j = h-1  \\
\end{array}
\end{cases}
$

\textbf{Subcase (i)}: Assume $\widehat{HFK}(K,h+1) = \mathbb{F}_{(a)}$, $\widehat{HFK}(K,h) = 0$, and $\widehat{HFK}(K,h-1) = \mathbb{F}_{(a-3)}$. 

Denote the first nonzero generator above height $-h+1$ by $\mathbb{F}_{(b)}$. Computation gives $a = \alpha-2, b = a-2h-1 = \alpha-2h-3$. We construct an hv-basis and check that $\partial^2 x_{(-h-1,\alpha-2h-2)}\neq 0$, which rules out this subcase.

\textbf{Subcase (ii)}: Assume that $\widehat{HFK}(K,h+1) = \mathbb{F}_{(a)}$, $\widehat{HFK}(K,h) = \mathbb{F}_{(a-1)}\oplus\mathbb{F}_{(b)}$, and $\widehat{HFK}(K,h-1) = \mathbb{F}_{(b-1)}$. 

Denote the first non-zero generator above height $-h+1$ by $\mathbb{F}_{(c)}$. Computation gives five possibilities, as shown in Table \ref{table:131-18}.

\begingroup
\setlength{\intextsep}{1pt}
\begin{table}[H]
\centering
\setlength{\abovecaptionskip}{1pt}
\begin{tabular}{|c|c|c|c|c|c|}
\hline
$a$  & $\alpha+1$ & $a$ & $\alpha$ &  $\alpha$  &  $a$  \\
\hline
$b$  & $\alpha-1$ & $b$ & $b$      &  $\alpha-1$  &  $a-1$  \\
\hline
$c$  & $\alpha-2h+1$ & $\alpha-2h-1$ & $\alpha-2h-1$  &  $\alpha-2h-1$  &  $\alpha-2h-1$  \\
\hline
\end{tabular}
\caption{}
\label{table:131-18}
\end{table}
\endgroup

The latter four columns are included in the second column.

For the first column, we construct an hv-basis and verify $\partial^2 x_{(-h,\alpha-2h-1)}\neq 0$, this contradiction rules out this possibility.

For the second column, \textit{if $b = a-1$},  we construct an hv-basis under the constraints of Maslov grading and $H_*(Y_p), H_*(Z_m)$. The resulting hv-basis is a staircase plus two boxes, as shown in Figure \ref{figure:(13131) (1,2,2,1)-1}.

\begingroup
\setlength{\intextsep}{1pt}
\begin{figure}[H]
	\centering
    \setlength{\abovecaptionskip}{1pt}
    \begin{minipage}[c]{0.4\textwidth}
        \centering
	   \includegraphics[scale = 0.7]{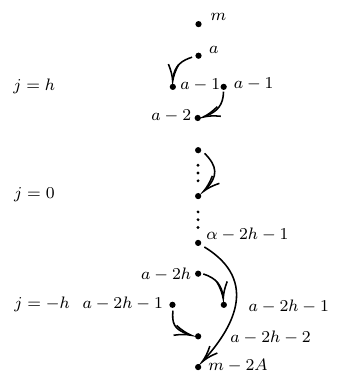}
    \end{minipage}
    \begin{minipage}[c]{0.5\textwidth}
        \centering
	   \includegraphics[scale = 0.8]{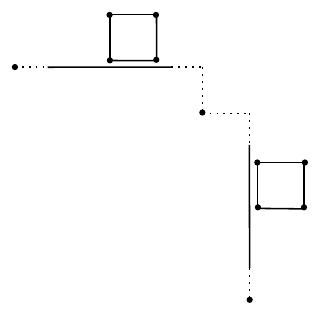}
    \end{minipage}
    \caption{}
    \label{figure:(13131) (1,2,2,1)-1}
\end{figure}
\endgroup

\textit{If $b\neq a-1$}, we construct an hv-basis under the constraints of Maslov grading and $H_*(Y_p), H_*(Z_m)$, then we check that either $\partial^2 x_{(h+1,a)}\neq 0$ or $\partial^2 x_{(h,b)}\neq 0$, which rules out this possibility.
\newpage

\subsection{Cases of special genera}\label{subsection:5.2 Cases of special genera}

In this section, we present the classification results for hv-bases of height 4 knots $K$ where $g(K) = 2$ and $g(K) = h+1$.

Table \ref{table:131-19} illustrates the case $g(K) = 2$. The hv-basis in this case can be viewed as a truncation of the general case. In contrast to the proof for the general case, we do not need to handle the generators at the middle height. Instead, we start from the height $3$ and trace downward to height $-2$. The cancellation lemma (cf.\ Lemma \ref{lemma: cancellation lemma}) and the conditions $H_*(A_{\pm 1})\cong \mathbb{F}^3$ are applied in a similar way.

\begin{table}[H]
\centering
\begin{tabular}{|c|c|c|}
\hline

\makecell{$rk H_*(X_p)$ \\ $p = 3,2,1$ } & $\widehat{HFK}(K,i), i = 2,1,0$ & hv-basis \\

\hline

(0,1,1) & \makecell{$\mathbb{F}_{(a)}$ \\ $\mathbb{F}_{(a-1)}\oplus\mathbb{F}_{(a-1)}$ \\ $\mathbb{F}_{(a-2)}\oplus \mathbb{F}_{(a-2)}\oplus\mathbb{F}_{(e)}$ }  &{\includegraphics[scale = 0.9, valign = c]{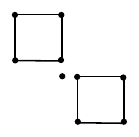}}\\

\hline

\multirow{2}{*}{(0,1,2)} & \makecell{ $\mathbb{F}_{(a)}$ \\ $\mathbb{F}_{(a-1)}\oplus\mathbb{F}_{(a-1)}\oplus\mathbb{F}_{(m-1)}$  \\ $\mathbb{F}_{(a-2)}\oplus\mathbb{F}_{(a-2)}\oplus\mathbb{F}_{(m-2)}$ } &{\includegraphics[scale = 0.9, valign = c]{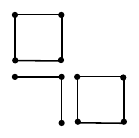}}\\

\cline{2-3}

& \makecell{ $\mathbb{F}_{(a)}$ \\ $\mathbb{F}_{(a-1)}$  \\ $\mathbb{F}_{(a-3)}\oplus\mathbb{F}_{(a-2)}\oplus\mathbb{F}_{(a-2)}$ } & \includegraphics[scale = 0.9, valign = c]{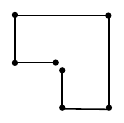} \\

\hline	
 
\multirow{2}{*}{(0,2,1)} &  \makecell{ $\mathbb{F}_{(a-1)}\oplus\mathbb{F}_{(a)}$ \\ $\mathbb{F}_{(a-2)}$  \\ $\mathbb{F}_{(a-3)}\oplus\mathbb{F}_{(a-3)}\oplus\mathbb{F}_{(a-2)}$ } &\includegraphics[scale = 0.90, valign = c]{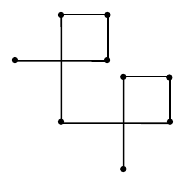}\\

\cline{2-3}

& \makecell{ $\mathbb{F}_{(a)}\oplus\mathbb{F}_{(m-1)}$ \\ $\mathbb{F}_{(a-1)}\oplus\mathbb{F}_{(a-1)}\oplus\mathbb{F}_{(m-2)}$  \\ $\mathbb{F}_{(a-2)}\oplus\mathbb{F}_{(a-2)}\oplus\mathbb{F}_{(m-3)}$ } & {\includegraphics[scale = 0.9, valign = c]{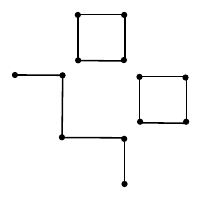}} \\

\hline

(0,2,2) & \makecell{ $\mathbb{F}_{(a)}\oplus\mathbb{F}_{(b)}$ \\ $\mathbb{F}_{(b-1)}\oplus\mathbb{F}_{(b-1)}$  \\ $\mathbb{F}_{(a-3)}\oplus\mathbb{F}_{(b-2)}\oplus\mathbb{F}_{(b-2)}$ } &{\includegraphics[scale = 0.9, valign = c]{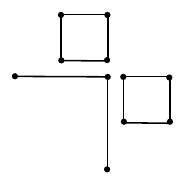}} \\
 
\hline
\end{tabular}
\caption{}
\label{table:131-19}
\end{table}

\newpage

Table \ref{table:131-20} and table \ref{table:131-21}
present classification results of hv-bases for a height 4 knot $K$ of the $(1,3,1,3,1)$ type with parameter $h$ such that $g(K) = h+1$.

The results under this condition can be viewed as a truncation of the general case at height $h+1$. The proof is essentially identical to that of the general case. 
 
\begin{table}[H]
\centering
\begin{tabular}{|c|c|c|}
\hline

\makecell{$rk H_*(X_p)$ \\ $p = h+2,h+1,h,h-1$ } & \makecell{$\widehat{HFK}(K,i)$\\ $i = h+1,h,h-1$} & hv-basis \\

\hline

(0,1,1,0)& \makecell{$\mathbb{F}_{(a)}$\\ $\mathbb{F}_{(a-1)}\oplus\mathbb{F}_{(a-1)}$ \\ $\mathbb{F}_{(a-2)}$ } & {\includegraphics[scale = 0.85, valign = c]{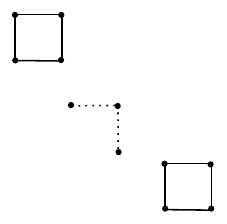}} \\

\hline
(0,1,1,1)& \makecell{$\mathbb{F}_{(a)}$\\ $\mathbb{F}_{(a-1)}\oplus\mathbb{F}_{(a-1)}$ \\ $\mathbb{F}_{(c)}\oplus\mathbb{F}_{(a-2)}$ } &{\includegraphics[scale = 0.85, valign = c]{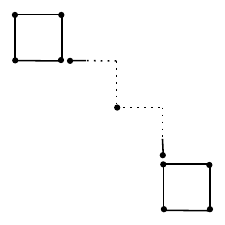}} \\

\hline

\multirow{2}{*}{(0,1,2,0)}& \makecell{$\mathbb{F}_{(a)}$ \\ $\mathbb{F}_{(a-1)}$ \\ $\mathbb{F}_{(a-3)}\oplus\mathbb{F}_{(a-2)}$ }&{\includegraphics[scale = 0.85, valign = c]{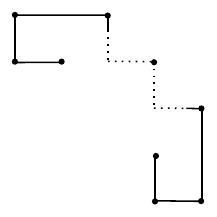}}  \\

\cline{2-3}

& \makecell{$\mathbb{F}_{(a)}$ \\ $\mathbb{F}_{(a-1)}\oplus\mathbb{F}_{(a-1)}\oplus\mathbb{F}_{(c)}$ \\ $\mathbb{F}_{(a-2)}\oplus\mathbb{F}_{(c-1)}$ } & {\includegraphics[scale = 0.85, valign = c]{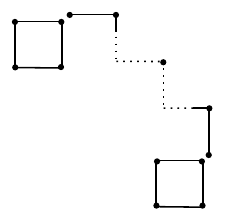}}  \\

\hline
\multirow{2}{*}{(0,1,2,1)}& \makecell{$\mathbb{F}_{(a)}$ \\ $\mathbb{F}_{(a-1)}$ \\ $\mathbb{F}_{(a-2)}$} & {\includegraphics[scale = 0.85, valign = c]{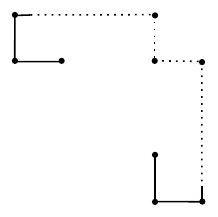}} \\

\cline{2-3}

& \makecell{$\mathbb{F}_{(a)}$ \\ $\mathbb{F}_{(a-1)}\oplus\mathbb{F}_{(a-1)}\oplus\mathbb{F}_{(b)}$ \\ $\mathbb{F}_{(a-2)}$ } & {\includegraphics[scale = 0.85, valign = c]{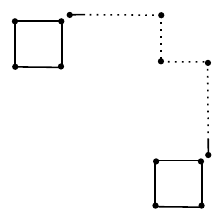}} \\

\hline
\end{tabular}
\caption{}
\label{table:131-20}
\end{table}

\newpage

\begin{table}[H]
\centering
\begin{tabular}{|c|c|c|}
\hline

\multirow{2}{*}{(0,2,1,0)}& \makecell{$\mathbb{F}_{(a-1)}\oplus\mathbb{F}_{(a)}$ \\ $\mathbb{F}_{(a-2)}$ \\ $\mathbb{F}_{(a-3)}$} & {\includegraphics[scale = 0.75, valign = c]{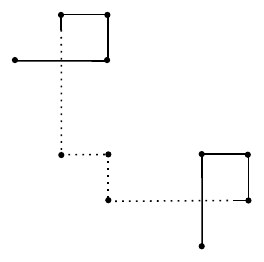}} \\

\cline{2-3}

& \makecell{ $\mathbb{F}_{(a)}\oplus\mathbb{F}_{(b)}$ \\ $\mathbb{F}_{(a-1)}\oplus\mathbb{F}_{(a-1)}\oplus\mathbb{F}_{(b-1)}$  \\ $\mathbb{F}_{(a-2)}$ } & {\includegraphics[scale = 0.75, valign = c]{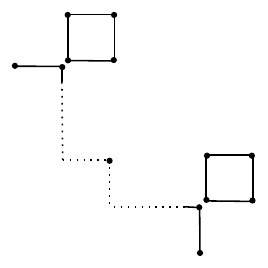}} \\

\hline

\multirow{2}{*}{(0,2,1,1)}& \makecell{$\mathbb{F}_{(a-1)}\oplus\mathbb{F}_{(a)}$ \\ $\mathbb{F}_{(a-2)}$ \\ $\mathbb{F}_{(a-3)}\oplus\mathbb{F}_{(a-2)}$} & {\includegraphics[scale = 0.75, valign = c]{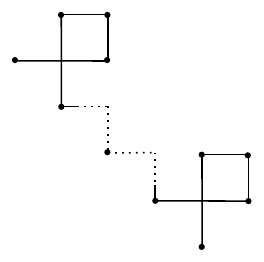}} 
\\

\cline{2-3}

& \makecell{ $\mathbb{F}_{(a)}\oplus\mathbb{F}_{(b)}$ \\ $\mathbb{F}_{(a-1)}\oplus\mathbb{F}_{(b-1)}\oplus\mathbb{F}_{(b-1)}$  \\ $\mathbb{F}_{(a-2)}\oplus\mathbb{F}_{(b-2)}$ } & {\includegraphics[scale = 0.75, valign = c]{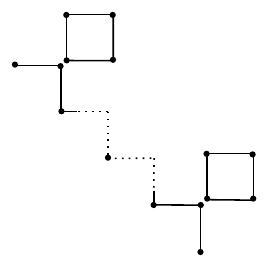}} 
\\

\hline

(0,2,2,0)&  \makecell{ $\mathbb{F}_{(a)}\oplus\mathbb{F}_{(b)}$ \\ $\mathbb{F}_{(b-1)}\oplus\mathbb{F}_{(b-1)}$  \\ $\mathbb{F}_{(a-3)}\oplus\mathbb{F}_{(b-2)}$ } & {\includegraphics[scale = 0.75, valign = c]{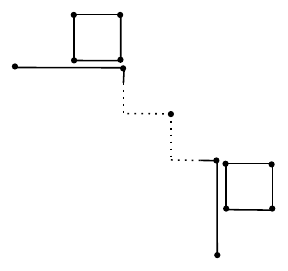}}  \\

\hline

(0,2,2,1)& \makecell{ $\mathbb{F}_{(a)}\oplus\mathbb{F}_{(b)}$ \\ $\mathbb{F}_{(a-1)}\oplus\mathbb{F}_{(a-1)}$  \\ $\mathbb{F}_{(a-2)}$ } & {\includegraphics[scale = 0.75, valign = c]{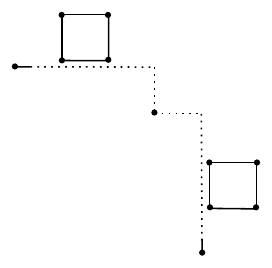}} \\

\hline
\end{tabular}
\caption{}
\label{table:131-21}
\end{table}

\bibliographystyle{alpha}
\bibliography{paper.bib}

\section*{Tool Use}
ChatGPT (OpenAI) was used interactively to assist with English-language editing, LaTeX typesetting, and the preparation and revision of TikZ figures. The author reviewed the resulting text and figures, checked the mathematical statements and proofs, and takes full responsibility for the final manuscript.

\end{document}